\documentclass[12pt]{amsart}
\usepackage{graphicx}
\usepackage[all]{xy}

\newtheoremstyle{slthm}
  {9pt}
  {5pt}
  {\slshape}
  {}
  {\bfseries}
  {.}
  {.5em}
  {\thmname{#1} \thmnumber{#2}{\rm \thmnote{ (#3)}}}

\newtheoremstyle{prcl}
  {9pt}
  {5pt}
  {\slshape}
  {}
  {\bfseries}
  {.}
  {.5em}
  {\thmname{#3} \thmnumber{ #2}}

\theoremstyle{slthm}
\newtheorem{theorem}{Theorem}[section]
\newtheorem{lemma}[theorem]{Lemma}
\newtheorem{lemmas}[theorem]{Lemmas}
\newtheorem{proposition}[theorem]{Proposition}
\newtheorem{corollary}[theorem]{Corollary}

\theoremstyle{definition}
\newtheorem{definition}[theorem]{Definition}

\newtheorem{defrmk}[theorem]{Definitions and Remarks}

\newtheorem{examples}[theorem]{Examples}

\theoremstyle{remark}

\newtheorem{notation}[theorem]{Notation}

\newtheorem{remark}[theorem]{Remark}
\newtheorem{remarks}[theorem]{Remarks}

\numberwithin{equation}{section}

\allowdisplaybreaks[2]

\newcommand{\A}{\mathcal{A}}
\newcommand{\B}{\mathcal{B}}
\newcommand{\C}{\mathcal{C}}
\newcommand{\D}{\mathcal{D}}

\newcommand{\F}{\mathcal{F}}
\newcommand{\G}{\mathcal{G}}
\newcommand{\I}{\mathcal{I}}

\renewcommand{\L}{\mathcal{L}}   %

\renewcommand{\O}{\mathcal{O}}   %
\renewcommand{\P}{\mathcal{P}} %

\newcommand{\R}{\mathcal{R}}
\renewcommand{\S}{\mathcal{S}}   %
\newcommand{\T}{\mathcal{T}}

\newcommand{\NN}{\mathbb{N}}

\newcommand{\QQ}{\mathbb{Q}}
\newcommand{\RR}{\mathbb{R}}

\newcommand{\PP}{\mathbb{P}}
\newcommand{\VV}{\mathbb{V}}

\renewcommand{\bar}[1]{\overline{#1}} %
\newcommand{\tld}[1]{\widetilde{#1}}
\newcommand{\lb}{\langle} 
\newcommand{\rb}{\rangle} 

\newcommand{\ps}[1]{[\hspace{-0.3ex}[#1]\hspace{-0.3ex}]}

\DeclareMathOperator{\rank}{rank}

\DeclareMathOperator{\adj}{adj}
\DeclareMathOperator{\lcm}{lcm}

\DeclareMathOperator{\dom}{dom}

\DeclareMathOperator{\im}{im}
\DeclareMathOperator{\Graph}{graph}

\newcommand{\Restr}[1]{\big|_{#1}}

\DeclareMathOperator{\isom}{isom}

\newcommand{\PD}[3]{\frac{\partial^{#1}#2}{\partial {#3}^{#1}}}

\newcommand{\PDn}[3]{\frac{\partial^{|#1|}#2}{\partial {#3}^{#1}}}

\newcommand{\PDmix}[6]{\frac{\partial^{#1}#2}
   {\partial{#3}^{#4}\partial{#5}^{#6}}}

\DeclareMathOperator{\sign}{sign}
\DeclareMathOperator{\ord}{ord}

\DeclareMathOperator{\Int}{int}
\DeclareMathOperator{\cl}{cl}
\DeclareMathOperator{\bd}{bd}
\DeclareMathOperator{\fr}{fr}

\DeclareMathOperator{\Th}{Th}

\DeclareMathOperator{\Cen}{Cen}

\DeclareMathOperator{\dist}{dist}

\DeclareMathOperator{\IF}{IF}

\DeclareMathOperator{\Div}{div}
\DeclareMathOperator{\name}{name}

\DeclareMathOperator{\Strat}{Strat}

\DeclareMathOperator{\row}{row}
\DeclareMathOperator{\col}{col}
\DeclareMathOperator{\supp}{supp}
\DeclareMathOperator{\id}{id}

\DeclareMathOperator{\Stratify}{Strat}

\DeclareMathOperator{\Loc}{Loc}

\DeclareMathOperator{\Index}{index}
\DeclareMathOperator{\ext}{ext}
\DeclareMathOperator{\BOX}{Box}
\DeclareMathOperator{\Bus}{Bus}

\title{Characterizing Decidability in a Quasianalytic Setting}

\author[D. J. Miller]{Daniel J. Miller}
\address{Emporia State University, Department of Mathematics, Computer Science and Economics, 1200 Commercial Street, Campus Box 4027, Emporia, KS 66801, U.S.A.}
\email{dmille10@emporia.edu}


\begin{document}

\begin{abstract}
Let $\RR_{\S}$ denote the expansion of the real ordered field by a family of real-valued functions $\S$, where each function in $\S$ is defined on a compact box and is a member of some quasianalytic class which is closed under the operations of function composition, division by variables, and implicitly defined functions.  It is shown that the first order theory of $\RR_{\S}$ is decidable if and only if two oracles, called the approximation and precision oracles for $\S$, are decidable.  Loosely stated, the approximation oracle for $\S$ allows one to approximate any partial derivative of any function in $\S$ to within any given error, and the precision oracle for $\S$ allows one to decide when a manifold $M\subseteq\RR^n$ is contained in a coordinate hyperplane $\{x\in\RR^n : x_i = 0\}$ when one is given $i\in\{1,\ldots,n\}$ and a system of equations which defines $M$ nonsingularly, where the functions occurring in the equations are rational polynomials of the coordinate variables $x = (x_1,\ldots,x_n)$ and the partial derivatives of the functions in $\S$.  A key component of the proof is the development of a local resolution of singularities procedure which is effective in the approximation and precision oracles for $\S$, and in the course of proving our main theorem, numerous theorems about the model theory of such structures $\RR_{\S}$ are also proven.
\end{abstract}

\maketitle

\setcounter{tocdepth}{1}
{\small \tableofcontents}

\section*{Introduction}\label{s:intro}

Analytic geometry --- in the most classical sense of the word, with origins in the work of Descartes --- studies subsets of Euclidean space defined using equations and inequalities.
For example, $x=y^2$ defines a certain parabola in $\RR^2$.  If, in addition to equations and inequalities, one allows the use of boolean logical connectives and quantification over the reals, one can say meaningful statements about such sets.  For example, $\forall x\exists y(x = y^2)$ says that every real number has a real square root, and $\forall x(x\geq 0 \rightarrow \exists y(x = y^2))$ says that every nonnegative real number has a real square root.  The first statement is clearly false and the second is true.  It is natural to ask whether there is an algorithmic way of determining which statements are true and which are false.  The answer to this question must surely depend on what information we know about the types of functions occurring in the equations and inequalities used to construct the statements.

To make this question more precise, we use the framework of first-order logic.  We fix a family of functions $\C$ which we call a ``quasianalytic IF-system'' (``IF'' stands for ``implicit function''): roughly speaking, this means that $\C$ is a family of rings of real-valued $C^\infty$ functions defined on the sets $[-r,r] = [-r_1,r_1]\times\cdots\times[-r_n,r_n]$ (for each $n\in\NN$ and $r = (r_1,\ldots,r_n)\in\QQ^{n}_{+}$, where $\QQ_{+}=\QQ\cap(0,+\infty)$), that $\C$ has injective Taylor maps, and that $\C$ is closed under the operations of function composition, division by variables, and implicitly defined functions. Such quasianalytic classes have been used by numerous authors when studying resolution of singularities and its applications --- such as by Bierstone and Milman \cite{BM04}, and by Rolin, Speissegger and Wilkie \cite{RSW} --- but unlike these authors, we do not assume that $\RR\subseteq\C$.  The prototypical example of a quasianalytic IF-system is the system of all real analytic functions, restricted to such compact boxes $[-r,r]$ in the interiors of their domains.  Now, let $\S$ denote a computably indexed family of functions $f:[-r,r]\to\RR$ in $\C$, where the choice of $n\in\NN$ and $r\in\QQ_{+}^{n}$ depend in a computable way on the choice of $f$ in $\S$.  Let $\RR_{\S}$ denote the expansion of the real ordered field by the family of functions $\{\widehat{f}\,\,\}_{f\in\S}$, where $\widehat{f}:\RR^n\to\RR$ is defined from $f:[-r,r]\to\RR$ by
\[
\widehat{f}(x) = \begin{cases}
f(x),   & \text{$x\in[-r,r]$,} \\
0,      & \text{$x\in\RR^n\setminus[-r,r]$,}
\end{cases}
\]
and let $\L_{\S}$ denote that language of the structure $\RR_{\S}$.  The {\bf theory} of $\RR_{\S}$, denoted by $\Th(\RR_{\S})$, is the set of first-order $\L_{\S}$-sentences which are true for the structure $\RR_{\S}$.  To say that the theory of $\RR_{\S}$ is {\bf decidable} means that there is an algorithm which allows one to determine which $\L_{\S}$-sentences are in $\Th(\RR_{\S})$ and which are not.

When studying questions of decidability, it is useful to work with a relative notion of computability by allowing the algorithms to call upon oracles.  An oracle is a theoretical construct which acts as an ``algorithmic black box'': an oracle functions like an algorithm, except that we do not require the oracle to be computable by an actual algorithm.  By allowing algorithms to call upon an oracle $\O$, just as an algorithm can call upon an actual subroutine, we obtain a broader notion of computability, what is called ``computable relative to $\O$''.  The purpose of using oracles when discussing decidability of theories is to isolate exactly what needs to be computed to decide the theory.

This paper gives an algorithm which decides the theory of $\RR_{\S}$ relative to two oracles, called the approximation and precision oracles for $\S$.  Loosely stated, the approximation oracle for $\S$ allows one to approximate any partial derivative of any function in $\S$ to within any given error, and the precision oracle for $\S$ allows one to decide when a manifold $M\subseteq\RR^n$ is contained in a coordinate hyperplane $\{x\in\RR^n : x_i = 0\}$ when one is given $i\in\{1,\ldots,n\}$ and a system of equations which defines $M$ nonsingularly, where the functions occurring in the equations are rational polynomials of the coordinate variables $x = (x_1,\ldots,x_n)$ and the partial derivatives of the functions in $\S$.  Not only do these two oracles decide the theory of $\RR_{\S}$, but it is easy to see that the theory of $\RR_{\S}$ decides these two oracles (this will be more apparent once the precise definitions of the oracles are given in Section \ref{s:oracles}), so this gives the following result, which is the main theorem of the paper.

\begin{theorem}[Characterization of Decidability]\label{introThm:main}
The theory of $\RR_{\S}$ is decidable if and only if the approximation and precision oracles for $\S$ are decidable.
\end{theorem}

Very often, it is easy to decide the approximation oracle for $\S$, but the precision oracle for $\S$ can be very hard to decide, so the later oracle gets at the heart of the difficulty of deciding the theory of such a structure $\RR_{\S}$.

In the course of proving Theorem \ref{introThm:main}, we give a local resolution of singularities procedure which is effective in the approximation and precision oracles for $\S$, and we give numerous theorems pertaining to the model theory of such structures $\RR_{\S}$, of which we now name a few.  These theorems do not deal with any computability issues, so we now drop the assumption that $\S$ is computably indexed.

To state the theorems, we need some additional terminology.  We use the word {\bf definable} to mean first-order definable with parameters, and we use the word {\bf $0$-definable} to mean first-order definable without parameters.  The {\bf natural stratification} of a compact box $\prod_{i=1}^{n}[a_i,b_i]$ is the set of all Cartesian products of the form $C = \prod_{i=1}^{n} C_i$, where $C_i \in\{\{a_i\}, (a_i,b_i),\{b_i\}\}$ for each $i\in\{1,\ldots,n\}$.  A {\bf rational box} is a Cartesian product of intervals with endpoints in $\QQ\cup\{\pm\infty\}$.  An {\bf $\S$-polynomial function} is a function $f = (f_1,\ldots,f_m):A\to\RR^m$ defined on a rational box $A\subseteq\RR^n$ such that $f_1,\ldots,f_m$ are rational polynomials of the coordinate variables $x_1,\ldots,x_n$ and the partial derivatives of functions in $\S$.  An {\bf $\S$-manifold} is a set $M = \{x\in A : f(x) = 0\}$, where $A\subseteq\RR^n$ is a rational box and $f:A\to\RR^{n-d}$ is an $\S$-polynomial function such that the Jacobian matrix $\frac{\partial f}{\partial(x_{\lambda(1)},\ldots,x_{\lambda(n-d)})}$ is nonsingular on $M$, for some increasing function $\lambda:\{1,\ldots,n-d\}\to\{1,\ldots,n\}$. (These last two definitions are somewhat loosely stated versions of the more precise Definitions \ref{def:Spoly} and \ref{def:Smanifold}.)

Theorem \ref{introThm:main} will be deduced from an effective version of the following parameterization theorem (see Theorem \ref{thm:SparamDefinable}).

\begin{theorem}[Parameterization]\label{introThm:param}
Let $A\subseteq\RR^m$ be $0$-definable in $\RR_{\S}$.  Then there exist finite families of $\C$-analytic immersions $\{F^{(j)} : U^{(j)}\to\RR^m\}_{j\in J}$, bounded open rational boxes $\{B^{(j)}\}_{j\in J}$, and maps $\{\psi^{(j)}\}_{j\in J}$ such that
\[
A = \bigcup_{j\in J} \psi^{(j)}\circ F^{(j)}(B^{(j)}),
\]
where for each $j\in J$, $\psi^{(j)}(x_1,\ldots,x_m) = (x_{1}^{\sigma_{1}(j)},\ldots,x_{m}^{\sigma_{m}(j)})$ for some $\sigma_{1}(j),\ldots,\sigma_{m}(j)\in\{-1,1\}$, $U^{(j)}$ is open in $\RR^{d(j)}$ for some $d(j)\leq m$, the closure of $B^{(j)}$ is contained in $U^{(j)}$, and $F^{(j)}(B^{(j)})$ is contained in the natural domain of $\psi^{(j)}$.  Furthermore, if $I\subseteq\{1,\ldots,m\}$ and $\{x_i : x\in A\}$ is bounded for each $i\in I$, then we can take $\sigma_{i}(j) = 1$ for each $j\in J$ and $i\in I$.  Finally, for each $j\in J$ and each set $C$ in the natural stratification of the closure of $B^{(j)}$, there exist an integer $p\geq d(j)$, an $\S$-manifold $M\subseteq\RR^{p}$, and a coordinate projection $\Pi:\RR^{p}\to\RR^{d(j)+m}$ such that $\Pi$ defines a $\C$-analytic isomorphism from $M$ onto the graph of $F^{(j)}\Restr{C}$.
\end{theorem}

Even though this theorem is stated for the $0$-definable sets of $\RR_{\S}$, by applying it to $\RR_{\S\cup\RR}$ instead of $\RR_{\S}$ one sees that the definable sets of $\RR_{\S}$ also have a parameterization theorem.  (There is a technical point to note here.  In Theorem \ref{introThm:param} we must assume that $\S\subseteq\C$ for some quasianalytic IF-system $\C$, but we do not assume that $\RR\subseteq\C$, so it may be that $\S\cup\RR \not \subseteq \C$.  However, in Section \ref{s:IFconstants} we show that there exists a smallest quasianalytic IF-system $\C(\RR)$ containing $\C$ and $\RR$, so Theorem \ref{introThm:param} can be safely applied to $\RR_{\S\cup\RR}$ since $\S\cup\RR\subseteq\C(\RR)$.)  Note that this parameterization theorem implies that each definable set $A$ has dimension, with $\dim(A) = \max\{d(j) : j\in J\}$, where $J$ and each $d(j)$ are as in the theorem.  Also, the following is an immediate corollary of the last sentence of the parameterization theorem.

\begin{theorem}\label{introThm:MC}
The structure $\RR_{\Delta(\S)}$ is model complete, where $\Delta(\S)$ is the family of all partial derivatives of functions in $\S$.
\end{theorem}

In o-minimality, it is very common to prove model completeness theorems by showing that definable sets are finite unions of projections of quantifier-free definable manifolds, with each projection being an immersion.  Theorem \ref{introThm:param} gives such a representation of the set $A$:  to see this, for each $j\in J$ write $p(j)$, $M^{(j)}$ and $\Pi^{(j)}:\RR^{p(j)}\to\RR^{d(j)+m}$ for the data $p$, $M$ and $\Pi:\RR^p\to\RR^{d(j)+m}$ given by the theorem for $C = B^{(j)}$, write $\Pi^{(j)} = \Pi^{(j)}_{\dom}\times\Pi^{(j)}_{\im}$ with $\Pi^{(j)}_{\dom}(M^{(j)}) = B^{(j)}$ and $\Pi^{(j)}_{\im}(M^{(j)}) = F^{(j)}(B^{(j)})$, and write $N^{(j)} = \{(x,y)\in M^{(j)}\times\RR^m : y = \psi^{(j)}\circ\Pi^{(j)}_{\im}(x)\}$; then $N^{(j)}$ projects onto $\psi^{(j)}\circ F^{(j)}(B^{(j)})$ via $(x,y)\mapsto y$, and this projection is an immersion.  However, Theorem \ref{introThm:param} says more than this since each function $F^{(j)}$ is defined in a neighborhood of the closure of $B^{(j)}$, not just on $B^{(j)}$ itself, and because of this it is easy to show that Hironaka's uniformization theorem holds for the $0$-definable sets of $\RR_{\S}$.  We explain this point further after some additional definitions.

Consider an expansion $\R$ of the real field.  We say that $\R$ is {\bf polynomially bounded} if for every function $f:\RR\to\RR$ definable in $\R$, there exist $M,N > 0$ such that $|f(t)| \leq t^N$ for all $t > M$.  We say that $\R$ has {\bf $C^\infty$-uniformization} if for every compact $0$-definable set $K\subseteq\RR^m$ there exists a $0$-definable surjective $C^\infty$-map $f:M\to K$, where $M\subseteq\RR^n$ is a compact $C^\infty$ manifold of the same dimension as $K$.  Two first-order structures with the same universe, but possibly different signatures, are said to be {\bf definably equivalent} if they have the same $0$-definable sets.

By taking $\S=\C$ in Theorem \ref{introThm:OminPolyUnif} below, the following two theorems can be seen to be converses of one another.  Collectively, these two theorems specify the model-theoretic properties which characterize, up to definable equivalence, the types of structures studied in this paper.

\begin{theorem}\label{introThm:OminPolyUnif}
The structure $\RR_{\S}$ is a polynomially bounded o-minimal expansion of the real field with $C^\infty$-uniformization.
\end{theorem}

\begin{theorem}\label{introThm:OminPolyUnifChar}
If $\R$ is a polynomially bounded o-minimal expansion of the real field with $C^\infty$-uniformization, then $\R$ is definably equivalent to $\RR_{\C}$ for some quasianalytic IF-system $\C$.
\end{theorem}

\subsection{Background and Motivation}
The motivation for this work is linked closely to the historical development of o-minimality, which is a subfield of model theory generalizing semialgebraic and subanalytic geometry.   The starting point is Tarski's theorem \cite{Tarski} stating that the theory of the real field is decidable.  (This theorem fits into the framework of this paper by taking $\S$ to be empty.)  He proved this by showing, in an effective manner, that the real ordered field has quantifier elimination (namely, that images of semialgebraic sets under coordinate projections are semialgebraic).  In the same paper, he also asked a very influential question: is the theory of the real exponential field decidable?  (The real exponential field is the expansion of the real field by the exponential function $\RR\to\RR:x\mapsto e^x$.)  In the early days of o-minimality, Van den Dries \cite{vdD84} made a very convincing argument that, from a practical point of view, this is the wrong question to ask about the real exponential field.  His argument was that in order to decide just the quantifier-free theory of this structure, one would have to answer many difficult open questions of transcendental number theory concerning the real exponential function, and for this reason, we should concentrate on studying the geometric and model-theoretic properties of the real exponential field, not the decidability of its theory.  And since that time, o-minimality has prospered into a thriving field of study which, for the most part, has concentrated on geometric and model-theoretic questions, but not on decidability.

However, in contrast to this stands the work of Macintyre and Wilkie \cite{MW}, who proved that the theory of the real exponential field is decidable if Schanuel's conjecture is true (Schanuel's conjecture is an open problem of transcendental number theory).  Simply put, this theorem means that the number-theoretic questions of concern to Van den Dries are really the \emph{only} stumbling block one must overcome to decide the theory of the real exponential field.  But since a proof of Schanuel's conjecture (or a suitable replacement if the conjecture is false) appears to be no easy feat, this is still a significant stumbling block, and the following question is left open by Macintyre and Wilkie's work:
\begin{enumerate}
\item[($*$)]
Does there exist any  o-minimal expansion of the real field with a decidable theory which defines a transcendental function?
\end{enumerate}
It is not hard to use Tarski's theorem to show that an expansion of the real field by a transcendental constant has a decidable theory, so in Question ($*$), it should be understood that the transcendental function is not $0$-ary.  This question is much less ambitious than the question addressed by Macintyre and Wilkie.  The question is not whether there is an interesting, natural, or useful o-minimal expansion of the real field with a decidable theory which defines a transcendental function.  The word ``any'' is literally interpreted, so this includes expansions of the real field by functions which do naturally appear in ``ordinary'' mathematics.

In \cite{DJMconstrDec}, which is a sequel to this paper, the author uses the characterization of decidability given here to answer the question ($*$) in the affirmative.  In fact, it is shown that in a certain topological sense, there are many o-minimal expansions of the real field with decidable theories which define transcendental functions.  The strategy of proof employed here and its sequel began with the following thought process: Schanuel's conjecture is a statement claiming that the exponential function is very generic with respect to rational polynomial equations, and in the isolated place in \cite{MW} where Macintyre and Wilkie actually use Schanuel's conjecture, this genericity property is used, but no other special properties of the exponential function are used in a critical way.  Therefore, if one could specially construct suitably generic functions, one should be able to adapt Macintyre and Wilkie's use of Schanuel's conjecture to obtain decidability results.  However, Macintyre and Wilkie's proof, overall, relies on Wilkie's proof that the real exponential field is model complete \cite{W96}, and this proof \emph{does} heavily rely on special properties of the exponential function (particularly, that it is Pfaffian).  Thus, in order to have complete freedom to construct sufficiently generic functions to obtain decidability results, Wilkie's model completeness proof should be replaced by a much more general model completeness construction which does not require that the functions in $\S$ satisfy any special algebraic or differential equations.  This general model completeness construction is the content of this paper, and the construction of generic families of functions $\S$ for which $\Th(\RR_{\S})$ is decidable is the main content of its sequel \cite{DJMconstrDec}.  The model completeness construction given here is a significantly revised adaptation of the construction of Rolin, Speissegger, and Wilkie \cite{RSW} and has been turned into a decision procedure for the theory of $\RR_{\S}$, relative to our two oracles for $\S$.

While this paper was being written, Jones and Servi \cite{JonesServi} also answered the question ($*$) in the affirmative by proving that the expansion of the real field by a power function $(0,+\infty)\to\RR:x\mapsto x^\alpha$ has a decidable theory if $\alpha$ is a computable real which is not $0$-definable in the real exponential field.  Both groups of authors answer ($*$) in a similar way, by expanding the real field by function(s) constructed in an artificial way so as to satisfy certain computability properties and also certain genericity properties.  However, neither work subsumes the other since they are about different types of o-minimal structures, and the techniques used in their proofs are quite different.  Jones and Servi's proof uses model-theoretic techniques, in that it relies on giving computable axiomatizations of theories.  And, in addition to adapting Macintyre and Wilkie's use of Schanuel's conjecture, Jones and Servi's proof also adapts many parts of Wilkie's model completeness construction \cite{W96} and relies on a recently proven Schanuel-like property of certain power functions due to Bays, Kirby, and Wilkie \cite{BKW08}.  The proof given here and its sequel relies on no such number-theoretic theorems, and it is based on a characterization of decidability, rather than just a sufficient condition for decidability.  This characterization of decidability is of interest in and of itself, probably more so than the answer to ($*$) given in the sequel.  The characterization is proven by a  geometric decision procedure, without recourse to enumeration of theories, and it could seemingly just as well be used to understand undecidability or relative decidability issues as it could be used to understand decidability.  Also, in the quasianalytic setting considered here, it gives some perspective on the relationship between computable analysis and effective model theory:  to conduct computable analysis one uses the approximation oracle for $\S$, but to decide the theory of $\RR_{\S}$ one must additionally use the precision oracle for $\S$.

\subsection{Notation}
We now fix some notation to be used throughout the paper.
\begin{enumerate}{\setlength{\itemsep}{5pt}
\item\emph{Miscellaneous}:

A {\bf family} is an indexed collection $C = \{c_i\}_{i\in I}$, where $I$ is an index set.  The map $I\to\{c_i : i\in I\} : i\mapsto c_i$ need not be injective.  We use standard set notation with families as if they were sets, such as $c\in C$ and $C\subseteq D$.

For any map $f:A\to B$,
\[
\dom(f) = A
\quad\text{and}\quad
\im(f) = f(A).
\]
For any $x\in\RR$, define the {\bf sign of $x$} by
\[
\sign(x) = \begin{cases}
1,  & \text{if $x > 0$,} \\
0,  & \text{if $x = 0$,} \\
-1, & \text{if $x < 0$.}
\end{cases}
\]
Write
\[
\QQ_+ = \QQ\cap(0,+\infty).
\]
Write $\NN = \{0,1,2,3,\ldots\}$ for the set of natural numbers. Define the {\bf support of $\alpha = (\alpha_1,\ldots,\alpha_n)\in\NN^n$} by
\[
\supp(\alpha) = \{i\in\{1,\ldots,n\} : \alpha_i \neq 0\}
\]
and the {\bf support of $A\subseteq\NN^n$} by
\[
\supp(A) = \bigcup_{\alpha\in A}\supp(\alpha).
\]
Empty sums equal $0$, and empty products equal $1$.
Define
\[
\RR^0 = \{0\},
\]
and for any $n\in\NN$ and $A\subseteq\RR^n$, define $\RR^0\times A = A\times \RR^0 = A$.

\item\emph{Topology}:

For any $A\subseteq\RR^n$, write $\cl(A)$, $\Int(A)$, $\bd(A) = \cl(A)\setminus \Int(A)$, and $\fr(A) = \cl(A)\setminus A$ for the topological closure, interior, boundary, and frontier of $A$ in $\RR^n$, respectively.  If $A\subseteq B\subseteq\RR^n$, the relativization of these concepts to the subspace $B$ are written as $\cl_B(A)$, $\Int_B(A)$, $\bd_B(A)$, and $\fr_B(A)$.

\item\emph{Coordinate Projections}:

Consider an injective map $\lambda:\{1,\ldots,d\}\to\{1,\ldots,n\}$ for integers $0 \leq d \leq n$.  If $x = (x_1,\ldots,x_n)$ are coordinates on $\RR^n$, write
\[
x_\lambda = (x_{\lambda(1)},\ldots,x_{\lambda(d)}),
\]
and define $\Pi_\lambda:\RR^n\to\RR^d$ by
\[
\Pi_\lambda(x) = x_\lambda,
\]
with the understanding that when $d=0$, $\lambda$ is the empty map and  $\Pi_{\emptyset}(x_1,\ldots,x_n) = 0$.  We call any such map $\Pi_\lambda$ a {\bf coordinate projection}.  We call the map $\lambda$ the {\bf name} for the coordinate projection $\Pi_\lambda$.

If $\lambda':\{1,\ldots,n-d\}\to\{1,\ldots,n\}$ is an injective map such that $\im(\lambda)\cup\im(\lambda') = \{1,\ldots,n\}$, we say that $\lambda'$ is {\bf complementary to $\lambda$}, and say that the coordinate projection $\Pi_{\lambda'}:\RR^n\to\RR^{n-d}$ is {\bf complementary to $\Pi_\lambda$}.

An important special case is when $\lambda(i) = i$ for all $i\in\{1,\ldots,d\}$.  In this cases we will also write $\Pi_d$ for $\Pi_\lambda$.  Namely,
\[
\Pi_d(x_1,\ldots,x_n) = (x_1,\ldots,x_d).
\]

If we do not wish to explicitly mention the map $\lambda$, we will use more generic notation to denote coordinate projections, such as $\Pi$, $\Pi'$, $\Pi^{(j)}$, etc., and will simply say in words that the map is a coordinate projection.

\item\emph{Multi-index Notation}:

For any tuples $x=(x_1,\ldots,x_n)$ and $y = (y_1,\ldots,y_n)$ in $\RR^n$, write
\begin{eqnarray*}
x \leq y
    & \text{iff} &
    \text{$x_i\leq y_i$ for all $i\in \{1,\ldots,n\}$,} \\
x < y
    & \text{iff} &
    \text{$x_i < y_i$ for all $i\in \{1,\ldots,n\}$.}
\end{eqnarray*}
For any $\alpha = (\alpha_1,\ldots,\alpha_n)\in\NN^n$, write $\alpha! = \alpha_1!\cdots\alpha_n!$ , $|\alpha| = \alpha_1+\cdots+\alpha_n$, and $x^\alpha = x_{1}^{\alpha_1}\cdots x_{n}^{\alpha_n}$.
For any $m\in\NN$, write
\begin{eqnarray*}
\NN^{n}_{<m} & = & \{\alpha\in\NN^n : |\alpha| < m\}, \\
\NN^{n}_{m} & = & \{\alpha\in\NN^n : |\alpha| = m\}.
\end{eqnarray*}

Consider a map $f=(f_1,\ldots,f_m):U\to\RR^m$, where $U$ is open in $\RR^n$.  If $f$ is $p$-times differentiable, then for any $\alpha=(\alpha_1,\ldots,\alpha_n)\in\NN^n$ with $|\alpha|\leq p$, write
\[
\PDn{\alpha}{f}{x}(x)
=
\left(
\frac{\partial^{|\alpha|} f_1}{\partial x_{1}^{\alpha_1} \cdots \partial x_{n}^{\alpha_n}}(x),
\ldots,
\frac{\partial^{|\alpha|} f_m}{\partial x_{1}^{\alpha_1} \cdots \partial x_{n}^{\alpha_n}}(x)
\right).
\]
If $f$ is differentiable, then for any injective map $\lambda:\{1,\ldots,d\}\to\{1,\ldots,n\}$, write
\[
\PD{}{f}{x_\lambda}(x)
=
\left(
\PD{}{f_i}{x_{\lambda(j)}}(x)
\right)_{(i,j)\in\{1,\ldots,m\}\times\{1,\ldots,d\}},
\]
which is an $m\times d$ matrix.

So, for example, if $g:\RR^m\to\RR^n$ and $f:\RR^n\to\RR^p$ are differentiable maps, and if $x = (x_1,\ldots,x_m)$ and $y = (y_1,\ldots,y_n)$ are coordinates on $\RR^m$ and $\RR^n$, respectively, then the chain rule can be written as
\[
\PD{}{(f\circ g)}{x}(x) = \PD{}{f}{y}(g(x))\PD{}{g}{x}(x).
\]

\item\emph{Other Index Sets}:

For any sets $A$ and $B$, we write $A^B$ to denote either the set of all functions from $B$ into $A$, or the set of all tuples $(a_b)_{b\in B}$ with entries $a_b$ in $A$.  These two meanings for $A^B$ are clearly identifiable, the difference between $B\to A:b\mapsto a_b$ and $(a_b)_{b\in B}$ being purely notational.

Consider $n\in\NN$ and $E\subseteq\{1,\ldots,n\}$, and write $x = (x_1,\ldots,x_n)$ for coordinates on $\RR^n$.  We view $\RR^E$ as the set of tuples of real numbers indexed by $E$, rather than the set of functions from $E$ into $\RR$, and by convention $\RR^\emptyset = \RR^0$ and $\RR^{\{1,\ldots,n\}} = \RR^n$.  Also, when working with index sets contained in $\{1,\ldots,n\}$ and $n$ is clear from context, we will use the superscript $c$ to denote complementation in $\{1,\ldots,n\}$, writing $D^c = \{1,\ldots,n\}\setminus D$ for any $D\subseteq\{1,\ldots,n\}$.

Define
\[
x_E = (x_i)_{i\in E},
\]
and define the projection $\Pi_E:\RR^n\to\RR^E$ by
\[
\Pi_E(x) = x_E.
\]

Generalizing our previous definitions for $\NN^n$, for any $\alpha = (\alpha_i)_{i\in E}\in\NN^E$ write $\alpha! = \prod_{i\in E} \alpha_i$, $|\alpha| = \sum_{i\in E}\alpha_i$, and $x_{E}^{\alpha} = \prod_{i\in E} x_{i}^{\alpha_i}$, and for any $m\in\NN$ write $\NN^{E}_{<m} =  \{\alpha\in\NN^E : |\alpha| < m\}$ and $\NN^{E}_{m} = \{\alpha\in\NN^E : |\alpha| = m\}$.

If $D_1,\ldots,D_k\subseteq\{1,\ldots,n\}$ are disjoint, write
\begin{equation}\label{eq:concat}
(x_{D_1},\ldots,x_{D_k}) = (x_i)_{i\in D_1\cup \cdots\cup D_k},
\end{equation}
and for any sets $A_1\subseteq\RR^{D_1},\ldots,A_k\subseteq\RR^{D_k}$,
define $A_1\times\cdots\times A_k \subseteq\RR^{D_1\cup\cdots\cup D_k}$ by
\[
A_1\times\cdots\times A_k = \{(a_1,\ldots,a_k) : a_1\in A, \ldots, a_k\in A_k\}.
\]
If $D,D'\subseteq\{1,\ldots,n\}$ are disjoint, $U\subseteq\RR^D$, and $g:U\to\RR^{D'}$, then the graph of $g$ is the subset of $\RR^{D\cup D'}$ defined by $\{(x_D,g(x_D)) : x_D\in U\}$.

We comment on this notation.  If $\lambda:\{1,\ldots,d\}\to\{1,\ldots,n\}$ is an injective map such that $\im(\lambda) = E$, then the difference between the projections $\Pi_\lambda:\RR^n\to\RR^d$ and $\Pi_E:\RR^n\to\RR^E$ is purely notational, the difference being whether the tuples in the images of these maps are indexed by $\{1,\ldots,d\}$ or by $E$.  The advantage of $\Pi_E$ over $\Pi_\lambda$ is due to the way that tuples indexed by disjoint subsets of $\{1,\ldots,n\}$ can be easily combined using the notation \eqref{eq:concat} without having any misconceptions about the ordering of the components of the tuples.  For example, $x = (x_E, x_{E^c}) = (x_{E^c},x_E)$.  For another example, if $i\in\{1,\ldots,n\}$ and $f:\RR^n\to\RR$ is a $C^1$-map such that $f(0) = 0$ and $\PD{}{f}{x_i}(0)\neq 0$, then by the implicit function theorem, there exists a $C^1$-map $g:U\to\RR^{\{i\}}$ defined in a neighborhood $U$ of the origin in $\RR^{\{i\}^c}$ such that $g(0) = 0$ and $f(x_{\{i\}^c}, g(x_{\{i\}^c})) = 0$ on $U$.  The point of this second example is that, even though the objects $\RR^{\{i\}^c}$ and $\RR^{\{i\}}$ are clearly identifiable with the objects $\RR^{n-1}$ and $\RR$, respectively, having $g$ map from a subset of $\RR^{\{i\}^c}$ into $\RR^{\{i\}}$ allows us to write $(x_{\{i\}^c}, g(x_{\{i\}^c}))$, which is an object in $\RR^n$, and the graph of $g$ equals the zero set of $f$ in a suitable neighborhood of $0$.
}\end{enumerate}

\section{Development of the Proof}\label{s:developProof}

In this section we discuss how the proof of Theorem \ref{introThm:main} was developed.  The purpose of the discussion is to highlight some key ideas of the proof in a heuristic fashion while avoiding many technical details, to discuss the relationship between this work and the previous work of others, and to explain why certain choices were made in the proof where one technique is favored over another.

The proof of Theorem \ref{introThm:main} is an effective version of the construction of Rolin, Speissegger, and Wilkie \cite{RSW} --- henceforth to be called the {\bf RSW-construction} --- which is a model completeness proof for $\RR_{\Delta(\S)}$ based on a local resolution of singularities procedure, where $\Delta(\S)$ is the family of all partial derivatives of functions in $\S$.  The majority of the paper (much of Sections \ref{s:compTop}-\ref{s:pres}) consists of a methodical development of all the tools needed to prove an effective local resolution of singularities theorem, which is given at the very beginning of Section \ref{s:desingThms}. The remaining of the paper (Sections \ref{s:desingThms}-\ref{s:MT}) then comes together rapidly in comparison.

Gabrielov \cite{Gab96} showed that $\RR_{\Delta(\S)}$ is model complete if the functions in $\S$ are analytic, and it is natural to wonder if his proof could be used instead of the RSW-construction in order to avoid all the work needed to prove an effective resolution of singularities theorem.  However, his proof relies on the following two facts:
\begin{enumerate}{\setlength{\itemsep}{3pt}
\item
If $A\subseteq(0,1]^2$ is globally subanalytic, and if $y > 0$ for all $(x,y)\in\cl(A)$ such that $x>0$, then there exist $c,\kappa > 0$ such that $y > c x^\kappa$ for all $(x,y)\in A$.\\
(\emph{Note}: The number $\kappa$ given by this property is used to determine how far to expand certain Taylor expansions in the proof of \cite[Lemma 1]{Gab96}.)

\item
If $\F$ is a family of real-valued analytic functions defined in a neighborhood of a compact set $K\subseteq\RR^n$, then there exists a finite $\G\subseteq\F$ such that
\[
\{x\in K : \text{$f(x) = 0$ for all $f\in\F$}\}
=
\{x\in K : \text{$f(x) = 0$ for all $f\in\G$}\}.
\]
(\emph{Note}:  This property, called topological Noetherianity, is used in the proof of \cite[Lemma 2]{Gab96} to know that the desingularization procedure it gives must eventually stop.)
}\end{enumerate}
Gabrielov's proof actually goes through when $\S\subseteq\C$ for an arbitrary quasianalytic IF-system $\C$, because property 1 can be proven by applying the curve selection lemma (see Lemma \ref{lemma:curveSelection}) to the component of $(0,1]^2\setminus\cl(A)$ whose closure in $\RR^2$ contains $(0,1]\times\{0\}$, and property 2 is a consequence of local resolution of singularities (see Lemma \ref{lemma:topNoeth}, which is also proven in \cite{DJMthesis} and \cite{BM04}).  Since Lemmas  \ref{lemma:curveSelection} and \ref{lemma:topNoeth} are proven using the RSW-construction and its main tool, local resolution of singularities, it appears that to make Gabrielov's proof effective in our oracles for $\S$, one could not avoid proving an effective local resolution of singularities procedure.  And once this is done, it makes more sense to use the RSW-construction since the additional work required is not that substantial.

To get an understanding of why resolution of singularities is so useful, it is helpful to compare this technique with Gabrielov's desingularization lemma, \cite[Lemma 2]{Gab96}.  For simplicity, consider $A = \{x\in[-1,1]^n : f(x) = 0\}$ for an analytic function $f:[-1,1]^n\to\RR$.  The proof of \cite[Lemma 2]{Gab96}  repeatedly applies the implicit function theorem to express $A$ as a finite, disjoint union of semianalytic manifolds which are implicitly defined from functions in the $\QQ$-algebra generated by the partial derivatives of $f$.  These manifolds are generally not compact, and as a result, it is difficult to tell in mid-process when the collection of manifolds constructed thus far actually cover all of $A$, at which point the procedure stops.  Topological Noetherianity is used to see that the process must stop eventually, whenever that may be.  Another way to partition $A$ into such manifolds is to consider the level sets of Bierstone and Milman's local invariant for $f$, which we shall call strata (see \cite{BM91} or \cite{BM97}).  This is similar to Gabrielov's proof because the local invariant is defined through repeated application of the implicit function theorem.  But it has an advantage: the local invariant organizes the strata, and in such a way so that the stratum with the maximum value of the local invariant is compact.  So to construct this maximum stratum, one only needs to apply the implicit function theorem on compact sets, which can be done effectively if the function $f$ satisfies certain computability assumptions.  And fortunately, this maximum stratum is the only stratum that one needs to compute, since it is used as the center of blowing-up.  After applying this blowing-up, the new maximum value of the local invariant is lowered, and one repeats the process of only applying the implicit function theorem on compact sets.

\subsection{The RSW-Construction}
We now give an exposition of the RSW-construction.  The goal is to show that $\RR_{\Delta(\S)}$ is model complete.  By using a number of rather standard reductions,\footnote{
    There was a comment added to the end of \cite{RSW} in response to an only half-thought-out observation I made shortly after the paper was published electronically: the comment states that the proof in \cite{RSW} only shows that $\RR_{\Delta(\S)}$ is model complete if $\S$ has the extension property, meaning that for each $f:[-r,r]\to\RR$ in $\S$ there exists $g:[-s,s]\to\RR$ in $\S$ such that $r < s$ and $f(x) = g(x)$ on $[-r,r]$.  This extra assumption that $\S$ has the extension property is not actually needed, and I apologize for my critique on such a trivial detail.  I am mentioning this here because one of the ``standard'' reductions alluded to above, which I failed to see when this comment was added to \cite{RSW}, involves first replacing $\S$ with another family $\S_{\ext}\subseteq\C$ such that $\RR_{\Delta(\S)}$ and $\RR_{\Delta(\S_{\ext})}$ have the same existentially $0$-definable sets and $\S_{\ext}$ has the extension property.  The proof in \cite{RSW} shows that $\RR_{\Delta(\S_{\ext})}$ is model complete, so $\RR_{\Delta(\S)}$ is also model complete.  There are many ways to construct such a family $\S_{\ext}$, but one way is to compose the functions in $\S$ with simple quadratic functions near the boundaries of their domains.  For example, suppose $\S = \{f\}$ for a single function $f:[-1,1]\to\RR$.  Then one could let $\S_{\ext} = \{f_{(i,r)}\}_{(i,r)\in\{-1,0,1\}\times(\QQ\cap(0,1))}$, where for each $r\in\QQ\cap(0,1)$ the functions $f_{(-1,r)},f_{(0,r)},f_{(1,r)}:[-r,r]\to\RR$ are defined by $f_{(-1,r)}(x) = f(-1+x^2)$, $f_{(0,r)}(x) = f(x)$, and $f_{(1,r)}(x) = f(1-x^2)$.
    }
it suffices to fix integers $0\leq m\leq n$ and a set $A\subseteq[-1,1]^n$ of the form
\begin{equation}\label{eq:Adefined}
A = \left\{x\in[-1,1]^n : \bigwedge_{s=1}^{k} \sign(f_s(x)) = \sigma_j\right\}
\end{equation}
for some $\sigma_1,\ldots,\sigma_k\in\{-1,0,1\}$ and some $\C$-analytic functions
$f_1,\ldots,f_k:U\to\RR$ which are existentially $0$-definable in $\RR_{\Delta(\S)}$, where $U$ is a neighborhood of $[-1,1]^n$, and then prove that $[-1,1]^m\setminus\Pi_m(A)$ is existentially $0$-definable in $\RR_{\Delta(\S)}$.
(Saying that $f_1,\ldots,f_k$ are ``$\C$-analytic'' means that they are locally represented by functions in $\C$ --- see Definition \ref{def:Canalytic} for a precise definition.)  To accomplish this, the RSW-construction uses three techniques: a local resolution of singularities procedure based on Bierstone and Milman \cite[Theorem 4.4]{BM88}, fiber cutting, and the theorem of the complement of Van den Dries and Speissegger \cite[Theorem 2.7]{vdDS98}.

The resolution of singularities procedure in \cite{RSW} uses power substitutions and local blowings-up with centers of codimension $2$, but here we shall speak as if the RSW-construction uses only local blowings-up with centers of arbitrary codimension because that is what is done in this paper (and in \cite[Theorem 4.4]{BM88}), and whether or not power substitutions are used is rather irrelevant.  Also, the resolution procedure in \cite{RSW} relies on the assumption that $\RR\subseteq\S$, so it only proves that $\RR_{\Delta(\S)}$ is model complete when $\RR\subseteq\S$.  However, here we will not assume that $\RR\subseteq\S$.  The author first showed in his Ph.D. thesis \cite{DJMthesis} how to modify the resolution procedure of \cite{RSW} so as to remove the assumption that $\RR\subseteq\S$, and of course, the more sophisticated resolution procedure used here also does not assume that $\RR\subseteq\S$ since we require that $\S$ is recursively indexed in all of our effectivity results.  All of the noneffective model-theoretic results in this paper can actually be proven by the simpler construction given in the thesis \cite{DJMthesis}.
\vspace*{5pt}

\noindent\emph{The Resolution Procedure}:

The RSW-construction applies its local resolution procedure to the function $g:U\to\RR$ defined by
\[
g(x) = \left(\prod_{s=1}^{k} f_s(x)\right)\left(\prod_{i=1}^{n}(1-x_i)(1+x_i)\right),
\]
for $f_1,\ldots,f_k$ as in \eqref{eq:Adefined}; the functions $1-x_i$ and $1+x_i$ are included in this product because the box $[-1,1]^n$ is defined by the system of inequalities $1-x_i \geq 0$ and $1+x_i \geq 0$ for $i=1,\ldots,n$.  This resolution procedure constructs a finite family of $\C$-analytic maps $\{F^{(j)}:V^{(j)}\to U\}_{j\in J}$ and bounded open rational boxes $\{B^{(j)}\}_{j\in J}$ such that
\[
[-1,1]^n \subseteq \bigcup_{j\in J} F^{(j)}(B^{(j)}),
\]
and such that for each $j\in J$, $V^{(j)}$ is open in $\RR^{d(j)}$ for some $d(j)\in\{0,\ldots,n\}$, $\cl(B^{(j)})\subseteq V^{(j)}$, and $F^{(j)}$ is a composition of a finite sequence of maps of two types:
\begin{enumerate}{\setlength{\itemsep}{3pt}
\item[]Type 1: a local blowing-up $\tld{\pi}_i:W_i\to W$ with some smooth center $C\subseteq W$;

\item[]Type 2: an inclusion map $C\hookrightarrow W$ for such a center $C$.
}\end{enumerate}
The subscript $i$ in the notation $\tld{\pi}_i:W_i\to W$ is used to indicate that $\tld{\pi}_{i}^{-1}(C) = \{y\in W_i : y_i = 0\}$.  The set $\tld{\pi}_{i}^{-1}(C)$, and the function $y_i$ which defines this set, will both be called the exceptional divisor of $\tld{\pi}_i$.  The inclusion $C\hookrightarrow W$ does not have an exceptional divisor.  Each of these two types of maps have an associated way of transforming any given $\C$-analytic map $h:W\to \RR$:
\begin{enumerate}{\setlength{\itemsep}{3pt}
\item[]Type 1: The transform of $h$ by $\tld{\pi}_i$ is defined to be $y_i(h\circ\tld{\pi}_i)$, namely, the product of the exceptional divisor and the pullback of $h$ by $\tld{\pi}_i$.

\item[]Type 2: The transform of $h$ by $C\hookrightarrow W$ is defined to be $h\Restr{C}$, namely, the restriction of $h$ to $C$.
}\end{enumerate}
(Note: The above definitions are stated a little loosely.  More literally, $\tld{\pi}_i = \pi_i\circ G$, where $G$ is a $\C$-analytic coordinate transformation and $\pi_i$ is a ``standard chart'' of a blowing-up whose center $G^{-1}(C)$ is defined by $x_I = 0$ for some $I\subseteq\{1,\ldots,n\}$, as in Definition \ref{def:blowup}.  The inclusion $C\hookrightarrow W$ really refers to the map $x_{I^c}\mapsto G(x_{I^c},0)$, where $0\in\RR^I$, and $h\Restr{C}$ really refers to $x_{I^c}\mapsto h\circ G(x_{I^c},0)$.)

Now, fix $j\in J$, and for simplicity write $F:V\to U$, $d$, and $B$ in place of $F^{(j)}:V^{(j)}\to U$, $d(j)$, and $B^{(j)}$.  The function $F$ can be expressed as a composition $F = F_1\circ\cdots\circ F_l$ of maps $F_i : U_i \to U_{i-1}$ diagrammed as follows,
\begin{equation}\label{eq:resSeq}
\xymatrix{
    & \RR
    & \RR
    &
    & \RR
    & \RR
    &
    & \RR
    & \RR
\\
V = \hspace*{-30pt}
    & U_l \ar[r]^-{F_l} \ar[u]^-{g_l}
    & U_{l-1} \ar[r] \ar[u]^-{g_{l-1}}
    & \cdots \ar[r]
    & U_i \ar[r]^-{F_i} \ar[u]^-{g_i}
    & U_{i-1} \ar[r] \ar[u]^{g_{i-1}}
    & \cdots \ar[r]
    & U_1 \ar[r]^-{F_1} \ar[u]^-{g_1}
    & U_0 \ar[u]^-{g_0 = g}
    & \hspace*{-30pt} = U ,
\\
    &
    & C_l \ar@{^{(}->}[u]
    &
    & C_{i+1} \ar@{^{(}->}[u]
    & C_i \ar@{^{(}->}[u]
    &
    & C_2 \ar@{^{(}->}[u]
    & C_1 \ar@{^{(}->}[u]
}
\end{equation}
where for each $i\in\{1,\ldots,l\}$, the map $F_i:U_i\to U_{i-1}$ is either a local blowing-up with center $C_i\subseteq U_{i-1}$ or is an inclusion $C_i\hookrightarrow U_{i-1}$, and where $g_i$ is the transform of $g_{i-1}$ by $F_i$, with $g_0 = g$.  For each $i < l$, the function $g_i$ is used by the resolution procedure to choose the center $C_{i+1}$ which determines the next map $F_{i+1}$ in the sequence, and the resolution process stops when $g_l$ is normal crossings on $U_l$, meaning that $g_l$ is a product of a monomial in the coordinate variables and an analytic unit on $U_l$.  The use of transforms of Type 1 forces $g\circ F$ and the accumulated exceptional divisors of $F$ to be simultaneously normal crossings.  More specifically, we can write $g\circ F(y) = y_{E}^{\alpha} u(y)$ for some $\alpha = (\alpha_k)_{k\in E}\in\NN^E$ and $\C$-analytic unit $u:V\to\RR$ (meaning that $u$ has constant positive or negative sign on $V$), where $E\subseteq\{1,\ldots,d\}$ is such that the sets $\{y\in V : y_k = 0\}$, for each $k\in E$, are the accumulated exceptional divisors of $F$.  The use of Type 2 transforms means that the resolution procedure is being applied hereditarily in ambient spaces of all possible dimensions $0,\ldots,n$, not just of dimension $n$.

The reason for doing this is as follows.  Because $A$ is a subset of $U$ defined by sign conditions on factors of $g$, and $g\circ F(y) = y_{E}^{\alpha} u(y)$, it follows that $F^{-1}(A)$ is a finite union of sets of the form
\begin{equation}\label{eq:Vbox}
V_\xi = \left\{x\in V : \bigcup_{j\in E} \sign(x_j) = \xi_j\right\}
\end{equation}
for some choice(s) of $\xi = (\xi_j)_{j\in E} \in\{-1,0,1\}^E$.  Note that for each $k\in E$, the set $V_\xi$ is either contained in, or is disjoint from, the exceptional divisor $\{y\in V : y_k = 0\}$.  If $\xi_k\neq 0$ for each $k\in E$, then $F$ restricts to an isomorphism on $V_\xi$; note that in this case, $B\cap V_\xi$ is a bounded open rational box whose closure is contained in $V$.  If $\xi_k = 0$ for some $k\in E$, then there exists a least $i\in\{1,\ldots,l\}$ such that $F_i$ is a local blowing-up with $F_i\circ\cdots\circ F_l(V_\xi)\subseteq C_i$.  By the minimality of $i$, we have that $F_i\circ\cdots\circ F_l(V_\xi)$ is disjoint from all the accumulated exceptional divisors of $F_1\circ\cdots\circ F_{i-1}:U_{i-1}\to U$.  We may therefore ignore $V_\xi$ in this case, because by the hereditary application of the resolution procedure (beginning with the inclusion $C_i\hookrightarrow U_{i-1}$), $F_i\circ\cdots\circ F_l(V_\xi)$ is a union of isomorphic images of sets of the same form as \eqref{eq:Vbox} but in a lower dimensional space, and these images are pushed forward isomorphically into $U$ via the map $F_1\circ\cdots\circ F_{i-1}$.

Observe the following:
\begin{enumerate}
\item[L1.]
If $G = G_1\circ G_2$ for $\C$-analytic maps $G_1$ and $G_2$, and if all the partial derivatives of $G_1$ and of $G_2$ are existentially $0$-definable in $\RR_{\Delta(\S)}$, then all the partial derivatives of $G$ are existentially $0$-definable in $\RR_{\Delta(\S)}$.

\begin{proof}
Use the chain rule.
\end{proof}

\item[L2.]
If $G(x) = x_{k}^{p} H(x)$ for some $p\in\NN$ and $\C$-analytic maps $G$ and $H$, where $x = (x_1,\ldots,x_n)$ and $k\in\{1,\ldots,n\}$, and if all the partial derivatives of $G$ are existentially $0$-definable in $\RR_{\Delta(\S)}$, then all the partial derivatives of $H$ are existentially $0$-definable in $\RR_{\Delta(\S)}$.

\begin{proof}
For each $\alpha\in\NN^n$, repeated differentiation of $H(x) = G(x)/x_{k}^{p}$ shows that $\PDn{\alpha}{H}{x}\Restr{x_k \neq 0}$ is existentially $0$-definable in $\RR_{\Delta(\S)}$, and the formula
\[
\left.\frac{1}{\alpha!}\PDn{\alpha}{H}{x}\right|_{x_k = 0}
=
\left.\frac{1}{(\alpha+p e_k)!} \frac{\partial^{|\alpha|+p} G}{\partial x^{\alpha+p e_k}}\right|_{x_k=0}
\]
shows that $\PDn{\alpha}{H}{x}\Restr{x_k=0}$ is existentially $0$-definable in $\RR_{\Delta(\S)}$, where $e_k$ is the $k$th standard unit vector in $\NN^n$ (see Lemma \ref{lemma:derivRestr}).
\end{proof}

\item[L3.]
If $H$ is a multi-variate, real-valued function defined implicitly by a nonsingular equation $G(x,H(x)) = 0$ for some $\C$-analytic function $G$, and if all the partial derivatives of $G$ are existentially $0$-definable in $\RR_{\Delta(\S)}$, then all the partial derivatives if $H$ are existentially $0$-definable in $\RR_{\Delta(\S)}$.

\begin{proof}
Use implicit differentiation.
\end{proof}
\end{enumerate}
By using observations L1-L3, the fact that all the partial derivatives of each of the factors of $g_0$ are existentially $0$-definable in $\RR_{\Delta(\S)}$, and the way in which each center $C_i$ is constructed from $g_{i-1}$ (the details of which we will not delve into here), one can show by induction on $i$ that $F_1\circ\cdots\circ F_i$ is existentially $0$-definable in $\RR_{\Delta(\S)}$ for each $i\in\{0,\ldots,l\}$.  Combining this with the results of the previous paragraph gives the following representation of $A$:
\begin{equation}\label{eq:Arep}
\text{\parbox{5.5in}{
There exist finite families of $\C$-analytic maps $\{F^{(j)}:V^{(j)}\to U\}_{j\in J}$ and bounded open rational boxes $\{B^{(j)}\}_{j\in J}$ such that $A = \bigcup_{j\in J} F^{(j)}(B^{(j)})$ and such that for each $j\in J$, $V^{(j)}$ is an open rational box in $\RR^{d(j)}$ for some $d(j)\in\{0,\ldots,n\}$, $\cl(B^{(j)})\subseteq V^{(j)}$, $F^{(j)}$ restricts to a $\C$-analytic isomorphism on $B^{(j)}$, and the function $F^{(j)}$ and all of its partial derivatives are existentially $0$-definable in $\RR_{\Delta(\S)}$.   In particular, $A$ has dimension, with $\dim A = \max\{d(j) : j\in J\}$.
}}
\end{equation}
\hfill\vspace*{1pt}

\noindent\emph{Fiber Cutting}:

At this point, the RSW-construction applies fiber cutting.  Using the notation of \eqref{eq:Arep}, let $J'$ be the set of all $j\in J$ for which there exists a coordinate projection $\Pi^{(j)}:\RR^m\to\RR^{d(j)}$ such that $\Pi^{(j)}\circ\Pi_m\circ F^{(j)}\Restr{B^{(j)}}$ is an immersion.  We have
\begin{equation}\label{eq:beforeFC}
\Pi_m(A) = \left(\bigcup_{j\in J'}\Pi_m\circ F^{(j)}(B^{(j)})\right) \cup \left(\bigcup_{j\in J\setminus J'}\Pi_m\circ F^{(j)}(B^{(j)})\right).
\end{equation}
Fiber cutting is a procedure which expresses $\Pi_m(A)$ in the form
\[
\Pi_m(A) = \left(\bigcup_{j\in J'}\Pi_m\circ F^{(j)}(B^{(j)})\right) \cup \left(\bigcup_{j\in J\setminus J'}\Pi_m\circ F^{(j)}(C^{(j)})\right),
\]
where for each $j\in J\setminus J'$, $C^{(j)} = \bigcup_{s\in S^{(j)}} C^{(j,s)}$ for finitely many sets $C^{(j,s)} \subseteq B^{(j)}$ which either have the desired immersion property (as in the definition of $J'$) or have dimension less than $d(j)$, and where each set $C^{(j,s)}$ is defined as a subset of the box $B^{(j)}$ by sign conditions on a finite list of $\C$-analytic functions which are existentially $0$-definable in $\RR_{\Delta(\S)}$.  By applying the resolution procedure to each of the sets $C^{(j,s)}$, in the same way that the resolution procedure was applied to the set $A$, one constructs a new representation of $\Pi_m(A)$ which we write in the same way as \eqref{eq:beforeFC} but using tildes:
\[
\Pi_m(A) = \left(\bigcup_{j\in \tld{J}'}\Pi_m\circ \tld{F}^{(j)}(\tld{B}^{(j)})\right) \cup \left(\bigcup_{j\in \tld{J}\setminus \tld{J}'}\Pi_m\circ \tld{F}^{(j)}(\tld{B}^{(j)})\right),
\]
where either $\tld{J}' = \tld{J}$, or  $\tld{J}' \neq \tld{J}$ and $\max\{\tld{d}(j) : j\in \tld{J}\setminus\tld{J}'\} < \max\{d(j) : j\in J\setminus J'\}$.  By repeatedly applying the resolution procedure and fiber cutting in this manner, in successive alteration, we eventually arrive at the following representation of $\Pi_m(A)$:

\begin{equation}\label{eq:projArep}
\text{\parbox{5.5in}{
There exist finite families of $\C$-analytic maps $\{F^{(j)}:V^{(j)}\to \Pi_m(U)\}_{j\in J}$ and bounded open rational boxes $\{B^{(j)}\}_{j\in J}$ such that $\Pi_m(A) = \bigcup_{j\in J} F^{(j)}(B^{(j)})$ and such that for each $j\in J$, $V^{(j)}$ is an open rational box in $\RR^{d(j)}$ for some $d(j)\in\{0,\ldots,n\}$, $\cl(B^{(j)})\subseteq V^{(j)}$, there exists a coordinate projection $\Pi^{(j)}:\RR^m\to\RR^{d(j)}$ such that $\Pi^{(j)}\circ F^{(j)}\Restr{B^{(j)}}$ is an immersion, and the function $F^{(j)}$ and all of its partial derivatives are existentially $0$-definable in $\RR_{\Delta(\S)}$.  In particular, $\Pi_m(A)$ has dimension, with $\dim \Pi_m(A) = \max\{d(j) : j\in J\}$.
}}
\end{equation}
\hfill\vspace*{1pt}

\noindent\emph{Theorem of the Complement}:

At this point the RSW-construction applies the theorem of the complement given in \cite[Theorem 2.7]{vdDS98} to conclude that $[-1,1]^m\setminus\Pi_m(A)$ is existentially $0$-definable in $\RR_{\Delta(\S)}$, which completes the proof.  The hypothesis of \cite[Theorem 2.7]{vdDS98} assumes that $\Pi_m(A)$ is represented as a union of projections of certain manifolds, rather than images of open sets under $\C$-analytic maps as given in \eqref{eq:projArep}, but this difference is purely cosmetic.  The representation in \eqref{eq:projArep} can be seen to fit into the framework of \cite[Theorem 2.7]{vdDS98} by noting that the $\C$-analytic manifold $M^{(j)} = \Graph(F^{(j)}\Restr{B^{(j)}})$ projects onto $F^{(j)}(B^{(j)})$, and by composing this projection with $\Pi^{(j)}$ one obtains an immersion onto an open subset of $\RR^{d(j)}$, and by also noting that $\fr(M^{(j)}) = \Graph(F^{(j)}\Restr{\bd(B^{(j)})})$ is existentially $0$-definable in $\RR_{\Delta(\S)}$ and has dimension less than $d(j)$, which is the dimension of $M^{(j)}$.
\vspace*{5pt}

\noindent\emph{Summary of the RSW-Construction}:

After some preliminary reductions, one consider integers $0\leq m\leq n$ and a set $A\subseteq[-1,1]^n$ defined by sign conditions on a finite list of $\C$-analytic functions which are existentially $0$-definable in $\RR_{\Delta(\S)}$.  Starting with the product of the functions used to define $A$, repeatedly apply resolution of singularities (applied hereditarily) and fiber cutting, in successive alternation, until the set $\Pi_m(A)$ is represented as in \eqref{eq:projArep}.  Then apply the theorem of the complement \cite[Theorem 2.7]{vdDS98} to conclude that $[-1,1]^m\setminus\Pi_m(A)$ is existentially $0$-definable in $\RR_{\Delta(\S)}$.

\subsection{Development of the Effective RSW-Construction}

We now discuss in very general terms how the effective version of the RSW-construction given in this paper was developed.  This effective construction can be viewed abstractly as a sequence of steps, where at any given step, in order to know how to proceed one must determine the answer to a set of questions of the following form:
\begin{quote}
Given some kind of discrete representations for two geometric objects $O_1$ and $O_2$ (such as two sets or two functions), is $O_1 = O_2$ or is $O_1\neq O_2$?
\end{quote}
For each $i\in\{1,2\}$, the discrete representation for $O_i$ typically consists of two types of data structures:
\begin{enumerate}{\setlength{\itemsep}{3pt}
\item
Approximation Algorithms:

These algorithms rely on the approximation oracle for $\S$ and can be used to approximate the object $O_i$ to within any given error.

\item
A Lifting:

The lifting of $O_i$ is some discrete data which determines a very special, nonsingular, existential $\L_{\Delta(\S)}$-formula that defines the object $O_i$.
}\end{enumerate}
If $O_1\neq O_2$, then this fact can be discovered using the approximation algorithms for $O_1$ and $O_2$, which rely on the approximation oracle for $\S$.  If $O_1 = O_2$, then this fact can be discovered by using the liftings for $O_1$ and $O_2$ in conjunction with the precision oracle for $\S$.  Thus to answer our question, one runs both of these verification procedures simultaneously, using time sharing, until one procedure stops.

The key results pertaining to the approximation algorithms come straight out of computable analysis and are presented in Part I of the paper.  After presenting in Part II some basic tools needed for our resolution procedure, the key results pertaining to the liftings are presented in Part III of the paper.  Part IV then applies all of these concepts when presenting effective versions of the resolution procedure, fiber cutting, and the theorem of the compliment.  Here we will focus solely on the development of the effective resolution procedure, since it is the most nontrivial of these three components of the proof.

If the resolution procedure of \cite{RSW} was used, it would be possible to construct discrete representations of our objects $O_i$ consisting of a set of approximation algorithms for $O_i$ and an existential $\L_{\Delta(\S)}$-formula defining $O_i$, however this existential formula would not be of the special nonsingular form required of a ``lifting'', so the precision oracle for $\S$ (which only deals with nonsingular objects) could not be used to discover that $O_1 = O_2$.  So in order to construct the liftings, we use a different resolution procedure for which the centers of blowings-up and the accumulated exceptional divisors are always chosen to be simultaneously normal crossings, which is not the case in the resolution procedure in \cite{RSW}.   For this reason we base our resolution procedure on a variant of Bierstone and Milman's construction given in \cite{BM91} and \cite{BM97}.  This is their well-known global resolution procedure, but we are only interested in a variant of their local algorithm, solely because of this normal crossings property of the centers and exceptional divisors.

However, there is a certain troublesome quirk in the way that Bierstone and Milman present their algorithm in local coordinates, due to their liberal use of linear coordinate transformations.  In their construction, there are local coordinates $x = (x_1,\ldots,x_n)$ on $U\subseteq\RR^n$ which witness the fact that the center of blowing-up and accumulated exceptional divisors are simultaneously normal crossings, meaning that each exceptional divisor is given as $\{x\in U: x_k = 0\}$ for some $k\in\{1,\ldots,n\}$, and the center is given by $C = \{x\in U : x_I = 0\}$ for some $I\subseteq\{1,\ldots,n\}$.  There are also local coordinates $\tld{x} = (\tld{x}_1,\ldots,\tld{x}_n)$ on $\tld{U}\subseteq\RR^n$ for which the center is given by $\tld{C} = \{\tld{x}\in\tld{U} : \tld{x}_{\tld{I}} = 0\}$ for some $\tld{I} \subseteq\{1,\ldots,n\}$, and these coordinates $\tld{x}$ are used to define the local invariant on which their procedure is based.  The quirk is that the coordinates $x$ and $\tld{x}$ are not the same.  There is a fairly obvious isomorphism $F:U\to\tld{U}$ relating the coordinates $x$ and $\tld{x}$, with $F(C) = \tld{C}$, but this is still troublesome for our purposes.  The reason is that, to construct the liftings, the coordinates $x$ should be used, but to actually ``run'' the algorithm, the coordinates $\tld{x}$ should be used, for these are the coordinates which are used to actually find the center of blowing-up and are also the coordinates which are used to define the blowing-up in local coordinates, in the natural way, that witness the drop in the local invariant after blowing-up is performed.

If $\Pi:U'\to U$ and $\tld{\Pi}:\tld{U}'\to \tld{U}$ are the (globally defined) blowings-up of $U$ and $\tld{U}$ with centers $C$ and $\tld{C}$, respectively, the isomorphism $F:U\to\tld{U}$ does lift to an isomorphism $F':U'\to \tld{U}'$, so one could seemingly construct an effective resolution procedure using Bierstone and Milman's procedure where one continually keeps track of two coordinate systems, one for the liftings and the other to run the algorithm, and continually keeps track of the isomorphisms between the two, which would be described by a patchwork of local charts which are also defined by liftings.  However, this seems like it would be an absolute mess.  For this reason, this paper develops its own way of presenting the resolution procedure in local coordinates so that we always have $x = \tld{x}$.  This is achieved by using much more restrictive coordinate transformations when constructing the centers of blowings-up, so that exceptional divisors are always coordinate hyperplanes, and are never ``tilted'' as done by Bierstone and Milman through their use of rather general linear coordinate transformations.  It is most likely that the procedure given here is equivalent, up to isomorphism, to Bierstone and Milman's procedure, but I have not verified this.

The data structure used at each stage in the resolution procedure is called an ``$\S$-presentation'', which consists of a ``basic $\S$-presentation'' (to be discussed immediately) along with some additional discrete data (to be discussed below).  Roughly speaking, a basic presentation is a tuple $(\F,E;K)$, where $\F$ is a finite family of $\C$-analytic functions on an open set $U\subseteq\RR^n$, $E\subseteq\{1,\ldots,n\}$, and $K$ is a compact subset of $U$.  Saying that $(\F,E;K)$ is a ``basic $\S$-presentation'' means that it has a certain discrete representation consisting of approximation algorithms and liftings.  At each step of the resolution procedure, an $\S$-presentation is used to find the next center of blowing-up (analogous to how each function $g_i$ in \eqref{eq:resSeq} is used to find the center $C_{i+1}$), where the set $K$ is to be covered by the images of the various coordinate transformations or local blowings-up mapping into $U$ which are to be constructed next, and the sets $\{x\in U : x_k = 0\}$, for each $k\in E$, are the accumulated exceptional divisors from the previously applied blowings-up.

Let $\Div(\F,E) = d\in\NN^E$ be such that $x_{E}^{d}$ is the greatest common divisor of the functions in $\F$ which is a monomial in $x_E$, and write $f(x) = x_{E}^{d} f_E(x)$ for each $f\in\F$. For each $x\in U$, let $\ord(\F,E;x) = \min\{\ord(f_E;x) : f\in\F\}$, where $\ord(f_E;x)$ is the order of $f_E$ at $x$, and let $\ord(\F,E) = \sup\{\ord(\F,E;x) : x\in U\}$.  The goal of the resolution procedure is to apply local blowings-up so as to transform the basic $\S$-presentation $(\F,E;K)$ into new basic $\S$-presentations $(\F',E';K')$ such that $\ord(\F',E') = 0$.

Because Bierstone and Milman's resolution procedure is based on a local invariant which is defined using the orders of various functions, it may appear that one must first be able to compute $\ord(\F,E)$ in an effective manner when one is given a basic $\S$-presentation $(\F,E;K)$, and that this ability to compute $\ord(\F,E)$ would then enable one to effectively construct the centers of blowing-up.  A key insight is to realize that this is not the case, and that this idea should be turned on its head.  One should, in fact, use a search procedure to find the (locally defined) centers of blowing-up determined by $(\F,E;K)$, and once these centers have been found, one would have incidentally computed $\ord(\F,E)$.

To find the centers, one starts approximating the various partial derivatives of the functions $f_E$ for each $f\in\F$ on small rational boxes which collectively cover $K$.  In this way, nonvanishing partial derivatives can be found on families of sets covering $K$, so after pulling back $(\F,E;K)$ by a family of inclusions, we may assume that an upper bound on $\ord(\F,E)$ has been be established.  One then makes the guess that this upper bound on $\ord(\F,E)$ actually equals $\ord(\F,E)$.  After performing certain coordinate transformations, one may assume that the coordinates of $(\F,E;K)$ are suitably chosen so that we can define a ``refinement'' of $(\F,E;K)$, which is another basic $\S$-presentation $(\F_1,E_1;K_1)$ obtained by restricting a certain set of powers of accumulated exceptional divisors and partial derivatives of the functions in $\{f_E\}_{f\in\F}$ to a certain coordinate subspace of $U$.  One then repeats the procedure with $(\F_1,E_1;K_1)$.  By continuing in this manner, after pulling back by a sequence of suitable coordinate transformations, one constructs a sequence of refinements,
\[
(\F,E,K) = (\F_0,E_0;K_0), (\F_1,E_1;K_1), \ldots, (\F_k,E_k;K_k).
\]
This sequence of refinements is determined by what we call an ``$\S$-presentation'', which consists of the basic $\S$-presentation $(\F,E;K)$ along with some additional discrete data related to the various guesses made along the way.  The process of constructing these refinements will always stop, where the stopping condition is that either $\ord(\F_k,E_k) = 0$ or $\ord(\F_k;E_k) = \infty$ (where $\ord(\F_k,E_k) = \infty$ means that $\F_k$ is a family of zero functions).  An important point is that our oracles for $\S$ can be used to recognize this stopping condition.

Now, let $d_k = \Div(\F_k,E_k)$, and for each $i\in\{0,\ldots,k-1\}$ let $N_i \subseteq\{1,\ldots,n\}\setminus\bigcup_{j=0}^{i-1}N_j$ be such that the domain of $\F_{i+1}$ is obtained from the domain of $\F_i$ by setting $x_{N_i} = 0$.  There is a certain close relationship between the orders of a basic presentation and its refinement (see Lemma \ref{lemma:refine}), and this relationship implies that if $\ord(\F_k,E_k) = \infty$, or if $\ord(\F_k,E_k) = 0$ and $|d_k| \geq p_k$ for a certain special value of $p_k\in\NN$, then our guesses for the values of $\ord(F_i,E_i)$, for each $i\in\{0,\ldots,k\}$, were in fact all correct.  In this case, if $\ord(\F_k,E_k) = \infty$ then define $N_k = \emptyset$, and if $\ord(\F_k,E_k) = 0$ then choose  $N_k\subseteq E_k$ minimal such that $|d_{k,N_k}| \geq p_k$.  The desired center of blowing-up is $C = \{x\in U : x_I = 0\}$, where $I = \bigcup_{i=0}^{k} N_k$.  Now, if on the other hand we have that $\ord(\F_k,E_k) = 0$ and $|d_k| < p_k$, this implies that for some $i\in\{0,\ldots,k-1\}$, our guessed value of $\ord(\F_i,E_i)$ was too large.  By further approximating the various partial derivatives, we can reduce our upper bound for $\ord(\F_i,E_i)$ for some $i$, and we then repeat the process, starting with $(\F_i,E_i;K_i)$.  Since the upper bounds on the orders of the refinements are always lowered in a lexicographical fashion, this process must eventually stop, at which point we have found the center of blowing-up.

\section*{\bf Part I: Effective Approximation}

Part I develops the tools we shall need from computable analysis.  Section \ref{s:compTop} deals with point-set topological concepts, and Section \ref{s:compClos} proves various closure properties of $C^p$ functions which can be effectively approximated. 
\section{Topological concepts from computable analysis}\label{s:compTop}

\begin{definition}\label{def:box}
An {\bf interval} is a connected subset of $\RR$.  An interval is {\bf rational} if its infimum and supremum are in $\QQ\cup\{-\infty,+\infty\}$.  A {\bf (rational) box} is a finite Cartesian product of (rational) intervals.  A box in $\RR^n$ is {\bf degenerate} if its interior in $\RR^n$ is empty, and is {\bf nondegenerate} otherwise.  (For example, $\emptyset$ and the singletons $\{a\}$, for each $a\in\QQ$, are the degenerate rational intervals.)

The {\bf name} for a nonempty rational box consists of the unique string of symbols which is used to denote the rational box in the natural manner.  For example, $(-2,1]\times(\frac{1}{2},+\infty)$ and $[1,1]\times(-\infty,+\infty)$ are names for certain nonempty rational boxes in $\RR^2$.  If $B$ is a nonempty rational box, we shall write $\name(B)$ for its name.  Note that the set of names for nonempty rational boxes is a computable subset of the set of all strings from the alphabet containing every rational number, the infinity symbols $-\infty$ and $+\infty$, the parenthesis $($ and $)$, the square brackets $[$ and $]$, the product symbol $\times$, and the comma.  If $\{B_i\}_{i\in I}$ is a family of rational boxes, we call $\{\name(B_i)\}_{i\in I}$ a {\bf name} for $\{B_i\}_{i\in I}$ provided that the index set $I$, and the map $i\mapsto\name(B_i)$ on $I$, are computable.

If $a = (a_1,\ldots,a_n)$ and $b=(b_1,\ldots,b_n)$ are tuples in $\RR^n$, we write
\begin{eqnarray*}
(a,b)
    & = &
    (a_1,b_1)\times\cdots\times(a_n,b_n), \\
{[a,b]}
    & = &
    [a_1,b_1]\times\cdots\times[a_n,b_n].
\end{eqnarray*}
\end{definition}

\begin{definition}\label{def:computableDomain}
A set $D\subseteq\RR^n$ is called a {\bf computable domain} if it is the union of nondegenerate compact rational boxes, and if the function from the set of names for nondegenerate compact rational boxes in $\RR^n$ into $\{0,1\}$ defined by
\begin{equation}\label{eq:compDomRep}
\name(B) \mapsto \begin{cases}
1, & \text{if $B\subseteq D$,}\\
0, & \text{if $B\not\subseteq D$,}
\end{cases}
\end{equation}
is computable.  We call an algorithm which computes the map \eqref{eq:compDomRep} a {\bf representation algorithm} for $D$.
\end{definition}

Some simple examples of computable domains are the empty set, $\RR^n$, $(0,1)^n$, $[0,1]^n$, and more generally, any finite union of nondegenerate rational boxes in $\RR^n$.

\begin{definition}\label{def:ceopen}
Let $D$ be a computable domain in $\RR^n$.  A set $U\subseteq D$ is {\bf c.e.\ open in $D$} (i.e., computably enumerably open in $D$) if $U$ is open in $D$ and if there is an algorithm which acts as follows:
\begin{quote}
Given the name for a nondegenerate compact rational box $B$ contained in $D$, the algorithm stops if and only if $B\subseteq U$.
\end{quote}
Any such algorithm is called a {\bf c.e.\ open representation algorithm for $U$ in $D$}.
\end{definition}

Note that any c.e.\ open representation algorithm for $U$ in $D$ can be viewed as implicitly relying on a representation algorithm for $D$ in order to check whether it is, in fact, given a name for a nondegenerate compact rational box contained in $D$.

\begin{remarks}\label{rmk:ceopen}
Let $D$ be a computable domain in $\RR^n$.
\begin{enumerate}{\setlength{\itemsep}{5pt}
\item
Let $U$ be open in $D$.  The set $U$ is c.e.\ open in $D$ if and only if the set of all names of nondegenerate compact rational boxes contained in $U$ is computably enumerable.

\item
The requirements in Definitions \ref{def:computableDomain} and \ref{def:ceopen} that $B$ be nondegenerate may be dropped without altering these definitions.  (This follows from the fact that if $B$ is a compact rational box contained in $D$, then $B\subseteq A\subseteq D$ for some nondegenerate compact rational box $A$, and likewise with $U$ in place of $D$.)  Thus if $U$ is open in $D$, the set $U$ is c.e.\ open in $D$ if and only if the set of all names of compact rational boxes contained in $U$ is computably enumerable.

\item
Let $\T$ be the set of all c.e.\ open subsets of $D$.  Then $\T$ forms a computable topology, in the following sense:
\begin{enumerate}
\item
The sets $\emptyset$ and $D$ are in $\T$.

\item
For any computable family $\{U_i\}_{i\in I}$ of members of $\T$, $\bigcup_{i\in I} U_i$ is in $\T$.

\item
The intersection of any finitely many members of $\T$ is in $\T$.
\end{enumerate}
}
\end{enumerate}
\end{remarks}

The proofs of Remarks \ref{rmk:ceopen}.1-3 are straightforward, so they could all be left to the reader.  However, we shall give a proof of Remark \ref{rmk:ceopen}.1 in order to raise a certain point.

\begin{proof}[Proof of Remark \ref{rmk:ceopen}.1]
Suppose that $U$ is c.e.\ open in $D$.  Using a representation algorithm for $D$, we can construct a computable enumeration $\{B_i\}_{i\in\NN}$ of the set of all nondegenerate compact rational boxes contained in $D$.  By using time sharing and a c.e.\ open representation algorithm for $U$, we can computably enumerate the boxes $B_i$ which are contained in $U$.

Conversely, suppose we are given a computable enumeration $\{B_i\}_{i\in\NN}$ of the set of all nondegenerate compact rational boxes contained in $U$.  Then given any nondegenerate compact rational box $B$ contained in $D$, for each $i\in\NN$ check if $B = B_i$.  Stop once such an $i$ has been found, and do not stop otherwise.  This is a c.e.\ open representation algorithm for $U$.
\end{proof}

Notice that Remark \ref{rmk:ceopen}.1 is a statement about the existence of one type of algorithm implying the existence of another type of algorithm, and conversely.  Namely, it states that if a c.e.\ open representation algorithm for $U$ exists, then an algorithm which enumerates the set of all names for compact rational boxes contained in $U$ exists, and conversely.  As literally stated, Remark \ref{rmk:ceopen}.1 does not claim that there is an effective procedure for constructing one type of algorithm from the other type of algorithm.  But upon reading the proof of Remark \ref{rmk:ceopen}.1, it becomes apparent that there  \emph{is} such an effective procedure, and in fact, the proof is just an informal description of this effective procedure.  Namely, the proof shows the following:
\begin{equation}\label{eq:ceopen}
\text{\parbox{5in}{
There is a partial computable function which, when given a representation algorithm for a computable domain $D$ in $\RR^n$ and a representation algorithm for a c.e.\ open set $U$ in $D$, outputs a computable enumeration of the set of names of all compact rational boxes contained in $U$.  And, there is another partial computable function which, when given a representation algorithm for a computable domain $D$ in $\RR^n$ and a computable enumeration of the set of names of all compact rational boxes contained in a set $U$ which is open in $D$, outputs a c.e.\ open representation algorithm for $U$.
}}
\end{equation}

The statement \eqref{eq:ceopen} is more precise and informative than Remark \ref{rmk:ceopen}.1, but it is much more cumbersome to state, almost unbearably so.  Therefore to simplify our language of discourse, we introduce the following concept and convention.

\begin{definition}\label{def:effectiveTruth}
Let us call any statement of the following form an {\bf algorithmic statement}:
\begin{quote}
If there exist algorithms $A_1,\ldots,A_k$ (of certain types), then there exists an algorithm $B$ (of a certain type).
\end{quote}
We shall say that this algorithmic statement is {\bf effectively true} if there is a partial computable map which, when given algorithms $A_1,\ldots,A_k$ (of certain types), it outputs an algorithm $B$ (of a certain type).
\begin{description}
\item[Convention]
Henceforth, whenever we claim that an  algorithmic statement is true (such as in a lemma, or a theorem, etc.), we are tacitly claiming that the statement is, in fact, effectively true.
\end{description}
Even with this convention, when there are nested algorithmic statements, we will sometimes use the phrase ``it is effectively true that \ldots'' for clarification.

There are two other phrases pertaining to effectivity that we shall use, one used to describe computability and the other used to describe computable enumerability.  Consider a set $\P$ of statements (with assigned truth values) in some computable language.  Suppose there is an algorithm which, when given any $P\in\P$, will stop and state whether or not $P$ is true; in this case, if we are given some $P\in\P$, we say that we can {\bf effectively determine} if $P$ is true.   Now suppose there is an algorithm which, when given any $P\in\P$, will stop if and only if $P$ is true; in this case, if we are given some $p\in \P$ which is true, we say that we can {\bf effectively verify} that $P$ is true.
\end{definition}

\begin{definition}\label{def:coceClosed}
Let $D$ be a computable domain in $\RR^n$.  A set $C\subseteq D$ is {\bf co-c.e.\ closed in $D$} if $D\setminus C$ is c.e.\ open in $D$.  A {\bf co-c.e.\ closed representation algorithm for $C$ in $D$} is a c.e.\ open representation algorithm for $D\setminus C$ in $D$.
\end{definition}

\begin{remark}\label{rmk:coceClosed}
Let $\C$ be the set of all co-c.e.\ closed subsets of a computable domain $D$.  Remark \ref{rmk:ceopen}.3 and DeMorgan's law imply the following:
\begin{enumerate}
\item
The sets $\emptyset$ and $D$ are in $\C$.

\item
For any computable family $\{C_i\}_{i\in I}$ of members of $\C$, $\bigcap_{i\in I} C_i$ is in $\C$.

\item
The union of any finitely many members of $\C$ is in $\C$.
\end{enumerate}
\end{remark}

\begin{definition}\label{def:coceCompact}
A set $K\subseteq\RR^n$ is {\bf co-c.e.\ compact} if $K$ is compact and there exists an algorithm acting as follows:
\begin{quote}
Given the name for a finite family $\{B_i\}_{i\in I}$ of bounded, open, rational boxes in $\RR^n$, the algorithm stops if and only if $K\subseteq\bigcup_{i\in I} B_i$.
\end{quote}
We call such an algorithm a {\bf co-c.e.\ compact representation algorithm} for $K$. \end{definition}

Note that a set $K\subseteq\RR^n$ is co-c.e.\ compact if and only if the set of all names for finite families $\{B_i\}_{i\in I}$ of bounded, open, rational boxes in $\RR^n$ which cover $K$ is computably enumerable.

\begin{lemma}\label{lemma:coceCompactInCeOpen}
If $U$ is c.e.\ open in $\RR^n$, and $K$ is a co-c.e.\ compact subset of $U$, then we can effectively verify that $K\subseteq U$.
\end{lemma}

\begin{proof}
There exists a finite family $\{B_i\}_{i\in I}$ of bounded, open, rational boxes such that
\begin{equation}\label{eq:Kcovered}
K\subseteq \bigcup_{i\in I}B_i
\end{equation}
and
\begin{equation}\label{eq:clBinU}
\text{$\cl(B_i)\subseteq U$ for all $i\in I$,}
\end{equation}
and both \eqref{eq:Kcovered} and \eqref{eq:clBinU} can be effectively verified.
\end{proof}

\begin{lemma}\label{lemma:cocecompact}
A set $K\subseteq\RR^n$ is co-c.e.\ compact if and only if there exists a computable sequence of sets $\{U_i\}_{i\in\NN}$ such that $K = \bigcap_{i\in\NN} U_i$, and such that for each $i\in\NN$, $U_i$ is a finite union of bounded open rational boxes and $\cl(U_{i+1}) \subseteq U_i$.
\end{lemma}

\begin{proof}
Suppose that $K$ is co-c.e. compact.  Fix a computable enumeration $\{V_i\}_{i\in\NN}$ of the set of all subsets of $\RR^n$ which contain $K$ and which are a finite union of bounded, open rational boxes.  Let $U_0 = V_0$.  Now let $j\geq 0$, and inductively assume that we have constructed $U_0,\ldots,U_{j-1}$ such that for all $k\in\{1,\ldots,j-1\}$, we have $U_k\in\{V_i\}_{i\in\NN}$ and $\cl(U_k) \subseteq U_{k-1}\cap V_k$.  The set $U_{j-1}\cap V_j$ is a finite union of bounded, open rational boxes and is a neighborhood of the co-c.e.\ compact set $K$, so we can effectively find some $m\in\NN$ such that $\cl(V_m) \subseteq U_{j-1}\cap V_j$.  Let $U_j = V_m$.  This completes the inductive construction of the sequence $\{U_i\}_{i\in\NN}$.  By construction, $\cl(U_{i+1})\subseteq U_i$ for each $i\in\NN$, and $K = \bigcap_{i\in\NN} U_i$ because $K \subseteq \bigcap_{i\in\NN} U_i \subseteq \bigcap_{i\in\NN} V_i = K$.

Conversely, assume that there exists a computable sequence of sets $\{U_i\}_{i\in\NN}$ such that $K = \bigcap_{i\in\NN} U_i$, and such that for each $i\in\NN$, $U_i$ is a finite union of bounded open rational boxes and $\cl(U_{i+1}) \subseteq U_i$.  We claim that the following is a co-c.e.\ compact representation algorithm for $K$:
\begin{quote}
Given the name for a finite family $\{B_i\}_{i\in I}$ of bounded open rational boxes in $\RR^n$, successively check for each $j\in\NN$ whether $U_j \subseteq \bigcup_{i\in I} B_i$.  Stop once a $j$ has been found, and do not stop otherwise.
\end{quote}
Note that we can effectively check whether $U_j \subseteq \bigcup_{i\in I} B_i$ since both sets are simply finite unions of rational boxes.  If $K\not\subseteq\bigcup_{i\in I} B_i$, this algorithm will not stop.  If $K\subseteq\bigcup_{i\in I} B_i$, then $\bigcup_{i\in I} B_i$ is a c.e.\ open neighborhood of $K$, so the following remark implies that there exists $j\in\NN$ such that $U_j \subseteq \bigcup_{i\in I} B_i$, so the algorithm will stop.
\end{proof}

\begin{remark}\label{rmk:cocecompact}
Let $K$ be co-c.e.\ compact, and let $\{U_i\}_{i\in\NN}$ be as in Lemma \ref{lemma:cocecompact}.  Since $K = \bigcap_{i\in\NN} U_i$ and $\cl(U_{i+1})\subseteq U_i$ for each $i$, it follows that $K = \bigcap_{i\in\NN}\cl(U_i)$, which represents $K$ as a decreasing union of compact sets.  Therefore if $V$ is an open set containing $K$, there exists $i\in\NN$ such that $\cl(U_i) \subseteq V$.  If $V$ is c.e.\ open, then we can effectively find such an $i$.
\end{remark}

\begin{proposition}\label{prop:compHeineBorel}
A subset of $\RR^n$ is co-c.e. compact if and only if it is co-c.e.\ closed in $\RR^n$ and there exists a rational $M > 0$ such that $K\subseteq(-M,M)^n$.
\end{proposition}

\begin{proof}
Suppose that $K\subseteq\RR^n$ is co-c.e.\ compact.  Fix a sequence $\{U_i\}_{i\in\NN}$ as in Lemma \ref{lemma:cocecompact}.  The set $U_0$ gives us a bound for $K$.  Since $K = \bigcap_{i\in\NN}\cl(U_i)$ and each set  $\cl(U_i)$ is co-c.e.\ closed in $\RR^n$,  Remark \ref{rmk:coceClosed} implies that $K$ is co-c.e.\ closed in $\RR^n$.

Conversely, suppose that $K$ is co-c.e.\ closed in $\RR^n$ and that $K \subseteq (-M,M)^n$ for a rational $M > 0$.  We claim that the following is a co-c.e.\ compact representation algorithm for $K$:
\begin{quote}
Given a finite family $\{A_i\}_{i\in I}$ of bounded open rational boxes in $\RR^n$, first express $[-M,M]\setminus\bigcup_{i\in I} A_i$ as a union of a finite family of compact rational boxes $\{B_j\}_{j\in J}$.  Then use a c.e.\ open representation algorithm for $\RR^n\setminus K$ to try to verify that $B_j\subseteq \RR^n\setminus K$ for each $j\in J$.
\end{quote}
This algorithm will stop if and only if $B_j\subseteq\RR^n\setminus K$ for all $j\in J$, which occurs if and only if $K\subseteq \bigcup_{i\in I}A_i$.
\end{proof}

We stress that Proposition \ref{prop:compHeineBorel} is effectively true, meaning that there are effective procedures which enable the co-c.e.\ compact representation algorithm for $K$ on the one hand, and the rational number $M > 0$ and co-c.e.\ closed representation algorithm for $K$ on the other hand, to be effectively constructed from one another.

\begin{definition}\label{def:compCp}
Consider a function $f:U\to\RR^m$, where $U$ is a c.e.\ open subset of some computable domain $D$ in $\RR^n$, and let $p\in\NN\cup\{\infty\}$.  We say that $f$ is {\bf computably $C^p$} if there is an algorithm which acts as follows:
\begin{equation}\label{eq:compCpRep}
\text{\parbox{5.3in}{Given $\alpha\in\NN^n$ such that $|\alpha|\leq p$, a name for an open rational box $I$ in $\RR^m$, and a name for a compact rational box $B$ in $D$, the algorithm stops if and only if $B\subseteq U$ and $\PDn{\alpha}{f}{x}(B)\subseteq I$.
}}
\end{equation}
Any such algorithm \eqref{eq:compCpRep} is called a {\bf $C^p$ approximation algorithm} for $f$. If $p=0$, we also say that $f$ is {\bf computably continuous}.  If $p=n=0$, we call $f(0)$ a {\bf computable point} in $\RR^m$.  If $p=n=0$ and $m=1$, we call $f(0)$ a {\bf computable real}. (Note that a point $a\in\RR^n$ is computable if and only if the set $\{a\}$ is co-c.e.\ compact.)  We may sometimes just say ``approximation algorithm'', rather than ``$C^0$ approximation algorithm'', when we are working with a computably continuous function, a computable point, or a computable real.

More generally, a family of functions $\S = \{S_\sigma\}_{\sigma\in\Sigma}$ is {\bf computably $C^p$} if the index set $\Sigma$ is computable, and if there is an algorithm which acts as a $C^p$ approximation algorithm for each function in $\S$, as indexed by $\Sigma$.  Such an algorithm is called a {\bf $C^p$ approximation algorithm} for the family $\S$.
\end{definition}

\begin{remarks}\label{rmk:compCpSimple}
Consider a function $f=(f_1,\ldots,f_m):U\to\RR^m$, where $U$ is an open subset of a computable domain $D$ in $\RR^n$.
\begin{enumerate}{\setlength{\itemsep}{5pt}
\item
Since $U$ is open in $D$, and $D$ is a union of nondegenerate compact rational boxes, it follows that for each $x\in U$ there exists a nondegenerate compact rational box $B$ such that $x\in B\subseteq U$.  Therefore all partial derivative of $f$ at $x$ can be defined using only points from $B$, either in a one-sided or a two-sided sense.  Thus saying $f$ is $C^p$ on $U$ makes sense.

\item
The function $f$ is computably $C^p$ if and only each of its component functions $f_1,\ldots,f_m$ are computably $C^p$.

\item
Every rational number is a computable real, and every constant function which takes the value of some computable real is computably $C^\infty$.

\item
Our definition of computably continuous is equivalent to what is known in the literature as ``Type 2 computable'' (see Weihrauch \cite{Weihrauch}), except we add in the additional requirement that the function be defined on a c.e.\ open subset of some computable domain $D$.  (No assumption is made about the domain in the standard definition of a Type 2 computable function.)
}
\end{enumerate}
\end{remarks}

\begin{proposition}\label{prop:compCont}
The following are equivalent for any function $f:U\to\RR^m$, where $U$ is an open subset of a computable domain $D$ in $\RR^n$:
\begin{enumerate}{\setlength{\itemsep}{3pt}
\item
The function $f$ is computably continuous.

\item
It is effectively true that for each c.e.\ open set $V$ in $\RR^m$, the set $f^{-1}(V)$ is c.e.\ open in $D$.  (In other words, there exists a partial computable function which, when given a representation algorithm for a c.e.\ open set $V$ in $\RR^m$, it outputs a representation algorithm for the c.e.\ open set $f^{-1}(V)$ in $D$.)

\item
It is effectively true that for each co-c.e.\ closed set $C$ in $\RR^m$, the set $f^{-1}(C)$ is co-c.e.\ closed in $D$.
}
\end{enumerate}
\end{proposition}

\begin{proof}
This is straightforward and is left to the reader.
\end{proof}

\begin{proposition}\label{prop:compContCompact}
If $U$ is c.e.\ open in $\RR^n$, $K\subseteq U$ is co-c.e.\ compact, and $f:U\to\RR^m$ is computably continuous, then $f(K)$ is co-c.e.\ compact.
\end{proposition}

\begin{proof}
Let $\A$ be the set of all finite families $\{(A_i,B_i)\}_{i\in I}$ such that $K\subseteq\bigcup_{i\in I} B_i$, where for each $i\in I$, the sets $A_i\subseteq\RR^m$ and $B_i\subseteq\RR^n$ are bounded, open, rational boxes such that $\cl(B_i)\subseteq U$ and $f(\cl(B_i)) \subseteq A_i$.  Note that for any open set $V\subseteq\RR^m$ such that $f(K)\subseteq V$, there exists $\{(A_i,B_i)\}_{i\in I} \in \A$ such that $\bigcup_{i\in I} A_i \subseteq V$.

Using a co-c.e.\ compact representation algorithm for $K$ and a $C^0$ approximation algorithm for $f$, we can construct a computable enumeration $\{\{(A_i,B_i)\}_{i\in I_j}\}_{j\in\NN}$ of $\A$.  Now, the set $\bigcup_{i\in I_0}A_i$ gives us a bound for $f(K)$, and we claim that the following is a co-c.e.\ closed representation algorithm for $f(K)$, so $f(K)$ is co-c.e. compact by Proposition \ref{prop:compHeineBorel}:
\begin{quote}
Given a name for a compact rational box $C\subseteq\RR^m$, successively check for each $j\in\NN$ whether $A_i\cap C = \emptyset$ for all $i\in I_j$.  Stop once a $j$ has been found, and do not stop otherwise.
\end{quote}
If $f(K)\cap C \neq \emptyset$, this algorithm will not stop.  If $f(K)\cap C = \emptyset$, then $\RR^m\setminus C$ is a neighborhood of $f(K)$, so there exists $j\in\NN$ such that $\bigcup_{i\in I_j} A_i \subseteq\RR^m\setminus C$, so the algorithm will stop.
\end{proof}

\section{Closure properties of computably $C^p$ functions}\label{s:compClos}

\begin{proposition}\label{prop:compArith}
The operations of addition and multiplication are computably $C^\infty$ maps from $\RR^2$ into $\RR$.
\end{proposition}

\begin{proof}
It suffices to show that the functions $S(x,y) = x+y$ and $P(x,y) = xy$ are computably continuous, since their partial derivatives are so trivial.
So consider rational boxes $B = [a_1,b_1]\times[a_2,b_2]$ and $I = (c,d)$.

Since $\PD{}{S}{x} = \PD{}{S}{y} = 1 > 0$, the function $S$ on $B$ is maximized at $(b_1,b_2)$ and minimized at $(a_1,a_2)$.  Thus $S(B)\subseteq I$ if and only if $c < a_1+a_2$ and $b_1+b_2<d$, so $S$ is computably continuous.

Similarly, since $\PD{}{P}{x} = y$ and $\PD{}{P}{y} = x$, the function $P$ on $B$ is maximized and minimized at points in $\{a_1,0,b_1\}\times\{a_2,0,b_2\}$ (which are easily determined according to whether $a_i < b_i \leq 0$, $a_i < 0 < b_i$, or $0\leq a_i < b_i$, for each $i\in\{1,2\}$), so $P$ is also computably continuous.
\end{proof}

\begin{proposition}\label{prop:compRecip}
The function $f:\RR\setminus\{0\}\to\RR$ defined by $f(x) = 1/x$ is computably $C^\infty$.
\end{proposition}

\begin{proof}
Consider rational intervals $B = [a,b]\subseteq\RR\setminus\{0\}$ and $I = (c,d)$, and let $n\in\NN$.  The $n$th derivative $f^{(n)}(x) = \frac{(-1)^n n!}{x^{n+1}}$ is either increasing or decreasing on $B$, according to whether $n$ is even or odd and whether $0 < a < b$ or $a < b < 0$.  So whether or not $f^{(n)}(B)\subseteq I$ is easily determined by comparing $f^{(n)}(a)$ and $f^{(n)}(b)$ with $c$ and $d$.
\end{proof}

\begin{proposition}\label{prop:compComp}
Let $p\in\NN\cup\{\infty\}$.  If $g:V\to U$ and $f:U\to\RR^k$ are computably $C^p$, then $f\circ g:V\to\RR^k$ is computably $C^p$.
\end{proposition}

\begin{proof}
When $p=0$, this follows easily from Proposition \ref{prop:compCont}.  So let $p > 0$, and inductively assume that the proposition holds for all computable $C^{p-1}$ functions.  Note that the induction hypothesis and Proposition \ref{prop:compArith} together imply that sums and products of computably $C^{p-1}$ functions are computably $C^{p-1}$.

Suppose that $U\subseteq\RR^m$ and $V\subseteq\RR^n$, write $x = (x_1,\ldots,x_m)$ and $y = (y_1,\ldots,y_n)$ for coordinates on $U$ and $V$, respectively, and write $f = (f_1,\ldots,f_k)$ and $g = (g_1,\ldots,g_m)$.  Then for each $l\in\{1,\ldots,k\}$ and $j\in\{1,\ldots,n\}$,
\[
\PD{}{(f_l\circ g)}{y_j}(x) = \sum_{i=1}^{m} \PD{}{f_l}{x_i}(g(y)) \PD{}{g_i}{y_j}(y).
\]
The functions $\PD{}{f_l}{x_i}$ and $\PD{}{g_i}{y_j}$ are computable $C^{p-1}$, so $\PD{}{f}{x_i}\circ g$ is computably $C^{p-1}$ by the induction hypothesis, and hence $\PD{}{(f_l\circ g)}{y_j}$ is computably $C^{p-1}$.  Thus $f\circ g$ is computably $C^p$.
\end{proof}

When speaking about Riemann integrals, a {\bf partition} of a compact rational interval $[a,b]$ is a finite collection of compact rational intervals $\P = \{[y_0,y_1],\ldots,[y_{k-1},y_k]\}$ with $a = y_0 < \cdots < y_k = b$.  If $P  \in\P$ and we are integrating with respect to a variable $y$, we write $\Delta y_P$ for the length of the interval $P$.  More generally, a {\bf partition} of a compact rational box $B = \prod_{i=1}^{n}[a_i,b_i]$ is a set $\P = \{P_1\times\cdots\times P_n : P_1\in\P_1, \ldots, P_n\in\P_n\}$, where $\P_i$ is a partition of $[a_i,b_i]$ for each $i\in\{1,\ldots,n\}$.

\begin{lemma}\label{lemma:compInteg}
Let $f:U\times[a,b]\to\RR$ be computably $C^p$, where $U\subseteq\RR^n$, and define
$F:U\to\RR$ by
\[
F(x) = \int_{a}^{b} f(x,y) dy.
\]
Then $F$ is computably $C^p$.
\end{lemma}

The hypothesis is loosely stated: it suffices that the functions $\PDn{\alpha}{f}{x}$, for each $\alpha\in\NN^{n}$ with $|\alpha|\leq p$, be continuous, and that there exist an approximation algorithm for only these partial derivatives in $x$.

\begin{proof}
We will prove the lemma for $p = 0$.  The lemma for a general $p$ follows by applying our proof to the formula
\[
\PDn{\alpha}{F}{x}(x) = \int_{a}^{b} \PDn{\alpha}{f}{x}(x,y) dy
\]
for each $\alpha\in\NN^n$ with $|\alpha|\leq p$.

Fix a compact rational box $B\subseteq U$ and a rational open interval $I$.  For each $\epsilon > 0$ there exists a partition $\P$ of $B\times[a,b]$ and a family of pairs of rational numbers $\{(m_P,M_P)\}_{P\in\P}$ such that
\begin{equation}\label{eq:ULcond}
\text{$m_P < f(x,y) < M_P$ for all $P\in\P$ and all $(x,y)\in P$,}
\end{equation}
and such that $\sum_{P\in\P} (M_P-m_P)\Delta y_P < \epsilon$.  Since $f$ is computably continuous, the condition \eqref{eq:ULcond} can be effectively verified, so we may construct a computable enumeration of all pairs $(\P, \{(m_P,M_P)\}_{P\in\P})$ which satisfy \eqref{eq:ULcond}.  For each such pair in this enumeration, check if
\begin{equation}\label{eq:integStop}
\text{$\sum_{P\in\P} m_P \Delta y_P$ and $\sum_{P\in\P} M_P \Delta y_P$ are in $I$,}
\end{equation}
and stop if such a pair satisfying \eqref{eq:integStop} is found.  Note that this algorithm stops if and only if $F(B)\subseteq I$.
\end{proof}

\begin{proposition}\label{prop:compDivVar}
Let $U\subseteq\RR^{n+1}$, and write $(x,y) = (x_1,\ldots,x_n,y)$ for coordinates on $\RR^{n+1}$.  Let $p\in\NN\cup\{\infty\}$, and suppose that $f:U\to\RR$ is a computably $C^{p+1}$ function such that $f(x,0) = 0$ on $\{x : (x,0)\in U\}$ (where $\infty + 1 = \infty$).  Then there exists a unique computably $C^p$ function $g:U\to\RR$ such that
\[
f(x,y) = y g(x,y)
\]
on $U$.
\end{proposition}

\begin{proof}
The uniqueness of $g$ is apparent, since there is at most one continuous function $g:U\to\RR$ satisfying $g(x,y) = f(x,y)/y$ for all $(x,y)\in U$ with $y\neq 0$.  To construct $g$, it suffices to fix a compact rational box $B\subseteq U$ and prove that there is a computably $C^p$ function  $g:B\to\RR$ such that $f(x,y) = y g(x,y)$ on $B$.

If $B\cap(\RR^n\times\{0\}) = \emptyset$, define $g:B\to\RR$ by $g(x,y) = f(x,y)/y$, and note that $g$ is computably $C^p$ by Propositions \ref{prop:compArith}-\ref{prop:compComp}. So suppose that $B\cap(\RR^n\times\{0\}) \neq \emptyset$, and define $g:B\to\RR$ by $g(x,y) = \int_{0}^{1}\PD{}{f}{y}(x,ty)dt$.  Then
\[
y g(x,y) = \int_{0}^{1} y \PD{}{f}{y}(x,ty) dt = \int_{0}^{y} \PD{}{f}{y}(x,s) ds = f(x,y) - f(x,0) = f(x,y).
\]
The function $g$ is computably $C^p$ by Lemma \ref{lemma:compInteg}.
\end{proof}

The final goal of the section is to prove an effective version of the implicit function theorem.  Consider a $C^1$ function $f=(f_1,\ldots,f_n):U\to\RR^n$, where $U$ is an open neighborhood of the origin in $\RR^m\times\RR^n$, and write $(x,y) = (x_1,\ldots,x_m,y_1,\ldots,y_n)$ for coordinates on $\RR^m\times\RR^n$.  For all $r\in(0,+\infty)^m$ and $s\in(0,+\infty)^n$ such that $[-r,r]\times[-s,s]\subseteq U$, we shall define a statement $\IF(f;r,s)$ which satisfies the following lemma.

\begin{lemma}\label{lemma:IF}
The statement $\IF(f;r,s)$ has the following properties.
\begin{enumerate}{\setlength{\itemsep}{5pt}
\item
If $\IF(f;r,s)$ holds, then the set
\begin{equation}\label{eq:IFgraph}
\{(x,y)\in[-r,r]\times[-s,s] : f(x,y) = 0\}
\end{equation}
is the graph of a $C^1$ function from $[-r,r]$ into $(-s,s)$, and $\det\PD{}{f}{y}\neq 0$ on \eqref{eq:IFgraph}.

\item
Suppose $\IF(f;r,s)$ holds.  Then there exists an open box $A\subseteq(0,+\infty)^{m+n}$ containing $(r,s)$ such that $\IF(f;u,v)$ holds for all $(u,v)\in A$.  Moreover, for any open box $V\subseteq U$ containing the origin and any open box $B\subseteq(0,+\infty)^{m+n}$ containing $(r,s)$ such that $V+B\subseteq A$,  the statement $\IF(f_{(a,b)}; u,v)$ holds for all $(a,b)\in V$ and all $(u,v)\in B$, where $f_{(a,b)}(x,y) = f(x+a,y+b)$.\vspace*{3pt}

(Note that $V+B\subseteq A$ will hold for all sufficiently small $V$ and $B$.)

\item
If $f(0) = 0$ and $\det\PD{}{f}{y}(0)\neq 0$, then there exists $(r,s)\in\QQ_{+}^{m}\times\QQ_{+}^{n}$ such that $\IF(f;r,s)$ holds.

\item
If $f$ is computably $C^1$, $r\in\QQ_{+}^{m}$, $s\in\QQ_{+}^{n}$, and $\IF(f;r,s)$ holds, then we can effectively verify that $\IF(f;r,s)$ holds.  In other words, there is an algorithm which acts as follows:
\begin{quote}
Given a $C^1$ approximation algorithm for $f$ and $(r,s)\in\QQ_{+}^{m}\times\QQ_{+}^{n}$, the algorithm stops if and only if $[-r,r]\times[-s,s]\subseteq U$ and $\IF(f;r,s)$ holds.
\end{quote}
}
\end{enumerate}
\end{lemma}

Our definition for $\IF(f;r,s)$ is based on the inductive proof of the implicit function theorem for $f:U\to\RR^n$ for a general value of $n$ from the case of $n=1$.  Thus we shall define $\IF(f;r,s)$ and prove Lemma \ref{lemma:IF} together by induction on $n$.  If one wanted to take a noninductive approach, one could instead base the definition of $\IF(f;r,s)$ on the proof of the implicit function which uses the contraction mapping principle (for example, see Rudin \cite{Rudin_baby}), such as done in McNicholl \cite{McNicholl} to prove an effective version of the implicit function theorem.  We use an inductive approach because it is rather intuitive in nature and relates more naturally with our definition of an implicit function system given in Section \ref{s:IFsystem}.

\begin{definition}[The Base Case]\label{def:IF}
If $n=1$, then $\IF(f;r,s)$ means that there exists $\sigma\in\{-1,1\}$ such that
$\sigma\cdot\PD{}{f}{y}(x,y)> 0$ for all $(x,y)\in[-r,r]\times[-s,s]$, and such that $\sigma\cdot f(x,-s) < 0 < \sigma\cdot f(x,s)$ for all $x\in[-r,r]$.
\end{definition}

\begin{proof}[Proof of Lemma \ref{lemma:IF}, the Base Case]
Assume $\IF(f;r,s)$ holds.  The intermediate value and increasing function theorems show that \eqref{eq:IFgraph} is the graph of a function from $[-r,r]$ into $(-s,s)$.  The implicit function theorem shows that this function is $C^1$, which proves 1.  Clauses 2 follows easily from the continuity of $f$ and $\PD{}{f}{y}$.  To prove 3, let $\sigma = \sign \PD{}{f}{y}(0)$, and note that since $\PD{}{f}{y}$ is continuous, there exists $s\in\QQ_+$ such that $\sigma \PD{}{f}{y}(0,y) > 0$ on $[-s,s]$. Hence, $\sigma f(0,-s) < 0 < \sigma f(0,s)$ because $f(0) = 0$.  Since $f$ and $\PD{}{f}{y}$ are continuous, there exists $r\in\QQ_{+}^{m}$ such that $\sigma \PD{}{f}{y}(x,y) > 0$ on $[-r,r]\times[-s,s]$ and such that $\sigma f(x,-s) < 0 < \sigma f(x,s)$ on $[-r,r]$, which proves 3.  Clause 4 is apparent.
\end{proof}

\noindent{\bf Definition \ref{def:IF}}\, (The Inductive Step){\bf .}
Now suppose that $n > 1$.  Then $\IF(f;r,s)$ means that there exist  $i,j\in\{1,\ldots,n\}$ such that, if we write
\begin{eqnarray*}
y' & = & (y_1, \ldots, y_{j-1}, y_{j+1}, \ldots, y_n), \\
s' & = & (s_1, \ldots, s_{j-1}, s_{j+1}, \ldots, s_n), \\
f' & = & (f_1, \ldots, f_{i-1}, f_{i+1}, \ldots, f_n),
\end{eqnarray*}
then $\IF(f_i;(r,s'), s_j)$ and $\IF(f'\circ H;r,s')$ both hold, where $H$ is defined as follows:
\begin{quote}
Since $\IF(f_i;(r,s'), s_j)$ holds, the base case of the induction shows that there exist tuples $R > r$ and $S' > s'$ such that $\IF(f_i;(R,S'),s_j)$ holds.  Let $U' = (-R,R)\times(-S',S')$, and let $h:U'\to(-s_j,s_j)$ be the $C^1$ function whose graph is the set
\[
\{(x,y',y_j)\in U'\times[-s_j,s_j] :f_i(x,y) = 0\}.
\]
Define $H:U'\to\RR^m\times\RR^n$ by
\[
H(x,y') = (x,y_1,\ldots,y_{j-1},h(x,y'),y_{j+1},\ldots,y_n).
\]
\end{quote}
We will also have use for the function $H':U'\to\RR^n$ defined by
\[
H'(x,y') = (y_1,\ldots,y_{j-1},h(x,y'),y_{j+1},\ldots,y_n).
\]

\begin{proof}[Proof of Lemma \ref{lemma:IF}, the inductive step]
Let $n > 1$, and assume the lemma holds for any $k\in\{1,\ldots,n-1\}$ in place of $n$.  We use the notation from the inductive step of Definition \ref{def:IF}.

We now prove 1.  Assume $\IF(f;r,s)$ holds.   The base case shows that $h$ is $C^1$, so $f'\circ H$ is $C^1$.  Since $\IF(f'\circ H;r,s')$ holds, the induction hypothesis shows that $\{(x,y') \in [-r,r]\times[-s',s'] : f'\circ H(x,y') = 0\}$ is the graph of a $C^1$ function $g':[-r,r]\to(-s',s')$.  The function $g:[-r,r]\to(-s,s)$ defined by $g(x) = H'\circ g'(x)$ is $C^1$, and \eqref{eq:IFgraph} is the graph of $g$.  This proves the first part of 1.

We need to prove that $\det\PD{}{f}{y}\neq 0$ on \eqref{eq:IFgraph}.  We have $\PD{}{f_i}{y_j}\neq 0$ on $[-r,r]\times[-s,s]$ by the definition of the base case, and $\det\PD{}{(f'\circ H)}{y'}(x,g'(x))\neq 0$ on $[-r,r]$ by the induction hypothesis.  Therefore
\[
0 \neq \det\PD{}{(f'\circ H)}{y'}(x,g'(x))
=
\left(\det\PD{}{f'}{y'}\circ H(x,g'(x))\right)
\left(\det\PD{}{H}{y'}(x,g'(x))\right),
\]
which shows that $\det\PD{}{H}{y'}(x,g'(x)) \neq 0$ on $[-r,r]$.  Now, differentiating the equation
\[
f_i(x,y_{<j},h(x,y'),y_{>j}) = 0
\]
shows that
\[
\PD{}{h}{y'}(x,y') = \left.
- \frac{\PD{}{f_i}{y'}(x,y)}{\PD{}{f_i}{y_j}(x,y)}\right|_{y_j = h(x,y')},
\]
so
\begin{eqnarray*}
\PD{}{H}{y'}(x,y')
    & = &
    \left.
    \PD{}{f'}{y'}(x,y) + \PD{}{f'}{y_j}(x,y) \PD{}{h}{y'}(x,y')
    \right|_{y_j = h(x,y')} \\
    & = &
    \frac{1}{\PD{}{f_i}{y_j}(x,y)}
    \left.\left(
    \PD{}{f'}{y'}(x,y) \PD{}{f_i}{y_j}(x,y)
    - \PD{}{f'}{y_j}(x,y) \PD{}{f_i}{y'}(x,y)
    \right)\right|_{y_j = h(x,y')}.
\end{eqnarray*}
Thus on \eqref{eq:IFgraph} we have
\begin{eqnarray}\label{eq:rowOps}
\det\PD{}{f}{y}
    & = &
    \det
    \left(\begin{matrix}
    \PD{}{f'}{y'}  & \PD{}{f'}{y_j} \\
    \PD{}{f_i}{y'} & \PD{}{f_i}{y_j}
    \end{matrix}\right)
    \\
    & = &
    \det
    \left(\begin{matrix}
    \PD{}{f'}{y'} \PD{}{f_i}{y_j}  & \PD{}{f'}{y_j} \PD{}{f_i}{y_j}\\
    \PD{}{f_i}{y'}                  & \PD{}{f_i}{y_j}
    \end{matrix}\right)
    \nonumber\\
    & = &
    \det
    \left(\begin{matrix}
    \PD{}{f'}{y'} \PD{}{f_i}{y_j} - \PD{}{f'}{y_j}\PD{}{f_i}{y'}  & 0 \\
    \PD{}{f_i}{y'}                                                &\PD{}{f_i}{y_j}
    \end{matrix}\right)
    \nonumber\\
    & = &
    \left(\det\PD{}{H}{y'}\right) \left(\det\PD{}{f_i}{y_j}\right)^2
    \nonumber\\
    & \neq &
    0,
    \nonumber
\end{eqnarray}
which proves 1.

We now prove 2.  Assume $\IF(f;r,s)$ holds.  By the base case of the induction, there exists an open box $A\subseteq (0,+\infty)^{m+n}$ containing $(r,s)$ such that $\IF(f_i;(u,v'),v_j)$ holds for all $(u,v)\in A$.  Let $A' = \{(x,y') : (x,y)\in A\}$. By applying the induction hypothesis and by shrinking $A'$, the statement $\IF(f'\circ H; u,v')$ holds for all $(u,v')\in A'$.  Therefore $\IF(f;u,v)$ holds for all $(u,v)\in A$.  Now, fix an open box $V\subseteq U$ containing the origin and an open box  $B\subseteq(0,+\infty)^{m+n}$ containing $(r,s)$ such that $V+B\subseteq A$.  Fix $(a,b)\in V$ and $(u,v)\in B$.  By the induction hypothesis, $\IF((f_i)_{(a,b)};(u,v'),v_j)$ and $\IF((f'\circ H)_{(a,b')};u,v')$ both hold.  If we define $h_{(a,b)}$ and $H_{(a,b')}$ from $(f_i)_{(a,b)}$ in the same way that $h$ and $H$ are defined from $f_i$, then for all $(x,y)\in [-u,u]\times[-v,v]$,
\begin{eqnarray*}
y_j = h_{(a,b)}(x,y')
    & \text{iff} &
    (f_i)_{(a,b)}(x,y) = 0, \\
    & \text{iff} &
    f_i(x+a,y+b) = 0, \\
    & \text{iff} &
    y_j + b_j = h(x+a,y'+b'),
\end{eqnarray*}
so $h_{(a,b)}(x,y') + b_j = h(x+a,y'+b')$, and hence
\begin{eqnarray*}
(f'\circ H)_{(a,b')}(x,y')
    & = &
    f'(x+a,y_{<j}+b_{<j}, h(x+a,y'+b'), y_{>j} + b_{>j}) \\
    & = &
    f'(x+a,y_{<j}+b_{<j}, h_{(a,b)}(x,y') + b_j, y_{>j} + b_{>j}) \\
    & = &
    f'_{(a,b)}\circ H_{(a,b)}(x,y').
\end{eqnarray*}
This shows that $(f'\circ H)_{(a,b')} = f'_{(a,b)}\circ H_{(a,b')}$.  So $\IF((f_i)_{(a,b)};u,v)$ and $\IF(f'_{(a,b)}\circ H_{(a,b')};u,v')$ both hold, and hence $\IF(f_{(a,b)};u,v)$ holds.  This proves 2.

We now prove 3.  Suppose that $f(0) = 0$ and $\PD{}{f}{y}(0)$ is nonsingular.  Then there exist $i,j\in\{1,\ldots,n\}$ such that $\PD{}{f_i}{y_j}(0)$ and $\PD{}{f'}{y'}(0)$ are nonsingular.  We can use $f_i$ to define $H$ near $0$ by the base case, and clearly $H(0) = 0$.  The computation \eqref{eq:rowOps} shows that $(\det\PD{}{H}{y'}(0))(\det\PD{}{f_i}{y_j}(0))^2 = \det\PD{}{f}{y}(0) \neq 0$, so $\PD{}{H}{y'}(0)$ is nonsingular.  Therefore $\PD{}{(f'\circ H)}{y'}(0)$ is nonsingular since $\PD{}{(f'\circ H)}{y'}(0) = \PD{}{f'}{y'}(0) \PD{}{H}{y'}(0)$.  The base case shows that $\IF(f_i;(r,s'),s_j)$ holds for some $(r,s)\in\QQ^{m}_{+}\times\QQ^{n}_{+}$, and the induction hypothesis shows that by shrinking $(r,s')$, $\IF(f'\circ H;r,s')$ also holds.  Thus $\IF(f;r,s)$ holds, which proves 3.

To prove 4, simply note that $f_i$ is computably $C^1$, so $h$ is computably $C^1$ by the base case, and hence so is $f'\circ H$ by Proposition \ref{prop:compComp}.  So the induction hypothesis shows that we can effectively verify that $\IF(f_i;(r,s'),s_j)$ and $\IF(f'\circ H;r,s')$ are true, which verifies that $\IF(f;r,s)$ is true.
\end{proof}

\begin{proposition}\label{prop:compIFT}
Let $f:U\to\RR^n$ be computably $C^p$, where $p\in\{1,2,3,\ldots,\}\cup\{\infty\}$ and $U$ is a neighborhood of the origin in $\RR^m\times\RR^n$.  If $\IF(f;r,s)$ holds, then the function $g:[-r,r]\to(-s,s)$ whose graph is
\[
\{(x,y)\in[-r,r]\times[-s,s] : f(x,y) = 0\}
\]
is computably $C^p$.
\end{proposition}

\begin{proof}
The proof is by induction on $n$.  Suppose $n=1$.  We may assume without loss of generality that $\PD{}{f}{y}(x,y) > 0$ on $[-r,r]\times[-s,s]$, and that $f(x,-s) < 0$ and $f(x,s) > 0$ on $[-r,r]$.  For any compact rational box $B \subseteq [-r,r]$ and rational interval $(a,b) \subseteq (-s,s)$, $g(B) \subseteq (a,b)$ if and only if
\begin{equation}\label{eq:fab}
\text{$f(x,a) < 0$ and $f(x,b) > 0$ for all $x\in B$.}
\end{equation}
Since $f$ is computably continuous and $a$ and $b$ are rational numbers, Proposition \ref{prop:compComp} implies that $x\mapsto f(x,a)$ and $x\mapsto f(x,b)$ are computably continuous.  So, if $g(B)\subseteq (a,b)$, then a $C^0$ approximation algorithm for $f$ can verify \eqref{eq:fab}.  Thus $g$ is computably continuous. It now follows from Propositions \ref{prop:compArith}-\ref{prop:compComp} and repeated differentiation of the equation $f(x,g(x)) = 0$ that $g$ is computably $C^p$.

Now suppose that $n>1$.  We use the notation of the inductive step of Definition \ref{def:IF}.  The base case of the induction shows that function $h$ is computably $C^p$, so $f'\circ H$ is as well by Proposition \ref{prop:compComp}.  Therefore the induction hypothesis shows that the function $g':[-r,r]\to(-s',s')$ implicitly defined by the equation $f'\circ H(x,y') = 0$ is computably $C^p$, and so by Proposition \ref{prop:compComp} again, the function $g = H'\circ g':[-r,r]\to(-s,s)$ implicitly defined by the equation $f(x,y) = 0$ is computably $C^p$.
\end{proof}

In Definition \ref{def:IF} we considered boxes $[-r,r]\times[-s,s]$ centered about the origin, but this was only done for convenience of notation.  In Definitions \ref{def:IFbox} and \ref{def:IFsection} below, we now define two modifications of the statement $\IF(f;r,s)$.

\begin{definition}\label{def:IFbox}
Let $f:U\to\RR^n$ be a $C^1$ function on an open set $U\subseteq \RR^{m+n}$, and  let $A\subseteq\RR^m$ and $B\subseteq\RR^n$ be a nondegenerate compact boxes such that $A\times B \subseteq U$.  The statement $\IF(f;A,B)$ is defined as follows:
\begin{quote}
Fix $a\in\RR^m$, $b\in\RR^n$, $r\in(0,+\infty)^m$, and $b\in(0,+\infty)^n$ such that $A = [a-r,a+r]$ and $B = [b-s,b+s]$, and define $T_{(a,b)}:\RR^{m+n}\to\RR^{m+n}$ by $T_{(a,b)}(x,y) = (x+a,y+b)$.  The statement $\IF(f;A,B)$ means $\IF(f\circ T_{(a,b)};r,s)$.
\end{quote}
\end{definition}
The obvious modifications of Lemma \ref{lemma:IF} and Proposition \ref{prop:compIFT} hold for the statement $\IF(f;A,B)$.

\begin{definition}\label{def:ratBoxManifold}
A set $M\subseteq\RR^m$ is a called a {\bf rational box manifold} if it is a rational box and a submanifold of $\RR^m$.
\end{definition}

Thus $M\subseteq\RR^m$ is a rational box manifold if and only if there exist $E\subseteq\{1,\ldots,m\}$, an open rational box  $U\subseteq\RR^E$, and a point $a\in\QQ^{E^c}$ such that $M = U\times\{a\}$.  Note that $\dim(M) = |E|$.

The purpose of the next modification of the statement $\IF(f;r,s)$ is to obtain a form of the implicit function theorem which is suitable for use by the precision oracle for $\S$, which will be defined in Section \ref{s:oracles}.

\begin{definition}\label{def:IFsection}
Consider a rational box manifold $M\subseteq\RR^m$, $d\in\{0,\ldots,\dim(M)\}$, a $C^1$ map $f:M\to\RR^{\dim(M) - d}$, an injection $\lambda:\{1,\ldots,d\}\to\{1,\ldots,m\}$ such that $\Pi_\lambda(M)$ is open in $\RR^d$, and a bounded rational box manifold $B$ which is open in $M$ with $\cl(B)\subseteq M$.  Define the statement $\IF_\lambda(f;B)$ as follows:
\begin{quote}
Write $M = U\times\{a\}$ for a set $E\subseteq\{1,\ldots,m\}$, open rational box  $U\subseteq\RR^E$, and $a\in\QQ^{E^c}$.  Write $B = (b-R,b+R)\times\{a\}$ for some $b\in\QQ^E$ and $R\in\QQ_{+}^{E}$ such that $[b-R,b+R] \subseteq U$.  Define $T_b:\RR^E \to \RR^m$ by $T_b(x_E) = (x_E+b,a)$.  Note that $\im(\lambda)\subseteq E$.  Extend $\lambda$ to a bijection $\sigma:\{1,\ldots,\dim(M)\}\to E$, and write $\Pi_\sigma(R) = (r,s)$, where $r\in\QQ_{+}^{d}$ and $s\in\QQ_{+}^{\dim(M)-d}$.
The statement $\IF_\lambda(f;B)$ means $\IF(f\circ T_b\circ \Pi_{\sigma}^{-1}; r,s)$.
\end{quote}
\end{definition}

The following remarks follow directly from Lemma \ref{lemma:IF}, Proposition \ref{prop:compIFT}, and Definition \ref{def:IFsection}.

\begin{remarks}\label{rmk:IFsection}
Consider $E\subseteq\{1,\ldots,m\}$, $M = U\times\{a\}\subseteq\RR^m$ for an open rational box $U\subseteq\RR^E$ and $a\in\QQ^{E^c}$, a $C^1$ function $f:M\to\RR^{\dim(M)-d}$, an injection $\lambda:\{1,\ldots,d\}\to\{1,\ldots,m\}$ with $\im(\lambda)\subseteq E$, and a bijection $\sigma:\{1,\ldots,\dim(M)\}\to E$ extending $\lambda$.   Define $\lambda':\{1,\ldots,\dim(M)-d\}\to\{1,\ldots,m\}$ by $\lambda'(i) = \sigma(i+d)$ for all $i$.
\begin{enumerate}{\setlength{\itemsep}{3pt}
\item
Write $B = V\times\{a\}$ for an open rational box $V\subseteq \RR^E$ with $\cl(V)\subseteq U$, and suppose that $\IF_\lambda(f;B)$ holds.
\begin{enumerate}{\setlength{\itemsep}{3pt}
\item
There exists a $C^1$ section $\varphi:\Pi_\lambda(B)\to B$ of the projection $\Pi_\lambda:\RR^m\to\RR^d$ such that
\[
\im(\varphi) = \{x\in M : f(x) = 0\},
\]
and $\det\PD{}{f}{x_{\lambda'}}(x)\neq 0$ for all $x\in\im(\varphi)$.

\item
There exists $\epsilon > 0$ such that for all $B' = V'\times\{a\}$, where $V'\subseteq\RR^E$ is an open rational box, if $\bd(V') \subseteq \{x\in\RR^E : \dist(x,\bd(V)) < \epsilon\}$, then $\cl(V') \subseteq U$ and $\IF_\lambda(f;B')$ holds.

\item
If $f$ is computably $C^1$, then we can effectively verify that $\IF_\lambda(f;B)$ holds.
}\end{enumerate}

\item
If $b\in U$ is such that $f(a,b) = 0$ and $\det\PD{}{f}{x_{\lambda'}}(a,b)\neq 0$, then there exists $R\in(0,+\infty)^E$ such that $[b-R,b+R]\subseteq U$ and $\IF_\lambda(f; \{a\}\times(b-R,b+R))$ holds.
}\end{enumerate}
\end{remarks}

Simply put, Remark 1(a) says that the equation $f(x) = 0$ defines the image of $\varphi$ nonsingularly as a subset of $\RR^E\times\{a\}$.  If $f$ happened to extend to a function on $U\times W$ for some open set $W\subseteq\RR^{E^c}$ containing $a$, then we could implicitly define $\im(\varphi)$ as a subset of $\RR^m$ by writing
\[
\im(\varphi) = \{x\in U\times W : f(x) = 0, x_{E^c} - a = 0\},
\]
and noting that $\det\PD{}{(f(x), x_{E^c} - a)}{(x_{\lambda'},x_{E^c})}(x) \neq 0$ for all $x\in\im(\varphi)$.  However, this cannot be done if we do not have access to such an extension, which is a situation that will occur when studying the structure $\RR_{\S}$, since the functions in $\S$ are defined on compact rational boxes.  Even though the functions in $\S$ extend in neighborhoods of their domains to functions in $\C$, we do not have access to these extensions in $\S$ when studying $\RR_{\S}$.  So instead we note that if $f:[-r,r]\to\RR$ is in $\S$, where $r\in\QQ_{+}^{n}$, then each set $M$ in the natural stratification of $[-r,r]$ is a rational box manifold.  By using this observation, we will be able to work with the form of the implicit function theorem given in Definition \ref{def:IFsection}.

\section*{\bf Part II: Tools for Desingularization}

Part II develops the basic tools we need for our local resolution of singularities procedure.  Section \ref{s:IFsystem} discusses quasianalytic implicit function systems, which are the ambient quasianalytic classes in which we work.  Section  \ref{s:order} gives some basic facts about the order of a function at a point and along a coordinate submanifold.  Section \ref{s:linTrans} contains a computational result about certain linear transformations that we will use when constructing our centers of blowing-up.  And finally, Section \ref{s:blowup} gives numerous results we shall need about blowings-up.

\section{Implicit Function Systems}\label{s:IFsystem}

\begin{definition}\label{def:IFsystem}
Let $\C = \bigcup_{n\in\NN, r\in\QQ_{+}^{n}} \C_r$, where each $\C_r$ is a collection of $C^\infty$ functions from $[-r,r]$ into $\RR$. We are using the convention that $\QQ_{+}^{0} = \{0\}$ and that functions from $[-0,0] = \{0\}$ into $\RR$ are identified with real numbers.  So $\C_0\subseteq\RR$, and $\C = \C_0 \cup \bigcup_{n>0,r\in\QQ_{+}^{n}}\C_r$.  We say that $\C$ is an {\bf implicit function system} (or an {\bf IF-system}, for short) if the following hold:
\begin{enumerate}{\setlength{\itemsep}{3pt}
\item
$\C_0$ is a field.

\item
The following hold for all $n > 0$ and $r\in\QQ_{+}^{n}$:
\begin{enumerate}{\setlength{\itemsep}{3pt}
\item
$\C_r$ is a ring containing the coordinate projection functions $x\mapsto x_i$, for each $i\in\{1,\ldots,n\}$.

\item
\emph{Extension Property}:

For each $f\in\C_r$ there exist $s>r$ and $g\in\C_s$ such that $f(x) = g(x)$ for all $x\in[-r,r]$.
}
\end{enumerate}

\item
The following hold for all $m,n\in\NN$, $r\in\QQ_{+}^{m}$, and $s\in\QQ_{+}^{n}$:
\begin{enumerate}{\setlength{\itemsep}{3pt}
\item
\emph{Closure under Composition}:

If $g\in \C_{r}^{n}$ is such that $g([-r,r])\subseteq[-s,s]$ and $f\in\C_s$, then $f\circ g \in \C_r$.

\item
\emph{Closure under Division by Variables}:

If $f\in\C_r$ is such that $f(x_1,\ldots,x_{m-1},0) = 0$ on $[-r_1,r_1]\times \cdots \times[-r_{m-1},r_{m-1}]$, then there exists $g\in\C_r$ such that $f(x) = x_m g(x)$ on $[-r,r]$.

\item
\emph{Closure under Implicit Functions}:

If $n=1$, and if $f\in\C_{(r,s)}$ is such that $\IF(f;r,s)$ holds, then the function $g:[-r,r]\to(-s,s)$ implicitly defined by $f(x,g(x)) = 0$ on $[-r,r]$ is in $\C_r$.
}
\end{enumerate}
}
\end{enumerate}
\end{definition}

\begin{definition}\label{def:quasianal}
An IF-system $\C$ is {\bf quasianalytic} if for all $n > 0$, $r\in\QQ_{+}^{n}$, and $a\in[-r,r]$, the Taylor map from $\C_r$ into $\RR\ps{x}$ given by
\[
f \mapsto \sum_{\alpha\in\NN^n} \frac{1}{\alpha!} \PDn{\alpha}{f}{x}(a) x^\alpha
\]
is injective.
\end{definition}

\begin{definition}\label{def:Canalytic}
Consider an IF-system $\C$ and a function $f:U\to\RR^m$ defined on an open set $U \subseteq \RR^n$.  We say that $f$ is {\bf $\C$-analytic at $a\in U$} if there exist $b\in\QQ^n$ and $r\in\QQ_{+}^{n}$ such that $a\in (b-r,b+r)$, $[b-r,b+r]\subseteq U$, and the function $[-r,r]\to\RR^m : x\mapsto f(x+b)$ is in $\C_{r}^{m}$.  We say that $f$ is {\bf $\C$-analytic} if $f$ is $\C$-analytic at every point of $U$.  More generally, for any set $A\subseteq\RR^n$, we say that a function $g:A\to\RR^m$ is $\C$-analytic if there exists a $\C$-analytic function $h:V\to\RR^m$ defined on an open set $V$ such that $A\subseteq V\subseteq \RR^n$ and $g(x) = h(x)$ for all $x\in A$.  We write $\C[A]$ for the ring of all real-valued $\C$-analytic functions on $A$.
\end{definition}

Note that since an IF-system $\C$ has the extension property, every member of $\C$ is $\C$-analytic.

\begin{examples}\label{ex:IFsystems}
\hfill
\begin{enumerate}{\setlength{\itemsep}{5pt}
\item
If for each $r$ we let $\C_r$ be the set of all real-valued functions on $[-r,r]$ which extend to a $C^\infty$ function in a neighborhood of $[-r,r]$, then $\C = \bigcup_{r}\C_r$ is the largest IF-system.  This IF-system is not quasianalytic.

\item
If for each $r$ we let $\C_r$ be the set of all real-valued functions on $[-r,r]$ which extend to a computably $C^\infty$ function on a neighborhood of $[-r,r]$, then the results of Section \ref{s:compClos} show that $\C = \bigcup_{r}\C_r$ is an IF-system.  This IF-system is also not quasianalytic.  Note that $\C_0\neq \RR$ since $\C_0$ is the set of computable reals.

\item
If for each $r$ we let $\C_r$ be the set of all real-valued functions on $[-r,r]$ which extend to an analytic function in neighborhood of $[-r,r]$, then $\C = \bigcup_{r}\C_r$ is a quasianalytic IF-system.  This is, in fact, the prototypical example of a quasianalytic IF-system, and is why we use the word ``$\C$-analytic''.

\item
More generally, fix a sequence of real numbers $\{M_n\}_{n\in\NN}$ such that $1 \leq M_0 \leq M_1\leq M_2 \leq \cdots$, $M_{n}^{2} \leq M_{n-1} M_{n+1}$ for all $n > 0$, and $\sum_{n\in\NN}M_n = +\infty$.  Let $\C = \bigcup_{r}\C_r$, where $\C_r$ denotes the set of all real-valued functions on $[-r,r]$ which extend to a $C^\infty$ function $f:U\to\RR$ on a neighborhood $U$ of $[-r,r]$ for which there exist positive constants $A$ and $B$ such that
\[
\left|\PDn{\alpha}{f}{x}(x)\right| \leq A B^{|\alpha|} M_{|\alpha|}
\]
for all $\alpha\in\NN^n$ and all $x\in U$.  Then $\C$ is a quasianalytic IF-system, called the Denjoy-Carleman class determined by the sequence $\{M_n\}_{n\in\NN}$.  When $M_n = n!$ for each $n\in\NN$, $\C$ is the IF-system of analytic functions given in the previous example.  (See \cite{RSW} for references.)

\item
If for each $r$ we let $\C_r$ be the set of all real-valued functions on $[-r,r]$ which extend to a function $f:U\to\RR$ defined on a neighborhood $U$ of $[-r,r]$ such that $f$ is $C^\infty$ and algebraic over $\QQ[x]$ (equivalently, $C^\infty$ and $0$-definable in the real field), then $\C = \bigcup_{r}\C_r$ is an IF-system.  Every function in $\C$ is analytic (see \cite{BCR}, the section on Nash functions), so $\C$ is quasianalytic.  Note that $\C_0$ is the field of algebraic reals, so $\C_0\neq \RR$.

\item
More generally, if we fix an expansion $\R$ of the real field, and  if for each $r$ we let $\C_r$ be the collection of all real-valued functions on $[-r,r]$ which extend to a function $f:U\to\RR$ defined in a neighborhood $U$ of $[-r,r]$  such that $f$ is $C^\infty$ and $0$-definable in $\R$, then $\C=\bigcup_{r}\C_r$ is an IF-system.  The set $\C_0$ is the set of $0$-definable constants of $\R$, so $\C_0$ need not equal $\RR$.  If $\R$ is an o-minimal structure, then $\C$ is quasianalytic if and only if $\R$ is polynomially bounded (see C. Miller \cite{CM95}).  In light of this result, if $\R$ is polynomially bounded and o-minimal, it is natural to wonder when $\R$ is definably equivalent to $\RR_\C$.  This relates to Theorem \ref{introThm:OminPolyUnifChar}, which will be proven in the last section in Theorem \ref{thm:strucChar}.  (Note: The proof of Theorem \ref{thm:strucChar} could be read right away.  It is simple and is independent of the rest of the paper.)
}
\end{enumerate}
\end{examples}

We now list some useful closure properties of IF-systems.

\begin{remarks}\label{rmk:IFsystem}
Let $\C$ be an IF-system.
\begin{enumerate}
\item
Let $r\in\QQ^{n}_{+}$, $f\in\C_r$, and $i\in\{1,\ldots,n\}$.  Then $\PD{}{f}{x_i}\in\C_r$.

\begin{proof}
Let $s>r$ and $g\in\C_s$ be such that $f(x) = g(x)$ for all $x\in[-r,r]$.  Then $H(x,y) = g(x_1,\ldots,x_{i-1}, x_i + y, x_{i+1},\ldots,x_n) - g(x)$ is in $\C_{(r,s_i-r_i)}$, and $H(x,0) = 0$, so $H(x,y) = y h(x,y)$ for some $h\in\C_{(r,s_i-r_i)}$.  Therefore $\PD{}{f}{x_i}(x) = h(x,0)$ is in $\C_r$.
\end{proof}

\item
Let $f\in\C_r$ be such that $f(x)\neq 0$ for all $x\in[-r,r]$, where $r\in\QQ_{+}^{n}$.  Then $1/f\in\C_r$.

\begin{proof}
We may assume without loss of generality that $f(x) > 0$ for all $x\in[-r,r]$. Fix a rational number $s > \max\{\frac{1}{f(x)} : x\in[-r,r]\}$, and define $g\in\C_{(r,s)}$ by
\[
g(x,y) = y f(x) - 1.
\]
For all $(x,y)\in[-r,r]\times[-s,s]$, we have $\PD{}{g}{y}(x,y) = f(x) > 0$, $g(x,s) = s f(x) - 1 > 0$, and $g(x,-s) = -sf(x) - 1 < 0$.  So $\IF(g;r,s)$ holds.  Since $\{(x,y)\in[-r,r]\times[-s,s] : g(x,y) = 0\}$ is the graph of $1/f$, we have $1/f\in\C_r$.
\end{proof}

\item
Let $(r,s)\in\QQ_{+}^{m}\times\QQ_{+}^{n}$ and $f\in\C_{(r,s)}$ be such that $\IF(f;r,s)$ holds, and let $g:[-r,r]\to(-s,s)$ be the $C^1$ function whose graph is $\{(x,y)\in[-r,r]\times[-s,s] : f(x,y)=0\}$.  Then $g\in\C_{r}^{n}$.

\begin{proof}
This follows from the inductive step in the proof of Lemma \ref{lemma:IF}.1, since that proof only uses the fact the $C^1$ functions are closed under composition and implicitly defined functions when $n=1$, and $\C$ has these closure properties.
\end{proof}

\item
Let $f:U\to\RR^n$ be $\C$-analytic, where $U$ is an open set in $\RR^m\times\RR^n$, and suppose that $(a,b)\in U$ is such that $f(a,b) = 0$ and $\det \PD{}{f}{y}(a,b) \neq 0$, where $x=(x_1,\ldots,x_m)$ and $y=(y_1,\ldots,y_n)$.  Let $g$ be the $C^\infty$ function implicitly defined in a neighborhood of $(a,b)$ by the equations $f(x,g(x)) = 0$ and $g(a) = b$.  Then $g$ is $\C$-analytic at $(a,b)$.

\begin{proof}
It follows from clauses 2 and 3 of Lemma \ref{lemma:IF} that there exist $(r,s)\in\QQ_{+}^{m}\times\QQ_{+}^{n}$ and $(c,d)\in\QQ^m\times\QQ^n$ such that $(a,b) \in (c-r,c+r)\times(d-s,d+s)$ and $\IF(f_{(c,d)};r,s)$ holds, where $f_{(c,d)}(x,y) = f(x+c,y+d)$.  Since $f_{(c,d)}\in\C_{(r,s)}$, we are done by the previous remark.
\end{proof}

\item
If $f:U\to(0,+\infty)$ is a $\C$-analytic function on an open set $U$ in $\RR^n$, and $q\in\QQ$, then $f^q$ is $\C$-analytic on $U$.

\begin{proof}
Because $\C$-analytic functions are closed under multiplication, and because of Remark 2 above, we may assume without loss of generality that $q = 1/k$ for a positive integer $k$.  The function $g:U\times(0,+\infty)\to\RR$ defined by $g(x,y) = y^k - f(x)$ is $\C$-analytic, $\PD{}{g}{y}(x,y) = k y^{k-1} > 0$ for all $y>0$, and $g(x,f(x)^{1/k}) = 0$ for all $x\in U$. So $f^{1/k}$ is $\C$-analytic on $U$ by the previous remark.
\end{proof}
\end{enumerate}
\end{remarks}

Because of these closure properties, we can do elementary differential geometry relative to the class of $\C$-analytic functions.  By simply using the word ``$\C$-analytic'' in place of ``$C^\infty$'' in the standard definitions from differential geometry (see Spivak \cite{spivak65} and \cite{spivak79}), we can define the notion of a ``$\C$-analytic submanifold of $\RR^n$'' (or more generally, an abstract $\C$-analytic manifold, which is a differential manifold with $\C$-analytic transition maps), a ``$\C$-analytic function'' on a $\C$-analytic manifold, a ``$\C$-analytic isomorphism'', and a ``$\C$-analytic embedding''.

We conclude this section with some simple observations about zero sets of $\C$-analytic functions when $\C$ is quasianalytic.

\begin{proposition}\label{prop:realZeros}
If $\C$ is quasianalytic, and $f:[a,b]\to\RR$ is a $\C$-analytic function which is not identically zero, then $\{x\in[a,b] : f(x) = 0\}$ is a finite set of points in $\C_0$.  Thus, $\C_0$ is a real-closed field, since all polynomial functions with coefficients in $\C_0$ are $\C$-analytic.
\end{proposition}

\begin{proof}
Fix $c\in[a,b]$ such that $f(c) = 0$.  Since $\C$ is quasianalytic, there exists $k\in\NN$ such that $f^{(k)}(c) = 0$  and $f^{(k+1)}(c)\neq 0$.  Since derivatives of $\C$-analytic functions are $\C$-analytic, we may simply assume that $f(c) = 0$ and $f'(c)\neq 0$.  Thus $c$ (considered as a function from $\{0\}$ into $\RR$) is in $\C_0$, by the Remark \ref{rmk:IFsystem}.4.  Since this argument also shows that $c$ is an isolated zero of $f$, the fact that $\{x\in[a,b] : f(x)=0\}$ is finite follows from the compactness of $[a,b]$.
\end{proof}

\begin{proposition}\label{prop:denseZeros}
Let $\C$ be quasianalytic, and let $\F$ be a finite set of real-valued $\C$-analytic functions on an open set $U$ in $\RR^n$, where $n\in\NN$.  Put
\[
A = \{x\in U : \text{$f(x) = 0$ for all $f\in\F$}\}.
\]
\begin{enumerate}{\setlength{\itemsep}{3pt}
\item
The set $U\setminus A$ is open in $U$.  If $A\neq U$ and $U$ is connected, then $U\setminus A$ is also dense in $U$.

\item
The set $A\cap \C_{0}^{n}$ is dense in $A$.
}\end{enumerate}
\end{proposition}

\begin{proof}
We first prove 1.  The set $U\setminus A$ is clearly open since the functions in $\F$ are continuous.  Now, suppose that $U$ is connected and that $U\setminus A$ is not dense in $U$.  Fix an open set $V\subseteq U$ such that $V\cap(U\setminus A) = \emptyset$.  Thus $ V\subseteq A$, so every function in $\F$ vanishes identically on $V$, and hence vanishes identically on $U$, since $U$ is connected and $\C$ is quasianalytic.  Thus $A = U$.  This proves 1.

We now prove 2 by generalizing the technique used to prove Proposition \ref{prop:realZeros}.  We proceed by induction on $n$.  The result is trivial when $n=0$,  so let $n > 0$ and assume the Proposition holds with $n-1$ in place of $n$.  There is nothing to prove if $A$ is empty, so we fix $a\in A$ and a connected open set $V$ such that $a\in V\subseteq U$; we must show that $A\cap V\cap\C_{0}^{n}\neq\emptyset$.  The result is trivial if every function in $\F$ is identically zero on $V$, so assume that we can fix $F\in\F$ which is is not identically zero on $V$.  Let $d = \inf\{|\alpha| : \alpha\in\NN^n, \PDn{\alpha}{F}{x}(a)\neq 0\}$, and note that $d < \infty$ because $\C$ is quasianalytic.  Since $a\in A$ and $F\in\F$, we have $F(a)=0$, so $d > 0$.  Fix $\alpha\in\NN^n$ such that $|\alpha| = d-1$ and $\PDn{\alpha+e_i}{F}{x}(a)\neq 0$ for some $i\in\{1,\ldots,n\}$, where $e_i$ is the $i$th standard unit vector in $\NN^n$.  Write $x' = (x_1,\ldots,x_{i-1},x_{i+1},\ldots,x_n)$ for coordinates on $\RR^{\{1,\ldots,n\}\setminus\{i\}}$.  Fix open sets $W\subseteq\RR^{\{1,\ldots,n\}\setminus\{i\}}$ and $I\subseteq\RR^{\{i\}}$ such that $a\in W\times I \subseteq V$ and such that $\{(x',x_i)\in W\times I : \PDn{\alpha}{F}{x}(x',x_i) = 0\}$ is the graph of a $\C$-analytic function $g:W\to I$.  Let
\[
B = \{x'\in W : \text{$f(x',g(x')) = 0$ for all $f\in\F$}\}.
\]
Note that $B\neq\emptyset$ since $a'\in B$.  By the induction hypothesis, $B\cap \C_{0}^{n-1}$ is dense in $B$.  Since $\{(x',g(x')) : x\in B\}\subseteq A\cap V$ and the map $x'\mapsto (x',g(x'))$ sends point in $\C_{0}^{n-1}$ to points in $\C_{0}^{n}$, it follows that $A\cap V\cap\C_{0}^{n}$ is nonempty.
\end{proof}

\section{Order}\label{s:order}

This section covers the basic facts we shall need about the order a function at a point and along a coordinate submanifold of $\RR^n$.  Throughout this section, $\C$ denotes an IF-system.

\begin{definition}\label{def:orderPoint}
Consider a real-valued $\C$-analytic function $f$ on an open set $U$ in $\RR^n$.  For any $a\in U$, the {\bf order of $f$ at $a$} is defined by
\[
\ord(f;a) = \inf\left\{|\alpha| : \alpha\in\NN^n, \PDn{\alpha}{f}{x}(a)\neq 0\right\},
\]
with the understanding that $\inf\emptyset = \infty$.
\end{definition}

\begin{proposition}\label{prop:order}
Let $f:U\to\RR$ and $g:U\to\RR$ be $\C$-analytic functions, and let $a\in U$.  In Statement 5, also let $h:V\to U$ be $\C$-analytic, and suppose that $b\in V$ is such that $h(b) = a$.  Then the following hold:
\begin{enumerate}{\setlength{\itemsep}{3pt}
\item
If $\C$ is quasianalytic, then $\ord(f;a) = \infty$ if and only if $f(x) = 0$ for all $x$ in the connected component of $U$ containing $a$.

\item
There exists a neighborhood $W$ of $a$ such that $\ord(f;x)\leq \ord(f;a)$ for all $x\in W$.  Moreover, the set $\{x\in W:\ord(f;x) = \ord(f;a)\}$ is closed in $W$.

\item
$\ord(f+g;a) \geq \min\{\ord(f;a), \ord(g;a)\}$, and equality holds if $\ord(f;a) \neq \ord(g;a)$.

\item
$\ord(fg;a) = \ord(f;a) + \ord(g;a)$.

\item
$\ord(f;a)\leq \ord(f\circ h;b)$, and equality holds if $m=n$ and $h$ is a $\C$-analytic isomorphism.
}
\end{enumerate}
\end{proposition}

\noindent We omit the proof of Proposition \ref{prop:order} since these properties are all very well-known.

\begin{definition}\label{def:orderAlong}
Let $\emptyset\neq C\subseteq U$ for an open set $U$ in $\RR^n$, and let $f:U\to\RR$ be $\C$-analytic.  Define the {\bf order of $f$ along $C$} by
\[
\ord_C(f) = \min\{\ord(f;x) : x\in C\}.
\]
If $a\in C$, define the {\bf order of $f$ along $C$ at $a$} by
\[
\ord_C(f;a) = \sup\{\ord_{C\cap V}(f) : \text{$V$ is a neighborhood of $a$ in $U$}\}.
\]
\end{definition}

\begin{remarks}\label{rmk:orderAlong}
Consider the situation of Definition \ref{def:orderAlong}.
\begin{enumerate}{\setlength{\itemsep}{5pt}
\item
If $V$ and $W$ are neighborhoods of $a\in U$ such that $V\subseteq W\subseteq U$, then $\ord_{C\cap W}(f) \leq \ord_{C\cap V}(f) \leq \ord(f;a)$.  It follows that if $\ord(f;a) < \infty$, or if $\ord(f;a) = \infty$ and $\C$ is quasianalytic, then $\ord_{C\cap V}(f) = \ord_C(f;a)$ for any sufficiently small neighborhood $V$ of $a$.

\item
If $F:V\to U$ is a $\C$-analytic isomorphism, Proposition \ref{prop:order}.5 shows that $\ord_C(f) = \ord_{F^{-1}(C)}(f\circ F)$.
}
\end{enumerate}
\end{remarks}

For simplicity, and because this is sufficient for our needs, we state the following proposition only for coordinate submanifolds of $\RR^n$.  It generalizes easily to arbitrary $\C$-analytic submanifolds of $U$ by using coordinate charts.  The content of the proposition is also rather well-known, but possibly not all of it to the extent of Proposition \ref{prop:order}, so we include a proof.

\begin{proposition}\label{prop:orderAlong}
Let $U\subseteq\RR^n$ be open and $f:U\to\RR$ be $\C$-analytic, and let $C=  \{x\in U : x_I = 0\}$ for some $I\subseteq\{1,\ldots,n\}$.
\begin{enumerate}{\setlength{\itemsep}{5pt}
\item
For every $k\leq \ord_C(f)$,
\begin{equation}\label{eq:orderAlong}
f(x) = \sum_{\alpha\in\NN^{I}_{k}} x_{I}^{\alpha} g_\alpha(x)
\end{equation}
on $U$, for some $\C$-analytic functions $g_\alpha :U\to\RR$ such that $g_\alpha(x) = \frac{1}{\alpha!}\PDn{\alpha}{f}{x_I}(x)$ for all $x\in C$.

\item
Let
\begin{eqnarray*}
d_1
    & = &
    \ord_C(f),
\\
d_2
    & = &
    \sup\{k\in\NN : f\in\lb x_I\rb^k\},
\\
d_3
    & = &
    \inf\left\{|\alpha| : \text{$\alpha\in\NN^I$, $\PDn{\alpha}{f}{x_I}(a)\neq 0$ for some $a\in C$}\right\},
\end{eqnarray*}
where $\lb x_I\rb$ is the ideal of $\C[U]$ generated by $\{x_i\}_{i\in I}$.  Then $d_1 = d_2 \leq d_3$.  If we additionally assume that $\C$ is quasianalytic, then $d_1 = d_2 = d_3$.

\item
The following two statements are equivalent if $\C$ is quasianalytic and $\ord_C(f) < \infty$:
\begin{enumerate}
\item
For all $a\in C$, $\ord(f;a) = \ord_C(f)$.

\item
For all $a\in C$ there exists $\alpha\in\NN^I$ such that $|\alpha| = \ord_C(f)$ and $\PDn{\alpha}{f}{x_I}(a)\neq 0$.
\end{enumerate}
}\end{enumerate}
\end{proposition}

\begin{proof}
We first prove 1.  Let $k\leq\ord_C(f)$.  We proceed by induction on $(n,k)$ ordered lexicographically.  There is nothing to prove if $n = 0$ or $k=0$, so we assume that $n$ and $k$ are both positive.  Choose $i\in I$, and let $I' = I\setminus\{i\}$.  Write $x = (x',x_i)$ for coordinates on $\RR^n$, where $x' = (x_1,\ldots,x_{i-1},x_{i+1},\ldots,x_n)$, and similarly write $a' = (a_1,\ldots,a_{i-1},a_{i+1},\ldots,a_n)$ for any point $a = (a_1,\ldots,a_n)\in\RR^n$.  Now,
\begin{equation}\label{eq:fSplit}
f(x) = f(x',0) + x_i h(x)
\end{equation}
for some $\C$-analytic function $h:U\to\RR$.  Since $x_i$ divides $x_i h(x)$, and $f(x',0)$ is independent of $x_i$, it follows that for any $a\in C$,
\[
\ord(f;a) = \min\{\ord(f(x',0);a'), \ord(x_i h(x);a)\},
\]
so
\[
k\leq \ord_C(f) \leq \ord(f;a) = \min\{\ord(f(x',0);a'), \ord(h;a) + 1\},
\]
and hence $\ord(f(x',0);a') \geq k$ and $\ord(h;a)\geq k-1$.  Since $a$ was an arbitrary point of $C$, we have $\ord_C(f(x',0)) \geq k$ and $\ord_C(h) \geq k-1$.  By the induction hypothesis,
\begin{eqnarray}
\label{eq:fRestr}
f(x',0)
    & = &
    \sum_{\alpha\in\NN^{I'}_{k}}
    x_{I'}^{\alpha} g_\alpha(x),
\\
\label{eq:h}
h(x)
    & = &
    \sum_{\beta\in\NN^{I}_{k-1}}
    x_{I}^{\beta} h_\beta(x),
\end{eqnarray}
on $U$, for some $\C$-analytic functions $g_\alpha:U\to\RR$ (which are independent of $x_i$) and $h_\beta:U\to\RR$ such that $g_\alpha(x) = \frac{1}{\alpha!}\PDn{\alpha}{f}{x_{I'}}(x)$ and $h_\beta(x) = \frac{1}{\beta!}\PDn{\beta}{h}{x_I}(x)$ for all $x\in C$.  For any $\alpha\in\NN^I$ with $\alpha_i > 0$,
\[
\PDn{\alpha}{f}{x}(x)
=
\PDn{\alpha}{(x_i h(x))}{x}
=
\alpha_i \frac{\partial^{|\alpha|-1} h}{\partial x_{I}^{\alpha - e_i}}(x)
+ x_i \PDn{\alpha}{h}{x}(x),
\]
where $e_i$ is the $i$th standard unit vector in $\NN^n$, so
\begin{equation}\label{eq:xhDiff}
\frac{1}{\alpha!}\PDn{\alpha}{f}{x}(x)
= \frac{1}{(\alpha-e_i)!}
\frac{\partial^{|\alpha|-1} h}{\partial x_{I}^{\alpha - e_i}}(x)
=
h_{\alpha-e_i}(x)
\end{equation}
on $C$.  By writing $g_\alpha = h_{\alpha-e_i}$ for each $\alpha\in\NN^n$ with $|\alpha|=k$ and $\alpha_i > 0$, \eqref{eq:h} gives
\begin{equation}\label{eq:xh}
x_i h(x) = \sum_{\alpha\in\NN^{I}_{k} \atop \alpha_i > 0}
x_{I}^{\alpha} g_\alpha(x),
\end{equation}
and \eqref{eq:xhDiff} gives $\frac{1}{\alpha!}\PD{\alpha}{f}{x}(x) = g_\alpha(x)$ on $C$.  Statement 1 now follows from \eqref{eq:fSplit}, \eqref{eq:fRestr}, and \eqref{eq:xh}.

We now prove 2.  Statement 1 implies that for any $k\leq d_1$ we have $f\in\lb x_I\rb^k$, so $k\leq d_2$, and hence $d_1\leq d_2$ since $k$ was arbitrary.  By the definition of $d_2$, for any $k\leq d_2$ we have
\[
f(x) = \sum_{\alpha\in\NN^{I}_{k}} x_{I}^{\alpha} g_\alpha(x)
\]
for some $\C$-analytic functions $g_\alpha:U\to\RR$, so for each $a\in C$,
\[
\ord(f;a)
\geq
\min\{\ord(x_{I}^{\alpha} g_\alpha(x);a) : \alpha\in\NN^{I}_{k}\}
\geq
k,
\]
which shows that $d_1 \geq k$, and hence $d_1\geq d_2$ since $k$ was arbitrary.  This shows that $d_1 = d_2$.  It is clear that $d_1\leq d_3$, so we now assume that $\C$ is quasianalytic and prove that $d_3 \leq d_1$.  We assume that $d_1 < \infty$, for else there is nothing to prove.  Fix $a = (a_{I^c},0)\in C$ and $\alpha\in\NN^n$ such that $|\alpha| = d_1$ and $\PDn{\alpha}{f}{x}(a) \neq 0$.  Thus the $\C$-analytic map $C\to\RR:x_{I^c}\mapsto \PDn{\alpha_I}{f}{x_I}(x_{I^c},0)$ has a nonzero Taylor series at $a_{I^c}$.  Since $\C$ is quasianalytic, the map does not vanish identically in some neighborhood of $a_{I^c}$, which means that there exists $b = (b_{I^c},0)\in C$ such that $\PDn{\alpha_I}{f}{x_I}(b)\neq 0$.  So $d_3 \leq |\alpha_I| \leq |\alpha| = d_1$.  This proves 2.

We now prove 3.  Assume (a). Write $f$ in the form \eqref{eq:orderAlong} with $k = \ord_C(f)$.  Then for any $a\in C$,
\begin{eqnarray*}
\ord_C(f)
    & = &
    \ord(f;a)
\\
    & \geq &
    \min\left\{\ord(x_{I}^{\alpha} g_\alpha(x); a) : \alpha\in\NN^I, |\alpha| = \ord_C(f)\right\}
\\
    & = &
    \ord_C(f)
    +
    \min\left\{\ord\left(\PDn{\alpha}{f}{x_I};a\right) : \alpha\in\NN^I, |\alpha| = \ord_C(f)\right\}
\\
    & \geq &
    \ord_C(f),
\end{eqnarray*}
so $\min\left\{\ord\left(\PDn{\alpha}{f}{x_I};a\right) : \alpha\in\NN^I, |\alpha| = \ord_C(f)\right\}$ must equal $0$, which  means that there exists $\alpha\in\NN^I$ such that $|\alpha| = \ord_C(f)$ and $\PDn{\alpha}{f}{x}(a)\neq 0$.  This proves (b).

Now assume (b).  Then for all $a\in C$, $\ord_C(f) \leq \ord(f;a)\leq \ord_C(f)$, with the first inequality following from the definition of $\ord_C(f)$ and the second inequality following from (b).  This proves (a).
\end{proof}

\begin{corollary}\label{cor:orderAlong}
Let $U\subseteq\RR^n$ be open and $f:U\to\RR$ be $\C$-analytic, let $C=  \{x\in U : x_I = 0\}$ for some $I\subseteq\{1,\ldots,n\}$, and let $d\in\NN$.  Then
\[
f(x) = \sum_{\alpha\in\NN^{I}_{<d}} \frac{1}{\alpha!}\PDn{\alpha}{f}{x}(x_{I^c},0) x_{I}^{\alpha}
+
\sum_{\alpha\in\NN^{I}_{d}}
x_{I}^{\alpha} f_\alpha(x)
\]
on $U$, for some $\C$-analytic functions $f_\alpha :U\to\RR$ such that $f_\alpha(x) = \frac{1}{\alpha!}\PDn{\alpha}{f}{x_I}(x)$ for all $x\in C$.
\end{corollary}

\begin{proof}
Let $R(x) = f(x) - \sum_{\alpha\in\NN^{I}_{<d}} \frac{1}{\alpha!}\PDn{\alpha}{f}{x}(x_{I^c},0) x_{I}^{\alpha}$, and note that $\ord_C(R)\geq d$.  Now apply Proposition \ref{prop:orderAlong}.1 to $R$.
\end{proof}

\section{Linear Transformations}\label{s:linTrans}

To construct the centers of blowing-up in our resolution procedure, we will perform two types of coordinate transformations: generically chosen linear transformations, and translations by implicitly defined functions.  This section derives the basic computational facts we will need to construct the first type of transformation. To efficiently state these facts, we first introduce some 2-dimensional multi-index notation.

\begin{notation}\label{notation:multi-index}
Consider $E\subseteq\{1,\ldots,n\}$ and a nonempty set $X\subseteq\RR$.  Write $E^c = \{1,\ldots,n\}\setminus E$.  Let $X^{E^c \times E}$ denote the set of matrices $(x_{i,j})_{(i,j)\in E^c\times E}$ with entries $x_{i,j}$ in $X$, with the understanding that $X^\emptyset = \RR^0 = \{0\}$.  Thus $E^c$ serves as the set of row indices and $E$ serves as the set of column indices.  If $\gamma\in\NN^{E^c\times E}$ and $w\in\RR^{E^c\times E}$, define
\begin{eqnarray*}
w^\gamma
    & = &
    \prod_{(i,j)\in E^c\times E} w_{i,j}^{\gamma_{i,j}}, \\
\gamma!
    & = &
    \prod_{(i,j)\in E^c\times E} \gamma_{i,j}!, \\
|\gamma|_{\col}
    & = &
    \left(\sum_{j\in E} \gamma_{i,j}\right)_{i\in E^c}, \\
|\gamma|_{\row}
    & = &
    \left(\sum_{i\in E^c} \gamma_{i,j}\right)_{j\in E}.
\end{eqnarray*}
Thus $|\gamma|_{\col}$ is obtained from $\gamma$ by summing over the column indices, and $|\gamma|_{\row}$ is obtained by summing over the row indices.
\end{notation}

\begin{definition}\label{def:TE}
Define $T_E:\RR^{E^c\times E}\times\RR^n\to\RR^n$ by
\[
T_E(w,y) = \left(\left(y_i + \sum_{j\in E} w_{i,j}\, y_j\right)_{i\in E^c}, (y_i)_{i\in E}\right).
\]
For each $\lambda\in\RR^{E^c\times E}$, write $T_\lambda:\RR^n\to\RR^n$ for the map defined by $T_\lambda(y) = T_E(\lambda,y)$.
\end{definition}

Note that if $E$ or $E^c$ is empty, then $\RR^{E^c\times E} = \{0\}$ and $T_E(0,y) = T_0(y) = y$.  Also note that
\begin{equation}\label{eq:Linv}
\text{$x = T_E(w,y)$ if and only if $y = T_E(-w,x)$,}
\end{equation}
so $T_{\lambda}^{-1} = T_{-\lambda}$ for each $\lambda\in\RR^{E^c\times E}$.

\begin{lemma}\label{lemma:diff_fT}
Let $f:U\to\RR$ be $C^\infty$, where $U\subseteq\RR^n$ is open, and let $E$ be a nonempty proper subset of $\{1,\ldots,n\}$.  Then for all $\alpha \in\NN^{E}$ and all $(w,y)\in T_{E}^{-1}(U)$,
\begin{equation}\label{eq:diff_fL}
\PDn{\alpha}{(f\circ T_E)}{y_{E}}(w,y)
=
\sum_{\gamma\in\NN^{E^c\times E} \atop |\gamma|_{\row} \leq \alpha}
\frac{\alpha!}{\gamma!(\alpha-|\gamma|_{\row})!}
\PDmix{|\alpha|}{f}{x_{E^c}}{|\gamma|_{\col}}{x_{E}}{\alpha-|\gamma|_{\row}}\circ T_E(w,y) w^{\gamma}
\end{equation}
\end{lemma}

\begin{proof}
For each $k\in E$, an induction on $\alpha_k$ shows that for each $\alpha_k\in\NN$,
\[
\PD{\alpha_k}{(f\circ T_E)}{y_k}(w,y)
=
\sum_{\beta_k\in\NN \atop \beta_k\leq\alpha_k} {\alpha_k\choose \beta_k} \sum_{j_1,\ldots,j_{\beta_k}\in E^c} \frac{\partial^{\alpha_k} f}{\partial x_{j_1}\cdots \partial x_{j_{\beta_k}} \partial {x_k}^{\alpha-\beta}}\circ T_E(w,y) w_{j_1,k}\cdots w_{j_{\beta_k},k}.
\]
Write $w_k = (w_{i,k})_{i\in E^c}$, and note that for each $\beta_k\leq\alpha_k$ and $\gamma_k\in\NN^{E^c}$ with $|\gamma_k| = \beta_k$, there are exactly $\frac{\beta_k!}{\gamma_k!}$ many choices of $j_1,\ldots,j_{\beta_k}\in E^c$ such that $w_{k}^{\gamma_k} = w_{j_1,k}\cdots w_{j_{\beta_k},k}$.  So combining like terms gives
\[
\PD{\alpha_k}{(f\circ T_E)}{y_k}(w,y)
=
    \sum_{\beta_k\in\NN \atop \beta_k\leq\alpha_k}
    \sum_{\gamma_k\in\NN^{E^c} \atop |\gamma_k|=\beta_k} \frac{\alpha_k!}{\gamma_k!(\alpha_k-\beta_k)!}
    \PDmix{\alpha_k}{f}{x_{E^c}}{\gamma_k}{x_k}{\alpha_k-\beta_k}\circ T_E(w,y) w_{k}^{\gamma_k}.
\]
Applying this formula successively to each of the variables $y_k$ with $k\in E$ gives
\begin{eqnarray*}
\PDn{\alpha}{(f\circ T_E)}{y_{E}}(w,y)
    & = &
    \sum_{(\beta_k)_{k\in E}\in\NN^E \atop \forall k\in E(\beta_k\leq\alpha_k)}
    \sum_{(\gamma_k)_{k\in E} \in (\NN^{E^c})^E \atop \forall k\in E(|\gamma_k| = \beta_k)}
    \left(
    \left(\prod_{k\in E} \frac{\alpha_k!}{\gamma_k!(\alpha_k-\beta_k)!}\right)
    \right.
\\
    & &
    \left.
    \frac{\partial^{\sum_{k\in E}\alpha_k} f}
    {\partial x_{E^c}^{\sum_{k\in E} \gamma_k}
    \left(\partial x_{k}^{\alpha_k - \beta_k}\right)_{k\in E}} \circ T_E(w,y)
    \left(\prod_{k\in E} w_{k}^{\gamma_k}\right)
\right).
\end{eqnarray*}
This can be more succinctly written as
\[
\PDn{\alpha}{(f\circ T_E)}{y_{E}}(w,y)
=
\sum_{\beta\in\NN^E \atop \beta\leq\alpha}
\sum_{\gamma\in\NN^{E^c\times E}\atop |\gamma|_{\row} = \beta}
\frac{\alpha!}{\gamma!(\alpha-\beta)!}
\PDmix{|\alpha|}{f}{x_{E^c}}{|\gamma|_{\col}}{x_{E}}{\alpha-\beta}\circ T_E(w,y) w^{\gamma},
\]
which is equivalent to \eqref{eq:diff_fL}.
\end{proof}

\begin{lemma}\label{lemma:RowColSums}
Let $E$ be a nonempty proper subset of $\{1,\ldots,n\}$.  For all $\beta\in\NN^n$ and $\alpha\in\NN^E$ such that $|\alpha| = |\beta|$ and $\alpha\geq \beta_E$, there exists $\gamma\in \NN^{E^c\times E}$ such that
\begin{eqnarray*}
|\gamma|_{\col}\, & = & \beta_{E^c}, \\
|\gamma|_{\row} & = & \alpha - \beta_E.
\end{eqnarray*}
\end{lemma}

\begin{proof}
To have more concrete notation, assume that $E = \{e+1,\ldots,n\}$ for some $e\in\{1,\ldots,n-1\}$.  We want to find $\gamma$ such that the following holds:
\begin{equation}\label{eq:RowColSum}
\begin{matrix}
\gamma_{1,e+1} & \cdots & \gamma_{1,n} & \to & \beta_1\\
\vdots         &        & \vdots       &     & \vdots \\
\gamma_{e,e+1} & \cdots & \gamma_{e,n} & \to & \beta_e \\
\downarrow     &        & \downarrow \\
\alpha_{e+1} - \beta_{e+1} & \cdots & \alpha_n - \beta_n &   & & ,
\end{matrix}
\end{equation}
where the arrows denote summation of the numbers in the corresponding row or column.
We proceed by induction on $n = |E^c| + |E|$.  Because $0 < e < n$, the base case is when $e=1$ and $n=2$, in which case \eqref{eq:RowColSum} is simply
\[
\begin{matrix}
\gamma_{1,2} & \to & \beta_1\\
\downarrow \\
\alpha_{2} - \beta_{2} & & &.
\end{matrix}
\]
We are assuming that $\beta = (\beta_1,\beta_2)$ and $\alpha = (\alpha_2)$ are such that $\beta_1 + \beta_2 = \alpha_2$, so let $\gamma_{1,2} = \beta_1$.

Now assume that $n > 2$.   We will assume that $\beta_e\leq \alpha_n-\beta_n$, because the case that $\beta_e \geq \alpha_n - \beta_n$ can be handled similarly.  Define $\gamma_{e,n} = \beta_e$.  Making this choice for $\gamma_{e,n}$ forces us to also define $\gamma_{e,e+1} = \cdots = \gamma_{e,n-1} = 0$, so \eqref{eq:RowColSum} now reduces to finding $\gamma' = (\gamma_{i,j})$ in $\NN^{(E^c\setminus\{e\}) \times E}$ satisfying
\[
\begin{matrix}
\gamma_{1,e+1}   & \cdots   & \gamma_{1,n-1}   & \gamma_{1,n}   & \to & \beta_1\\
\vdots           &          & \vdots           & \vdots         &     & \vdots \\
\gamma_{e-1,e+1} & \cdots   & \gamma_{e-1,n-1} & \gamma_{e-1,n} & \to & \beta_{e-1} \\
\downarrow       &          & \downarrow       & \downarrow     \\
\alpha_{e+1} - \beta_{e+1} & \cdots & \alpha_{n-1} - \beta_{n-1} & (\alpha_n - \beta_e) - \beta_n &   & & .
\end{matrix}
\]
Let $\alpha' = (\alpha_{e+1},\ldots,\alpha_{n-1},\alpha_n-\beta_e)$ and $\beta' = (\beta_1,\ldots,\beta_{e-1},\beta_{e+1},\ldots,\beta_n)$.  Note that $|\alpha'| = |\beta'|$, so we are done by the induction hypothesis.
\end{proof}

\begin{proposition}\label{prop:Gen}
Let $f:U\to\RR$ be $C^\infty$, where $U\subseteq\RR^n$ is open.  Let $a\in U$, and assume that $m = \ord(f;a) < \infty$.  Let $E$ be a nonempty proper subset of $\{1,\ldots,n\}$, and fix $\alpha\in\NN^E$ such that $|\alpha| = m$.  Then
\begin{equation}\label{eq:fTpoly}
\PDn{\alpha}{(f\circ T_E)}{y_E}(T_E^{-1}(a))
=
\sum_{\gamma\in\NN^{E^c\times E} \atop |\gamma|_{\row} \leq \alpha}
\frac{\alpha!}{\gamma!(\alpha-|\gamma|_{\row})!}
\PDmix{|\alpha|}{f}{x_{E^c}}{|\gamma|_{\col}}{x_E}{\alpha-|\gamma|_{\row}}(a) w^{\gamma}
\end{equation}
is a nonzero polynomial in $w$ if and only if
\begin{equation}\label{eq:Gen}
\text{\parbox{5in}{$\alpha\geq\beta_E$ for some $\beta\in\NN^n$ such that $|\beta| = m$ and $\PDn{\beta}{f}{x}(a)\neq 0$.}}
\end{equation}
\end{proposition}

An important special case is when $|E| = 1$, in which case $\alpha = (m)$ and \eqref{eq:Gen} is automatically satisfied.

\begin{proof}
Lemma \ref{lemma:diff_fT} shows that \eqref{eq:fTpoly} holds, and this function is clearly a polynomial in $w$, so the only issue is whether or not this polynomial is nonzero.  If \eqref{eq:Gen} holds, as witnessed by $\beta$, then because $|\alpha| = m = |\beta|$, Lemma \ref{lemma:RowColSums} shows that there exists $\gamma\in\NN^{E^c\times E}$ such that $|\gamma|_{\col} = \beta_{E^c}$ and $|\gamma|_{\row} = \alpha - \beta_E$, so \eqref{eq:fTpoly} is a nonzero polynomial.  If \eqref{eq:Gen} does not hold, then because every derivative occurring in the sum \eqref{eq:fTpoly} is of the form $\PDn{\beta}{f}{x}(a)$ for some $\beta\in\NN^n$ with $|\beta| = m$ and $\beta_E\leq \alpha$, \eqref{eq:fTpoly} is the zero polynomial.
\end{proof}

\section{Blowing-up}\label{s:blowup}

This section proves the key computational lemmas about blowings-up that we will use in our resolution procedure.  We will use the following notation throughout this section.

\begin{notation}
Fix a quasianalytic IF-system $\C$.  Also fix an open set $U\subseteq\RR^n$.  We write $x = (x_1,\ldots,x_n)$ for coordinates on $U$, and if we are given a function $F:V\to U$ where $V$ is open in $\RR^n$, we write $y = (y_1,\ldots,y_n)$ for coordinates on $V$.  Fix a nonempty set $I\subseteq\{1,\ldots,n\}$, and define
\[
C = \{x\in U : x_I = 0\}.
\]
For each $i\in I$, let
\[
I_i = I\setminus\{i\}.
\]
For any $J\subseteq\{1,\ldots,n\}$, write $J^c = \{1,\ldots,n\}\setminus J$.  For each $i\in I$, define $\pi_i:\RR^n\to\RR^n$ by $\pi_i(y) = x$, where for each $j\in\{1,\ldots,n\}$,
\[
x_j = \begin{cases}
y_i y_j, & \text{if $j\in I_i$}, \\
y_j,  & \text{if $j\in I_{i}^{c}$.}
\end{cases}
\]
Thus
\[
\pi_i(y) = (y_{I_{i}^{c}}, y_i y_{I_i}) = (y_{I^c}, y_i, y_i y_{I_i}),
\]
where $y_i y_{I_i} = (y_i y_j)_{j\in I_i}$.  We also let
\[
U_{i} = \pi_{i}^{-1}(U)
\]
for each $i\in I$.
\end{notation}

\begin{definition}\label{def:blowup}
We call the family of maps $\{\pi_i:U_{i}\to U\}_{i\in I}$ the {\bf standard charts for the blowing-up of $U$ with center $C$}, and we call
\[
\pi_{i}^{-1}(C) = \{y\in U_{i} : y_i = 0\}
\]
the {\bf exceptional divisor} of $\pi_i$.
\end{definition}

\begin{remark}\label{rmk:blowup}
Note that $U_{i}\setminus \pi_{i}^{-1}(C) = \{y\in U_{i} : y_i \neq 0\}$ and that the restriction of $\pi_i$ to $U_{i}\setminus \pi_{i}^{-1}(C)$ is a $\C$-analytic isomorphism from $U_{i}\setminus \pi_{i}^{-1}(C)$ onto $\{x\in U : x_i \neq 0\}$.
\end{remark}

We do not need the following remark, but it explains our choice of terminology in Definition \ref{def:blowup}.

\begin{remark}\label{rmk:blowupGlobal}
Define
\begin{eqnarray*}
U'
    & = &
    \left\{(x,\xi)\in U\times\PP^{|I|-1}(\RR) : x_I\in\xi\right\} \\
    & = &
    \left\{(x,\xi)\in U\times\PP^{|I|-1}(\RR) : \bigwedge_{i,j\in I} x_i\xi_j = x_j\xi_i\right\},
\end{eqnarray*}
where $\PP^{|I|-1}(\RR)$ is the $(|I|-1)$-dimensional projective space of all lines through the origin in $\RR^{|I|}$, and $\xi = [\xi_i]_{i\in I}$ are homogeneous coordinates for the line $\xi$.  The projection $\Pi:U'\to U$ given by
\[
\Pi(x,\xi) = x
\]
is called a {\bf blowing-up of $U$ with center $C$}, and $\Pi^{-1}(C) = C\times\PP^{|I|-1}(\RR)$ is called the {\bf exceptional divisor} of $\Pi$.  The set $U'$ is an algebraic manifold with coordinate charts $\{\psi_i:U'_i\to U_i\}_{i\in I}$, where for each $i\in I$,
\begin{eqnarray*}
U'_i
    & = &
    \{(x,\xi)\in U' : \xi_i\neq 0\}, \\
\psi_i(x,\xi)
    & = &
    \left(\left(x_j\right)_{j\in I_{i}^{c}}, \left(\frac{\xi_j}{\xi_i}\right)_{j\in I_i}\right).
\end{eqnarray*}
The map $\pi_i:U_i\to U$ is the pushout of $\Pi\Restr{U'_i}$ by $\psi_i$, namely, $\pi_i = \Pi\circ\psi_{i}^{-1}$.
\end{remark}

\begin{lemma}\label{lemma:graph}
Consider a finite, transitive, directed graph such that for all nodes $p$ and $q$, there is an edge from $p$ to $q$ or from $q$ to $p$.  Then there exists a node $p$ such that for every node $q$ there is an edge from $q$ to $p$.
\end{lemma}

Note that in the hypothesis of Lemma \ref{lemma:graph}, we allow the possibility of an edge from $p$ to $q$ and from $q$ to $p$.  Also, if one takes $p = p$ in the hypothesis of the lemma, we see that there is an edge from $p$ to $p$ for each node $p$.

\begin{proof}
Let $\{0,\ldots,k\}$ be the nodes.  The result of the lemma is obvious if $k = 0$, so assume that $k > 0$ and that the lemma holds for all such graphs with $k$ nodes.  We may assume that there is a node $q\in\{0,\ldots,k-1\}$ with an edge from $k$ to $q$, for if this was not the case, then for every node $r\in\{0,\ldots,k-1\}$ there would be an edge from $r$ to $k$, and we would be done.  By the induction hypothesis, there exists $p\in\{0,\ldots,k-1\}$ such that for every node $r\in\{0,\ldots,k-1\}$ there is an edge from $r$ to $p$.  There is also an edge from $k$ to $p$ since there is an edge from $k$ to $q$ and from $q$ to $p$, and the graph is transitive.
\end{proof}

\begin{lemma}\label{lemma:blowupCover}
Write $I = \bigcup_{l=0}^{k} N_l$ for disjoint nonempty subsets $N_0,\ldots,N_k$ of $I$, and for each $i\in I$ let $\ell(i)\in\{0,\ldots,k\}$ be such that $i\in N_{\ell(i)}$.  Let $K\subseteq U$, and let $\epsilon > 0$ be rational.  For each $i\in I$, define
\[
K_i = \{x\in \pi_{i}^{-1}(K) : \text{$|y_j| \leq \epsilon^{\ell(j) - \ell(i)}$ for all $j\in I$}\}.
\]
Then
\[
K = \bigcup_{i\in I} \pi_i(K_i).
\]
Furthermore, if $K$ is co-c.e.\ compact, then each set $K_i$ is co-c.e.\ compact.
\end{lemma}

\begin{proof}
Clearly $\bigcup_{i\in I}\pi_i(K_i) \subseteq K$ since $K_i\subseteq \pi_{i}^{-1}(K)$ for each $i\in I$, so we must prove the reverse inclusion.  Note that for each $i\in I$,
\begin{equation}\label{eq:KiImage}
\pi_i(K_i) = \{x\in K : \text{$|x_j| \leq \epsilon^{\ell(j)-\ell(i)}|x_i|$ for all $j\in I$}\}.
\end{equation}

Now, fix $x\in K$.  For each $p\in\{0,\ldots,k\}$, fix $n(p)\in N_p$ such that $|x_j| \leq |x_{n(p)}|$ for all $j\in N_p$.  Thus
\begin{equation}\label{eq:lnId}
\text{$\ell(n(p)) = p$ for all $p\in\{0,\ldots,k\}$,}
\end{equation}
and
\begin{equation}\label{eq:nlBd}
\text{$|x_j| \leq |x_{n(\ell(j))}|$ for all $j\in I$.}
\end{equation}
Note that for all $p,q\in\{0,\ldots,k\}$, either $|x_{n(p)}| \leq \epsilon^{p-q}|x_{n(q)}|$ or $|x_{n(q)}| \leq \epsilon^{q-p}|x_{n(p)}|$.  Consider the directed graph with nodes $\{0,\ldots,k\}$ and edges specified as follows: for each $p,q\in\{0,\ldots,k\}$, the graph has an edge from $q$ to $p$ if and only if $|x_{n(q)}| \leq \epsilon^{q-p}|x_{n(p)}|$.  This graph satisfies the hypothesis of Lemma \ref{lemma:graph}, so we may fix $p\in\{0,\ldots,k\}$ such that
\begin{equation}\label{eq:pMaxNode}
\text{$|x_{n(q)}| \leq \epsilon^{q-p}|x_{n(p)}|$ for all $q\in\{0,\ldots,k\}$.}
\end{equation}
Statements \eqref{eq:lnId}, \eqref{eq:nlBd}, and \eqref{eq:pMaxNode} imply that for all $j\in I$,
\[
|x_j| \leq |x_{n(\ell(j))}| \leq \epsilon^{\ell(j)-p}|x_{n(p)}| = \epsilon^{\ell(j)-\ell(n(p)))}|x_{n(p)}|,
\]
so $x\in\pi_{n(p)}(K_{n(p)})$.  Therefore $K \subseteq \bigcup_{i\in I}\pi_i(K_i)$.

To finish, suppose that $K$ is co-c.e.\ compact.  Each set $K_i$ is co-c.e.\ closed by Proposition \ref{prop:compCont} and is also bounded since $K$ is bounded and $\pi_i$ acts trivially in the coordinates $y_{I^c}$.  Thus $K_i$ is co-c.e.\ compact by Proposition \ref{prop:compHeineBorel}.
\end{proof}

\begin{lemma}\label{lemma:blowupCoverShrink}
For each $i\in I$ let $A_i \subseteq U_i$ and $\delta_i > 0$, and let
\[
A = \bigcup_{i\in I} \pi_i(A_i).
\]
Then
\[
\{x\in A : \text{$|x_i|\leq\delta_i$ for all $i\in I$}\}
\subseteq
\bigcup_{i\in I} \pi_i\left(\{y\in A_i : |y_i|\leq\delta_i\}\right).
\]
\end{lemma}

\begin{proof}
Let $x\in A$ be such that $|x_i|\leq\delta_i$ for all $i\in I$.  Fix $i\in I$ and $y\in A_i$ such that $\pi_i(y) = x$.  Since $y_i = x_i$, $|y_i| \leq\delta_i$.
\end{proof}

\begin{lemma}\label{lemma:strictTrans}
Consider a $\C$-analytic function $f:U\to\RR$, and let $i\in I$.  Then
\[
\ord_{\pi_{i}^{-1}(C)}(f\circ\pi_i) = \ord_C(f).
\]
\end{lemma}

\begin{proof}
An induction on $l$ shows that for any $i\in I$ and any $l\in\NN$,
\[
\PD{l}{(f\circ\pi_i)}{y_i}(y) = \sum_{j=0}^{l} {l \choose j}
\sum_{k_1,\ldots,k_j \in I_i}
\frac{\partial^l f}{\partial x_{k_1}\cdots\partial
x_{k_j}\partial x_{i}^{l-j}}\circ\pi_{i}(y) y_{k_1}\cdots y_{k_j}.
\]
By combining like terms, this formula can be rewritten as
\[
\PD{l}{(f\circ\pi_i)}{y_i}(y) = \sum_{j=0}^{l} \frac{l!}{j!(l-j)!}
\sum_{\alpha_{I_i}\in\NN^{I_i} \atop |\alpha_{I_i}| = j}
\frac{j!}{\alpha_{I_i}!} \PDmix{l}{f}{x_{I_i}}{\alpha_{I_i}}{x_i}{l-j}\circ \pi_i(y) y_{I_i}^{\alpha_{I_i}}.
\]
This simplifies to
\begin{equation}\label{eq:diffBlowup}
\frac{1}{l!} \PD{l}{(f\circ\pi_i)}{y_i}(y)
=
\sum_{\alpha\in\NN^{I} \atop |\alpha| = l} \frac{1}{\alpha!} \PDn{\alpha}{f}{x_I}\circ\pi_i(y) y_{I_i}^{\alpha_{I_i}}.
\end{equation}
Setting $y_i = 0$ in \eqref{eq:diffBlowup} gives
\begin{equation}\label{eq:diffBlowupEval}
\left.\frac{1}{l!} \PD{l}{(f\circ\pi_i)}{y_i}(y)\right|_{y_i = 0}
=
\sum_{\alpha\in\NN^{I} \atop |\alpha| = l} \frac{1}{\alpha!} \PDn{\alpha}{f}{x_I}(y_{I^c},0) y_{I_i}^{\alpha_{I_i}}.
\end{equation}
Since $\C$ is quasianalytic, Proposition \ref{prop:orderAlong}.2 shows that \eqref{eq:diffBlowupEval} is identically zero if $l < \ord_C(f)$ and is a nonzero function if $l=\ord_C(f)$.  Thus $\ord_{\pi_{i}^{-1}(C)}(f\circ\pi_i) = \ord_C(f)$.
\end{proof}

\begin{definition}\label{def:strictTrans}
Let $f:U\to\RR$ be $\C$-analytic, and suppose that $m = \ord_C(f) < \infty$.  Lemma \ref{lemma:strictTrans} implies that $m$ is the greatest integer for which there exists a $\C$-analytic function $f'_i:U_{i}\to\RR$ such that
\begin{equation}\label{eq:strictTransform}
f\circ\pi_i(y) = y_{i}^{m} f'_i(y)
\end{equation}
on $U_{i}$.  We call $f'_i$ the {\bf strict transform of $f$ by $\pi_i$}.
\end{definition}

We do not need the following remark, but it explains our choice of terminology in Definition \ref{def:strictTrans}.

\begin{remark}\label{rmk:strictTransGlobal}
Consider the situation of Remark \ref{rmk:blowupGlobal} and Definition \ref{def:strictTrans}. A calculation (which we omit) shows that for all $i,j\in I$, $f'_i\circ\psi_i(x,\xi)$ and $f'_j\circ\psi_j(x,\xi)$ are equivalent up to a unit on $U'_i\cap U'_j$.  So $\{f'_i\circ\psi_i:U'_i\to\RR\}_{i\in I}$ generates a principle ideal sheaf on $U'$.  This principle ideal sheaf is called the {\bf strict transform of $f$ by $\Pi$}.
\end{remark}

\begin{lemma}\label{lemma:STHi}
Fix distinct $i,j\in I$ and a positive integer $m$.  Let $f:U\to\RR$ be a $\C$-analytic function such that $\ord(f;x) = m$ for all $x\in C$, and such that $\PD{m-1}{f}{x_j}(x) = 0$ for all $x\in U$ with $x_j = 0$.  Then $\PD{m-1}{f'_i}{y_j}(y) = 0$ for all $y\in U_i$ with $y_j = 0$.
\end{lemma}

\begin{proof}
Note that $y_{i}^{m} f'_i(y) = f\circ\pi_i(y) = f\left(y_{I_{i}^{c}},y_i y_{I_i}\right)$ on $U_i$.  Differentiating gives
\begin{equation}\label{eq:fiDiff}
y_{i}^{m}\PD{m-1}{f'_i}{y_j}(y) = y_{i}^{m-1}\PD{m-1}{f}{x_j}\left(y_{I_{i}^{c}},y_i y_{I_i}\right).
\end{equation}
Because $\PD{m-1}{f}{x_j}(x) = 0$ for all $x\in U$ with $x_j=0$, the mean value theorem shows that in some neighborhood of $\{x\in U : x_j = 0\}$ we have
\begin{equation}\label{eq:fMVT}
\PD{m-1}{f}{x_j}(x) = x_j\PD{m}{f}{x_j}(x_1,\ldots,x_{j-1},\xi_j(x),x_{j+1},\ldots,x_n),
\end{equation}
for some $\xi_j(x)$ which is between $x_j$ and $0$ when $x_j\neq 0$ and which equals zero when $x_j = 0$.  By writing $x = \pi_i(y)$ in \eqref{eq:fMVT} (so $x_j = y_i y_j$), substituting into \eqref{eq:fiDiff}, and canceling $y_{i}^{m}$, we get
\begin{equation}\label{eq:fiMVT}
\PD{m-1}{f'_i}{y_j}(y) = y_j \PD{m}{f}{x_j}\left(y_{I_{i}^{c}}, y_i y_{I\setminus\{i,j\}}, \xi_j(y_{I_{i}^{c}},y_i y_{I_i})\right)
\end{equation}
on $U_i$.  Setting $y_j=0$ in \eqref{eq:fiMVT} gives $\PD{m-1}{f'_i}{y_j}(y)\Restr{y_j = 0} = 0$.
\end{proof}

\begin{lemma}\label{lemma:derivRestr}
Let $U\subseteq\RR^n$ be open.  Let $i\in\{1,\ldots,n\}$, write $x = (x',x_i)$ for coordinates on $\RR^n$, where $x' = (x_1,\ldots,x_{i-1},x_{i+1},\ldots,x_n)$, and put $U' = \{x' : (x',0)\in U\}$.  Suppose that $f:U\to\RR$ is a $C^{p+d}$ function such that $\PD{j}{f}{x_i}(x',0) = 0$ for all $j\in\{0,\ldots,d-1\}$ and all $x'\in U'$.  Let $g:U\to\RR$ be the unique $C^p$ function satisfying
\[
f(x) = x_{i}^{d} g(x)
\]
on $U$ (the existence and uniqueness of $g$ follows from repeated application of Proposition \ref{prop:compDivVar}, ignoring the computability assumptions).  Then for all $\alpha\in\NN^n$ with $|\alpha|\leq p$,
\[
\frac{1}{\alpha_i!} \PDn{\alpha}{g}{x}(x',0)
=
\frac{1}{(\alpha_i+d)!} \frac{\partial^{|\alpha|+d} f}{\partial x^{\alpha+d e_i}}(x',0)
\]
on $U'$, where $ e_i$ is the $i$th standard unit vector in $\NN^n$.
\end{lemma}

\begin{proof}
Let $\alpha\in\NN^n$ with $|\alpha|\leq p$, and write $\alpha' = (\alpha_1,\ldots,\alpha_{i-1},\alpha_{i+1},\ldots,\alpha_n)$.  Then on $U$,
\begin{eqnarray*}
\frac{\partial^{|\alpha|+d} f}{\partial x^{\alpha+d e_i}}(x)
    & = &
    \PD{\alpha_i+d}{}{x_i}\left(x_{i}^{d} \PDn{\alpha'}{g}{(x')}(x)\right)
\\
    & = &
    \sum_{j=0}^{\alpha_i+d} {\alpha_i + d \choose j}
    \frac{\partial^{|\alpha'|+j} g}{\partial (x')^{\alpha'} \partial x_{i}^{j}}(x)
    \,\PD{\alpha_i+d-j}{(x_{i}^{d})}{x_n}
\\
    & = &
    \sum_{j=\alpha_i}^{\alpha_i+d} \frac{(\alpha_i+d)!}{j!(\alpha_i+d-j)!}
    \frac{\partial^{|\alpha'|+j} g}{\partial (x')^{\alpha'} \partial x_{i}^{j}}(x)
    \frac{d!}{(j-\alpha_i)!} x_{i}^{j-\alpha_i}.
\end{eqnarray*}
When this is evaluated at $x_i=0$, all terms with  $j > \alpha_i$ are zero, so this gives
\[
\frac{\partial^{|\alpha|+d} f}{\partial x^{\alpha+d e_i}}(x',0)
=
\frac{(\alpha_i+d)!}{\alpha_i!}\PDn{\alpha}{g}{x}(x',0)
\]
on $U'$.
\end{proof}

\begin{lemma}\label{lemma:STorderBound}
Let $f:U\to\RR$ be $\C$-analytic, and suppose that $\ord(f;x) = m < \infty$ for all $x\in C$. Let $i\in I$ and $b\in\pi_{i}^{-1}(C)$, and write $a = \pi_i(b)$.  Choose $\alpha\in\NN^{I}$ such that
\begin{equation}\label{eq:alpha}
\text{$|\alpha| = m$ and $\PDn{\alpha}{f}{x_I}(a)\neq 0$.}
\end{equation}
(Such an $\alpha$ exists by Proposition \ref{prop:orderAlong}.3.)
Then
\[
\PDn{\alpha_{I_i}}{f'_i}{y_{I_i}}(b)\neq 0
\]
if either of the following two conditions hold:
\begin{enumerate}{\setlength{\itemsep}{3pt}
\item[]\hspace*{-10pt}
\emph{Condition 1}:
$\alpha_i$ is minimal with respect to $\alpha$ satisfying \eqref{eq:alpha}.

\item[]\hspace*{-10pt}
\emph{Condition 2}:
$b_I = 0$.
}
\end{enumerate}
\end{lemma}

\begin{proof}
Fix $\alpha\in\NN^I$ satisfying \eqref{eq:alpha}.  Define $\xi\in\NN^I$ by $\xi_{I_i} = \alpha_{I_i}$ and $\xi_i = m$.  Assuming that either Condition 1 or Condition 2 holds, we will prove that $\PDn{\xi}{(f\circ\pi_i)}{y_I}(b)\neq 0$.  This will suffice to prove the lemma because $f\circ\pi_i(y) = y_{i}^{m}f'_i(y)$, and hence
\[
\frac{1}{\alpha_{I_i}!}\PDn{\alpha_{I_i}}{f'_i}{y_{I_i}}(b)
=
\frac{1}{\xi!} \PDn{\xi}{(f\circ\pi_i)}{y_I}(b)
\]
by Lemma \ref{lemma:derivRestr}.

Note that $a_{I^c} = b_{I^c}$, $a_I = 0$, and $b_i = 0$.  The Taylor series for $x_{I}\mapsto f(a_{I^c},x_{I})$ at $0$ is given by
\[
\sum_{\beta\in\NN^I} \frac{1}{\beta!} \PDn{\beta}{f}{x_I}(a) x_{I}^{\beta}.
\]
By formally composing this series with $(y_i(y_{I_i} + b_{I_i}), y_i)$, which are the $I$-components of
\[
\pi_i(a_{I^c},y_{I_i} + b_{I_i}, y_i) = (a_{I^c}, y_i(y_{I_i} + b_{I_i}), y_i),
\]
we obtain the Taylor series for $y_I \mapsto f\circ\pi_i(a_{I_c},y_I)$ at $b_I$.  So
\begin{eqnarray}\label{eq:TaylorSeries1}
\sum_{\beta\in\NN^I} \frac{1}{\beta!}\PDn{\beta}{(f\circ\pi_i)}{y_I}(b) y_{I}^{\beta}
    & = &
    \sum_{\beta\in\NN^I} \frac{1}{\beta!} \PDn{\beta}{f}{x_I}(a)
    \left(y_{I_i} + b_{I_i}\right)^{\beta_{I_i}} y_{i}^{|\beta|}
    \nonumber\\
    & = &
    \sum_{\beta\in \NN^I} \frac{1}{\beta!} \PDn{\beta}{f}{x_I}(a)
    \left(
    \sum_{\gamma\in\NN^{I_i} \atop \gamma\leq\beta_{I_i}}
    \frac{\beta_{I_i}!}{\gamma!(\beta_{I_i} - \gamma)!}
    y_{I_i}^{\gamma} b_{I_i}^{\beta_{I_i} - \gamma}
    \right)
    y_{i}^{|\beta|}
    \nonumber\\
    & = &
    \sum_{\delta\in\NN^I}
    \left(
    \sum_{\beta\in\NN^I \atop |\beta| = \delta_i, \beta_{I_i} \geq \delta_{I_i}}
    \frac{1}{\beta_i! \delta_{I_i}! (\beta_{I_i} - \delta_{I_i})!}
    \PDn{\beta}{f}{x_I}(a) b_{I_i}^{\beta_{I_i} - \delta_{I_i}}
    \right)
    y_{I}^{\delta},
\end{eqnarray}
where the second equality is by the multinomial theorem, and the third equality comes from reindexing by setting $\delta_{I_i} = \gamma$ and $\delta_i = |\beta|$.

First suppose that Condition 1 holds. To show that $\PDn{\xi}{(f\circ\pi_i)}{y_I}(b) \neq 0$, it suffices to show that the coefficient of $y_{I}^{\xi}$ in \eqref{eq:TaylorSeries1} is nonzero.  To analyze this coefficient, consider $\beta\in\NN^I$ with $|\beta| = \xi_i$ and $\beta_{I_i} \geq \xi_{I_i}$.  The definition of $\xi$ shows that $|\beta| = m$ and $\beta_{I_i} \geq \alpha_{I_i}$.  There are now two cases to consider:
\begin{enumerate}{\setlength{\itemsep}{3pt}
\item
First suppose that $\beta_{I_i} = \alpha_{I_i}$.  Since $|\beta| = |\alpha|$, it follows that $\beta_i = \alpha_i$.  So $\beta = \alpha$.

\item
Now suppose that $\beta_{I_i} \neq \alpha_{I_i}$.  Since $\beta_{I_i} \geq \alpha_{I_i}$, this means that $|\beta_{I_i}| > |\alpha_{I_i}|$, so
$\beta_i = m - |\beta_{I_i}| < m - |\alpha_{I_i}| = \alpha_i$.  By the minimality of $\alpha_i$, $\PDn{\beta}{f}{x_I}(a) = 0$.
}\end{enumerate}
These two cases and \eqref{eq:TaylorSeries1} show that
\[
\frac{1}{\xi!}\PDn{\xi}{(f\circ\pi_i)}{y_I}(b)
=
\frac{1}{\alpha!}\PDn{\alpha}{f}{x_I}(a)
\neq 0,
\]
as desired.

Now suppose that Condition 2 holds.  Equation \eqref{eq:TaylorSeries1} simplifies to
\begin{eqnarray}\label{eq:TaylorSeries2}
\sum_{\beta\in\NN^I} \frac{1}{\beta!}\PDn{\beta}{(f\circ\pi_i)}{y_I}(b) y_{I}^{\beta}
    & = &
    \sum_{\beta\in\NN^I} \frac{1}{\beta!} \PDn{\beta}{f}{x_I}(a)
    y_{I_i}^{\beta_{I_i}} y_{i}^{|\beta|}
    \nonumber\\
    & = &
    \sum_{\delta\in\NN^I \atop \delta_i \geq |\delta_{I_i}|} \frac{1}{(\delta_i - |\delta_{I_i}|)!\delta_{I_i}!}
    \frac{\partial^{\delta_i} f}{\partial x_{I_i}^{\delta_{I_i}}
    \partial x_{i}^{\delta_i - |\delta_{I_i}|}}(a) y_{I}^{\delta},
\end{eqnarray}
where the second equality comes from reindexing by setting $\delta_{I_i} = \beta_{I_i}$ and $\delta_i = |\beta|$.  Comparing the coefficients of $y_{I}^{\xi}$ in \eqref{eq:TaylorSeries2} shows that
\[
\frac{1}{\xi!}\PDn{\xi}{(f\circ\pi_i)}{y_I}(b)
=
\frac{1}{\alpha!}\PDn{\alpha}{f}{x_I}(a)\neq 0.
\]
\end{proof}

\section*{\bf Part III: Lifting}

Throughout Part III, we fix an IF-system $\C$.  Section \ref{s:fiberProd} proves a fiber product lemma. Section \ref{s:blowupSet} discusses blowup sets, which are a kind of set that occurs naturally as the domains of our maps in our resolution procedure.  Section \ref{s:lifting} defines the notion of a lifting, which is a nonsingular representation of a family of $\C$-analytic functions defined on an open blowup set.  The section concludes with the lifting lemmas, which are a collection of results showing how liftings are preserved under various operations.

\section{The Fiber Product Lemma}\label{s:fiberProd}

\begin{definition}\label{def:goodDir}
A coordinate projection $\Pi:\RR^n\to\RR^d$ is called a {\bf good direction} for a $\C$-analytic submanifold $M$ of $\RR^n$ if $\Pi$ defines a $\C$-analytic isomorphism from $M$ onto an open subset of $\RR^d$.  (Of course, this notion also applies to coordinate projections $\Pi_E:\RR^n\to\RR^E$, where $E\subseteq\{1,\ldots,n\}$.)
\end{definition}

\begin{lemma}\label{lemma:goodDir}
Suppose that $M$ is a $\C$-analytic submanifold of $\RR^n$ defined by
\[
M = \{x\in U : f(x) = 0\}
\]
for some open set $U\subseteq\RR^n$ and $\C$-analytic function $f:U\to\RR^{n-d}$ such that $\rank\PD{}{f}{x} = n-d$ on $M$.  Fix complementary coordinate projections $\Pi:\RR^n\to\RR^d$ and $\Pi':\RR^n\to\RR^{n-d}$, and write $x = (y,y')$, where $\Pi(x) = y$ and $\Pi'(x) = y'$.
\begin{enumerate}{\setlength{\itemsep}{3pt}
\item
If $\det\PD{}{f}{y'}(x) \neq 0$ on $M$, and if $M = \{x\in U : g(x) = 0\}$ for some $\C$-analytic function $g:U\to\RR^{n-e}$ such that $\rank\PD{}{g}{x} = n-e$ on $M$, then $e = d$ and $\det\PD{}{g}{y'} \neq 0$ on $M$.

\item
The projection $\Pi$ is a good direction for $M$ if and only if $\Pi$ is injective on $M$ and $\det\PD{}{f}{y'} \neq 0$ on $M$.
}\end{enumerate}
\end{lemma}

\begin{proof}
We first prove 1.  Write
\[
M = \{x\in U : f(x) = 0\} = \{x\in U : g(x) = 0\},
\]
for $\C$-analytic functions $f = (f_1,\ldots,f_{n-d}):U\to\RR^{n-d}$ and $g = (g_1,\ldots,g_{n-e}):U\to\RR^{n-e}$, where $\det\PD{}{f}{y'}\neq 0$ on $M$ and $\PD{}{g}{x}$ has constant rank $n-e$ on $M$.  Applying the implicit function theorem to $f$ locally about any point of $M$ shows that $M$ has pure dimension $d$, and doing the same for $g$ shows that $M$ has pure dimension $e$, so $d = e$ (by invariance of domain).  Now, fix a neighborhood $V$ of $M$ such that $\det\PD{}{f}{y'}\neq 0$ on $V$. For each $i\in\{1,\ldots,n-d\}$, $g_i(x) = 0$ on $M$, so Proposition \ref{prop:orderAlong}.1 shows that there exist $h_{i,1},\ldots,h_{i,n-d}\in\C[V]$ such that
\[
g_i(x) = \sum_{j=1}^{n-d}h_{i,j}(x) f_j(x)
\]
on $V$.  Thus $g(x) = h(x)f(x)$ on $V$, where $h$ is the matrix of functions $(h_{i,j})_{i,j=1,\ldots,n-d}$, and $f$ and $g$ are written as column vectors.  Fix $a\in M$.  Thus
\[
\PD{}{g}{x}(a) = h(a)\PD{}{f}{x}(a).
\]
Because $\PD{}{g}{x}(a)$ and $\PD{}{f}{x}(a)$ have rank $n-d$, $h(a)$ must be invertible.  Therefore $\PD{}{g}{y'}(a)$ must be invertible, since
\[
\PD{}{g}{y'}(a) = h(a)\PD{}{f}{y'}(a)
\]
and $\PD{}{f}{y'}(a)$ is invertible.  This proves 1.

We now prove 2.  Assume that $\Pi$ is a good direction for $M$.  We may assume that $\Pi(x) = (x_1,\ldots,x_d)$.  The inverse of $\Pi\Restr{M}:M\to\Pi(M)$ is of the form $\Pi(M)\to M:y\mapsto(y,g(y))$ for a $\C$-analytic function $g:\Pi(M)\to\RR^{n-d}$.  Note that $M = \{(y,y')\in U : y' - g(y) = 0\}$ and $\det \PD{}{(y'-g(y))}{y'} = 1$, so $\det\PD{}{f}{y'}\neq 0$ on $M$ by Statement 1.  This proves the forward implication of 2.  The converse follows directly from the implicit function theorem.
\end{proof}

\begin{definition}\label{def:Fmanifold}
Let $\F$ be a $\QQ$-algebra of real-valued $\C$-analytic functions defined on a rational box manifold $X\subseteq\RR^m$; fix $D\subseteq\{1,\ldots,m\}$, an open rational box $A\subseteq\RR^M$, and $a\in\RR^{M^c}$ such that $X = A\times\{a\}$, where $D^c = \{1,\ldots,m\}\setminus D$.  A set $M$ is called an {\bf $\F$-manifold} if there exist $d\in\{0,\ldots,\dim(M)\}$, and injection $\lambda:\{1,\ldots,d\} \to \{1,\ldots,m\}$ with $\im(\lambda)\subseteq D$, and a $\C$-analytic function $f:X\to\RR^{\dim(M)-d}$ such that
\[
M = \{x\in X : f(x) = 0\},
\]
$\rank\PD{}{f}{x_D}(x) = \dim(M)-d$ for all $x\in M$, and $\Pi_\lambda:\RR^n\to\RR^d$ is a good direction on $M$.  We say that $M$ is {\bf defined nonsingularly by $f$}. (We may also say that $M$ is defined nonsingularly by $f(x) = 0$.)
\end{definition}

Therefore by Lemma \ref{lemma:goodDir}, if we extend $\lambda$ to a bijection $\sigma:\{1,\ldots,\dim(M)\}\to D$ and define $\lambda':\{1,\ldots,\dim(M)-d\}\to\{1,\ldots,m\}$ by $\lambda'(i) = \sigma(i+d)$ for all $i$, then $\det\PD{}{f}{x_{\lambda'}}(x) \neq 0$ for all $x\in M$.

\begin{definition}
If $\F$ and $\G$ are $\QQ$-algebras of real-valued functions defined on the
sets $X$ and $Y$, respectively, then their {\bf tensor product}, $\F\otimes\G$,
is the $\QQ$-algebra of functions on $X\times Y$ generated by the functions
$h:X\times Y\to\RR$ of the form
\begin{enumerate}
\item
$h(x,y) = f(x)$ for some $f\in\F$, or

\item
$h(x,y) = g(y)$ for some $g\in\G$.
\end{enumerate}
\end{definition}

Note that in the definition of $\F\otimes\G$, the variables $x$ and $y$ are disjoint.  So if we identify polynomials with rational coefficients with the functions from $\RR^n$ into $\RR$ that they define (so, for example, $\QQ[x_1,x_2]$ and $\QQ[x_4,x_5]$ denote the same $\QQ$-algebra of functions), then we have, for example,
\[
\QQ[x_1,x_2,x_3]\otimes\QQ[x_1,x_2] = \QQ[x_1,x_2,x_3]\otimes\QQ[x_4,x_5] = \QQ[x_1,x_2,x_3,x_4,x_5].
\]

\begin{definition}\label{def:fiberProd}
For any sets $X$, $Y$, and $Z$, if two maps $\varphi:X\to Z$ and $\psi:Y\to Z$
are specified, define the {\bf fiber product of $X$ and $Y$ over $Z$ (with
respect to $\varphi$ and $\psi$)} by
\[
X\times_Z Y := \{(x,y)\in X\times Y : \varphi(x) = \psi(y)\}.
\]
\end{definition}

Note that we have natural projections $X\times_Z Y\to X : (x,y)\mapsto x$ and $X\times_Z Y \to Y : (x,y)\mapsto y$ and that the following diagram commutes:
\[
\xymatrix{
X\times_Z Y \ar[r]\ar[d] & Y \ar[d]^{\psi}\\
X \ar[r]^{\varphi}       & Z \quad.}
\]
\hfill\\

We now describe the situation to be considered in the fiber product lemma.  Consider the diagram,
\begin{equation}\label{eq:fiberProductGiven}
\xymatrix{
    & N \ar[d]^-{\Pi^{N}_{P}} \ar[r]_-{\Pi^{N}_{V}}^-{\isom}
    & V
\\
M \ar[r]_{\Pi^{M}_{P}}^-{\isom}
    & P \ar[r]_-{\Pi^{P}_{U}}^-{\isom}
    & U \quad ,
}
\end{equation}
where $\C$-analytic isomorphisms are denoted by ``isom'' and the sets and maps in \eqref{eq:fiberProductGiven} have the following meaning:
\begin{quote}
$ $\indent Let $\F$ be a $\QQ$-algebra of real-valued $\C$-analytic functions defined on a rational box manifold $X\subseteq\RR^m$, and assume that $\F$ contains all the coordinate variables $x_1,\ldots,x_m$. Let $d\in\{0,\ldots,\dim(X)\}$, let
\[
M = \{x\in X : f(x) = 0\}
\]
be an $\F$-manifold defined nonsingularly by $f\in\F^{\dim(X)-d}$ with a good direction $\Pi^{M}_{U}:\RR^m\to\RR^d$, and let $U = \Pi^{M}_{U}(M)$.     Consider two coordinate projections $\Pi^{M}_{P}:\RR^{m}\to\RR^p$ and $\Pi^{P}_{U}:\RR^p\to\RR^d$ such that $\Pi^{M}_{U} = \Pi^{P}_{U}\circ\Pi^{M}_{P}$, where $d\leq p\leq m$, and let
\[
P = \Pi^{M}_{P}(M).
\]

Similarly, let $\G$ be a $\QQ$-algebra of real-valued $\C$-analytic functions defined on a rational box manifold $Y\subseteq\RR^n$, where $n\geq p$, and assume that $\G$ contains all the coordinate variables $y_1,\ldots,y_n$. Let $e\in\{0,\ldots,\dim(Y)\}$, let
\[
N = \{y\in Y: g(y) = 0\}
\]
be a $\G$-manifold defined nonsingularly by $g\in\G^{\dim(Y)-e}$ with a good direction $\Pi^{N}_{V}:\RR^n\to\RR^e$, and let $V = \Pi^{N}_{V}(N)$.  Assume that $\Pi^{N}_{P}:\RR^n\to\RR^p$ is a coordinate projection such that $\Pi^{N}_{P}(N)\subseteq P$, and let $\Pi^{N}_{U} = \Pi_{U}^{P}\circ\Pi^{N}_{P}$.
\end{quote}
Now, consider the fiber product
\[
M\times_{P}N = \{(x,y)\in M\times N : \Pi^{M}_{P}(x) = \Pi^{N}_{P}(y)\},
\]
and complete the diagram \eqref{eq:fiberProductGiven} to
\begin{equation}\label{eq:fiberProductCompleted}
\xymatrix{
M\times_{P}N \ar[r] \ar[d]
    & N \ar[d]^{\Pi^{N}_{P}} \ar[r]_{\Pi^{N}_{V}}^{\isom}
    & V \\
M \ar[r]_{\Pi^{M}_{P}}^{\isom}
    & P \ar[r]_{\Pi^{P}_{U}}^{\isom}
    & U \quad .
}
\end{equation}
Since $\Pi_{U}^{P}:P\to U$ is a bijection, $M\times_{P}N$ is the set of all $(x,y)\in X\times Y$ such that
\begin{equation}\label{eq:fiberProdFormula}
\left(f(x) = 0\right)\wedge
\left(g(y) = 0\right)
\wedge\left(\Pi^{M}_{U}(x) - \Pi^{N}_{U}(y) = 0\right),
\end{equation}
which is a system of equations with functions in $\F\otimes\G$.

\begin{lemma}[Fiber Product Lemma]\label{lemma:fiberProduct}
For the situation just described, $M\times_{P} N$ is an $\F\otimes\G$-manifold defined nonsingularly by \eqref{eq:fiberProdFormula}, and the natural projection $M\times_P N\to N$ is a $\C$-analytic isomorphism.
\end{lemma}

Since $\Pi_{N}^{V}$ is a good direction for $N$, saying that the natural projection $M\times_P N \to N$ is a $\C$-analytic isomorphism is equivalent to saying that the composition of the natural projection $M\times_P N\to N$ with $\Pi^{N}_{V}:N\to V$ is a good direction for $M\times_P N$.

\begin{proof}
Consider points $(x_1,y)$ and $(x_2,y)$ in $M\times_P N$.  Then $\Pi^{M}_{P}(x_1) = \Pi^{N}_{P}(y) = \Pi^{M}_{P}(x_2)$, so $x_1=x_2$ because $\Pi^{M}_{P}$ is injective.  This shows that the natural projection $M\times_P N\to N$ is injective.

Now consider $y\in N$.  Then $\Pi^{N}_{P}(y)\in P$, so there exists $x\in M$ such that $\Pi^{M}_{P}(x) = \Pi^{N}_{P}(y)$ because $\Pi^{M}_{P}$ maps $M$ onto $P$.  Thus $(x,y)\in M\times_P N$.  This shows that the natural projection $M\times_P N\to N$ is surjective.

To finish, we must show that \eqref{eq:fiberProdFormula} defines $M\times_P N$ nonsingularly as a subset of $X\times Y$ with good direction $\Pi^{N}_{V}$.  In order to simplify our notation, we begin with a reduction of the problem.  Fix $D\subseteq\{1,\ldots,m\}$, an open rational box $A\subseteq\RR^D$, and $a\in\QQ^{D^c}$ such that $X = A\times\{a\}$, and similarly fix $E\subseteq\{1,\ldots,n\}$, an open rational box $B\subseteq\RR^E$, and $b\in\QQ^{E^c}$ such that $Y = B\times\{b\}$.  Define $\tld{f}:A\times\RR^{D^c}\to\RR^{m-d}$ by $\tld{f}(x) = (f(x_D,a), x_{D^c} - a)$, and define $\tld{g}:B\times\RR^{E^c}\to\RR^{n-e}$ by $\tld{g}(y) = (g(y_E,b),y_{E^c} - b)$.  Observe that
\[
\left(\tld{f}(x) = 0\right)\wedge
\left(\tld{g}(y) = 0\right)
\wedge\left(\Pi^{M}_{U}(x) - \Pi^{N}_{U}(y) = 0\right)
\]
defines $M\times_P N$ nonsingularly as a subset of $(A\times\RR^{D^c})\times(B\times\RR^{E^c})$ with good direction $\Pi^{N}_{V}$
if and only if \eqref{eq:fiberProdFormula} defines $M\times_P N$ nonsingularly as a subset of $X\times Y$ with good direction $\Pi^{N}_V{}$.  We may therefore assume, without loss of generality, that $X$ and $Y$ are open rational boxes.

Partition the tuples of variables $x$ and $y$ into $x = (x_1,x_2,x_3,x_4,x_5)$
and $y=(y_1,y_2,y_3,y_4,y_5,y_6)$, where the $x_i$'s and $y_j$'s are not just
single variables but are (possibly empty) tuples of variables, so that
\begin{eqnarray*}
\Pi^{M}_{P}(x_1,x_2,x_3,x_4,x_5) & = & (x_1,x_2,x_3,x_4), \\
\Pi^{N}_{P}(y_1,y_2,y_3,y_4,y_5,y_6) & = & (y_1,y_2,y_3,y_4) \\
            &   & \text{(where $x_i$ corresponds to $y_i$ for each $i\in\{1,2,3,4\}$)}, \\
\Pi^{P}_{U}(x_1,x_2,x_3,x_4) & = & (x_1, x_2), \\
\Pi^{N}_{V}(y_1,y_2,y_3,y_4,y_5,y_6) & = & (y_1,y_3,y_5).
\end{eqnarray*}
This notation allows for all possible coincidences and noncoincidences of the
variables in the images of the four coordinate projections $\Pi^{M}_{P}$,
$\Pi^{N}_{P}$, $\Pi^{P}_{U}$ and $\Pi^{N}_{V}$.  With this notation,
\[
M\times_P N = \{(x,y)\in X\times Y : f(x) = 0, g(y) = 0, x_1-y_1 = 0, x_2 -
y_2 = 0\}.
\]
Now,
\[
\PD{}{(f,g,x_1-y_1,x_2-y_2)}{(x,y_2,y_4,y_6)}
\]
is the $(m+n-e)\times(m+n-e)$ matrix given in block form by
\begin{equation}\label{eq:Nmatrix}
\left(\begin{matrix} \PD{}{f}{x_1} & \PD{}{f}{x_2} &
\PD{}{f}{x_3}&\PD{}{f}{x_4} & \PD{}{f}{x_5} & 0 & 0 & 0 \\
0 & 0 & 0 & 0 & 0 & \PD{}{g}{y_2} & \PD{}{g}{y_4} & \PD{}{g}{y_6} \\
\id & 0 & 0 & 0 & 0 & 0 & 0 & 0 \\
0 & \id & 0 & 0 & 0 & -\id & 0 & 0
\end{matrix}\right),
\end{equation}
where each ``$\id$'' denotes an identity matrix.  Since $x\mapsto(x_1,x_2)$ and $y\mapsto(y_1,y_3,y_5)$ restrict to isomorphisms on $M$ and $N$, respectively, Lemma \ref{lemma:goodDir}.2 implies that
$\det\PD{}{f}{(x_3,x_4,x_5)}\neq 0$ and $\det\PD{}{g}{(y_2,y_4,y_6)} \neq 0$ on $M\times N$.  Therefore \eqref{eq:Nmatrix} is nonsingular on $M\times N$, so \eqref{eq:fiberProdFormula} defines $M\times_P N$ nonsingularly with good direction $(x,y)\mapsto \Pi^{N}_{V}(y)$, again by Lemma \ref{lemma:goodDir}.2.
\end{proof}

\section{Blowup Sets}\label{s:blowupSet}

In this section we work with sequences of blowings-up with various centers, so we shall use the following notation for the standard charts of our blowings-up, which is more detailed than the notation used in Section \ref{s:blowup}.

\begin{notation}\label{notation:blowupChartI}
If $\emptyset\neq I\subseteq\{1,\ldots,n\}$ and $i\in I$, write $\pi_{I,i}:\RR^n\to\RR^n$ for the $i$th-standard chart of the blowing-up of $\RR^n$ with center $\{x\in\RR^n : x_I = 0\}$.  Thus $\pi_{I,i}(y) = (y_{I^c}, y_i, y_i y_{I_i})$, where $I_i = I\setminus\{i\}$ and the superscript $c$ always denotes complementation in $\{1,\ldots,n\}$.
\end{notation}

\begin{definition}\label{def:blowupSet}
A set $Y\subseteq\RR^n$ is a {\bf blowup set} if it is a blowup set of length $k$ for some $k\in\NN$, which is defined inductively as follows:
\begin{enumerate}{\setlength{\itemsep}{5pt}
\item
The set $Y$ is a blowup set of length $0$ if it is a rational box.

\item
For $k > 0$, the set $Y$ is a blowup set of length $k$ if $Y = \pi_{I,i}^{-1}(X)\cap(\RR^{I^c}\times B)$ for some nonempty $I\subseteq\{1,\ldots,n\}$, $i\in I$, rational box $B\subseteq \RR^I$, and blowup set $X\subseteq\RR^n$ of length $k-1$ such that $\{x\in X : x_I = 0\}$ is nonempty.
}\end{enumerate}
\end{definition}

A blowup set is, by definition, a blowup set of length $k$ for some $k$, but the number $k$ need not be uniquely determined by the set itself.  For example, if $X$ is a blowup set of length $k$ and $\{x\in X : x_i = 0\}$ is nonempty for some $i\in\{1,\ldots,n\}$, then taking $I = \{i\}$ and $B = \RR^I$ in clause 2 of Definition \ref{def:blowupSet} shows that $Y = X$; thus $X$ is a blowup set of length $l$ for all $l\geq k$.

\begin{notation}\label{notation:CenBox}
For any $Y\subseteq\RR^n$ and $y\in\RR^n$, define
\begin{eqnarray*}
\Cen(Y)
    & = &
    \left\{j\in\{1,\ldots,n\} : 0\in\Pi_{\{j\}}(Y)\right\},
\\
\BOX(y,Y)
    & = &
    \left\{\left((y_j)_{j\in\Cen(Y)^c}, (t_j y_j)_{j\in\Cen(Y)}\right) : \text{$0\leq t_j\leq 1$ for all $j\in\Cen(Y)$}\right\}.
\end{eqnarray*}
\end{notation}

\begin{notation}
In Definition \ref{def:box}, we defined $[a,b] = [a_1,b_1]\times \cdots \times [a_n,b_n]$ for any $a = (a_1,\ldots,a_n)$ and $b = (b_1,\ldots,b_n)$ in $\RR^n$. However, we now make an exception.  From now through the proof of Lemma \ref{lemma:blowupSet}, for any $a,b\in\RR^n$ we will write
\[
[a,b] = \{(1-t)a+tb : 0\leq t\leq 1\},
\]
which is the line segment from $a$ to $b$.
\end{notation}

\begin{definition}\label{def:polyConn}
A set $Y\subseteq\RR^n$ is {\bf polygonally connected} if for all $x,y\in Y$ there exist finitely many points $x_0,\ldots,x_n\in\RR^n$ such that $x = x_0$, $y = x_n$, and $[x_{j-1},x_j]\subseteq Y$ for all $j\in\{1,\ldots,n\}$.
\end{definition}

Clearly, a set which is polygonally connected is connected.

\begin{lemma}\label{lemma:blowupSet}
Let $Y\subseteq\RR^n$ be a blowup set.
\begin{enumerate}{\setlength{\itemsep}{5pt}
\item
Suppose that $Y = \pi_{I,i}^{-1}(X)\cap(\RR^{I^c}\times B)$, as in clause 2 of Definition \ref{def:blowupSet}, where $B = \prod_{j\in I}B_j$ for rational intervals $B_j$.  Then
\[
\Cen(Y) = \left(\Cen(X)\setminus I\right)\cup\{j\in I : 0\in B_j\}.
\]

\item
For all $y\in Y$, $\BOX(y,Y)\subseteq Y$.

\item
For all $J\subseteq\{1,\ldots,n\}$, the set $\{y\in Y : y_J = 0\}$ is nonempty if and only if $J\subseteq\Cen(Y)$.

\item
The set $Y$ is polygonally connected.
}\end{enumerate}
\end{lemma}

\begin{proof}
Suppose that $Y$ is a blowup set of length $k$.  If $k = 0$, then statement 1 does not apply, and statements 2-4 are trivial.  So assume that $k > 0$ and that the lemma holds for all blowup sets of length $k-1$.  Write
\begin{equation}\label{eq:Ydef}
Y = \pi_{I,i}^{-1}(X)\cap(\RR^{I^c}\times B),
\end{equation}
as in clause 2 of Definition \ref{def:blowupSet}, where $B = \prod_{j\in I}B_j$ for rational intervals $B_j$.  Note that since $\{x\in X : x_I = 0\}$ is nonempty, applying statement 3 to the set $X$ shows that
\begin{equation}\label{eq:IcenX}
I\subseteq\Cen(X).
\end{equation}

We now prove 1.  Let $j\in\Cen(Y)$.  Fix $y\in Y$ such that $y_j = 0$, and put $x = \pi_{I,i}(y)$.  Then $x_j = 0$ and $x\in X$, so $j\in\Cen(X)$.  Thus $\Cen(Y)\subseteq\Cen(X)$, so in particular, $\Cen(Y)\setminus I \subseteq \Cen(X)\setminus I$.  Also, \eqref{eq:Ydef} implies that $\Cen(Y)\cap I\subseteq\{j\in I : 0 \in B_j\}$, so $\Cen(Y) \subseteq \left(\Cen(X)\setminus I\right)\cup\{j\in I : 0\in B_j\}$.

To show the reverse inclusion, let $j\in(\Cen(X)\setminus I) \cup\{j\in I : 0\in B_j\}$.  Fix $y\in Y$, and note that $\Pi_{I,i}(y)\in X$ and $y_I\in B$ by \eqref{eq:Ydef}.  Define $z = (z_1,\ldots,z_n)$ by $z_j = 0$ and $z_l = y_l$ for all $l\in\{1,\ldots,n\}\setminus\{j\}$.  Note that $j\in\Cen(X)$ by \eqref{eq:IcenX}, so applying statement 2 to $X$ shows that $\pi_{I,i}(z) \in \BOX(\pi_{I,i}(y),X) \subseteq X$.  If $j\in \Cen(X)\setminus I$, then $z_I = y_I \in B$.  If $j\in I$, then $0\in B_j$, so $z_I\in B$.  Either way, we have $z_I\in B$.  So $z\in Y$ by \eqref{eq:Ydef}, and hence $j\in\Cen(Y)$.  This proves 1.

We now prove 2.  Let $y\in Y$.  Thus $(y_{I^c},y_i,y_iy_{I_i})\in X$ and $y_I\in B$.  Let $t=(t_j)_{j\in\Cen(Y)}\in[0,1]^{\Cen(Y)}$, and put $z = ((y_j)_{j\in\Cen(Y)^c}, (t_jy_j)_{j\in\Cen(Y)})$.  Then
\[
\pi_{I,i}(z) = \left((t_j)_{j\in I^c\cap\Cen(Y)^c}, (t_jy_j)_{j\in I^c\cap\Cen(Y)},
t_i y_i, (y_iy_j)_{j\in I_i\cap\Cen(Y)^c}, (t_it_jy_iy_j)_{j\in I_i\cap\Cen(Y)}\right),
\]
where $0\leq t_i\leq 1$ if $0\in B_i$, and $t_i = 1$ if $0\not\in B_i$.  Since the $t_j$ and $t_it_j$ are in $[0,1]$, $\pi_{I,i}(z)\in\BOX(\Pi_{I,i}(y),X) \subseteq X$.  Also, by Statement 1 we have $0\in B_j$ for all $j\in\Cen(Y)\cap I$, so $z_I\in B$, and hence $z\in Y$.  Thus $\BOX(y,Y)\subseteq Y$, which proves 2.

We now prove 3.  Let $J\subseteq\{1,\ldots,n\}$.  If $\{y\in Y:y_J=0\}$ is nonempty, then $0\in\Pi_{\{j\}}(Y)$ for all $j\in J$, so $J\subseteq\Cen(Y)$.  Conversely, if $J\subseteq\Cen(Y)$, fix $b\in Y$, and note that $(b_{J^c},0)\in\BOX(b,Y)\subseteq Y$ by statement 2, so $\{y\in Y : y_J = 0\}$ is nonempty.  This proves 3.

We now prove 4.  Let $x,y\in Y$.  Define $z = (z_j)_{j\in I}\in\RR^I$ by
\[
z_j = \begin{cases}
x_j,    & \text{if $x_j y_j \geq 0$ and $|x_j| \leq |y_j|$,} \\
y_j,    & \text{if $x_j y_j \geq 0$ and $|y_j| \leq |x_j|$,} \\
0,      & \text{if $x_j y_j < 0$,}
\end{cases}
\]
for each $j\in I$.  Applying statement 1 to $\pi_{I,i}^{-1}(X)$ shows that $I\subseteq\Cen(\pi_{I,i}^{-1}(X))$, so $[x,(x_{I^c},z)]\subseteq\BOX(x,\pi_{I,i}^{-1}(X))$.  Applying statement 2 to $\pi_{I,i}^{-1}(X)$ shows that $\BOX(x,\pi_{I,i}^{-1}(X)) \subseteq \pi_{I,i}^{-1}(X)$, and hence $[x,(x_{I^c},z)] \subseteq \pi_{I,i}^{-1}(X)$.  Also, $[x_I,z]\subseteq B$, so $[x,(x_{I^c},z)] \subseteq Y$.  Likewise, $[y,(y_{I^c},z)]\subseteq Y$.  Since $\pi_{I,i}(x_{I^c},z)$ and $\pi_{I,i}(y_{I^c},z)$ are in $\{w\in X : w_I = (z_i,z_i z_{I_i})\}$, which is a blowup set of length $k-1$, the induction hypothesis implies that there exist finitely many $w^{(0)},\ldots,w^{(l)}\in\RR^{I^c}$ such that $x_{I^c} = w^{(0)}$, $y_{I^c} = w^{(l)}$, and $[(w^{(j-1)},z_i,z_iz_{I_i}), (w^{(j)},z_i,z_iz_{I_i})] \subseteq \{w\in X : w_I = (z_i,z_i z_{I_i})\}$ for all $j\in\{1,\ldots,l\}$.   Hence $[(w^{(j-1)},z),(w^{(j)},z)]\subseteq Y$ for all $j\in\{1,\ldots,l\}$.   Thus $x,(w^{(0)},z),\ldots,(w^{(l)},z),y$ are the vertices of a polygonal path from $x$ to $y$ in $Y$, which proves 4.
\end{proof}

A set $X\subseteq\RR^n$ is a blowup set of length $k$ if and only if there exist nonempty $I_1,\ldots,I_k\subseteq\{1,\ldots,n\}$, $(i_1,\ldots,i_k)\in I_1\times\cdots\times I_k$, and rational boxes $A_0\subseteq\RR^n$, $A_1\subseteq\RR^{I_1}$, \ldots, $A_k\subseteq\RR^{I_k}$ such that $X = X_k$, where
\begin{eqnarray}\label{eq:blowupSetConst}
X_0
    & = &
    A_0,
\\
X_j
    & = &
    \pi^{-1}_{I_j,i_j}(X_{j-1})\cap(\RR^{I_{j-1}^{c}}\times A_j),
    \quad
    \text{for each $j\in\{1,\ldots,k\}$,}
    \nonumber
\end{eqnarray}
with the stipulation that $I_j\subseteq\Cen(X_{j-1})$ for all $j\in\{1,\ldots,k\}$. Note that we may easily compute $\Cen(X_j)$ for all $j\in\{0,\ldots,k\}$ using Lemma \ref{lemma:blowupSet}.1.  Also note that the construction of $X$ given in \eqref{eq:blowupSetConst} is not uniquely determined by the set $X$.

\begin{notation}\label{notation:lifting}
If $X$ is the blowup set constructed in \eqref{eq:blowupSetConst}, and this construction is implicitly understood from context (or is unimportant), then we shall write
\[
\bar{X} = \left\{(x_0,\ldots,x_k)\in(\RR^n)^{k+1} : (x_0\in A_0) \wedge
 \left(\bigwedge_{j=1}^{k}(\pi_{I_j,i_j}(x_j) = x_{j-1})\wedge(x_{j,I_j}\in A_j)\right)\right\},
\]
and shall also write $\bar{\Pi}:(\RR^n)^{k+1}\to\RR^n$ for the coordinate projection
\[
\bar{\Pi}(x_0,\ldots,x_k) = x_k.
\]
Note that $\bar{\Pi}$ defines a bijection from $\bar{X}$ onto $X$, with inverse map
\[
x_k\mapsto (\pi_{I_1,i_1}\circ\cdots\pi_{I_k,i_k}(x_k),\ldots,\pi_{I_k,i_k}(x_k),x_k).
\]
If the construction of $X$ given in \eqref{eq:blowupSetConst} needs to be specified, then for $I = (I_1,\ldots,I_k)$, $i=(i_1,\ldots,i_k)$, and $A = A_0\times\cdots\times A_k$, we shall write
\begin{eqnarray*}
\Bus(I,i,A) & = & X,\\
\bar{\Bus}(I,i,A) & = & \bar{X}.
\end{eqnarray*}
(The notation ``$\Bus$'' is an acronym for ``\underline{B}low\underline{u}p \underline{s}et''.)
\end{notation}

\begin{definition}\label{def:sameBlowup}
We say that a blowup set $X$ is {\bf defined by the sequence of blowings-up} $\pi_{I_1,i_1},\ldots,\pi_{I_k,i_k}$ if $X = \Bus(I,i,A)$ for some choice of $A$, where $I = (I_1,\ldots,I_k)$ and $i=(i_1,\ldots,i_k)$.  Two blowup sets $X$ and $Y$ are {\bf defined by the same sequence of blowings-up} if there exists a sequence of blowings-up $\pi_{I_1,i_1},\ldots,\pi_{I_k,i_k}$ which define both $X$ and $Y$.
\end{definition}

\begin{definition}\label{def:blowupSetLifting}
Suppose $X = \Bus(I,i,A)$ and $\bar{X} = \bar{\Bus}(I,i,A)$.  We call the map $\bar{\Pi}:\bar{X}\to X$ a {\bf lifting} of $X$, and we call
$(I,i,\name(A))$ a {\bf name} for this lifting. The data $(I,i,\name(A))$
determines the map $\bar{\Pi}:\bar{X}\to X$, and therefore also determined the sets $\bar{X}$ and $X$, so we shall also refer to $(I,i,\name(A))$
as a name for $X$ and as a name for $\bar{X}$.
\end{definition}

\begin{remarks}\label{rmk:blowupSet}
Consider a lifting $\bar{\Pi}:\bar{X}\to X$ with name $(I,i,\name(A))$, where $I = (I_1.\ldots,I_k)$, $i = (i_1,\ldots,i_k)$, and $A = A_0\times\cdots\times A_k$.
\begin{enumerate}{\setlength{\itemsep}{5pt}
\item
If $Y = \Bus(I,i,B)$ and $B\subseteq A$, then $Y\subseteq X$.

\item
If the rational box $A$ is (open/closed/bounded/compact), then $X$ and $\bar{X}$ are (c.e.\ open/co-c.e.\ closed/bounded/co-c.e.\ compact).\\

\noindent Motivated by these two observation, we will always use the following conventions. \vspace*{5pt}

\noindent{\bf Conventions:}

If we state that a blowup set $X$ is (open/closed/bounded/compact), we tacitly mean that $X$ has been constructed as $X = \Bus(I,i,A)$ for some rational box $A$ which is (open/closed/bounded/compact).  If we state that $Y\subseteq X$ for blowup-sets $X = \Bus(I,i,A)$ and $Y = \Bus(I,i,B)$, we tacitly mean that $B\subseteq A$.

\item
For any rational box $B\subseteq\RR^n$ there exists a rational box $A'\subseteq A$ such that
\begin{equation}\label{eq:BusIntersect}
X\cap B = \Bus(I,i,A').
\end{equation}

\begin{proof}
Write $B = \prod_{j=1}^{n} B_j$ for rational intervals $B_1,\ldots,B_n$.  Put $I_0 = \{1,\ldots,n\}$, and for each $l\in\{0,\ldots,k\}$ define
\begin{eqnarray*}
J_l
    & = &
    I_l\setminus\bigcup_{s=l+1}^{k} I_s, \\
A'_l
    & = &
    A_l\cap\left(\RR^{J_{l}^{c}}\times\prod_{j\in J_l} B_j\right).
\end{eqnarray*}
Then \eqref{eq:BusIntersect} holds for $A' = A'_0\times\cdots\times A'_k$.
\end{proof}
}\end{enumerate}
\end{remarks}

\begin{definition}\label{def:blowupManifold}
We call $M\subseteq\RR^n$ a {\bf blowup manifold} if there exist an open blowup set $U\subseteq\RR^n$ and a rational box manifold $B\subseteq\RR^n$ such that $M = U\cap B$.  If $U = \Bus(I,i,A)$ (with $A$ open, by convention), then Remark \ref{rmk:blowupSet}.2 and its proof imply that $M$ is a blowup set defined by $M = \Bus(I,i,A')$ for some rational box manifold $A'\subseteq A$.  If $\Pi_E:\RR^n\to\RR^E$ is the unique good direction for $B$, where $E\subseteq\{1,\ldots,n\}$, and if $M$ is nonempty, then $M$ is a connected $|E|$-dimensional manifold, and $\Pi_E$ is also the unique good direction for $M$.  A {\bf name} for the blowup manifold $M$ consists of the following data:
\begin{itemize}
\item
$(I,i,A')$ (that is, a name for the blowup set $M$);

\item
the set $E$.
\end{itemize}
\end{definition}

\begin{notation}\label{notation:variety}
For a family $\F$ of real-valued functions on a set $U\subseteq\RR^n$, and a set $X\subseteq U$, denote the variety of $\F$ on $X$ by
\[
\VV(\F;X) = \bigcap_{f\in\F}\VV(f;X),
\]
where
\[
\VV(f;X) = \{x\in X : f(x) = 0\}.
\]
\end{notation}

\begin{lemma}\label{lemma:blowupSetVar}
Let $\F$ be a finite family of real-valued computably continuous functions on a c.e.\ open set $U$ in $\RR^n$, and consider a blowup set $\Bus(I,i,A) \subseteq U$, where $A$ is compact.
\begin{enumerate}{\setlength{\itemsep}{5pt}
\item
If $\VV(\F;\Bus(I,i,A)) = \emptyset$, then we can effectively find an open rational box $B$ containing $A$ such that $\Bus(I,i,B)\subseteq U$ and $\VV(\F;\Bus(I,i,B)) = \emptyset$.

\item
If $\VV(f;\Bus(I,i,A)) = \emptyset$ for all $f\in\F$, then we can effectively find an open rational box $B$ containing $A$ such that $\Bus(I,i,B)\subseteq U$ and $\VV(f;\Bus(I,i,B)) = \emptyset$ for all $f\in\F$.
}\end{enumerate}
\end{lemma}

\begin{proof}
Suppose that $\VV(\F;\Bus(I,i,A)) = \emptyset$.  Using the computability assumptions on $U$ and $\F$, and the fact that $\Bus(I,i,A)$ is co-c.e.\ compact, we can effectively find a finite family $\{D_j\}_{j\in J}$ of open rational boxes such that $\Bus(I,i,A)\subseteq \bigcup_{j\in J} D_j$ and such that for each $j\in J$, $\cl(D_j)\subseteq U$ and we have found some $f_j\in\F$ such that $f_j(x)\neq 0$ on $\cl(D_j)$.  The set $\Bus(I,i,A)$ is compact and $\Bus(I,i,A) = \bigcap_{B \supseteq A}\Bus(I,i,\cl(B))$, where the intersection is over all bounded open rational boxes $B$ containing $A$.  So there exists such a $B \supseteq A$ with $\Bus(I,i,\cl(B)) \subseteq \bigcup_{j\in J}D_j$, and hence $\VV(\F;\Bus(I,i,B)) = \emptyset$.  Since each set $\Bus(I,i,\cl(B))$ is co-c.e.\ compact, we can effectively find such a $B$, which proves Statement 1.  Statement 2 is proven similarly, the only difference being that one constructs the set $\bigcup_{j\in J}D_j$ so that $f(x)\neq 0$ on $\cl(D_j)$ for all $j\in J$ and all $f\in\F$.
\end{proof}

\section{Liftings}\label{s:lifting}

For this entire section, fix a family $\S\subseteq\C$.  Write
\[
\S = \{S_\sigma\}_{\sigma\in\Sigma}
\]
for functions $S_\sigma:[-r_\sigma,r_\sigma]\to\RR$, where $r_\sigma\in\QQ_{+}^{\eta(\sigma)}$.  Assume that the map
\[
\Sigma\to\bigcup_{n\in\NN}\QQ_{+}^{n} : \sigma\mapsto r_\sigma
\]
is computable.  Thus, the index set $\Sigma$ and the arity map $\eta:\Sigma\to\NN$ are also computable.

\begin{definition}\label{def:pullback}
Consider a family $\F = \{F_j\}_{j\in J}$ of functions on a set $X$, and a function $g:Y\to X$.  The {\bf pullback of $\F$ by $g$} is the family of functions on $Y$ given by
\[
g^*\F = \{F_j\circ g\}_{j\in J}.
\]
For $Y\subseteq X$, the {\bf restriction of $\F$ to $Y$}, denoted by $\F\Restr{Y}$, is the pullback of $\F$ by the inclusion map $Y\hookrightarrow X$.
\end{definition}

\begin{definition}\label{def:naturalStratBox}
The {\bf natural stratification} of a nonempty interval $I\subseteq\RR$ is the set $\Strat(I)$ consisting of $\Int(I)$ (when $I$ is nondegenerate) and any connected components of $\bd(I)$ that are contained in $I$.  For example, $\Strat(\{a\}) = \{\{a\}\}$, $\Strat((a,b]) = \{(a,b),\{b\}\}$, $\Strat([a,b]) = \{\{a\},(a,b),\{b\}\}$, and $\Strat(\RR) = \RR$.   The natural stratification of a nonempty box $B = \prod_{i=1}^{n}B_i \subseteq\RR^n$ is defined by
\[
\Strat(B) = \left\{\prod_{i=1}^{n}C_i : \text{$C_i\in\Strat(B_i)$ for all $i\in\{1,\ldots,n\}$}\right\}.
\]
\end{definition}

\begin{definition}\label{def:Salgebra}
An {\bf $\S$-algebra on a natural domain} is a finite tensor product of algebras of the form $\QQ[x_1]$ (defined on $\RR$) or of the form $\QQ[x_1,\ldots,x_{\eta(\sigma)},\PDn{\alpha}{S_\sigma}{x}]$ (defined on $[-r_\sigma,r_\sigma]$) for some $\sigma\in\Sigma$ and $\alpha\in\NN^{\eta(\sigma)}$.  If $\F$ is an $\S$-algebra on a natural domain $D\subseteq\RR^n$, and $B$ is a rational box manifold which is open in some member of the natural stratification of $D$, then we call $\F\Restr{B}$ an {\bf $\S$-algebra}.
\end{definition}

\begin{definition}\label{def:Spoly}
A function $P:D\to\RR^k$ is an {\bf $\S$-polynomial map} if $P\in\F^k$ for some  $\S$-algebra $\F$ on $D$.  If $\F$ is defined on its natural domain, then we call $D$ the {\bf natural domain of $P$}.
\end{definition}

\begin{definition}
If $P:D\to\RR^k$ is an $\S$-polynomial map, with $D\subseteq\RR^m$, then a {\bf name} for $P$ is a tuple
\[
(p(x,y),\sigma,\alpha,\xi,\name(D))
\]
such that
\begin{equation}\label{eq:name}
P(x) =
p\left(x,\PDn{\alpha(1)}{S_{\sigma(1)}}{x}\circ\Pi_{\xi(1)}(x), \ldots, \PDn{\alpha(n)}{S_{\sigma(n)}}{x}\circ\Pi_{\xi(n)}(x)\right),
\end{equation}
where
\begin{enumerate}{\setlength{\itemsep}{3pt}
\item
$n\in\NN$ and $p(x,y)\in\QQ[x,y]^k$, with $x = (x_1,\ldots,x_m)$ and $y=(y_1,\ldots,y_n)$;

\item
$\sigma$, $\alpha$ and $\xi$ are maps with domain
$\{1,\ldots,n\}$ such that
\begin{enumerate}{\setlength{\itemsep}{3pt}
\item
for each $i\in\{1,\ldots,n\}$, $\sigma(i)$ is a member of $\Sigma$,
$\alpha(i)$ is a member of $\NN^{\eta\circ\sigma(i)}$, and
$\xi(i)$ is an increasing map from
$\{1,\ldots,\eta\circ\sigma(i)\}$ into $\{1,\ldots,m\}$;

\item
the images of $\xi(1),\ldots,\xi(n)$ are disjoint (so
$\eta\circ\sigma(1)+\cdots+\eta\circ\sigma(n)\leq m$).
}\end{enumerate}
}\end{enumerate}
\end{definition}

For any $\S$-algebra $\F$ with domain $D$, there exist maps $\sigma$, $\alpha$, and $\xi$ such that
\begin{equation}\label{eq:names}
\left\{(p(x,y),\sigma,\alpha,\xi,\name(D)) : p(x,y)\in\QQ[x,y]^k\right\}
\end{equation}
is a collection of names for the members of $\F^k$.  The map from the set
\eqref{eq:names} to $\F^k$ that sends each name
$(p(x,y),\sigma,\alpha,\xi,\name(D))$ to $P(x)$, as defined by \eqref{eq:name}, is surjective but is not necessarily injective.

\begin{definition}\label{def:Smanifold}
An {\bf $\S$-manifold} is an $\F$-manifold for some $\S$-algebra $\F$.  If $M\subseteq \RR^n$ is an $\S$-manifold defined nonsingularly by an $\S$-polynomial map $P:D\to\RR^{\dim(D)-d}$ with good direction $\Pi:\RR^n\to\RR^d$, and if $\name(F)$ is a name for $F$, then we call $(\name(P),\name(\Pi))$ a {\bf name} for $M$.
\end{definition}

\begin{definition}\label{def:basicLifting}
Let $f:U\to\RR^m$ be a function defined on a set $U\subseteq\RR^n$, and let $M$ be a blowup manifold contained in $U$.  Suppose that for some integer $n'\geq n+m$ there exists an $\S$-manifold $M'\subseteq\RR^{n'}$ and a coordinate projection $\Pi:\RR^{n'}\to\RR^{m+n}$ which defines a $\C$-analytic isomorphism from $M'$ onto the graph of $f\Restr{M}$.  Then we call the following data a {\bf basic $\S$-lifting of $f$ on $M$}:
\begin{itemize}
\item
a name for the blowup manifold $M$;

\item
a name for an $\S$-polynomial map $P$ which defines $M'$ nonsingularly by $M = \{x\in U' : P(x) = 0\}$, where $U'\subseteq\RR^{n'}$;

\item
the name for the coordinate projection $\Pi$.
\end{itemize}
\end{definition}

Consider the situation of Definition \ref{def:basicLifting}.  Let $E$ be the unique subset of $\{1,\ldots,n\}$ such that $\Pi_E:\RR^n\to\RR^E$ defines an isomorphism from $M$ onto an open set $V\subseteq\RR^E$.  Then we have the following diagram of coordinates projections, all of which are isomorphisms:
\[
\xymatrix{
M' \ar[r]^-{\Pi}
    & \Graph\left(f\Restr{M}\right) \ar[r]^-{\Pi_n}
    & M \ar[r]^{\Pi_E}
    & V.
}
\]
Thus the projection $\Pi_E\circ\Pi_n\circ\Pi:\RR^{n'}\to\RR^E$ is a good direction for $M'$, so
\[
\left(\name(P), \name(\Pi_E\circ\Pi_n\circ\Pi)\right)
\]
is a name for the $\S$-manifold $M'$.

Basic $\S$-liftings are purely syntactic objects, but it is convenient to speak about them in a semantic manner.  We will write
\begin{equation}\label{eq:basicLiftingNotation}
\xymatrix{
M' \ar[r]^-{\Pi}
    & \Graph(f\Restr{M})
}
\end{equation}
to denote the basic $\S$-lifting given in Definition \ref{def:basicLifting}.  If we do not wish to specify the projection $\Pi$, it will be omitted in \eqref{eq:basicLiftingNotation}.

\begin{lemmas}[Basic Lifting Lemmas]\label{lemma:basicLifting}\hfill
\begin{enumerate}{\setlength{\itemsep}{5pt}
\item{\bf Arithmetic Functions.}
The graphs of the following functions are all $\emptyset$-manifolds:
\begin{enumerate}
\item
$\RR^n\to\RR : x\mapsto p(x)$, for any $n\in\NN$ and $p\in\QQ[x]$, where $x = (x_1,\ldots,x_n)$;

\item
$(-\infty,0)\to\RR:x\mapsto \frac{1}{x}$ and $(0,+\infty)\to\RR:x\mapsto \frac{1}{x}$.
\end{enumerate}
So these functions all have basic $\emptyset$-liftings, since their graphs are their own liftings.

\item{\bf Identity Map.}
Let $\bar{\Pi}:\bar{M}\to M$ be a lifting of a blowup manifold $M$, where $M\subseteq\RR^n$ and $\bar{M}\subseteq\RR^{\bar{n}}$.  Define
\[
M' = \{(x,y) \in\RR^{\bar{n}}\times\RR^n : x\in \bar{M}, y = \bar{\Pi}(x)\},
\]
and define $\Pi':\RR^{\bar{n}+n}\to\RR^{2n}$ by $\Pi'(x,y) = (\bar{\Pi}(x),y)$.  Then
\[
\xymatrix{
M' \ar[r]^-{\Pi'}
    & \Graph(\id_M)
}
\]
is a basic $\emptyset$-lifting of the identity map $\id_M:M\to M$.

\item{\bf Pairing.}
If $M_1\to\Graph(f\Restr{M})$ and $M_2\to\Graph(f_2\Restr{M})$
are basic $\S$-liftings of $f_1$ and $f_2$ on $M$, then
\[
M_1\times_{M} M_2\to\Graph(f_1\times f_2\Restr{M})
\]
is a basic $\S$-lifting of $f_1\times f_2$ on $M$.

\item{\bf Composition.}
If $M'\to\Graph(f\Restr{M})$ is basic $\S$-lifting of $f$ on $M$, $N'\to\Graph(g\Restr{N})$ is a basic $\S$-lifting of $g$ on $N$, and $g(N)\subseteq M$, then
\[
M'\times_M N'\to\Graph(f\circ g\Restr{N})
\]
is a basic $\S$-lifting of $f\circ g$ on $N$.

\item{\bf Implicit Functions.}
Let $f:U\to\RR^n$ be a $\C$-analytic function, with $U$ open in $\RR^{m+n}$.  Suppose that $\IF(f;\cl(A),\cl(B))$ holds for some bounded, open, rational boxes $A\subseteq\RR^m$ and $B\subseteq\RR^n$ such that $\cl(A)\times \cl(B)\subseteq U$, and let $g:\cl(A)\to B$ be the function implicitly defined by $f(x,g(x)) = 0$ on $\cl(A)$.  Fix $I\subseteq\{1,\ldots,m\}$, and put $A_I = \{x\in A : x_I = 0\}$.  Suppose that $\xymatrix{M \ar[r]^-{\Pi\times\Pi'} & \Graph\left(f\Restr{A_I\times B}\right)}$ is a basic $\S$-lifting of $f$ on $A_I\times B$, with $\Pi(M) = A_I\times B$ and $\Pi'(M) = f(A_I\times B)$, and define
\[
N = \{x\in M : \Pi'(x) = 0\}.
\]
Then
\[
\xymatrix{
N \ar[r]^-{\Pi} & \Graph\left(g\Restr{A_I}\right)
}
\]
is a basic $\S$-lifting of $g$ on $A_I$.
}\end{enumerate}
\end{lemmas}

\begin{proof}
Statements 1 and 2 are obvious.  Statements 3 and 4 are corollaries of Lemma \ref{lemma:fiberProduct}, and their proofs are best expressed in diagrams.  We use solid arrows to denote  assumed information and use dotted arrows to denote deduced information.

To prove 3, write $M\to U$ for the good direction for $M$, and note that
\[
\xymatrix{
M_1\times_M M_2
    \ar@{-->}[rr]^-{\isom}
    \ar@{-->}[dd]^-{\isom}
    \ar@{-->}[ddrr]^-{\Pi}_-{\isom}
    \ar@{-->}@/_6pc/[dddr]_-{\Pi_1}
    \ar@{-->}@/^5pc/[drrr]^-{\Pi_2}
    & & M_2 \ar[d]^-{\isom} \\
& & \Graph(f_2\Restr{M}) \ar[r] \ar[d]^-{\isom}
    & f_2(M) \\
M_1 \ar[r]^-{\isom} & \Graph(f_1\Restr{M}) \ar[d] \ar[r]^-{\isom}
    & M \ar[ur]_{f_2} \ar[dl]^{f_1} \ar[r]^-{\isom}
    & U \\
& f_1(M) &  & ,}
\]
so $\xymatrix{M_1\times_{M} M_2 \ar[rr]^-{\Pi\times\Pi_1\times\Pi_2} & & \Graph\left(f_1\times f_2\Restr{M}\right)}$ is a basic $\S$-lifting of $f_1\times f_2$ on $M$.

To prove 4, write $M\to U$ and $N\to V$ for the good directions for $M$ and $N$, and note that
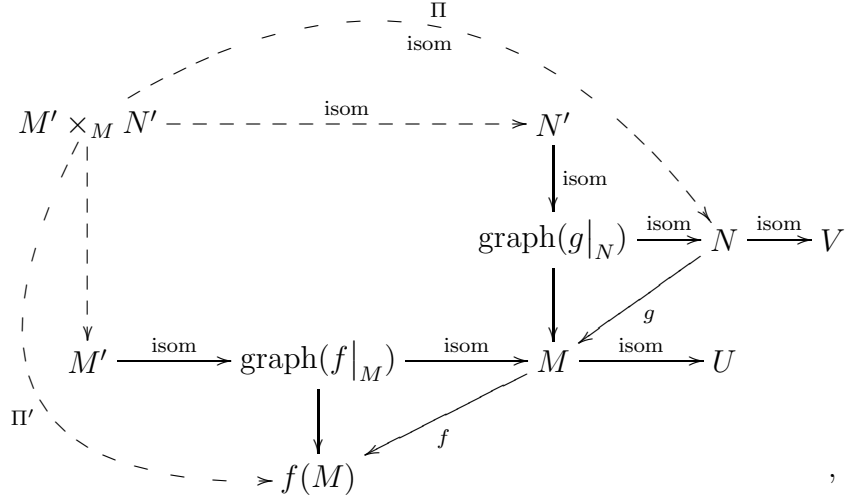
\[
\xymatrix{ & & & \\
M'\times_M N' \ar@{-->}[rr]^-{\isom} \ar@{-->}[dd]
    \ar@{-->}@/^5pc/[rrrd]^{\Pi}_{\isom}
    \ar@{-->}@/_6pc/[dddr]_{\Pi'}
    & & N' \ar[d]^-{\isom} \\
& & \Graph(g\Restr{N}) \ar[r]^-{\isom} \ar[d]
    & N \ar[ld]^{g} \ar[r]^-{\isom}
    & V \\
M' \ar[r]^-{\isom}
    & \Graph(f\Restr{M}) \ar[r]^-{\isom} \ar[d]
    & M \ar[ld]^{f} \ar[r]^-{\isom}
    & U \\
&  f(M) & & & ,}
\]
so
$
\xymatrix{M'\times_M N' \ar[r]^-{\Pi\times\Pi'} & \Graph(f\circ g)}
$
is a basic $\S$-lifting of $f\circ g$ on $N$.

We now prove 5.  Write
\[
M = \{v\in C : F(v) = 0\}
\]
for a rational box $C\subseteq\RR^{n'}$ and an $\S$-polynomial map $F:C\to\RR^{n'+|I|-(m+n)}$ which defines $M$ nonsingularly with good direction $\Pi$.  Write the $n'$-tuple of variables $v$ as $v = (w,x,y,z)$ for tuples of variables $w$ (an $(n' - (m+2n))$-tuple), $x$ (an $m$-tuple), $y$ (an $n$-tuple), and $z$ (an $n$-tuple), where $\Pi(v) = (x,y)$ and $\Pi'(v) = z$.  Also write $x = (x_{I^c},x_I)$, where $I^{c} = \{1,\ldots,m\}\setminus I$.  It is clear that $\Pi$ defines a bijection from $N$ onto the graph of $g\Restr{A_I}$.  We must show that $\PD{}{(F, z)}{(w,x_I,y,z)}$ is nonsingular on $N$.

Write $\tld{A}_I = \Pi_{I^c}(A_I)$ (so $A_I = \tld{A}_I\times\{0\}$, with $\tld{A}_I\subseteq\RR^{I^c}$ and $0\in\RR^I$).  Since the maps
\[
\xymatrix{M \ar[r]^-{\Pi\times\Pi'} & \Graph\left(f\Restr{A_I\times B}\right) \ar[r]^-{\Pi} & A_I\times B \ar[rr]^-{\Pi_{I^c}\times \id_B} & & \tld{A}_I\times B}
\]
are all isomorphisms, it follows that
\[
F\left(h(x_{I^c},y), x_{I^c}, 0, y, f(x_{I^c},0,y)\right) = 0
\]
on $\tld{A}_I\times B$ for some $\C$-analytic function $h:\tld{A}_I\times B \to\RR^{n'-(m+2n)}$, where $0\in\RR^I$ in these functions.  Differentiating in $y$ gives
\begin{equation}\label{eq:Fdiffx2}
\PD{}{F}{w}\PD{}{h}{y} + \PD{}{F}{y} + \PD{}{F}{z}\PD{}{f}{y} = 0
\end{equation}
on $A_I\times B$, with the understanding that the partial derivatives of $F$, $f$, and $h$ are evaluated at $\left(h(x_{I^c},y), x_{I^c}, 0, y, f(x_{I^c},0,y)\right)$, $(x_{I^c},0,y)$, and $(x_{I^c},y)$, respectively.  On $N$, the matrix $\PD{}{F}{(w,x_I,y)}$ is nonsingular because by \eqref{eq:Fdiffx2},
\[
\left[\begin{matrix}
\PD{}{F}{w} & \PD{}{F}{x_I} & \PD{}{F}{y}
\end{matrix}\right]
=
\left[\begin{matrix}
\PD{}{F}{w} & \PD{}{F}{x_I} & \PD{}{F}{z}
\end{matrix}\right]
\left[\begin{matrix}
\id & 0     & -\PD{}{h}{y} \\
0   & \id   & 0 \\
0   & 0     & -\PD{}{f}{y}
\end{matrix}\right],
\]
which expresses $\PD{}{F}{(w,x_I,y)}$ as a product of two nonsingular matrices.  It follows that
\[
\PD{}{(F,z)}{(w,x_I,y,z)}
=
\left[\begin{matrix}
\PD{}{F}{w} & \PD{}{F}{x_I} & \PD{}{F}{y} & \PD{}{F}{z} \\
0           & 0             &             & \id
\end{matrix}\right]
\]
is nonsingular on $N$.
\end{proof}

\begin{remark}\label{rmk:basicLifting}
If $f,g:M\to\RR$ have basic $\S$-liftings, then so do $f+g$ and $fg$.
If $f:M\to\RR$ has a basic $\S$-lifting, then so does $1/f$, provided that $f(x)\neq 0$ on $M$.
\end{remark}

\begin{proof}
Apply Lemma \ref{lemma:basicLifting}.1-3.
\end{proof}

\begin{definition}\label{def:Estrat}
Let $U\subseteq\RR^n$ be an open blowup set, and let $E\subseteq\{1,\ldots,n\}$.  For each $\sigma\in\{-1,0,1\}^E$, let
\begin{equation}\label{eq:Estrat}
U_\sigma = \left\{x\in U : \bigwedge_{i\in E} \sign(x_i) = \sigma_i\right\}.
\end{equation}
The {\bf $E$-stratification of $U$} is the family
\[
\Strat(U,E) = \left\{U_\sigma\right\}_{\sigma\in\{-1,0,1\}^E}.
\]
\end{definition}

Note that for each $\sigma\in\{-1,0,1\}^E$, the set $U_\sigma$ is a blowup manifold defined from the same sequence of blowings-up as $U$, and
\[
\fr_U\left(U_\sigma\right)
=
\bigcup\left\{U_\xi : \text{$\xi\in\{-1,0,1\}^E$, $\xi\neq \sigma$, and $\xi_i\in\{0,\sigma_i\}$ for all $i\in E$}\right\}.
\]
Thus $\Strat(U,E)$ is indeed a stratification of $U$.

\begin{definition}\label{def:SElifting}
Let $U\subseteq\RR^n$ be an open blowup set and $E\subseteq\{1,\ldots,n\}$.  Consider a family $\F = \{f_j:U\to\RR^{m_j}\}_{j\in J}$ of $\C$-analytic functions on $U$, where $J$ is some computable index set and the function $J\to\NN:j\mapsto m_j$ is computable.  An {\bf $(\S,E)$-lifting of $\F$} is a computable map which assigns to each $j\in J$, $\alpha\in\NN^{m_j}$, and $\sigma\in\{-1,0,1\}^E$, a basic $\S$-lifting of $\PDn{\alpha}{f_j}{x}$ on $U_\sigma$.  If $\F = \{f\}$ for a single function $f$, we say $(\S,E)$-lifting ``of $f$'' rather than ``of $\{f\}$''.
\end{definition}

\begin{definition}\label{def:Ecompatible}
For open blowup sets $U\subseteq\RR^n$ and $V\subseteq\RR^m$, and sets $E\subseteq\{1,\ldots,n\}$ and $D\subseteq\{1,\ldots,m\}$, a map $g:V\to U$ is {\bf $(D,E)$-compatible} if for all $N\in\Stratify(V,D)$ there exists $M\in\Stratify(U,E)$ such that $g(N)\subseteq M$.
\end{definition}

\begin{notation}\label{notation:refine}
For any $A\subseteq\RR^n$ and $N\subseteq\{1,\ldots,n\}$, let
\[
A|_N = \{x_{N^c} : x\in A, x_N = 0\},
\]
where $N^c = \{1,\ldots,n\}\setminus N$.  If $A\subseteq U$, then for any function $f$ on $U$, write $f\Restr{A|_N}$ for the pullback of $f$ by the map $A|_N\to\RR:x_{N^c}\mapsto (x_{N^c},0)$, where $0\in\RR^N$.  Likewise, for any family $\F$ of functions on $U$, write $\F\Restr{A|_N}$ for the pullback of $\F$ by $x_{N^c}\mapsto (x_{N^c},0)$.
\end{notation}

In Definition \ref{def:pullback} we wrote $\F\Restr{A}$ for restriction of $\F$ to $A$, which is the pullback of $\F$ by the inclusion map $A\hookrightarrow U$.  The notations $\F\Restr{A}$ and $\F\Restr{A|_N}$ can be distinguished by considering the ambient space of the sets $A$ and $A|_N$.  For example, for $0\in\RR^N$, the family $\F\Restr{A|_N\times\{0\}}$ is the restriction of $\F$ to $A|_N\times\{0\}$, which is a family of functions in $x$, but $\F\Restr{A|_N}$ is a family of functions in $x_{N^c}$.

\begin{lemmas}[Lifting Lemmas, Part I]\label{lemma:lifting1}
Let $\F = \{F_j:U\to\RR^{m_j}\}_{j\in J}$ be a computably indexed family of $\C$-analytic functions on an open blowup set $U\subseteq\RR^n$, and let $E\subseteq\{1,\ldots,n\}$.
\begin{enumerate}{\setlength{\itemsep}{5pt}
\item{\bf Slicing.}
If $\F$ has an $(\S,E)$-lifting and $E\subseteq D\subseteq\{1,\ldots,n\}$, then $\F$ has an $(\S,D)$-lifting.

\item{\bf Restricting.}
If $U = \Bus(I,i,A)$ and $V = \Bus(I,i,B)$ for open rational boxes $A$ and $B$, with $B\subseteq A$, and if $\F$ has an $(\S,E)$-lifting, then $\F\Restr{V}$ has an $(\S,E)$-lifting.

\item{\bf Compatible Composition.}
If $\F$ has an $(\S,E)$-lifting, and if $g:V\to U$ is a $(D,E)$-compatible $\C$-analytic function with an $(\S,D)$-lifting, where $V\subseteq\RR^m$ and $D\subseteq\{1,\ldots,m\}$, then $g^*\F$ has an $(\S,D)$-lifting.

\item{\bf Division by Variables.}
Let $k\in E$ and $d\in\NN$ be such that for each $j\in J$ there exists a $\C$-analytic function $g_j:U\to\RR^{m_j}$ such that $f_j(x) = x_{k}^{d}g_j(x)$ on $U$.  Define $x_{k}^{-d}\F = \{g_j\}_{j\in J}$.  If $\F$ has an $(\S,E)$-lifting, then $x_{k}^{-d}\F$ has an $(\S,E)$-lifting.

\item{\bf Refinement.}
Let $N\subseteq\{1,\ldots,n\}$ be such that $U|_N$ is nonempty.  If $\F$ has an $(\S,E)$-lifting, then $\F\Restr{U|_N}$ has an $(\S,E\setminus N)$-lifting.

\item{\bf Maps with Trivial Components.}
Let $N\subseteq\{1,\ldots,n\}$, $V\subseteq\RR^{N^c}$ be an open blowup set, $B\subseteq\RR^N$ be an open rational box, and  $g:V\to\RR^{N^c}$ be $\C$-analytic.  Define $G:V\times B\to\RR^n$ by $G(y) = (g(y_{N^c}), y_N)$.  If $g$ has an $(\S,E\setminus N)$-lifting and is $(E\setminus N, E\setminus N)$-compatible, then $G$ has an $(\S,E\setminus N)$-lifting and is $(E,E)$-compatible.
}\end{enumerate}
\end{lemmas}

\begin{proof}
Statements 1 and 2 follow from Lemma \ref{lemma:basicLifting}.2,4.

We now prove 3.  Let $N\in\Strat(V,D)$, and choose the unique $M\subseteq\Strat(U,E)$ such that $g(N)\subseteq M$.  For each $f\in\F$, $\alpha\in\NN^n$, and $\beta\in\NN^m$, the functions $\PDn{\alpha}{f}{x}\Restr{M}$ and $\PDn{\beta}{g}{y}\Restr{N}$ have basic $\S$-liftings, where $x = (x_1,\ldots,x_n)$ and $y = (y_1,\ldots,y_m)$.  Thus by Lemma \ref{lemma:basicLifting}.4, each function $\PDn{\alpha}{f}{x}\circ g\Restr{N}$ has a basic $\S$-lifting.  Therefore by Remark \ref{rmk:basicLifting} and repeated application of the chain rule, each $\PD{\beta}{(f\circ g)}{y}$ has a basic $\S$-lifting on $N$.  This proves 3.

We now prove 4.  Consider some $f\in\F$, $\alpha\in\NN^n$, and $M\in\Strat(U,E)$.  If $M\subseteq\{x\in U : x_k\neq 0\}$, then differentiating $g(x) = x_{j}^{-d} f(x)$ using the the product rule, and applying Remark \ref{rmk:basicLifting}, shows that $\PD{\alpha}{g}{x}$ has a basic $\S$-lifting on $M$. If $M\subseteq\{x\in U : x_k = 0\}$, then by Lemma \ref{lemma:derivRestr},
\[
\frac{1}{\alpha!} \PDn{\alpha}{g}{x}(x)
=
\frac{1}{(\alpha_k+d)!}
\frac{\partial^{|\alpha|+d} f}{\partial x^{\alpha+de_k}}(x),
\quad\text{for $x\in U$ with $x_k = 0$,}
\]
which shows that $\PDn{\alpha}{g}{x}$ has a basic $\S$-lifting on $M$.  This proves 4.

We now prove 5.  By Statement 1, $\F$ has an $(\S,E\cup N)$-lifting.  By only considering the members of $\Strat(U,E\cup N)$ contained in $\{x\in\RR^n : x_N = 0\}$ and only considering partial derivatives $\PDn{\alpha}{f}{x_{N^c}}$ with $f\in\F$ and $\alpha\in\NN^{N^c}$, we obtain an $(\S,E\setminus N)$-lifting of $\F|_N$.  This proves 5.

Statement 6 is obvious.
\end{proof}

\begin{lemmas}[Lifting Lemmas, Part II]\label{lemma:lifting2}
\hfill
\begin{enumerate}{\setlength{\itemsep}{5pt}
\item{\bf Linear transformations.}
Let $E\subseteq\{1,\ldots,n\}$.  For each $\lambda\in\QQ^{E^c\times E}$ and open rational box $A\subseteq\RR^n$, the map $T_\lambda:A\to\RR^n$ is $(E,E)$-compatible and has an $(\emptyset,\emptyset)$-lifting.

\item{\bf Translations by implicitly defined functions.}
Let $f:U\to\RR^n$ be a $\C$-analytic function, where $U\subseteq\RR^{m+n}$ is open.  Suppose that $\IF(f;\cl(A),\cl(B))$ holds for some bounded, open, rational boxes $A\subseteq\RR^m$ and $B\subseteq\RR^n$ such that $\cl(A)\times \cl(B)\subseteq U$, and let $g:\cl(A)\to B$ be the function implicitly defined by $f(x,g(x)) = 0$ on $\cl(A)$.  Fix $\epsilon\in\QQ^{n}_{+}$ such that $[g(x)-\epsilon,g(x)+\epsilon]\subseteq B$ for all $x\in \cl(A)$, and define $G:A\times(-\epsilon,\epsilon)\to U$ by
\[
G(x,y) = (x, y + g(x)).
\]
If $E\subseteq\{1,\ldots,m\}$ and $f$ has an $(\S,E)$-lifting, then $G$ has an $(\S,E)$-lifting and is $(E,E)$-compatible.

\item{\bf Blowings-up.}
Let $I\subseteq\{1,\ldots,n\}$, and let $U\subseteq\RR^n$ be an open blowup set such that $C = \{x\in U : x_I = 0\}$ is nonempty.  Let $i\in I$, and let $\pi_i:U_i\to U$ be the $i$th standard chart of the blowing-up of $U$ with center $C$.  Then for any $E\subseteq\{1,\ldots,n\}$, $\pi_i:U_i\to U$ has an $(\emptyset,\emptyset)$-lifting and is $(E\cup\{i\},E)$-compatible.
}\end{enumerate}
\end{lemmas}

\begin{proof}
By definition, $T_\lambda(y) = \left(\left(y_i + \sum_{j\in E} \lambda_{i,j} y_j\right)_{i\in E^c}, y_E\right)$, so Statement 1 is obvious.

We now prove 2.  The $(E,E)$-compatibility of $G$ is clear.  Fix $M\in\Strat(A,E)$.  Lemma \ref{lemma:basicLifting}.5 shows that $g$ has a basic $\S$-lifting on $M$.  Basic $\S$-liftings on $M$ of all the partial derivatives of $g$ can be constructed by repeatedly differentiating the equation $f(x,g(x)) = 0$ to solve for the partial derivatives of $g$, using the formula $\left(\PD{}{f}{y}\right)^{-1} = \adj\left(\PD{}{f}{y}\right) / \det\left(\PD{}{f}{y}\right)$, and then applying Remark \ref{rmk:basicLifting}.  Basic $\S$-liftings for all partial derivatives of $G$ on $M\times B$ can be constructed using the basic $\S$-liftings of the partial derivatives of $g$ on $M$ and Remark \ref{rmk:basicLifting}. This proves 2.

We now prove 3.  The set $U_i = \pi_{i}^{-1}(U)$ is an open blowup set since $U$ is, and the fact that $\pi_i:U_i\to U$ has an $(\emptyset,\emptyset)$-lifting and is $(E\cup\{i\},E)$-compatible follows easily from its defining formula, $\pi_i(y) = \left(y_{I^c}, y_i, y_i y_{I_i}\right)$.  This proves 3.
\end{proof}

\section*{\bf Part IV: Desingularization}

Throughout all of Part IV, we fix an IF-system $\C$ and a family $\S\subseteq\C$. Unless explicitly stated otherwise, it should be assumed that $\C$ is quasianalytic.  Write
\[
\S = \{S_\sigma\}_{\sigma\in\Sigma}
\]
for functions $S_\sigma:[-r_\sigma,r_\sigma]\to\RR$, where $r_\sigma\in\QQ_{+}^{\eta(\sigma)}$.  For all statements involving $\S$ which have any effective content, we assume that the map
\[
\Sigma\to\bigcup_{n\in\NN}\QQ_{+}^{n} : \sigma\mapsto r_\sigma
\]
is computable, and hence, that the index set $\Sigma$ and the arity map $\eta:\Sigma\to\NN$ are also computable.  Write
\[
\Delta(\S) = \left\{\PDn{\alpha}{S_\sigma}{x}\right\}_{\sigma\in\Sigma, \alpha\in\NN^{\eta(\sigma)}}.
\]

We briefly outline Part IV.  Section \ref{s:oracles} defines the approximation and precision oracles for $\S$.  Sections \ref{s:basicPres} and \ref{s:pres} discuss basic $\S$-presentations and $\S$-presentations, respectively, which are data structures used in our effective desingularization theorems.  These two sections complete all the hard work needed to obtain an effective local resolution of singularities theorem, which is then given rather effortlessly at the beginning of Section \ref{s:desingThms} (see Theorem \ref{thm:presDesing}).  Section \ref{s:desingThms} then proceeds to prove  an effective fiber cutting theorem (see Proposition \ref{prop:SparamFC}) and an effective theorem of the complement (see Theorem \ref{thm:compl}), which are then combined to give an effective parameterization theorem for the $0$-definable sets of $\RR_{\S}$ (see Theorem \ref{thm:SparamDefinable}).  Section \ref{s:IFconstants} constructs the smallest IF-system $\C(E)$ containing $\C$ and $E$, where $E$ is an additive group subgroup of $\RR$ which contains $\C_0$.  The IF-system $\C(E)$ is useful when discussing definability, rather than $0$-definability, since the former uses parameters.  Finally, Section \ref{s:MT} completes the proofs of the main model theoretic results of the paper, the first of which being our characterization of the decidability of the theory of $\RR_{\S}$.

\section{The Approximation and Precision Oracles}\label{s:oracles}

\begin{definition}\label{def:approxOracle}
The {\bf approximation oracle for $\S$} is an oracle which acts as a $C^\infty$ approximation algorithm for the family $\S$.
\end{definition}

Thus, given $\sigma\in\Sigma$, $\alpha\in\NN^{\eta(\sigma)}$, and names for a compact rational box $B\subseteq[-r_\sigma,r_\sigma]$ and a rational open interval $I$, the approximation oracle for $\S$ stops if and only if
\begin{equation}\label{eq:approx}
\PDn{\alpha}{S_\sigma}{x}(B) \subseteq I.
\end{equation}

\begin{definition}\label{def:precOracle}
The {\bf precision oracle for $\S$} acts as follows:
\begin{quote}
Given the following data:
\begin{itemize}
\item
a name for an $\S$-polynomial map $P:D\to\RR^{\dim(D)-d}$, where $D\subseteq\RR^m$ and $d\in\{0,\ldots,\dim(D)\}$,

\item
an injection $\lambda:\{1,\ldots,d\}\to\{1,\ldots,m\}$ such that $\Pi_\lambda(D)$ is open in $\RR^d$,

\item
a name for a bounded rational box manifold $B$ which is open in $D$ with $\cl(B)\subseteq D$,

\item
$i\in\{1,\ldots,m\}$,
\end{itemize}
if we write $\varphi = (\varphi_1,\ldots,\varphi_m):\Pi_\lambda(B)\to B$ for the section of the projection $\Pi_\lambda:\RR^m\to\RR^d$ implicitly defined by $f\circ\varphi(x_\lambda) = 0$ on $\Pi_\lambda(B)$, the oracle stops if and only if
\begin{equation}\label{eq:prec}
\text{$\varphi_i(x_\lambda) = 0$ for all $x_\lambda\in \Pi_\lambda(B)$.}
\end{equation}
\end{quote}
\end{definition}

The theory of $\RR_{\S}$ clearly decides the approximation and precision oracles for $\S$, since \eqref{eq:approx} and \eqref{eq:prec} are expressible as $\L_{\S}$-sentences which can be effectively constructed from the data given as input to the two oracles.  The main purpose of this paper is to prove that the converse is also true, namely, that these two oracles decide the theory of $\RR_{\S}$.

\begin{proposition}\label{prop:precOracle}
Relative to the approximation oracle for $\S$, the following oracle is computably equivalent to the precision oracle for $\S$:
\begin{equation}\label{eq:FmanifoldPrec}
\text{\parbox{5.5in}{
Given a name for some $\S$-manifold $M\subseteq\RR^m$, and given $i\in\{1,\ldots,m\}$, the oracle states whether or not $M$ is a subset of $\{x\in\RR^m : x_i = 0\}$.
}}
\end{equation}
\end{proposition}

\begin{proof}
Assume that the oracle \eqref{eq:FmanifoldPrec} is decidable.  Suppose we are given names for the data $P:D\to\RR^{\dim(D) - d}$, $B$, $\lambda$, and $i$ from Definition \ref{def:precOracle}, and let $\varphi = (\varphi_1,\ldots,\varphi_m):\Pi_\lambda(B)\to B$ for the section of the projection $\Pi_\lambda:\RR^m\to\RR^d$ implicitly defined by $P\circ\varphi(x_\lambda) = 0$ on $\Pi_\lambda(B)$.  Then $M = \{x\in B : P(x) = 0\}$ is an $\S$-manifold, and $M\subseteq\{x\in\RR^m : x_i = 0\}$ if and only if $\varphi(x_\lambda) = 0$ for all $x_\lambda\in\Pi_\lambda(B)$.  So the precision oracle for $\S$ is decidable.

Conversely, assume that the approximation and precision oracles for $\S$ are decidable.  Suppose we are given a name $(\name(P),\lambda)$ for an $\S$-manifold $M = \{x\in D : P(x) = 0\} \subseteq\RR^m$ and some $i\in\{1,\ldots,m\}$, where $P:D\to\RR^{\dim(D) - d}$ and $\lambda:\{1,\ldots,d\}\to\{1,\ldots,m\}$.  We want to determine if
\begin{equation}\label{eq:xi=0?}
M\subseteq\{x\in\RR^m : x_i = 0\}
\end{equation}
is true.  Since $\Pi_\lambda(M)$ is open in $\RR^d$, then \eqref{eq:xi=0?} is clearly false if $i\in\im(\lambda)$.  So assume that $i\not\in\im(\lambda)$.  Using the approximation oracle for $\S$, by Remark \ref{rmk:IFsection}.1(c) we may find a name for a bounded rational box manifold $B$ such that $B$ is open $D$, $\cl(B) \subseteq D$, and $\IF_\lambda(P;B)$ holds.  Let $\varphi = (\varphi_1,\ldots,\varphi_m):\Pi_\lambda(B)\to B$ be the section of the projection $\Pi_\lambda:\RR^m\to\RR^d$ implicitly defined by $P\circ\varphi(x_\lambda) = 0$ on $\Pi_\lambda(B)$.  Since $M$ is connected, the quasianalyticity of $\C$ implies that \eqref{eq:xi=0?} holds if and only if $\varphi_i(x_\lambda) = 0$ for all $x_\lambda\in\Pi_\lambda(B)$.  If $\varphi_i(x_\lambda) \neq 0$ for some $x_\lambda \in \Pi_\lambda(B)$, then this can be effectively verified using the approximation oracle, since we can search for a compact rational box $A\subseteq \Pi_\lambda(B)$ such that $\varphi_\lambda(A) \subseteq \RR\setminus\{0\}$. If $\varphi_i(x_\lambda) =0$ for all $x_\lambda \in \Pi_\lambda(B)$, then this can be effectively verified using the precision oracle for $\S$.  Therefore by running these two verification procedures simultaneously, we may effectively determine if \eqref{eq:xi=0?} is true.
\end{proof}

Remark \ref{rmk:IFsection}.1(c) implies that if we are given valid input data for the precision oracle for $\S$, then we can use the approximation oracle for $\S$ to effectively verify that the input data is actually valid.  In contrast, suppose we are given a name $(\name(f),\lambda)$ for an $\S$-manifold $M = \{x\in D : P(x) = 0\}$.  How can we effectively verify that the set $M$ is actually an $\S$-manifold?  It does not seem possible to effectively verify this using only the approximation oracle for $\S$, for it does not seem possible to verify that $M$ is connected, and when $M$ is not compact, it does seem possible to verify that $\det\PD{}{P}{x_{\lambda'}}(x)\neq 0$ for all $x\in M$, where $\lambda'$ is defined as in Remark \ref{rmk:IFsection}.  For this reason we define the precision oracle for $\S$ as in Definition \ref{def:precOracle}, not \eqref{eq:FmanifoldPrec}.

\section{Basic Presentations}\label{s:basicPres}

\begin{notation}\label{notation:EA}
For each $i\in\{1,\ldots,n\}$, let
\[
H_i = \{x\in\RR^n : x_i = 0\}.
\]
For each $A\subseteq\RR^n$ and $E\subseteq\{1,\ldots,n\}$, let
\[
E_A = \{i\in E : A\cap H_i\neq\emptyset\}.
\]
For each $a = (a_1,\ldots,a_n)\in\RR^n$, we simply write $E_a$ instead of $E_{\{a\}}$.  Thus
\[
E_a = \{i\in E : a_i = 0\}.
\]
\end{notation}

\begin{definition}\label{def:basicPres}
A {\bf basic presentation on $U$} is a tuple
\begin{equation}\label{eq:basicPres}
(\F,E;K),
\end{equation}
where $\F$ is a finite nonempty family of $\C$-analytic functions on an open blowup set $U\subseteq\RR^n$, and $K$ is a compact blowup set which is contained in $U$ and which is defined by the same sequence of blowings-up as $U$.
For each $a\in U$, let $\F_a$ be the family of germs at $a$ of the functions in $\F$.  We call $(\F_a,E_a)$ the {\bf germ of $(\F,E)$ at $a$}.
\end{definition}

For the rest of the section, $(\F,E;K)$ denotes a basic presentation on $U\subseteq\RR^n$.  In addition, fix a set $I\subseteq\Cen(K)$, let
\[
C = \{x\in U : x_I = 0\},
\]
and for each $i\in I$, let $\pi_i:U_i\to U$ be the $i$th standard chart of the blowing-up of $U$ with center $C$.

\begin{definition}\label{def:basicSPres}
We call $(\F,E;K)$ a {\bf basic $\S$-presentation} if $\F$ has an $(\S,E)$-lifting and if $\F$ is computably $C^\infty$ relative to the approximation oracle for $\S$. A {\bf representation} for a basic $\S$-presentation $(\F,E;K)$ consists of the following data:
\begin{itemize}
\item
the number $n$ and the set $E$;

\item
names for liftings of the blowup sets $U$ and $K$ (using the same sequence of blowings-up);

\item
an $(\S,E)$-lifting of $\F$;

\item
a $C^\infty$ approximation algorithm for $\F$ relative to the approximation oracle for $\S$.
\end{itemize}
\end{definition}

\begin{definition}\label{def:Forder}
For each $a\in U$, define the {\bf order of $\F$ at $a\in U$} by
\[
\ord(\F;a) = \min\{\ord(f;a) : f\in\F\}.
\]
For each $A\subseteq U$, define the {\bf order of $\F$ on $A$} by
\[
\ord(\F;A) = \sup\{\ord(\F;a) : a\in A\},
\]
and define the {\bf order of $\F$ along $A$} by
\[
\ord_A\F = \min\{\ord(\F;a) : a\in A\}.
\]
Finally, define the {\bf order of $\F$} by
\[
\ord\F = \ord(\F;U).
\]
\end{definition}

\begin{notation}\label{notation:fE}
For each $a\in U$ and each $i\in E_a$, let
\[
d_{a,i} = \min\{\ord_{H_i}(f;a) : f\in\F\},
\]
and define
\[
\Div(\F,E;a) = d_a = (d_{a,i})_{i\in E_a}.
\]
If $\ord(\F;a) < \infty$, then $d_a\in\NN^{E_a}$ and $x_{E_a}^{d_a}$ is the greatest common divisor of $\F_a$ which is a monomial in $x_{E_a}$; in this case, for each $f\in\F$ we define $f_{E_a}$ to be the unique $\C$-analytic germ at $a$ such that $f(x) = x_{E_a}^{d_a} f_{E_a}(x)$ near $a$.   If $\ord(\F;a) = \infty$, then $d_a = (\infty)_{i\in E}$ and $\F_a$ is a family of germs of zero functions; in this case, for each $f\in\F$ we define $f_{E_a} = 0$.

Similarly, for each $i\in E$ let
\[
d_i = \min\{\ord_{U\cap H_i}(f) : f\in\F\},
\]
and define
\[
\Div(\F,E) = d = (d_i)_{i\in E}.
\]
If $\F$ contains a nonzero function, then $x_{E}^{d}$ is the greatest common divisor for $\F$ which is a monomial in $x_E$; in this case, for each $f\in\F$ we define $f_E:U\to\RR$ be the unique $\C$-analytic function such that $f(x) = x_{E}^{d} f_E(x)$ on $U$.   If $\F$ is a family of zero functions, then $d = (\infty)_{i\in E}$; in this case, for each $f\in\F$ we define $f_E:U\to\RR$ to be the zero function on $U$.

Define
\[
\F_E = \{f_E\}_{f\in\F}.
\]
For each $e\in\NN^E$ such that $e\leq d$, define
\[
x^{-e}_{E}\F = \{x^{-e}_{E}f\}_{f\in\F}.
\]
Thus $\F_E = x_{E}^{-d}\F$ when $d$ is finite.
\end{notation}

\begin{definition}\label{def:FEorder}
For each $a\in U$, define the {\bf order of $(\F,E)$ at $a\in U$} by
\[
\ord(\F,E;a) = \min\{\ord(f_{E_a};a):f\in\F\}.
\]
For each $A\subseteq U$, define the {\bf order of $(\F,E)$ on $A$} by
\[
\ord(\F,E;A) = \sup\{\ord(\F,E;a) : a\in A\},
\]
and define the {\bf order of $(\F,E)$ along $A$} by
\[
\ord_A(\F,E) = \min\{\ord(\F,E;a) : a\in A\}.
\]
Finally, define the {\bf order of $(\F,E)$} by
\[
\ord(\F,E) = \ord(\F,E;U).
\]
\end{definition}

\begin{remarks}\label{rmk:Forder}\hfill
\begin{enumerate}{\setlength{\itemsep}{5pt}
\item
If $x$ is sufficiently close to $a\in U$, then $E_x = \{i\in E_a : a_i=0\}$, $\ord(\F;x)\leq \ord(\F;a)$, and $\ord(\F,E;x)\leq \ord(\F,E;a)$.

\item
The set $U$ is open and is connected by Lemma \ref{lemma:blowupSet}.4.  Thus for any $f\in\F$, the quasianalyticity of $\C$ implies that if $f$ vanishes identically on a neighborhood of a point $a\in U$, then $f$ vanishes identically on all of $U$.  It follows that $\F$ is a family of zero functions if and only if $\ord(\F;a) = \infty$ for some $a\in U$.

Similarly, for each $i\in\{1,\ldots,n\}$, the set $U\cap H_i$ is connected because it is a blowup manifold.  Thus for any $f\in\F$, if $k\in\NN$ is such that $f(x) = \PD{}{f}{x_i}(x) = \cdots = \PD{k-1}{f}{x_i}(x) = 0$ for all $x\in H_i$ sufficiently close to a point $a\in U\cap H_i$, then $f(x) = \PD{}{f}{x_i}(x) = \cdots = \PD{k-1}{f}{x_i}(x) = 0$ for all $x\in U\cap H_i$.  This means that if $x_{i}^{k}$ divides $f$ locally near a point $a\in U\cap H_i$, then $x_{i}^{k}$ divides $f$ on all of $U$.  It follows that for all $a\in U$ and all $f\in\F$,
\[
\ord(f_{E_a};a) = \ord(f_E;a),
\]
so
\begin{equation}\label{eq:ordLocal}
\ord(\F,E;a) = \ord(\F_E;a)
\end{equation}
and
\begin{equation}\label{eq:divLocal}
d_a = (d_i)_{i\in E_a}.
\end{equation}
Also,
\begin{equation}\label{eq:ordDivLocal}
\ord(\F;a) = |\Div(\F,E;a)| + \ord(\F,E;a),
\end{equation}
so
\begin{equation}\label{eq:ordDiv}
\ord\F \leq |\Div(\F,E)| + \ord(\F,E).
\end{equation}
}\end{enumerate}
\end{remarks}

\begin{lemma}\label{lemma:SpresZero}
Suppose that $(\F,E;K)$ is a basic $\S$-presentation.  Write $\F = \{F_j\}_{j\in J}$ for  a finite index set $J$, let $N\subseteq\{1,\ldots,n\}$, and suppose that $A = \{x\in U : x_N = 0\}$ is nonempty.  Then relative to the approximation and precision oracles for $\S$, the set
\begin{equation}\label{eq:SpresZero}
\left\{(j,\alpha)\in J\times\NN^n : \text{$\PDn{\alpha}{F_j}{x}(x) = 0$ for all $x\in A$}\right\}
\end{equation}
is computable.
\end{lemma}

\begin{proof}
By Lemma \ref{lemma:lifting1}.1, $\F$ has an $(\S,E\cup N)$-lifting.  Consider $f\in\F$ and $\alpha\in\NN^n$.  Fix $M\in\Strat(U,E\cup N)$ which is open in $A$, and let
\[
\xymatrix{M' \ar[r]^-{\Pi_\lambda} & \Graph\left(\PD{\alpha}{f}{x}\Restr{M}\right)} \]
be a basic $\S$-lifting given by the $(\S,E\cup N)$-lifting of $\F$, where $M'\subseteq\RR^{n'}$ and $\lambda:\{1,\ldots,n+1\}\to\{1,\ldots,n'\}$ is an injection.  The set $A$ is a connected manifold and $M$ is open in $A$, so $\PDn{\alpha}{f}{x}$ vanishes identically on $A$ if and only if $\PDn{\alpha}{f}{x}$ vanishes identically on $M$, which occurs if and only if $M' \subseteq\{x\in\RR^{n'} : x_{\lambda(n+1)} = 0\}$.  By Proposition \ref{prop:precOracle}, whether or not the later condition is true can be effectively determined relative to our oracles.
\end{proof}

\begin{lemma}\label{lemma:SpresZeroDivOrder}
If $(\F,E;K)$ is a basic $\S$-presentation, then relative to the approximation and precision oracles for $\S$, we can effectively determine if $\F$ contains a nonzero function.  If it is determined that $\F$ contains a nonzero function, then relative to our two oracles, we can compute $\Div(\F,E)$ and can effectively find some $m\in\NN$ such that $\ord(\F,E;K) \leq m$.
\end{lemma}

\begin{proof}
To see if $\F$ contains a nonzero function, apply Lemma \ref{lemma:SpresZero} with $N = \emptyset$ and compute the fiber of \eqref{eq:SpresZero} over $\alpha = 0$.

Suppose it was determined that $\F$ contains a nonzero function.  Then for each $i\in E$ and nonzero $f\in\F$, we can compute $\ord_{H_i\cap U}(f)$ using Proposition \ref{prop:orderAlong}.2 and Lemma \ref{lemma:SpresZero} with $N = \{i\}$.  We can therefore compute $\Div(\F,E)$.  Thus by Proposition \ref{prop:compDivVar}, for each $f\in\F$ we have a computably $C^\infty$ approximation algorithm for $f_E$.

Now, let $\B$ be the set of all finite families of bounded, open, rational boxes $\{B_i\}_{i\in I}$ such that $K\subseteq\bigcup_{i\in I} B_i$ and $\cl(B_i)\subseteq U$ for all $i\in I$.  There exist $\{B_i\}_{i\in I}\in\B$ and corresponding families $\{f_i\}_{i\in I}\subseteq\F$ and $\{\alpha_i\}_{i\in I}\subseteq\NN^n$ such that $\PDn{\alpha_i}{(f_i)_E}{x}(B_i)\subseteq\RR\setminus\{0\}$ for all $i\in I$.
Since $K$ is a co-c.e.\ compact subset of the c.e.\ open set $U$, we can construct a computable enumeration $\{\{B_i\}_{i\in I_j}\}_{j\in\NN}$ of $\B$.  By using a time sharing procedure, for each $j\in\NN$, $i\in I_j$, $f\in\F$, and $\alpha\in\NN^n$, we can try to verify that $\PDn{\alpha}{f_E}{x}(B_i) \subseteq\RR\setminus\{0\}$.  Eventually we will find some $j\in\NN$ and families $\{f_i\}_{i\in I_j}\subseteq\F$ and $\{\alpha_i\}_{i\in I_j}\subseteq\NN^n$ such that $\PDn{\alpha_i}{(f_i)_E}{x}(B_i) \subseteq\RR\setminus\{0\}$ for each $i\in I_j$, so $\ord(\F,E;K) \leq \max\{|\alpha_i| : i\in I_j\}$.
\end{proof}

\begin{lemma}\label{lemma:fEblowup}
If $\ord_C(\F,E) = m < \infty$, then for all $i\in I$ and $f\in\F$,
\[
(f\circ\pi_i)_{E\cup\{i\}}(y) = y_{i}^{-m}(f_E\circ\pi_i)(y) .
\]
\end{lemma}

\begin{proof}
Fix $i\in I$.  For each $f\in\F$, $f(x) = x_{E}^{d} f_E(x)$, so
\[
f\circ\pi_i(y) = y_{E\setminus\{i\}}^{d_{E\setminus\{i\}}} y_{i}^{|d_{E\cap I}|} f_E\circ\pi_i(y).
\]
Since $\ord_C(f_E)\geq m$, Lemma \ref{lemma:strictTrans} shows that
\[
f_E\circ\pi_i(y) = y_{i}^{m}\tld{f}_i(y)
\]
for a $\C$-analytic function $\tld{f}_i:U_i\to\RR$.  Thus $y_{E\setminus\{i\}}^{d_{E\setminus\{i\}}} y_{i}^{|d_{E\cap I}|+m}$ is a common divisor of $\pi_{i}^{*}\F$ which a monomial in $y_{E\cup\{i\}}$.  To finish we need to show that this is the greatest common divisor of $\pi_{i}^{*}\F$ which is a monomial in $y_{E\cup\{i\}}$.

Suppose for a contradiction that there exists $j\in E\cup\{i\}$ such that $y_j$ divides $\tld{f}_i(y)$ for all $f\in\F$.  If $j = i$, then Lemma \ref{lemma:strictTrans} shows that $\ord_C(f_E) > m$ for all $f\in\F$, which contradicts the assumption that $\ord_C(\F,E) = m$.  So suppose that $j \neq i$.  Writing $x = \pi_i(y)$ gives
\[
0
=
y_{i}^{m}\tld{f}_i(y)\Restr{y_j=0}
=
f_E(y_{I_{i}^{c}}, y_i y_{I_i})\Restr{y_j = 0}
=
f_E(x)\Restr{x_j=0}.
\]
Since this holds for all $y\in U_i$ with $y_j = 0$, this holds for all $x\in U$ with $x_j = 0$.  So $x_j$ divides $f_E(x)$ for all $f\in\F$, which contradicts the definition of $f_E$.
\end{proof}

\begin{lemma}\label{lemma:orderBlowup}
If $\ord_C(\F,E) = \ord(\F,E) = m < \infty$, then for all $i\in I$,
\[
\ord(\pi_{i}^{*}\F,E\cup\{i\}) \leq m.
\]
\end{lemma}

(Note that the hypothesis of the lemma implies that $\ord(\F,E;x) = m$ for all $x\in C$.)

\begin{proof}
Let $i\in I$ and $b\in U_i$, and write $a = \pi_i(b)$.  If $b_i\neq 0$, then $\ord(\F,E\cup\{i\};b) = \ord(\F,E;a) \leq m$ because $\pi_i$ restricts to an isomorphism on $U_i\setminus\pi_{i}^{-1}(C)$.  So we may assume that $b_i = 0$, and hence $a\in C$.
Lemma \ref{lemma:fEblowup} implies that for all $f\in\F$,
\begin{equation}\label{eq:orderBound1}
\ord(\pi_{i}^{*}\F,E\cup\{i\};b) \leq \ord\left(y_{i}^{-m}(f_E\circ\pi_i)(y); b\right).
\end{equation}

Since $\ord(\F,E;a) = m$, we may now fix $f\in\F$ such that $\ord(f_E;a) = m$.
Because $\ord_C(\F,E) = m$, $\ord(f_E;x) = m$ for all $x\in C$ sufficiently close to $a$.  Therefore Proposition \ref{prop:orderAlong}.3 implies that there exists $\alpha\in\NN^I$ such that $|\alpha| = m$ and $\PDn{\alpha}{f_E}{x_I}(a)\neq 0$.  We may choose such an $\alpha$ with $\alpha_i$ minimal.  Therefore Lemma \ref{lemma:STorderBound} shows that
\begin{equation}\label{eq:fEaST}
\PDn{\alpha_{I_i}}{(f_E)'_i}{y_{I_i}}(b)\neq 0,
\end{equation}
where $(f_E)'_i(y) = y_{i}^{-m}(f_E\circ\pi_i)(y)$, which is the strict transform of $f_E$ by $\pi_i$.  From \eqref{eq:orderBound1} and \eqref{eq:fEaST} we get
\begin{equation}\label{eq:orderBound2}
\ord(\pi_{i}^{*}\F,E\cup\{i\};b) \leq \ord((f_E)'_i;b) \leq |\alpha_{I_i}| = |\alpha| - \alpha_i \leq |\alpha| = m.
\end{equation}
\end{proof}

Note that in the proof of Lemma \ref{lemma:orderBlowup}, if $\alpha_i$ happened to be positive, then $|\alpha_{I_i}| < |\alpha|$, in which case \eqref{eq:orderBound2} would give $\ord(\pi_{i}^{*}\F,E\cup\{i\};b) < m$.  This observation leads naturally to the following definition and to Lemma \ref{lemma:refinableBlowup}, which improves upon Lemma \ref{lemma:orderBlowup} under stronger assumptions.

\begin{definition}\label{def:refinable}
Let $m$ be a positive integer, let $N$ be a nonempty subset of $\{1,\ldots,n\}$, and let $r\in\{0,|N|\}$.  For  each $A\subseteq U$, we say that $(\F,E)$ is {\bf $(m,N,r)$-refinable on $A$} if $\{x\in A : x_N = 0\}$ is nonempty, and if there exists an open set $V$, with $A\subseteq V\subseteq U$, such that the following hold:
\begin{enumerate}{\setlength{\itemsep}{5pt}
\item[]\emph{Case 1}: $r = 0$.

In this case, $N = \{i\}$ for some $i\in\{1,\ldots,n\}\setminus E$, and there exists $f_i\in\F$ such that $\PD{m}{(f_i)_E}{x_i}(x)\neq 0$ for all $x\in V$, and such that $\PD{m-1}{(f_i)_E}{x_i}(x) = 0$ for all $x\in V\cap H_i$.

\item[]\emph{Case 2}: $r = |N|$.

In this case, $N\subseteq E$, and for each $i\in N$ there exist $f_i\in\F$ and $\alpha_i=(\alpha_{i,j})_{j\in N}\in\NN^{N}$ such that $|\alpha_i| = m$, $\alpha_{i,i} > 0$, and $\PDn{\alpha_i}{(f_i)_E}{x_N}(x)\neq 0$ for all $x\in V$.
}\end{enumerate}
In both cases 1 and 2, we say that $\{(f_i,\alpha_i)\}_{i\in N}$ {\bf witnesses} the fact that $(\F,E)$ is $(m,N,r)$-refinable on $A$, where $\alpha_i = (m)\in\NN^{\{i\}}$ when $r=0$.  For each $a\in U$, we say that $(\F,E)$ is {\bf $(m,N,r)$-refinable at $a$} (or that the germ $(\F_a,E_a)$ is $(m,N,r)$-refinable) if $(\F,E)$ is $(m,N,r)$-refinable on $\{a\}$.
We say that the basic presentation $(\F,E;K)$ is {\bf $(m,N,r)$-refinable} if $(\F,E)$ is $(m,N,r)$-refinable on $U$ and $N\subseteq\Cen(K)$.
\end{definition}

\begin{remark}\label{rmk:refinableOnK}
Suppose that $(\F,E;K)$ is a basic $\S$-presentation such that $(\F,E)$ is $(m,N,r)$-refinable on $K$ and $N\subseteq\Cen(K)$.
The condition $N\subseteq\Cen(K)$ can be effectively verified because $\Cen(K)$ can be computed through repeated application of Lemma \ref{lemma:blowupSet}.1.  Also, relative to the approximation and precision oracles for $\S$, we can effectively verify that $(\F,E)$ is $(m,N,r)$-refinable on $K$, and we can use Lemma \ref{lemma:blowupSetVar} to effectively construct an open blowup set $V$ such that $(\F\Restr{V},E;K)$ is $(m,N,r)$-refinable, where $K\subseteq V \subseteq U$ and $V$ is defined by the same sequence of blowings-up as $K$.
\end{remark}

\begin{remarks}\label{rmk:refinable}
Suppose that $(\F,E)$ is $(m,N,r)$-refinable on $U$.
\begin{enumerate}{\setlength{\itemsep}{5pt}
\item
Then $\ord(\F,E) \leq m$.

\begin{proof}
Let $\{(f_i,\alpha_i)\}_{i\in N}$ witness the fact that $(\F,E)$ is $(m,N,r)$-refinable on $U$, and fix $i\in N$.  Then for all $a\in U$, $\ord(\F,E;a) \leq \ord((f_i)_E;a) \leq |\alpha_i| = m$.
\end{proof}

\item
We have $\ord_C(\F,E) = m$ if and only if $\ord(\F,E;x) = m$ for all $x\in C$.

\begin{proof}
This follows from the previous remark and Definition \ref{def:FEorder}.
\end{proof}

\item
If $r=0$, where $N = \{i\}$, then $\ord(\F,E;x) < m$ for all $x\in U$ with $x_i\neq 0$.

\begin{proof}
Fix $f\in\F$ such that $\PD{m}{f_E}{x_i}(x)\neq 0$ for all $x\in U$, and such that $\PD{m-1}{f_E}{x_i}(x) = 0$ for all $x\in U\cap H_i$.  Since $U$ is connected, $\PD{m}{f_E}{x_i}$ has a constant nonzero sign on $U$. For each $a\in \Pi_{\{i\}^c}(U)$, the fiber $\{x_i : x\in U, \Pi_{\{i\}^c}(x) = a\}$ is an open interval (the connectedness follows from Lemma \ref{lemma:blowupSet}.4), so the map $x_i\mapsto \PD{m-1}{f_E}{x_i}(a,x_i)$ is strictly monotonic on this fiber and is zero at $x_i = 0$.  So for all $x\in U$ with $x_i\neq 0$, we have $\PD{m-1}{f_E}{x_i}(x)\neq 0$, and hence $\ord(\F,E;x) < m$.
\end{proof}

\item
If $\ord_C(\F,E) = m$, then $N\subseteq I$. \\

\noindent If $r=0$, this is obvious from the previous remark.  The following proof handles both cases $r=0$ and $r = |N|$ at once.

\begin{proof}
Let $i\in N$, and fix $\alpha\in\NN^n$ and $f\in\F$ such that $\supp(\alpha)\subseteq N$, $|\alpha| = m$, $\alpha_i > 0$, and $\PDn{\alpha}{f_E}{x}(x)\neq 0$ for all $x\in U$.  By Corollary \ref{cor:orderAlong},
\[
f_E(x) = \sum_{\beta\in\NN^{I}_{m}} \frac{1}{\beta!} x_{I}^{\beta} \PDn{\beta}{f_E}{x_I}(x_{I^c},0) + g(x)
\]
on $U$ for some $g\in\lb x_I\rb^{m+1}$.  Thus
\begin{equation}\label{eq:fEdiff}
\PDn{\alpha}{f_E}{x}(x)
=
\sum_{\beta\in\NN^{I}_{m} \atop \beta\geq\alpha_I}
\frac{1}{(\beta - \alpha_I)!} x_{I}^{\beta - \alpha_I}
\PDmix{m+|\alpha_{I^c}|}{f_E}{x_{I^c}}{\alpha_{I^c}}{x_I}{\beta}(x_{I^c},0)
+ \PDn{\alpha}{g}{x}(x).
\end{equation}
Suppose for a contradiction that $i\not\in I$.  Then $|\alpha_I| < m$, so $\beta\neq\alpha_I$ for all $\beta\in\NN^{I}_{m}$ such that $\beta\geq \alpha_I$.  Therefore evaluating the right hand side of \eqref{eq:fEdiff} at $x_I = 0$ gives $0$, which contradicts the fact that $\PDn{\alpha}{f_E}{x}(x)\neq 0$ for all $x\in U$.  So in fact $i\in I$.
\end{proof}

\item
If $\ord_C(\F,E) = m$, then $\ord_C(\F) = |d_{E\cap I}| + m$.

\begin{proof}
For any $a\in C$, \eqref{eq:ordDivLocal} shows that
\[
\ord(\F;a) = |\Div(\F,E;a)| + \ord(\F,E;a) \geq |d_{E\cap I}| + m.
\]
Equality holds when $a$ is chosen so that $a_j\neq 0$ for all $j\in E\setminus I$.
\end{proof}

\item
Let $p\in\NN$.  If $\ord_C(\F,E) = m$ and $\ord_C(\F) \geq p$, then $y_{i}^{p}$ divides $\pi_{i}^{*}\F$ for each $i\in I$.

\begin{proof}
For each $f\in\F$, we have $f(x) = x_{E}^{d} f_E(x)$, so
\[
f\circ\pi_i(y) = y_{E\setminus\{i\}}^{d_{E\setminus\{i\}}} \, y_{i}^{|d_{E\cap I}| + m} \left(y_{i}^{-m} (f_E\circ \pi_i)(y)\right),
\]
and $|d_{E\cap I} | + m \geq p$ by the previous remark.
\end{proof}
}\end{enumerate}
\end{remarks}

\begin{lemma}\label{lemma:refinableBlowup}
Suppose that $(\F,E;K)$ is an $(m,N,r)$-refinable basic $\S$-presentation, and let $\{(f_j,\alpha_j)\}_{j\in N}$ witness this fact, where each $\alpha_j = (\alpha_{j,k})_{k\in N}$.  Also suppose that $\ord_C(\F,E) = m$  (so $N\subseteq I$ by Remark \ref{rmk:refinable}.4).  Then relative to the approximation and precision oracles for $\S$, the following hold for all $i\in I$:
\begin{enumerate}{\setlength{\itemsep}{5pt}
\item
If $i\in I\setminus N$, then for any co-c.e.\ compact set $K_i\subseteq U_i$, we can effectively find a rational number $\delta_i > 0$ such that $\left(\pi_{i}^{*}\F, E\cup\{i\}\right)$ is $(m,N,r)$-refinable on $\{y\in K_i : |y_i|\leq\delta_i\}$.

\item
If $i\in N$, then we can effectively find a rational number $\epsilon_i > 0$ such that
\[
\ord(\pi_{i}^{*}\F;E\cup\{i\};y) \leq m - \alpha_{i,i} < m
\]
for all $y\in\pi_{i}^{-1}(K)$ such that $|y_j|\leq\epsilon_i$ for all $j\in I$.
}\end{enumerate}
\end{lemma}

\begin{proof}
Let $i\in I$.  Lemma \ref{lemma:fEblowup} implies that for each $j\in N$, $(f_j\circ\pi_i)_{E\cup\{i\}}$ is the strict transform of $(f_j)_E$ by $\pi_i$.

We first prove 1.  Suppose that $i\in I\setminus N$, and let $K_i\subseteq U_i$ be co-c.e.\ compact.  For each $j\in N$, since the $i$th component of $(0,\alpha_j)\in\NN^{I\setminus N}\times\NN^N$ is $0$, which is clearly minimal, Condition 1 in Lemma \ref{lemma:STorderBound} implies that
\[
\PDn{\alpha_j}{(f_j\circ\pi_i)_{E\cup\{i\}}}{y_N}(y)\neq 0
\]
for all $y\in\pi_{i}^{-1}(C)$ and $j\in N$.  When $r=0$, note also that by Lemma \ref{lemma:STHi}, $\PD{m-1}{(f_j\circ\pi_i)_{E\cup\{i\}}}{y_j}(y) = 0$ for all $y\in U_i\cap H_j$.  Since $K_i\cap\pi_{i}^{-1}(C)$ is co-c.e.\ compact and the family $\pi_{i}^{*}\F$ is computably $C^\infty$, we can effectively find a rational number $\delta_i > 0$ such that
\begin{equation}\label{eq:InotN}
\PDn{\alpha_j}{(f_j\circ\pi_i)_{E\cup\{i\}}}{y_N}(y)\neq 0
\end{equation}
for all $j\in N$ and all $y\in K_i$ with $|y_i|\leq\delta_i$.  This proves 1.

We now prove 2.  Suppose that $i\in N$.  Condition 2 in Lemma \ref{lemma:STorderBound} implies that
\begin{equation}\label{eq:N}
\PDn{(\alpha_i)_{N_i}}{(f_j\circ\pi_i)_{E\cup\{i\}}}{y_N}(y)\neq 0
\end{equation}
for all $y\in U_i$ with $y_I = 0$. Since $\{y\in \pi_{i}^{-1}(K) : y_I = 0\} = \{x\in K : x_I = 0\}$, which is co-c.e.\ compact, we can effectively find a rational number $\epsilon_i > 0$ such that \eqref{eq:N} holds for all $y\in K_i$ satisfying $|y_j|\leq \epsilon_i$ for all $j\in I$.  Thus for all such $y$,
\begin{equation}\label{eq:orderDrop}
\ord(\pi_{i}^{*}\F,E\cup\{i\};y) \leq |(\alpha_i)_{N_i}| = m - \alpha_{i,i} < m.
\end{equation}
This proves 2.
\end{proof}

Below we use Notation \ref{notation:refine}.

\begin{definition}\label{def:refine}
Let $p,m\in\NN$, with $m > 0$, and let $N\subseteq\{1,\ldots,n\}$ be nonempty.  Let
\begin{eqnarray*}
A
    & = &
    \left\{\alpha\in\NN^{N}_{<m} : \text{$\left.\PDn{\alpha}{f_E}{x_N}\right|_{U|_N}\not\equiv 0$ for some $f\in\F$}\right\},
\\
q
    & = &
    \max\left\{0, p - \left(m+\left|d_{E\cap N}\right|\right)\right\},
\end{eqnarray*}
and define
\[
\mu^{p,m}_{N}(\F,E) =
\lcm\left(\{q\}\cup\{m-|\alpha| : \alpha\in A\}\right).
\]
For brevity, write $\tld{p} = \mu^{p,m}_{N}(\F,E)$.  Define
\[
\partial^{p,m}_{N,E} \F
=
\begin{cases}
\left\{\left(\left.\PDn{\alpha}{f_E}{x_N}\right|_{U|_N}\right)^{\tld{p}/(m-|\alpha|)}
\right\}_{\alpha\in\NN^{N}_{<m},\, f\in\F},
    &
    \text{if $q = 0$,}
\\
\hfill\\
\left\{ \left(x_{E\setminus N}^{d_{E\setminus N}}\right)^{\tld{p}/q} \right\}
\cup
\left\{\left(\left.\PDn{\alpha}{f_E}{x_N}\right|_{U|N}\right)^{\tld{p}/(m-|\alpha|)}
\right\}_{\alpha\in\NN^{N}_{<m},\, f\in\F},
    &
    \text{if $q > 0$,}
\end{cases}
\]
and define
\[
\partial^{p,m}_{N}(\F,E;K) = \left(\partial^{p,m}_{N,E}\F, E\setminus N; K|_N\right).
\]
We call $\partial^{p,m}_{N}(\F,E;K)$ the {\bf $(p,m,N)$-refinement of $(\F,E;K)$}. It is a basic presentation on $U|_N$.
\end{definition}

The following lemma relates the orders of a basic presentation with its refinement.

\begin{lemma}\label{lemma:refine}
Suppose that $(\F,E;K)$ is $(m,N,r)$-refinable, write $\tld{p} = \mu^{p,m}_{N}(\F,E)$, and let $a\in U$ with $a_N = 0$.  Then
\begin{equation}\label{eq:refine1}
\text{$\ord(\F;a) \geq p$ and $\ord(\F,E;a) = m$}
\end{equation}
if and only if
\begin{equation}\label{eq:refine2}
\ord\left(\partial^{p,m}_{N,E}\F ; a_{N^c}\right) \geq \tld{p}.
\end{equation}
\end{lemma}

\begin{proof}
We use the notation of Definition \ref{def:refine}.  Consider $f\in\F$.  Corollary \ref{cor:orderAlong} implies that
\[
f_E(x) = \sum_{\alpha\in\NN^{N}_{<m}} \frac{1}{\alpha!}\PDn{\alpha}{f_E}{x_N}(x_{N^c},0) x_{N}^{\alpha} + \sum_{\alpha\in\NN^{N}_{m}} x_{N}^{\alpha} f_{\alpha}(x)
\]
for some $\C$-analytic functions $f_\alpha:U\to\RR$.  Thus $\ord(f_E;a)\geq m$ if and only if $\ord\left(\PDn{\alpha}{f_E}{x_N};a\right) \geq m-|\alpha|$ for all $\alpha\in\NN^{N}_{<m}$.  This and Remark \ref{rmk:refinable}.1 imply that $\ord(\F,E) = m$ if and only if
\[
\ord\left(\left\{\left(\left.\PDn{\alpha}{f_E}{x_N}\right|_{U|_N}\right)^{\tld{p}/(m-|\alpha|)}
\right\}_{\alpha\in\NN^{N}_{<m},\, f\in\F} ; a_{N^c}\right)
\geq \tld{p}.
\]
Thus \eqref{eq:refine1} and \eqref{eq:refine2} are both false if $\ord(\F,E) < m$.  So assume that $\ord(\F,E) = m$.  Then
\[
\ord(\F;a) = \ord(x_{E\setminus N}^{d_{E\setminus N}};a_{N^c}) + |d_{E\cap N}| + m, \]
so $\ord(\F;a) \geq p$ if and only if $\ord(x_{E\setminus N}^{d_{E\setminus N}};a_{N^c}) \geq p - (m + |d_{E\cap N}|)$.  This holds if $q = 0$.  If $q > 0$, this holds if and only if
\[
\ord\left(\left(x_{E\setminus N}^{d_{E\setminus N}} \right)^{\tld{p}/q};a_{N^c}\right) \geq \tld{p}.
\]
Thus \eqref{eq:refine1} holds if and only \eqref{eq:refine2} holds.
\end{proof}

The following lemma shows that, up to dividing by appropriate powers of the exceptional divisor $y_i$, the operations of blowing-up and refining commute.

\begin{lemma}\label{lemma:blowupRefineCommute}
Suppose that $(\F,E;K)$ is $(m,N,r)$-refinable and that $\ord_C\F\geq p$ and $\ord_C(\F,E) = m$.  (Thus $N\subseteq I$ by Remark \ref{rmk:refinable}.4.)  Let $i\in I\setminus N$, and write $\pi_i:U_i\to U$ and $\tld{\pi}_i:U_i|_N \to U|_N$ for the $i$th standard charts for the blowing-up of $U$ with center $C$ and the blowing-up of $U|_N$ with center $C|_N$, respectively.  Write $\tld{p} = \mu^{p,m}_{N}(\F,E)$.  Then
\begin{equation}\label{eq:blowupRefineCommute}
\partial^{p,m}_{N,E\cup\{i\}}\left(y_{i}^{-p}\, \pi_{i}^{*} \F\right)
=
y_{i}^{-\tld{p}}\, \tld{\pi}_{i}^{*}\, \partial^{p,m}_{N,E}\F.
\end{equation}
\end{lemma}

\begin{proof}
We use the notation of Definition \ref{def:refine}.  We also write $\tld{x} = x_{N^c}$ and $\tld{y} = y_{N^c}$.  Note that
$\pi_i(y) = \left(y_{I^c}, y_i, y_i y_{I_i}\right)$ and
$\tld{\pi}_i(\tld{y}) = \left(y_{I^c}, y_i, y_i y_{I_i\setminus N}\right)$.
We assume that $q > 0$.  (The case $q = 0$ is similar and simpler.)

Consider $f\in\F$, and write $f(x) = x_{E}^{d} f_E(x)$ and
\[
f_E(x) = \sum_{\alpha\in\NN^{N}_{<m}} \frac{1}{\alpha!}\PDn{\alpha}{f_E}{x_N}(\tld{x},0) x_{N}^{\alpha}
+
\sum_{\alpha\in\NN^{N}_{m}} x_{N}^{\alpha} f_\alpha(x).
\]
Thus
\[
f\circ\pi_i(y) = y_{E\setminus\{i\}}^{d_{E\setminus\{i\}}} y_{i}^{|d_{E\cap I}| + m}
\left(y_{i}^{-m} (f_E\circ\pi_i)(y)\right),
\]
and
\[
y_{i}^{-m} (f_E\circ\pi_i)(y)
=
\sum_{\alpha\in\NN^{N}_{<m}} \frac{1}{\alpha!} y_{i}^{|\alpha|-m} \PDn{\alpha}{f_E}{x_N}(\tld{\pi}_i(\tld{y}),0) \,y_{N}^{\alpha}
+
\sum_{\alpha\in\NN^{N}_{m}} y_{N}^{\alpha} (f_\alpha\circ\pi_i)(y);
\]
note that for each $\alpha\in\NN^{N}_{<m}$, the function $y_{i}^{|\alpha|-m} \PDn{\alpha}{f_E}{x_N}(\tld{\pi}_i(\tld{y}),0)$ is indeed $\C$-analytic because Lemma \ref{lemma:strictTrans} shows that
\[
\ord_{\pi^{-1}_{i}(C)}\left(\PDn{\alpha}{f_E}{x_N}(\tld{\pi}_i(\tld{y}),0)\right) = \ord_C\left(\PDn{\alpha}{f_E}{x_N}(\tld{x},0)\right) \geq m - |\alpha|.
\]
So $\partial^{p,m}_{N,E\cup\{i\}}\left(y_{i}^{-p}\, \pi_{i}^{*} \F\right)$ consists of
\begin{equation}\label{eq:blowupRefineCommute1}
\left(
y_{E\setminus(N\cup\{i\})}^{d_{E\setminus(N\cup\{i\})}}
\,y_{i}^{|d_{E\cap I}| + m - p}
\right)^{\tld{p}/q}
\end{equation}
and the functions
\begin{equation}\label{eq:blowupRefineCommute2}
\left(
y_{i}^{|\alpha|-m}
\PDn{\alpha}{f_E}{x_N}\left(\tld{\pi}_i(\tld{y}),0\right)
\right)^{\tld{p}/(m-|\alpha|)},
\quad
\text{for each $\alpha\in\NN^{N}_{<m}$ and $f\in\F$.}
\end{equation}
And, $y_{i}^{-\tld{p}}\, \tld{\pi}_{i}^{*}\, \partial^{p,m}_{N,E}\F$ consists of
\begin{equation}\label{eq:blowupRefineCommute3}
y_{i}^{-\tld{p}}
\left(
y_{E\setminus(N\cup\{i\})}^{d_{E\setminus(N\cup\{i\})}}
y_{i}^{|d_{(E\setminus N)\cap I}|}
\right)^{\tld{p}/q}
\end{equation}
and the functions
\begin{equation}\label{eq:blowupRefineCommute4}
y_{i}^{-\tld{p}}
\left(
\PDn{\alpha}{f_E}{x_N}\left(\tld{\pi}_i(\tld{y}),0\right)
\right)^{\tld{p}/(m-|\alpha|)},
\quad
\text{for each $\alpha\in\NN^{N}_{<m}$ and $f\in\F$.}
\end{equation}
Clearly \eqref{eq:blowupRefineCommute2} and \eqref{eq:blowupRefineCommute4} are the same, and
\begin{eqnarray*}
\left(|d_{E\cap I}| + m - p\right)\frac{\tld{p}}{q}
    & = &
    \left(|d_{(E\setminus N)\cap I}| + |d_{E\cap N}| + m - p\right) \frac{\tld{p}}{q}
\\
    & = &
    \left(|d_{(E\setminus N)\cap I}| - q\right)
    \frac{\tld{p}}{q}
\\
    & = &
    |d_{(E\setminus N)\cap I}| \frac{\tld{p}}{q} - \tld{p},
\end{eqnarray*}
so \eqref{eq:blowupRefineCommute1} and \eqref{eq:blowupRefineCommute3} are also the same.  This proves \eqref{eq:blowupRefineCommute}.
\end{proof}

\begin{lemma}\label{lemma:blowupCase0}
Suppose that $\ord(\F,E) = 0$ and $|d| \geq p > 0$, where $d = \Div(\F,E)$ and $p\in\NN$.  Suppose that $I\subseteq E$ is minimal such that $|d_I| \geq p$, and let $i\in I$ and $d' = \Div(y_{i}^{-p}\pi_{i}^{*}\F,E)$.  Then
\[
d'_{E\setminus\{i\}} = d_{E\setminus\{i\}}
\quad
\text{and}
\quad
d'_i < d_i.
\]
\end{lemma}

\begin{proof}
The pullback of $x_{E}^{d}$ by $\pi_i$ equals $y_{E\setminus\{i\}}^{d_{E\setminus\{i\}}} y_{i}^{|d_I|}$, and also equals $y_{E}^{d'} y_{i}^{p}$ since $\ord(\F,E) = 0$.  So
\[
y_{E}^{d'} = y_{E\setminus\{i\}}^{d_{E\setminus\{i\}}} \, y_{i}^{|d_I|-p}.
\]
Thus $d'_{E\setminus\{i\}} = d_{E\setminus\{i\}}$.  The minimality of $I$ implies that $p > |d_{I\setminus\{i\}}| = |d_I| - d_i$, so $d'_i = |d_I| - p < d_i$.
\end{proof}

The final task in this section is to show how to perform a change of coordinates to pullback a basic presentation to one which is refinable.

\begin{notation}\label{notation:germ}
For $a,b\in\RR^n$, write $F:\RR^{n}_{b}\to\RR^{n}_{a}$ to denote the germ at $b$ of a function $F:V\to\RR^n$ defined in a neighborhood $V$ of $b$, where $F(b) = a$.  We call $F:\RR^{n}_{b}\to\RR^{n}_{a}$ a  ``$\C$-analytic local isomorphism'' if $F$ is $\C$-analytic and $\det\PD{}{F}{x}(b)\neq 0$.
\end{notation}

\begin{definition}\label{def:localCoordTrans}
Let $a\in U$.  A {\bf local coordinate transformation for $(\F_a,E_a)$} is a $\C$-analytic local isomorphism $F = (F_1,\ldots,F_n):\RR^{n}_{b}\to\RR^{n}_{a}$ such that
\begin{equation}\label{eq:localCoordTrans}
F_i(y) = y_i
\quad\text{for all $i\in E_a$.}
\end{equation}
We call
\[
(F^*\F_a, E_a)
\]
the {\bf pullback of $(\F_a,E_a)$ by $F$}.  Note that $(F^*\F)_b = F^*\F_a$, and that \eqref{eq:localCoordTrans} implies that $(E_a)_b = E_a$.  Thus $(F^*\F_a, E_a)$ is the germ at $b$ of $(F^*\F,E_a)$.
\end{definition}

Local coordinate transformations may be composed in the natural way: if $F:\RR^{n}_{b}\to\RR^{n}_{a}$ is a local coordinate transformation for $(\F_a,E_a)$, and $G:\RR^{n}_{c}\to\RR^{n}_{b}$ is a local coordinates transformation for $(F^*\F_a,E_a)$, then $F\circ G:\RR^{n}_{c}\to\RR^{n}_{a}$ is a local coordinate transformation for $(\F_a,E_a)$, whose pullback by $F\circ G$ is $((F\circ G)^*\F_a,E_a)$.  These composition and pullback operations are associative.

\begin{definition}\label{def:localAdmissCoordTrans}
Consider the following two types of local coordinate transformations $F:\RR^{n}_{b}\to\RR^{n}_{a}$ for a germ $(\F_a,E_a)$, where $a\in U$.
\begin{enumerate}{\setlength{\itemsep}{5pt}
\item\emph{Linear Transformation for $(\F_a,E_a)$}:\\
Let $b = T^{-1}_{\lambda}(a)$ and
\[
F(y) = T_\lambda(y)
\]
for some $\lambda\in\QQ^{D^c\times D}$, where $D = \{i\}$ for some $i\in\{1,\ldots,n\}$ if $E_a$ is empty, and $D = E_a$ if $E_a$ is nonempty.

\item\emph{Translation by an Implicitly Defined Function for $(\F_a,E_a)$}:\\
Suppose that there exist $f\in\F$, $\alpha\in\NN^n$, and $i\in\{1,\ldots,n\}\setminus E_a$ such that
\[
\PDn{\alpha}{f_E}{x}(a) = 0 \quad\text{and}\quad \frac{\partial^{|\alpha|+1} f_E}{\partial x^\alpha \partial x_i}(a) \neq 0.
\]
Let $g:\RR^{\{i\}^c}_{a_{\{i\}^c}} \to \RR_{a_i}^{\{i\}}$ be the $\C$-analytic germ implicitly defined by
\[
\PDn{\alpha}{f_E}{x}\left(x_{\{i\}^c}, g(x_{\{i\}^c})\right) = 0.
\]
Let $b = (a_{\{i\}^c},0)$ and
\[
F(y) = \left(y_{\{i\}^c},\,\, y_i + g(y_{\{i\}^c})\right).
\]
}\end{enumerate}
A {\bf basic admissible local coordinate transformation for $(\F_a,E_a)$} is either a linear transformation for $(\F_a,E_a)$ or a translation by an implicitly defined function for $(\F_a,E_a)$.  An {\bf admissible local coordinate transformation for $(\F_a,E_a)$} is a local coordinate transformation for $(\F_a,E_a)$ which is a composition of finitely many basic admissible local coordinate transformations.
\end{definition}

We now give the analogs of Definitions \ref{def:localCoordTrans} and \ref{def:localAdmissCoordTrans} for basic presentations, rather than germs of basic presentations.

\begin{definition}\label{def:coordTrans}
A {\bf coordinate transformation for $(\F,E;K)$} consists of a $\C$-analytic embedding $F:U'\to U$ and a set $K'$, where $U'$ is an open blowup set contained in $\RR^n$, $K'$ is a compact blowup set which is contained in $U'$ and is defined by the same sequence of blowings-up as $U'$, and
\begin{equation}\label{eq:coordTrans}
\text{$F_i(y) = y_i$ for all $i\in E_{F(U')}$ and $y\in U'$.}
\end{equation}
We write
\[
F:(U';K')\to (U;K)
\]
to denote the coordinate transformation for $(\F,E;K)$ consisting of the map $F:U'\to U$ and compact set $K'$.  Note that \eqref{eq:coordTrans} implies that $(E_{F(U')})_{U'} = E_{F(U')}$.  Define the {\bf pullback of $(\F,E;K)$ by $F:(U';K')\to(U;K)$} to be the basic presentation
\[
(F^*\F,E_{F(U')};K').
\]
\end{definition}

Coordinate transformations for basic presentations can be composed in the natural way: if $F:(U';K')\to(U;K)$ is a coordinate transformation for $(\F,E;K)$ and $G:(U'';K'')\to(U';K')$ is a coordinate transformation for $(F^*\F,E_{F(U')};K')$, then $F\circ G:(U'';K'')\to(U;K)$ is a coordinate transformation for $(\F,E;K)$, whose pullback by  $F\circ G:(U'';K'')\to(U;K)$ is $((F\circ G)^*\F,E_{F\circ G(U'')}; K'')$.  Note that
\begin{equation}\label{eq:coordTransG}
\text{$G_i(y) = y_i$ for all $i\in (E_{F(U')})_{G(U'')}$ and  $y\in U''$.}
\end{equation}
It follows from \eqref{eq:coordTrans} and \eqref{eq:coordTransG} that $(E_{F(U')})_{G(U'')} = E_{F\circ G(U'')}$, so these composition and pullback operations are associative.

\begin{definition}\label{def:admissCoordTrans}
Consider the following three types of coordinate transformations $F:(U';K')\to(U;K)$ for a basic presentation $(\F,E;K)$.
\begin{enumerate}{\setlength{\itemsep}{5pt}
\item\emph{Inclusion for $(\F,E;K)$}:

The set $K'$ is given by $K' = K\cap B$ for some closed rational box $B\subseteq\RR^n$, the set $U'$ is an open blowup set such that $K' \subseteq U' \subseteq U$, where $U'$ is defined by the same sequence of blowings-up as $K$, and $F:U'\to U$ is the inclusion map.

\item\emph{Linear Transformation for $(\F,E;K)$}:\\
Given
\begin{itemize}
\item
a bounded open rational box $B$ such that $\cl(B)\subseteq U$,

\item
$\lambda\in\QQ^{D^c\times D}$, where $D = \{i\}$ for some $i\in\{1,\ldots,n\}$ if $E_B$ is empty, and $D = E_B$ if $E_B$ is nonempty,

\item
a bounded open rational box $U'\subseteq\RR^n$ such that $T_\lambda(\cl(U')) \subseteq B$,
\end{itemize}
define $F:U'\to U$ by $F(y) = T_\lambda(y)$, and let $K'$ be any compact rational box contained in $U'$.

\item\emph{Translation by an Implicitly Defined Function for $(\F,E;K)$}:\\
Given
\begin{itemize}
\item
$f\in\F$,

\item
$\alpha\in\NN$,

\item
$i\in \{1,\ldots,n\}\setminus E$,

\item
bounded open rational boxes $A\subseteq \RR^{\{i\}^c}$ and $B\subseteq\RR^{\{i\}}$ such that $\cl(A)\times\cl(B)\subseteq U$ and such that
\[
\IF\left(\PDn{\alpha}{f_E}{x}; \cl(A), \cl(B)\right)
\]
holds,
\end{itemize}
let $g:\cl(A)\to B$ be the function implicitly defined by $\PDn{\alpha}{f_E}{x}(x_{\{i\}^c},g(x_{\{i\}^c})) = 0$ on $\cl(A)$.  Fix $\epsilon\in\QQ_{+}^{\{i\}}$ such that $[g(x)-\epsilon,g(x)+\epsilon]\subseteq B$ for all $x\in \cl(A)$.  Let $U' = A\times(-\epsilon,\epsilon)$, define $F:U'\to U$ by
\[
F(y) = \left(y_{\{i\}^c},y_i + g(y_{\{i\}^c})\right),
\]
and let $K'$ be any compact rational box contained in $U'$.
}\end{enumerate}
A {\bf basic admissible coordinate transformation for $(\F,E;K)$} is an inclusion for $(\F,E;K)$, a linear transformation for $(\F,E;K)$, or a translation by an implicitly defined function for $(\F,E;K)$.  An {\bf admissible coordinate transformation for $(\F,E;K)$} is a coordinate transformation for $(\F,E;K)$ which is a composition of finitely many basic admissible coordinate transformations.

The {\bf name} of a basic admissible coordinate transformation for $(\F,E;K)$ is defined according to the type of the transformation, as follows (we use the notation from 1-3 above):
\begin{enumerate}{\setlength{\itemsep}{3pt}
\item
$(\name(U'), \name(K'))$,

\item
$(\name(B),D,\lambda,\name(U'),\name(K'))$,

\item
$(\Index(f),\alpha,i,\name(A),\name(B),\epsilon)$, where $\Index(f)$ is an index for $f$ (recall that $\F$ is an indexed family of functions).
}\end{enumerate}
If $F:(U';K')\to(U;K)$ is an admissible coordinate transformation for $(\F,E;K)$ expressed as a composition $F = F_1\circ\cdots\circ F_l$ of basic admissible coordinate transformations, diagramed as follows,
\[
\xymatrix{
(U';K') = (U_l;K_l) \ar[r]^-{F_l}
    & (U_{l-1};K_{l-1}) \ar[r]
    & \cdots \ar[r]
    & (U_1;K_1) \ar[r]^-{F_1}
    & (U_0;K_0) = (U;K) ,
}
\]
then define the {\bf name} of $F:(U';K')\to(U;K)$ to be $(\name(F_1),\ldots,\name(F_l))$, where each $\name(F_i)$ is the name of $F_i:(U_i;K_i)\to(U_{i-1};K_{i-1})$.
\end{definition}

In Definition \ref{def:admissCoordTrans} we defined an ``inclusion for $(\F,E;K)$'', but we did not explicitly give a local analog of this in Definition \ref{def:localAdmissCoordTrans} because its local analog is simply the germ of the identity map at $a$, which is a special type of linear transformation for $(\F_a,E_a)$.  Such inclusions will not actually be used until Section \ref{s:pres}.

\begin{remarks}\label{rmk:admissTrans}
Suppose that $(\F,E;K)$ is a basic $\S$-presentation.
\begin{enumerate}{\setlength{\itemsep}{5pt}
\item
The pullback of $(\F,E;K)$ by an admissible coordinate transformation is a basic $\S$-presentation.

\begin{proof}
This follows from the results of Section \ref{s:compClos} and the Lifting Lemmas \ref{lemma:lifting1} and \ref{lemma:lifting2}.
\end{proof}

\item
Relative to the approximation and precision oracles for $\S$, the set of names of all admissible coordinate transformations for $(\F,E;K)$ is computably enumerable.

\begin{proof}
The name of an admissible coordinate transformation for $(\F,E;K)$ is a finite string of symbols in the alphabet containing all the symbols used to name rational boxes (see Definition \ref{def:box} for a list of these symbols), along with the index set for $\F$ and the curly braces $\{$ and $\}$, which are used along with the integers $1,\ldots,n$ to name subsets of $\{1,\ldots,n\}$.  It is easy to determine when a string is of a syntactically correct form to potentially be a name for an admissible coordinate transformation for $(\F,E;K)$.  We must prove that if we are given such a syntactically correct string, and if this string happens to actually be a name for an admissible coordinate transformation for $(\F,E;K)$, then this fact can be effectively verified relative to our two oracles for $\S$.  Because of the previous remark, it suffices to prove the analogous fact for basic admissible coordinate transformations for $(\F,E;K)$.  Below, we use the notation of Definition \ref{def:admissCoordTrans}.

Keeping in mind the convention set forth in Remark \ref{rmk:blowupSet}.2, this is trivial for admissible inclusions.  So consider a name for a linear transformation.  The condition $\cl(B)\subseteq U$ can be effectively verified because $\cl(B)$ is a compact rational box and $U$ is c.e.\ open.  Whether or not $E_B$ is empty can be determined since $B$ is a rational box.  Finally, the condition $T_\lambda(\cl(U'))\subseteq B$ can be effectively verified because $T_\lambda(\cl(U'))$ is co-c.e. compact by Proposition \ref{prop:compContCompact}.

Now consider a name for a translation by an implicitly defined function.  The condition $\cl(A)\times\cl(B)\subseteq U$ can be effectively verified because $\cl(A)\times\cl(B)$ is a compact rational box and $U$ is c.e.\ open.  The condition $\IF\left(\PDn{\alpha}{f_E}{x}; \cl(A), \cl(B)\right)$ can be effectively verified by Lemma \ref{lemma:IF}.4.  Finally, the condition, $[g(x) - \epsilon, g(x) + \epsilon]\subseteq B$ for all $x\in\cl(A)$, can be effectively verified because the image of $\cl(A)\times[-\epsilon,\epsilon]$ under the map $(x,t)\mapsto g(x)+t$ is co-.c.e\ compact by Proposition \ref{prop:compContCompact}.
\end{proof}

\item
Let $a\in U$, let $F:\RR^{n}_{b}\to\RR^{n}_{a}$ be a local admissible coordinate transformation for the germ $(\F_a,E_a)$, and suppose that the germ $(F^*\F_a,E_a)$ is $(m,N,r)$-refinable.  Then there is a representative $F:U'\to\RR^n$ for the germ $F:\RR^{n}_{b}\to\RR^{n}_{a}$ and a compact rational box $K'$ such that $F:(U';K')\to(U;K)$ is an admissible coordinate transformation for $(\F,E;K)$, $K'$ is a neighborhood of $b$, $E_{F(U')} = E_a$, and $(F^*\F,E_{F(U')})$ is $(m,N,r)$-refinable on $K'$.  Relative to the approximation and precision oracles for $\S$, Remark \ref{rmk:refinableOnK} shows that we can effectively verify that $(F^*\F,E_{F(U')})$ is $(m,N,r)$-refinable on $K'$ and that by possibly shrinking $U'$, in an effective manner, we can ensure that $(\F^*\F,E_{F(U')};K')$ is $(m,N,r)$-refinable.
}\end{enumerate}
\end{remarks}

\begin{lemma}\label{lemma:transRefinable1}
Let $(\F,E;K)$ be a basic $\S$-presentation, and suppose that we know that $\ord(\F,E;K)\leq m$, where $m\in\NN$.  Then relative to the approximation and precision oracles for $\S$, we can effectively find a finite family
\begin{equation}\label{eq:admTransFamily1}
\{F^{(j)}:(U^{(j)};K^{(j)})\to(U;K)\}_{j\in J}
\end{equation}
of admissible coordinate transformations for $(\F,E;K)$ such that
\begin{equation}\label{eq:covering1}
K\subseteq\bigcup_{j\in J} F^{(j)}(\Int(K^{(j)})),
\end{equation}
and such that for each $j\in J$, either
\begin{enumerate}{\setlength{\itemsep}{3pt}
\item
$\ord(\F^{(j)}, E^{(j)}) = 0$, or

\item
$(\F^{(j)},E^{(j)};K^{(j)})$ is $(m^{(j)},N^{(j)},r^{(j)})$-refinable for some $m^{(j)}\leq m$, $N^{(j)}$, and $r^{(j)}$,
}\end{enumerate}
where $(\F^{(j)},E^{(j)};K^{(j)})$ is the pullback of $(\F,E;K)$ by $F^{(j)}:(U^{(j)};K^{(j)})\to(U;K)$.
\end{lemma}

\begin{proof}
Fix $a\in K$.  We claim that either $\ord(\F,E;a)=0$ or there exists a local admissible coordinate transformation $F:\RR^{n}_{b}\to\RR^{n}_{a}$ for $(\F_a,E_a)$ such that the germ $(F^{*}\F_a,E_a)$ is $(m(a),N(a),r(a))$-refinable for some $m(a)\leq m$, $N(a)$, and $r(a)$.

We first show that the claim implies the lemma.  It follows from the claim and Remark \ref{rmk:admissTrans}.3 that for each $a\in K$ there exists an admissible coordinate transformation  $F^{(a)}:(U^{(a)};K^{(a)})\to (U;K)$ such that  $F^{(a)}(\Int(K^{(a)}))$ is a neighborhood of $a$, and such that relative to our oracles, we can effectively verify that the pullback of $(\F,E;K)$ by $F^{(a)}:(U^{(a)};K^{(a)})\to (U;K)$ either has order $0$ or is $(m(a),N(a),r(a))$-refinable for some $m(a)\leq m$, $N(a)$, and $r(a)$.  Since $K$ is compact, there exists a finite set $A\subseteq K$ such that $K\subseteq\bigcup_{a\in A}F^{(a)}(\Int(K^{(a)}))$, and this too can be effectively verified since $K$ is co-c.e.\ compact and $\bigcup_{a\in A}F^{(a)}(\Int(K^{(a)}))$ is c.e.\ open.  By Remark \ref{rmk:admissTrans}.2 we can computably enumerate all admissible coordinate transformations for $(\F,E;K)$, so we will eventually discover such a family $\{F^{(a)}:(U^{(a)};K^{(a)})\to(U;K)\}_{a\in A}$, which proves the lemma from the claim.

We now prove the claim.  Write $m_a = \ord(\F,E;a)$, and note that $m_a \leq m$.  There is nothing to prove if $m_a=0$, so assume that $m_a > 0$.  First suppose that $E_a$ is empty.  It follows from Proposition \ref{prop:Gen} that there exist $f\in\F$ and $\lambda\in\QQ^{\{1,\ldots,n-1\}\times\{n\}}$ such that $\PD{m_a}{(f\circ T_\lambda)}{y_n}(p) \neq 0$, where $p = T^{-1}_{\lambda}(a)$.  Let $g:\RR^{\{n\}^c}_{p_{\{n\}^c}}\to\RR_{p_n}^{\{n\}}$ be the $\C$-analytic germ implicitly defined by $\PD{m_a-1}{(f\circ T_\lambda)}{y_n}(y_{\{n\}^c},g(y_{\{n\}^c})) = 0$ for $y$ near $p_{\{n\}^c}$, and let $G(y) = (y_{\{n\}^c}, y_n + g(y_{\{n\}^c}))$.  Let $F = T_\lambda\circ G$ and $b = F^{-1}(a)$.  Thus $\PD{m_a}{(f\circ F)}{y_n}(b)\neq 0$ and $\PD{m_a-1}{(f\circ F)}{y_n}(y) = 0$ for all $y$ near $b$ with $y_n = 0$, so the pullback of $(\F_a,E_a)$ by $F$ is $(m_a,\{n\},0)$-refinable.

Now suppose that $E_a$ is nonempty.  Let $N_a$ be the support of the set
\begin{equation}\label{eq:Na}
\left\{\alpha\in\NN^{E_a}_{m_a} : \text{there exist $f\in\F$ and $\beta\in\NN^{n}_{m_a}$ such that $\PDn{\beta}{f_{E_a}}{x}(a)\neq 0$ and $\alpha\geq\beta_{E_a}$}\right\}.
\end{equation}
If $E_a = \{1,\ldots,n\}$, then $(\F_a,E_a)$ is $(m_a,N_a,|N_a|)$-refinable.  So suppose that $E_a\neq \{1,\ldots,n\}$.
Then Proposition \ref{prop:Gen} implies that there exists $\lambda\in\QQ^{E_{a}^{c}\times E_a}$ such that for all $\alpha$ in the set \eqref{eq:Na}, there exists $f\in\F$ such that $\PDn{\alpha}{(f_{E_a}\circ T_\lambda)}{y_{E_a}}(b)\neq 0$, where $b = T^{-1}_{\lambda}(a)$.  Thus the pullback of $(\F_a,E_a)$ by $T_\lambda$ is $(m_a,N_a,|N_a|)$-refinable.
\end{proof}

\begin{lemma}\label{lemma:transRefinable2}
Let $(\F,E;K)$ be a basic $\S$-presentation which is $(m,N,r)$-refinable, let $\epsilon > 0$ be rational, and let $i\in N$.  Then relative to the approximation and precision oracles for $\S$, we can effectively find a finite family
\begin{equation}\label{eq:admTransFamily2}
\{F_j:(U^{(j)};K^{(j)})\to(U;K)\}_{j\in J}
\end{equation}
of admissible coordinate transformations for $(\F,E;K)$ such that
\begin{equation}\label{eq:covering2}
\{x\in K : |x_i| \geq \epsilon\} \subseteq\bigcup_{j\in J} F^{(j)}(\Int(K^{(j)})),
\end{equation}
and such that for each $j\in J$,  either
\begin{enumerate}{\setlength{\itemsep}{3pt}
\item
$\ord(\F^{(j)}, E^{(j)}) < m$, or

\item
$(\F^{(j)},E^{(j)};K^{(j)})$ is $(m,N^{(j)},r^{(j)})$-refinable for some $N^{(j)}\subseteq N$ and $r^{(j)} < r$,
}\end{enumerate}
where $(\F^{(j)},E^{(j)};K^{(j)})$ is the pullback of $(\F,E;K)$ by $F^{(j)}:(U^{(j)};K^{(j)})\to(U;K)$.
\end{lemma}

\begin{proof}
Fix $a\in K$ with $|a_i| \geq \epsilon$.  As in the proof of Lemma \ref{lemma:transRefinable1}, it suffices to prove the existence of a local admissible coordinate transformation $F:\RR^{n}_{b}\to\RR^{n}_{a}$ such that the pullback of $(\F_a,E_a)$ by $F$ has the desired reduction in order or local refinability property at $b$.  If $\ord(\F,E;a) < m$, simply take $F$ to be the germ of the identity map.  Now assume that $\ord(\F,E;a) = m$.  Since $a_i\neq 0$, this assumption excludes the case that $r=0$, by Remark \ref{rmk:refinable}.3.  Thus $r = |N|$.

First suppose that $N_a$ is empty.  Since $(\F,E;K)$ is $(m,N,r)$-refinable, we may fix $\beta\in\NN^{E}_{m}$ and $f\in\F$ such that $\PDn{\beta}{f_{E_a}}{x_E}(a)\neq 0$ and $\supp(\beta)\subseteq N$.  Consider $\alpha\in\NN^E$ defined by $\alpha_i = m$ and $\alpha_{E\setminus\{i\}} = 0$.  Note that $\beta_{E_a} = 0$ because $N_a = \emptyset$, and also $\beta_i\leq |\beta| = m$, so $\alpha_{E_a\cup\{i\}} \geq \beta_{E_a\cup\{i\}}$.  Therefore Proposition \ref{prop:Gen} implies that there exists $\lambda\in\QQ^{(E_a\cup\{i\})^c\times(E_a\cup\{i\})}$ such that $\PD{m}{(f_{E_a}\circ T_\lambda)}{y_i}(p)\neq 0$, where $p = T^{-1}_{\lambda}(a)$.  Let $g:\RR^{\{i\}^c}_{p_{\{i\}^c}} \to \RR^{\{i\}}_{p_i}$ be the $\C$-analytic germ implicitly define by $\PD{m-1}{(f_{E_a}\circ T_\lambda)}{y_i}(y_{\{i\}^c}, g(y_{\{i\}^c})) = 0$ for $y_{\{i\}^c}$ near $p_{\{i\}^c}$. Let $G(y) = (y_{\{i\}^c}, y_i + g(y_{\{i\}^c}))$, and $F = T_\lambda\circ G$.  Then the pullback of $(\F_a,E_a)$ by $F$ is $(m,\{i\},0)$-refinable.

Now suppose that $N_a$ is nonempty.  For each $j\in N_a$ there exist $f_j\in\F$ and $\beta_j = (\beta_{j,k})_{k\in N}\in\NN^{N}_{m}$ such that $\beta_{j,j} > 0$ and $\PDn{\beta}{(f_j)_{E_a}}{x_N}(a)\neq 0$.  For each $j\in N_a$ choose $\alpha_j = (\alpha_{j,k})_{k\in N_a}\in\NN^{N_a}_{m}$ such that $\alpha\geq \beta_{j,N_a}$: for example, for each $k\in N_a$ one could let
\[
\alpha_{j,k} = \begin{cases}
\beta_{j,j} + \sum_{l\in N\setminus N_a} \beta_{j,l},
    & \text{if $k = j$,} \\
\beta_{j,k},
    & \text{otherwise.}
\end{cases}
\]
Since $\alpha\geq \beta_{j,N_a}$, we have $\alpha_j\geq \beta_{j,j} > 0$.
By Proposition \ref{prop:Gen} there exists $\lambda\in\QQ^{E_{a}^{c}\times E_a}$ such that for all $j\in N_a$, $\PDn{\alpha_j}{((f_j)_{E_a}\circ T_\lambda)}{y_{N_a}}(b)\neq 0$, where $b = T^{-1}_{\lambda}(a)$.  Thus the pullback of $(\F,E_a)$ by $T_\lambda$ is $(m,N_a,|N_a|)$-refinable.  Since $a_i\neq 0$, we have $i\in N\setminus N_a$, so $|N_a| < |N| = r$.
\end{proof}

\section{Presentations}\label{s:pres}

Fix a basic presentation $(\F,E;K)$ on $U\subseteq\RR^n$.  Let $k\in\NN$, let $m_0,\ldots,m_{k-1}$ be positive integers, let $m_k\in\NN\cup\{\infty\}$, let $N_0,\ldots,N_{k-1}$ be nonempty disjoint subsets of $\Cen(K)$, and for each $i\in\{0,\ldots,k-1\}$ let $r_i\in\{0,|N_i|\}$.  (Note that $k\leq \sum_{j=0}^{k-1}|N_j| \leq n$.)  Inductively define a sequence of basic presentations $(\F_0,E_0;K_0),\ldots,(\F_k,E_k;K_k)$ by setting
\begin{equation}\label{eq:presSeq1}
(\F_0,E_0;K_0) = (\F,E;K) \quad\text{and}\quad p_0 = 0,
\end{equation}
and by setting
\begin{equation}\label{eq:presSeq2}
(\F_i,E_i;K_i) = \partial^{p_{i-1},m_{i-1}}_{N_{i-1}}(\F_{i-1},E_{i-1};K_{i-1})
\quad
\text{and}
\quad
p_i = \mu^{p_{i-1},m_{i-1}}_{N_{i-1}}(\F_{i-1},E_{i-1})
\end{equation}
for each $i\in\{1,\ldots,k\}$.  For each $i\in\{0,\ldots,k\}$, let
\[
d_i = (d_{i,j})_{j\in E_i} = \Div(\F_i,E_i),
\]
and let $U_i\subseteq\RR^{M_i}$ be the domain of $\F_i$, where
\[
M_i = \left(\bigcup_{j=0}^{i-1} N_j\right)^c,
\]
and where the superscript $c$ always denotes complementation in $\{1,\ldots,n\}$.
If $m_k = 0$, define $r_k = |d_k|$.  If $m_k = 0$ and $|d_k|\geq p_k$, let $N_k$ be some subset of $E_k$ which is minimal such that $|d_{k,N_k}| \geq p_k$.  If $m_k > 0$, or if $m_k = 0$ and $|d_k| < p_k$, define $N_k = \emptyset$.

\begin{definition}\label{def:pres}
We call the sequence $\{(\F_i,E_i;K_i)\}_{i\in\{0,\ldots,k\}}$ an {\bf admissible sequence of refinements} for $(\F,E;K)$ if $(\F_i,E_i;K_i)$ is $(m_i,N_i,r_i)$-refinable for all $i\in\{0,\ldots,k-1\}$.  We call
\begin{equation}\label{eq:pres}
\P = \begin{cases}
(\F,E;K:m_0,N_0,r_0,\ldots,m_{k-1},N_{k-1},r_{k-1},m_k),
    & \text{if $m_k > 0$,}
\\
(\F,E;K:m_0,N_0,r_0,\ldots,m_{k-1},N_{k-1},r_{k-1},m_k,N_k,r_k),
    & \text{if $m_k = 0$,}
\end{cases}
\end{equation}
a {\bf presentation} of $(\F,E;K)$ if $\{(\F_i,E_i;K_i)\}_{i\in\{0,\ldots,k\}}$ is an admissible sequence of refinements for $(\F,E;K)$ and $\ord(\F_k,E_k)\leq m_k$.  Thus the data in the presentation $\P$ determines the admissible sequence of refinements $\{(\F_i,E_i;K_i)\}_{i\in\{0,\ldots,k\}}$ and also gives the additional information that $\ord(\F_k,E_k)\leq m_k$.
\end{definition}

For the rest of the section, we shall assume that $\P$, as defined in \eqref{eq:pres}, is a presentation of $(\F,E;K)$, and we shall use all the notation specified prior to Definition \ref{def:pres}; thus, $\{(\F_,E_i;K_i)\}_{i\in\{0,\ldots,k\}}$ denotes an admissible sequence of refinements for $(\F,E;K)$.  Also, for any set $A\subseteq U$ and $i\in\{0,\ldots,k\}$, we shall write
\[
A_i = A|_{N_0\cup\cdots\cup N_{i-1}},
\]
as in Notation \ref{notation:refine}.  Thus $A_i$ is defined from $A$ in the same way that $U_i$ is defined from $U$ and $K_i$ is defined from $K$.

Note that if $(\F,E;K)$ is a basic $\S$-presentation, then $(\F_i,E_i;K_i)$ is a basic $\S$-presentation for each $i\in\{0,\ldots,k\}$.  In fact, by using a representation for $(\F,E;K)$ and the data $m_0,N_0,\ldots,m_{k-1},N_{k-1}$, we can effectively construct a representation for $(\F_i,E_i;K_i)$ for each $i\in\{0,\ldots,k\}$.  For this reason, we make the following definition.

\begin{definition}\label{def:Spres}
We call $\P$ an {\bf $\S$-presentation} if $(\F,E;K)$ is a basic $\S$-presentation.  In this case, a {\bf representation} for $\P$ consists of the following data:
\begin{itemize}
\item
a representation for the basic $\S$-presentation $(\F,E;K)$;

\item
$m_0,N_0,r_0,\ldots,m_{k-1},N_{k-1},r_{k-1},m_k$, and also $N_k$ and $r_k$ when $m_k=0$;

\item
$p_0,\ldots,p_k$ and $d_0,\ldots,d_k$;

\item
for each $i\in\{0,\ldots,k-1\}$, the data $\{(\Index(f_{i,j}),\alpha_{i,j})\}_{j\in N_i}$, where $\{(f_{i,j},\alpha_{i,j})\}_{j\in N_i}$ witnesses the fact that $(\F_i,E_i;K_i)$ is $(m_i,N_i,r_i)$-refinable, and $\Index(f_{i,j})$ is an index for the function $f_{i,j}\in\F_i$ (recall that $\F_i$ is an \emph{indexed} family of functions).
\end{itemize}
\end{definition}

\begin{remark}\label{rmk:Spres}
The key information in a representation for an $\S$-presentation $\P$ is the representation for $(\F,E;K)$ and the data $m_0,N_0,r_0,\ldots,m_{k-1},N_{k-1},r_{k-1},m_k$, along with the choice of $N_k$ when $m_k = 0$ and $|d_k|\geq p_k$.  The other data is included only for computational convenience.  Namely, using the approximation and precision oracles for $\S$, the data $p_0,\ldots,p_k$ and $d_0,\ldots,d_k$ can be computed from a representation for $(\F,E;K)$ and  from the data $m_0,N_0,\ldots,m_{k-1},N_{k-1}$.  But since the data $p_0,\ldots,p_k$ and $d_0,\ldots,d_k$ will be used often, it makes sense to include it in our data structure so that we do not have to repeatedly recompute it.  Similarly, using our oracles, we can effectively verify that for each $i\in\{0,\ldots,k-1\}$, $\{(f_{i,j},\alpha_{i,j})\}_{j\in N_i}$ witnesses the fact that $(\F_i,E_i)$ is $(m_i,N_i,r_i)$-refinable on $K_i$.  So by Remark \ref{rmk:refinableOnK}, we could effectively construct $V$, with $K\subseteq V\subseteq U$, such that $\{(f_{i,j},\alpha_{i,j})\}_{j\in N_i}$ witnesses that $(\F_i\Restr{V_i},E_i;K_i)$ is $(m_i,N_i,r_i)$-refinable for each $i\in\{0,\ldots,k-1\}$.  But then we would be repeatedly and unnecessarily shrinking the neighborhood of $K$ on which we are working, and would be repeatedly recomputing the witness $\{(f_{i,j},\alpha_{i,j})\}_{j\in N_i}$, so it makes more sense to also include $\{(\Index(f_{i,j}),\alpha_{i,j})\}_{j\in N_i}$ in our data structure.
\end{remark}

\begin{definition}\label{def:presRank}
We call
\[
\rank\P =
\begin{cases}
(m_0,r_0,\ldots,m_{k-1},r_{k-1},m_k),
    & \text{if $m_k > 0$,}
\\
(m_0,r_0,\ldots,m_{k-1},r_{k-1},m_k,r_k),
    & \text{if $m_k = 0$,}
\end{cases}
\]
the {\bf rank} of $\P$.  The rank of a presentation is, in general, a member of
\begin{equation}\label{eq:rankValues}
\bigcup_{i\in\NN}(\NN\cup\{\infty\})^i.
\end{equation}
The set of all possible ranks of presentations are well-ordered by giving \eqref{eq:rankValues} the following lexicographical order: for any $a = (a_1,\ldots,a_m)$ and $b = (b_1,\ldots,b_n)$ in \eqref{eq:rankValues}, define $a < b$ if and only if $(a_1,\ldots,a_m,\infty,\infty,\infty,\ldots)$ is less than $(b_1,\ldots,b_n,\infty,\infty,\infty,\ldots)$ lexicographically.
\end{definition}

\begin{definition}\label{def:complete}
We say that $\P$ is {\bf complete} if  $\F_k$ is a family of zero functions (in which case $m_k = \infty$), or if $m_k = 0$ and $|d_k| \geq p_k$.  We say that $\P$ is {\bf incomplete} if $\P$ is not complete.
\end{definition}

\begin{lemma}\label{lemma:completeCenter}
If $\P$ is complete, then for all $i\in\{0,\ldots,k\}$ and all $x\in C$,
\begin{equation}\label{eq:completeCenter}
\ord(\F_i;x_{M_i}) \geq p_i
\quad
\text{and}
\quad
\ord(\F_i,E_i;x_{M_i}) = m_i.
\end{equation}
\end{lemma}

\begin{proof}
The definition of complete implies that \eqref{eq:completeCenter} holds for $i = k$.  And, for any $j\in\{1,\ldots,k\}$, Lemma \ref{lemma:refine} shows that if \eqref{eq:completeCenter} holds for $i = j$, then \eqref{eq:completeCenter} holds for $i = j-1$.
\end{proof}

Our choice of the word ``complete'' in Definition \ref{def:complete} is due to the following two closely related reasons.  First, finding a presentation $\P$ for $(\F,E;K)$ serves as a way of computing an upper bound for $\ord(\F,E)$, since $\ord(\F_i,E_i)\leq m_i$ for all $i\in\{0,\ldots,k\}$.  When $\P$ is complete, Lemma \ref{lemma:completeCenter} implies that $\ord(\F_i,E_i) = m_i$ for all $i\in\{0,\ldots,k\}$, and thus the computation of $\ord(\F,E)$ has been ``completed''.  Second, when $\P$ is complete, we shall use the set
\begin{equation}\label{eq:C}
C = \{x\in U : x_I = 0\}
\end{equation}
as a center of blowing-up in our resolution procedure, where
\begin{equation}\label{eq:I}
I = \bigcup_{j=0}^{k} N_j,
\end{equation}
and thus we have ``completed'' the process of finding the center.

Our effective resolution theorem will employ ``admissible transformations'' for an $\S$-presentation $\P$ (see Definition \ref{def:presAdmTrans}).  Such admissible transformations are of two types: ``admissible coordinate transformation for $\P$'' (see Definition \ref{def:presAdmCoordTrans}) and ``admissible blowup transformations for $\P$'' (see Definitions and Remarks \ref{defrmk:presAdmBlowupTrans}).

\subsection{Admissible Coordinate Transformations for a Presentation}
\label{ss:presAdmCoordTrans}

Recall that we have fixed an admissible sequence of refinements $\{(\F_i,E_i;K_i)\}_{i\in\{0,\ldots,k\}}$ for $(\F,E;K)$, and that we use the notation given prior to Definition \ref{def:pres}.

\begin{definition}\label{def:presAdmissCoordTrans}
Let $l\in\{0,\ldots,k\}$, and suppose that $F_l:(U'_l;K'_l)\to(U_l;K_l)$ is an admissible coordinate transformation for $(\F_l,E_l;K_l)$.  Let $A\subseteq\RR^{N_0\cup\cdots\cup N_{l-1}}$ be a bounded open rational box such that $\cl(F_l(U_l))\times\cl(A)\subseteq U$, and let $B$ be a compact rational box contained in $A$.  Put $U' = U_l\times A$ and $K' = K'_l\times B$, and define $F:U'\to U$ by
\[
F(y) = \left(y_{N_0},\ldots,y_{N_{l-1}}, F_l(y_{M_l})\right).
\]
We call $F:(U';K')\to(U;K)$ a {\bf basic admissible coordinate transformation} for the sequence $\{(\F_i,E_i;K_i)\}_{i\in\{0,\ldots,k\}}$.  The {\bf pullback} of $\{(\F_i,E_i;K_i)\}_{i\in\{0,\ldots,k\}}$ by $F:(U';K')\to(U;K)$ is the admissible sequence of refinements for $(F^*\F,E_{F(U')};K')$ defined by
\[
\{(\F'_i,E'_i;K'_i)\}_{i\in\{0,\ldots,l\}},
\]
where $(\F'_i,E'_i;K'_i)$ is the pullback of $(\F_i,E_i;K_i)$ by $F_i:(U'_i;K'_i)\to (U_i;K_i)$, with
$F_i(y_{M_i}) = \left(y_{N_i},\ldots,y_{N_{l-1}}, F_l(y_{M_l})\right)$, $U'_i = U'\Restr{N_0\cup\cdots\cup N_{i-1}}$, and $K'_i = K'\Restr{N_0\cup\cdots\cup N_{i-1}}$.  Basic admissible coordinate transformations for admissible sequences of refinements can be composed in the natural way: if $F:(U';K')\to(U;K)$, $\{(\F_i,E_i;K_i)\}_{i\in\{0,\ldots,k\}}$, and $\{(\F'_i,E'_i;K'_i)\}_{i\in\{0,\ldots,l\}}$ are as above, and if $G:(U'';K'')\to(U';K')$ is a basic admissible coordinate transformation for $\{(\F'_i,E'_i;K'_i)\}_{i\in\{0,\ldots,l\}}$, with pullback $\{(\F''_i,E''_i;K''_i)\}_{i\in\{0,\ldots,m\}}$ for some $m\in\{0,\ldots,l\}$, then $F\circ G:(U'';K'')\to(U;K)$ pulls back $\{(\F_i,E_i;K_i)\}_{i\in\{0,\ldots,k\}}$ to $\{(\F''_i,E''_i;K''_i)\}_{i\in\{0,\ldots,m\}}$.  An {\bf admissible coordinate transformation} for $\{(\F_i,E_i;K_i)\}_{i\in\{0,\ldots,k\}}$ is a composition of finitely many basic admissible coordinate transformations for admissible sequences of refinements, the first of which acts on $\{(\F_i,E_i;K_i)\}_{i\in\{0,\ldots,k\}}$.
\end{definition}

\begin{remarks}\label{rmk:presAdmissCoordTrans}
Let $F:(U;K')\to(U;K)$ be an admissible coordinate transformation for $\{(\F_i,E_i;K_i)\}_{i\in\{0,\ldots,k\}}$, and let $\{(\F'_i,E'_i;K'_i)\}_{i\in\{0,\ldots,l\}}$ be the the pullback of $\{(\F_i,E_i;K_i)\}_{i\in\{0,\ldots,k\}}$ by $F:(U';K')\to(U;K)$.
\begin{enumerate}{\setlength{\itemsep}{5pt}
\item
In order to keep the terminology straight, we briefly compare Definitions \ref{def:coordTrans} and \ref{def:presAdmissCoordTrans}.  Note that $F:(U';K')\to(U;K)$ is also a coordinate transformation for the basic presentation $(\F;E;K)$.  If $(\F',E';K')$ is the pullback of $(\F;E;K)$ by $F:(U;K')\to(U;K)$, then $(\F',E';K') = (\F'_0,E'_0;K'_0)$.

\item
For each $i\in\{0,\ldots,l-1\}$, $(\F'_i,E'_i;K'_i)$ is $(m_i,N_i,r_i)$-refinable.
}\end{enumerate}
\end{remarks}

We now fix a common notation to be used in Lemmas \ref{lemma:presCoord1}-\ref{lemma:presCoord3}, since the statements of these three lemmas are similar.  Each lemma constructs a finite family of admissible coordinate transformations for $\{(\F_i,E_i;K_i)\}_{i\in\{0,\ldots,k\}}$, which we denote by
\begin{equation}\label{eq:Fcovering}
\{F^{(j)}:(U^{(j)};K^{(j)})\to(U;K)\}_{j\in J}.
\end{equation}
For each $j\in J$, write
\[
F^{(j)}(y) = \left(y_{N_0},\ldots,y_{N_{l_j-1}}, F^{(j)}_{l_j}(y_{M_{l_j}})\right)
\]
for an admissible coordinate transformation $F^{(j)}_{l_j}:(U_{l_j}^{(j)};K_{l_j}^{(j)}) \to (U_{l_j};K_{l_j})$ for $(\F_{l_j},E_{l_j};K_{l_j})$, where $l_j\in\{0,\ldots,k\}$, and write
\[
\text{$(\F^{(j)},E^{(j)};K^{(j)})$ and $\{(\F^{(j)}_{i},E^{(j)}_{i};K^{(j)}_{i})\}_{i\in\{0,\ldots,l_j\}}$}
\]
for the pullbacks of $(\F,E;K)$ and $\{(\F_i,E_i;K_i)\}_{i\in\{0,\ldots,k\}}$ by $F^{(j)}:(U^{(j)};K^{(j)})\to(U;K)$, respectively.

\begin{lemma}\label{lemma:presCoord1}
Suppose that $\P$ is an $\S$-presentation.  Fix $l\in\{0,\ldots,k-1\}$, $i\in N_l$, and a rational number $\epsilon > 0$.  Then relative to the approximation and precision oracles for $\S$, we can effectively construct a finite family \eqref{eq:Fcovering} of admissible coordinate transformations for $\{(\F_i,E_i;K_i)\}_{i\in\{0,\ldots,k\}}$ such that \begin{equation}\label{eq:Kcover1}
\{x\in K : |x_i| \geq \epsilon\} \subseteq \bigcup_{j\in J} F^{(j)}\left(\Int\!\left(K^{(j)}\right)\right),
\end{equation}
and such that for each $j\in J$, we have $l_j\in\{0,\ldots,l\}$ and either
\begin{enumerate}
\item
$\ord(\F_{l_j}^{(j)}, E_{l_j}^{(j)}) = 0$, or

\item
$\left(\F^{(j)}_{l_j},E^{(j)}_{l_j};K^{(j)}_{l_j}\right)$ is $(m_{l_j}^{(j)},N_{l_j}^{(j)},r_{l_j}^{(j)})$-refinable for some $(m_{l_j}^{(j)},N_{l_j}^{(j)},r_{l_j}^{(j)})$ with $(m_{l_j}^{(j)},r_{l_j}^{(j)})$ lexicographically less than $(m_{l_j},r_{l_j})$.
\end{enumerate}
\end{lemma}

\begin{proof}
The proof is by induction on $l$.
By applying Lemma \ref{lemma:transRefinable2}, and then Lemma \ref{lemma:transRefinable1}, we may effectively construct a finite family of admissible coordinate transformations $\{F_{l}^{(j)}:(U_{l}^{(j)};K_{l}^{(j)}) \to (U_l;K_l)\}_{j\in J}$ for $(U_l,E_l;K_l)$ such that
\[
\{x\in K_l : |x_i| \geq \epsilon\} \subseteq \bigcup_{j\in J} F_{l}^{(j)}\left(\Int\!\left(K_{l}^{(j)}\right)\right),
\]
and such that for each $j\in J$, either $\ord(\F_{l}^{(j)}, E_{l}^{(j)}) = 0$ or $(\F_{l}^{(j)}, E_{l}^{(j)};K_{l}^{(j)})$ is $(m_{l}^{(j)},N_{l}^{(j)},r_{l}^{(j)})$-refinable for some $(m_{l}^{(j)},N_{l}^{(j)},r_{l}^{(j)})$ with $(m_{l}^{(j)},r_{l}^{(j)})$ is lexicographically less than $(m_l,r_l)$.  We are done if $l = 0$, so assume that $l > 0$ and that the lemma holds with $l'$ in place of $l$, for each $l' < l$.  For each $j\in J$, since $F^{(j)}_{l}(\cl(U^{(j)}_{l})) = \cl(F^{(j)}_{l}(U^{(j)}_{l}))$, Proposition \ref{prop:compContCompact} implies that the set $\cl(F_{l}^{(j)}(U_{l}^{(j)}))\times\{0\}$, with $0\in\RR^{N_0\cup\cdots\cup N_{l-1}}$, is a co-c.e.\ compact subset of the c.e.\ open set $U$.  We may therefore effectively find a bounded open rational box $A^{(j)} \subseteq\RR^{N_0\cup\cdots\cup N_{l-1}}$ and a compact rational box $B^{(j)}\subseteq A^{(j)}$ such that $0\in\Int(B^{(j)})$ and $\cl(F_{l}^{(j)}(U_{l}^{(j)}))\times \cl(A^{(j)}) \subseteq U$.  For each $j\in J$, let $U^{(j)} = U_{k}^{(j)}\times A^{(j)}$ and $K^{(j)} = K_{k}^{(j)}\times B^{(j)}$, define $F^{(j)}:U^{(j)}\to U$ by
\[
F^{(j)}(y) = \left(y_{N_0},\ldots,y_{N_{l-1}}, F_{l}^{(j)}(y_{M_l})\right),
\]
and write
\[
B^{(j)} = \prod_{r\in\bigcup_{s=0}^{l-1}N_s} [a_{r}^{(j)}, b_{r}^{(j)}].
\]
Since $0\in\Int(B^{(j)})$, we have $a_{r}^{(j)} < 0 < b_{r}^{(j)}$ for all $j,r$.  For each $r\in\bigcup_{s=0}^{l-1}N_s$, let
\[
\epsilon_r = \min\{-a^{(j)}_{r}, b^{(j)}_{r} : j\in J\}.
\]
It follows from Lemma \ref{lemma:blowupSet}.2 that $K\subseteq \RR^{N_0\cup\cdots\cup N_{l-1}} \times K_l$, so
\begin{equation}\label{eq:Kl}
\left\{x\in K : (|x_i|\geq \epsilon)\wedge\left(\bigwedge_{r\in\bigcup_{s=0}^{l-1}N_s} |x_r| \leq \epsilon_r\right)\right\}
\subseteq
\bigcup_{j\in J} F^{(j)}(K^{(j)}).
\end{equation}
Applying the induction hypothesis to the sets
\begin{equation}\label{eq:K<l}
\{x\in K : |x_r|\geq \epsilon_r\},
\quad \text{for each $r\in\bigcup_{s=0}^{l-1}N_s$,}
\end{equation}
completes the proof of the lemma, since the left side of \eqref{eq:Kl} and the sets \eqref{eq:K<l} cover the left side of \eqref{eq:Kcover1}.
\end{proof}

\begin{lemma}\label{lemma:presCoord2}
Suppose that $\P$ is an $\S$-presentation, that $l\in\{1,\ldots,k\}$, and that we know that $\ord(\F_l;x) < p_l$ for all $x\in K_l$.  Then relative to the approximation and precision oracles for $\S$, we can effectively construct a finite family \eqref{eq:Fcovering} of admissible coordinate transformations for $\{(\F_i,E_i;K_i)\}_{i\in\{0,\ldots,k\}}$ such that
\begin{equation}\label{eq:Kcover2}
K \subseteq \bigcup_{j\in J} F^{(j)}\left(\Int\!\left(K^{(j)}\right)\right),
\end{equation}
and such that for each $j\in J$, we have $l_j\in\{0,\ldots,l-1\}$ and either
\begin{enumerate}
\item
$\ord(\F_{l_j}^{(j)}, E_{l_j}^{(j)}) = 0$ and $|\Div(\F_{l_j}^{(j)}, E_{l_j}^{(j)})| \geq p_{l_j}$, or

\item
$\left(\F^{(j)}_{l_j},E^{(j)}_{l_j};K^{(j)}_{l_j}\right)$ is $(m_{l_j}^{(j)},N_{l_j}^{(j)},r_{l_j}^{(j)})$-refinable for some $(m_{l_j}^{(j)},N_{l_j}^{(j)},r_{l_j}^{(j)})$ with $(m_{l_j}^{(j)},r_{l_j}^{(j)})$ lexicographically less than $(m_{l_j},r_{l_j})$.
\end{enumerate}
\end{lemma}

\begin{proof}
The proof is by induction on $l$.  Lemma \ref{lemma:refine} implies that for all $x\in K_{l-1}$ with $x_{N_{l-1}} = 0$, either $\ord(\F_{l-1};x) < p_{l-1}$ or $\ord(\F_{l-1},E_{l-1};x) < m_{l-1}$.  It follows from Remarks \ref{rmk:Forder} and Proposition \ref{prop:compCont} that
\begin{equation}\label{eq:pDrop}
\{x\in U_{l-1} : \ord(\F_{l-1};x) < p_{l-1}\}
\end{equation}
and
\begin{equation}\label{eq:mDrop}
\{x\in U_{l-1} : \ord(\F_{l-1},E_{l-1};x) < m_{l-1}\}
\end{equation}
are c.e. open subsets of $U_{l-1}$.  These sets cover the co-c.e.\ compact set
\begin{equation}\label{eq:Krefine}
\{x\in K_{l-1} : x_{N_{l-1}} = 0\},
\end{equation}
so by computably enumerating all compact rational boxes $B\subseteq\RR^{M_{l-1}}$, with $N_{l-1}\subseteq\Cen(\Int(B))$ and with $B$ contained in either \eqref{eq:pDrop} or \eqref{eq:mDrop}, and by simultaneously trying to check if the interiors of the currently enumerated boxes cover \eqref{eq:Krefine}, we will find a finite family $\{B_j\}_{j\in J'}$ of compact rational boxes in $\RR^{M_{l-1}}$ whose interiors cover \eqref{eq:Krefine} and such that for each $j\in J'$, we have $N_{l-1}\subseteq\Cen(\Int(B_j))$ and $B_j$ is a subset of either \eqref{eq:pDrop} or \eqref{eq:mDrop}.

For each $j\in J'$ write
\[
B_j = \prod_{r\in M_{l-1}} [a_{j,r}, b_{j,r}].
\]
Thus $a_{j,r} < 0 < b_{j,r}$ for all $j\in J'$ and $r\in N_{l-1}$.  For each $r\in N_{l-1}$ let
\[
\epsilon_r = \min\{-a_{j,r}, b_{j,r} : j\in J'\}.
\]
Note that the sets
\begin{equation}\label{eq:KB}
\{x\in K : x_{M_{l-1}}\in B_j\}, \quad \text{for  $j\in J'$,}
\end{equation}
and the sets
\begin{equation}\label{eq:K>e}
\{x\in K : |x_r| \geq \epsilon_r\}, \quad \text{for  $r\in N_{l-1}$,}
\end{equation}
collectively cover $K$.

For each $r\in N_{l-1}$, apply Lemma \ref{lemma:presCoord1} to the set \eqref{eq:K>e}.  This constructs a finite family $\{F^{(j)} : (U^{(j)};K^{(j)})\to(U;K)\}_{j\in J}$ of admissible coordinate transformations for $\{(\F_i,E_i;K_i)\}_{i\in\{0,\ldots,k\}}$ such that
\[
\{x\in K : |x_r| \geq \epsilon_r\}
\subseteq
\bigcup_{j\in J} F^{(j)}(\Int(K^{(j)})),
\]
and such that for each $j\in J$, we have $l_j\in\{0,\ldots,l-1\}$, and either
$\ord(\F^{(j)}_{l_j},E^{(j)}_{l_j}) = 0$ or $(\F^{(j)}_{l_j},E^{(j)}_{l_j};K^{(j)}_{l_j})$ is $(m^{(j)}_{l_j}, N^{(j)}_{l_j}, r^{(j)}_{l_j})$-refinable for some $(m^{(j)}_{l_j}, N^{(j)}_{l_j}, r^{(j)}_{l_j})$ with $(m^{(j)}_{l_j},r^{(j)}_{l_j})$ lexicographically less than $(m_{l_j},r_{l_j})$.  We may assume that $|\Div(\F^{(j)}_{l_j},E^{(j)}_{l_j})| \geq p_{l_j}$ whenever $\ord(\F^{(j)}_{l_j},E^{(j)}_{l_j}) = 0$, because this is true when $l=1$, and if $l > 1$ we may apply the induction hypothesis to each $(\F^{(j)},E^{(j)}; K^{(j)})$ with $\ord(\F^{(j)}_{l_j},E^{(j)}_{l_j}) = 0$ and $|\Div(\F^{(j)}_{l_j},E^{(j)}_{l_j})| < p_{l_j}$.  We have thus covered each of the sets \eqref{eq:K>e} using admissible transformations with pullbacks of the desired form, so we may focus on covering the sets \eqref{eq:KB}.

Fix $j\in J'$.  To simplify notation, we may replace $K$ with $\{x\in K : x_{M_{l-1}} \in B_j\}$ and thereby assume that $K_{l-1}$ is a subset of either \eqref{eq:pDrop} or \eqref{eq:mDrop}.  If $l=1$, it is impossible for $K_{l-1}$ to be a subset of \eqref{eq:pDrop} since this set is empty.  If $l > 1$ and $K_{l-1}$ is a subset of \eqref{eq:pDrop}, then we are done by the induction hypothesis.  We may therefore assume that $K_{l-1}$ is a subset of \eqref{eq:mDrop}.  Applying Lemma \ref{lemma:transRefinable1} constructs a finite family
\begin{equation}\label{eq:transFamilyRefine2} \{F^{(j)}_{l-1}:(U^{(j)}_{l-1};K^{(j)}_{l-1})\to(U_{l-1};K_{l-1})\}_{j\in J}
\end{equation}
of admissible coordinate transformations for $(\F_{l-1},E_{l-1};K_{l-1})$ such that
\[
K_{l-1} \subseteq \bigcup_{j\in J} F^{(j)}_{l-1}(\Int(K^{(j)}_{l-1})),
\]
and such that for each $j\in J$, either $\ord(\F_{l-1}^{(j)},E_{l-1}^{(j)}) = 0$ or
$(\F_{l-1}^{(j)},E_{l-1}^{(j)};K_{l-1}^{(j)})$ is $(m_{l-1}^{(j)},N_{l-1}^{(j)},r_{l-1}^{(j)})$-refinable for some $(m_{l-1}^{(j)},N_{l-1}^{(j)},r_{l-1}^{(j)})$ with $m_{l-1}^{(j)} < m_{l-1}$.  We are done if $l=1$, so assume that $l>1$.  Then, as in the proof of Lemma \ref{lemma:presCoord1}, we may use \eqref{eq:transFamilyRefine2} to construct a corresponding family $\{F^{(j)}:(U^{(j)};K^{(j)}) \to (U;K)\}_{j\in J}$ of admissible transformations for $\P$ such that
\[
\left\{x\in K : \text{$|x_r|\leq \epsilon_r$ for all $r\in \bigcup_{s=0}^{l-2} N_s$}\right\}
\subseteq
\bigcup_{j\in J} F^{(j)}(\Int(K^{(j)}))
\]
for some rational numbers $\epsilon_r > 0$.  By applying the induction hypothesis to each $(\F^{(j)},E^{(j)};K^{(j)})$ with $\ord(\F_{l-1}^{(j)},E_{l-1}^{(j)}) = 0$ and $|\Div(\F_{l-1}^{(j)},E_{l-1}^{(j)})| < p_{l-1}$, we may assume that $|\Div(\F_{l-1}^{(j)},E_{l-1}^{(j)})| \geq p_{l-1}$ whenever $\ord(\F_{l-1}^{(j)},E_{l-1}^{(j)}) = 0$.  To finish, simply handle the sets
\[
\{x\in K : |x_r| \geq \epsilon_r\},
\quad\text{for each $r\in \bigcup_{s=0}^{l-2} N_s$,}
\]
in the same way as the sets \eqref{eq:K>e}.
\end{proof}

The following lemma builds upon Lemmas \ref{lemma:presCoord1} and \ref{lemma:presCoord2}, including them as special cases.  When applying the following lemma, the set $B$ will be taken to be either $B = \RR^n$ or $B = \{x\in\RR^n : |x_i|\geq \epsilon\}$ for some $i\in\{1,\ldots,n\}$ and rational $\epsilon > 0$.

\begin{lemma}\label{lemma:presCoord3}
Suppose that $\P$ is an $\S$-presentation with $m_k < \infty$, and let $L = K\cap B$, where $B\subseteq\RR^n$ is a finite union of closed rational boxes.  Then relative to the approximation and precision oracles for $\S$, we can effectively construct a finite family \eqref{eq:Fcovering} of admissible coordinate transformations for $\{(\F_i,E_i;K_i)\}_{i\in\{0,\ldots,k\}}$ such that
\begin{equation}\label{eq:Kcover3}
L \subseteq \bigcup_{j\in J} F^{(j)}\left(\Int\!\left(K^{(j)}\right)\right),
\end{equation}
and such that for each $j\in J$, the transformation $F^{(j)}:(U^{(j)};K^{(j)})\to(U;K)$ can be classified as being one of the following three types, where each type is defined according to properties of the pullback $(\F^{(j)},E^{(j)};K^{(j)})$:
\begin{description}
\item[Type 1]
We have $\ord(\F^{(j)}_{l_j},E^{(j)}_{l_j}) = 0$ and $|\Div(\F^{(j)}_{l_j},E^{(j)}_{l_j})| \geq p_{l_j}$.\\
(When applying this case, we shall choose some $N^{(j)}_{l_j} \subseteq E^{(j)}_{l_j}$ minimal such that $|\Div(\F^{(j)}_{l_j},E^{(j)}_{l_j})_{N^{(j)}_{l_j}}| \geq p_{l_j}$, and put $r^{(j)}_{l_j} = |N^{(j)}_{l_j}|$.)

\item[Type 2]
We have $l_j < k$, and $\left(\F^{(j)}_{l_j},E^{(j)}_{l_j};K^{(j)}_{l_j}\right)$ is $(m_{l_j}^{(j)},N_{l_j}^{(j)},r_{l_j}^{(j)})$-refinable for some $(m_{l_j}^{(j)},N_{l_j}^{(j)},r_{l_j}^{(j)})$ with $(m_{l_j}^{(j)},r_{l_j}^{(j)})$ lexicographically less than $(m_{l_j},r_{l_j})$.

\item[Type 3]
We have $l_j = k$, and $\left(\F^{(j)}_{k},E^{(j)}_{k};K^{(j)}_{k}\right)$ is $(m_{k}^{(j)},N_{k}^{(j)},r_{k}^{(j)})$-refinable for some $(m_{k}^{(j)},N_{k}^{(j)},r_{k}^{(j)})$ with $m_{k}^{(j)} \leq m_k$.
\end{description}
In addition, if we are in either of the following two special cases, then each transformation $F^{(j)}:(U^{(j)};K^{(j)})\to(U;K)$ is either of Type 1 with $l_j < k$, or of Type 2:
\begin{description}
\item[Special Case (a)]
We are given $l\in\{0,\ldots,k-1\}$, $i\in N_l$, and a rational number $\epsilon > 0$ such that $L \subseteq \{x\in K : |x_i|\geq \epsilon\}$.

\item[Special Case (b)]
We are given $l\in\{0,\ldots,k-1\}$, and we know that $\ord(\F_{l+1};x) < p_{l+1}$ for all $x\in K_{l+1}$.
\end{description}
\end{lemma}

\begin{proof}
First suppose we are in Special Case (a).  Apply Lemma \ref{lemma:presCoord1} to construct a finite family \eqref{eq:Fcovering} of admissible coordinate transformations such that
\[
\{x\in K : |x_i|\leq\epsilon\} \subseteq \bigcup_{j\in J} F^{(j)}(\Int(K^{(j)}))
\]
and such that for each $j\in\{0,\ldots,l\}$, either $\ord(\F^{(j)}_{l_j}, E^{(j)}_{l_j}) = 0$ or $\left(\F^{(j)}_{l_j},E^{(j)}_{l_j};K^{(j)}_{l_j}\right)$ is $(m_{l_j}^{(j)},N_{l_j}^{(j)},r_{l_j}^{(j)})$-refinable for some $(m_{l_j}^{(j)},N_{l_j}^{(j)},r_{l_j}^{(j)})$ with $(m_{l_j}^{(j)},r_{l_j}^{(j)})$ lexicographically less than $(m_{l_j},r_{l_j})$.  Now apply Lemma \ref{lemma:presCoord2} to each $\left(\F^{(j)}_{l_j},E^{(j)}_{l_j};K^{(j)}_{l_j}\right)$ with $\ord(\F^{(j)}_{l_j}, E^{(j)}_{l_j}) = 0$ and $|\Div(\F^{(j)}_{l_j}, E^{(j)}_{l_j})| < p_{l_j}$.  This proves Special Case (a).

Special case (b) is immediate from Lemma \ref{lemma:presCoord2}, so we now consider the general case (so we may as well assume that $L = K$).  Apply Lemma \ref{lemma:transRefinable1} to effectively construct a finite family of admissible coordinate transformations
\begin{equation}\label{eq:transFamilyRefine3}
\{F^{(j)}_{k} : (U^{(j)}_{k}; K^{(j)}_{k}) \to (U_k;K_k)\}_{j\in J}
\end{equation}
for $(\F_k,E_k;K_k)$ such that
\[
K_k\subseteq \bigcup_{j\in J} F^{(j)}_{k}(\Int(K^{(j)}_{k})),
\]
and such that for each $j\in J$, either $\ord(\F^{(j)}_{k}, E^{(j)}_{k}) = 0$ or $(\F^{(j)}_{k}, E^{(j)}_{k}; K^{(j)}_{k})$ is $(m^{(j)}_{k}, N^{(j)}_{k}, r^{(j)}_{k})$-refinable for some $(m^{(j)}_{k}, N^{(j)}_{k}, r^{(j)}_{k})$ with $m^{(j)}_{k} \leq m_k$.   As in the proof of Lemma \ref{lemma:presCoord1}, we may use \eqref{eq:transFamilyRefine3} to construct a corresponding family \begin{equation}\label{eq:transFamRef3}
\{F^{(j)}:(U^{(j)};K^{(j)}) \to (U;K)\}_{j\in J}
\end{equation}
of admissible transformations for $\P$ such that
\begin{equation}\label{eq:leqEpsilon}
\left\{x\in K : \text{$|x_r|\leq \epsilon_r$ for all $r\in \bigcup_{s=0}^{k-1} N_s$}\right\}
\subseteq
\bigcup_{j\in J} F^{(j)}(\Int(K^{(j)}))
\end{equation}
for some rational numbers $\epsilon_r > 0$.  Now finish by applying Special Case (b) to each $(F^{(j)},E^{(j)};K^{(j)})$ with $\ord(\F^{(j)}_{k}, E^{(j)}_{k}) = 0$ and $|\Div(\F^{(j)}_{k}, E^{(j)}_{k})| < p_k$, and by applying Special Case (a) to the sets $\{x\in K : |x_r| \geq \epsilon_r\}$, for each $r\in \bigcup_{s=0}^{k-1} N_s$.
\end{proof}

\begin{definition}\label{def:presAdmCoordTrans}
An {\bf admissible coordinate transformation for $\P$} is a Type 1, Type 2, or Type 3 admissible coordinate transformation for $\{(\F_i,E_i;K_i)\}_{i\in\{0,\ldots,k\}}$, as constructed in Lemma \ref{lemma:presCoord3}, or is a Type 0 admissible coordinate transformation for $\P$, which is defined as follows:
\begin{description}
\item[Type 0]
Suppose that $m_k = \infty$ and that $m'_k\in\NN$ is such that $\ord(\F_k,E_k;K_k) \leq m'_k$.  Relative to the approximation and precision oracles for $\S$ (assuming that $\P$ is an $\S$-presentation), we may use Lemma \ref{lemma:blowupSetVar} to effectively find an open blowup set $U'_k$ such that $K_k\subseteq U'_k \subseteq U_k$, $U'_k$ is defined by the same sequence of blowings-up as $K_k$, and $\ord(\F_k,E_k;U'_k) \leq m'_k$.  Put $U' = U\cap(\RR^{N_0\cup\cdots\cup N_{k-1}}\times U'_k)$.  We call the inclusion $\iota:(U';K)\to(U;K)$ a Type 0 admissible coordinate transformation for $\P$.  If $m'_k = 0$, define $r'_k = |d_k|$.  If $m'_k = 0$ and $|d_k|\geq p_k$, choose $N'_k\subseteq E_k$ minimal such that $|d_{k,N'_k}| \geq p_k$; otherwise, let $N'_k = \emptyset$.
\end{description}
The {\bf pullback} of $\P$ by an admissible coordinate transformation for $\P$ is the presentation $\P'$ defined as follows, according to the type of the transformation (we use the notation above for Type 0 transformations and the notation from Lemma \ref{lemma:presCoord3} for Type 1-3 transformations):
\[
\P' =
\begin{cases}
(\F\Restr{U'}, E; K: (m_j,N_j,r_j)_{j=0}^{k-1}, m'_k,N'_k,r'_k),
    & \text{Type 0 with $m'_ k =0$,}
\vspace*{5pt}
\\
(\F\Restr{U'}, E; K: (m_j,N_j,r_j)_{j=0}^{k-1}, m'_k),
    & \text{Type 0 with $m'_ k > 0$,}
\vspace*{5pt}
\\
(\F^{(j)}, E^{(j)}; K^{(j)}: (m_j,N_j,r_j)_{j=0}^{l_j-1}, 0,N^{(j)}_{l_j},r^{(j)}_{l_j}),
    & \text{Type 1,}
\vspace*{5pt}
\\
(\F^{(j)}, E^{(j)}; K^{(j)}:(m_j,N_j,r_j)_{j=0}^{l_j-1}, m^{(j)}_{l_j},N^{(j)}_{l_j},r^{(j)}_{l_j},\infty),
    & \text{Type 2 or 3.}
\end{cases}
\]
\end{definition}

The following is an important observation.

\begin{remark}\label{rmk:admissCoordTransRankDrop}
Suppose that $\P'$ is the pullback of a presentation $\P$ by any admissible coordinate transformation $F:(U';K')\to(U;K)$ for $\P$.  If $\P$ is incomplete, or if $F:(U';K')\to(U;K)$ is constructed as in Special Case (a) of Lemma \ref{lemma:presCoord3}, then
\[
\rank\P' < \rank\P.
\]
\end{remark}

\subsection{Admissible Blowup Transformations for a Complete Presentation}
\label{ss:presAdmBlowupTrans}

Now suppose that $\P$ is complete.

\begin{notation}\label{notation:blowupCompPres}
Define $C$ and $I$ as in \eqref{eq:C} and \eqref{eq:I}.  For each $i\in I$, let $\pi_i:U^{[i]}\to U$ be the $i$th standard chart of the blowing-up of $U$ with center $C$, and let $\ell(i)\in\{0,\ldots,k\}$ be such that $i\in N_{\ell(i)}$.  We use the notation $U^{[i]}$ for $\pi_{i}^{-1}(U)$, rather than the notation $U_i$ used prior in the paper, so as not to conflict with notation $U_i$ introduced at the beginning of this section (namely, $U_i = U|_{N_0\cup\cdots\cup N_{i-1}}$).  We use similar notation below.

Let $i\in I$.  For each $j\in\{0,\ldots,\ell(i)\}$, let $\pi_{i,j}:U^{[i]}_{j}\to U_j$ be the blowing-up of $U_j$ with center $C_j = \{x\in U_j : x_{I\cap M_j} = 0\}$.  Thus $\pi_{i,j}(y_{M_j}) = (y_{I^c}, y_i, y_i y_{I_i\cap M_j})$, with $I_i = I\setminus\{i\}$ as usual, and $U^{[i]}_{j} = U^{[i]}\Restr{N_0\cup\cdots\cup N_{j-1}}$. Let
\[
\F^{[i]} = \pi_{i}^{*}\F,
\quad
E^{[i]} = E\cup\{i\},
\quad\text{and}\quad
d^{[i]} = \Div(\F^{[i]},E^{[i]}).
\]
Also let
\[
\F^{[i]}_{0} = \F^{[i]},
\quad
E^{[i]}_{0} = E^{[i]},
\quad\text{and}\quad
d^{[i]}_{0} = d^{[i]},
\]
and for each $j\in\{1,\ldots,\ell(i)\}$ inductively define
\[
\F^{[i]}_{j} = \partial^{p_{j-1},m_{j-1}}_{N_{j-1},E^{[i]}_{j-1}}(\F^{[i]}_{j-1}),
\quad
E^{[i]}_{j} = E^{[i]}_{j-1}\setminus N_{j-1},
\quad\text{and}\quad
d^{[i]}_{j} = \Div(\F^{[i]}_{j},E^{[i]}_{j}).
\]
Thus $E^{[i]}_{j} = E_j\cup\{i\}$ for each $j\in\{0,\ldots,\ell(i)\}$.  Finally, for any $i\in I$ and any set $A\subseteq U^{[i]}$, write $A_{j} = A\Restr{N_0\cup\cdots\cup N_{j-1}}$ for each $j\in\{0,\ldots,\ell(i)\}$, as usual.
\end{notation}

We now make two important observations.  First, note that since $p_0 = 0$, repeated application of Lemma \ref{lemma:blowupRefineCommute} shows that for all $j\in\{0,\ldots,\ell(i)\}$,
\begin{equation}\label{eq:transFormula}
\F^{[i]}_{j} = y_{i}^{-p_j} \pi_{i,j}^{*}\F_j,
\end{equation}
and hence
\begin{equation}\label{eq:orderTransFormula}
\ord(\F^{[i]}_{j},E^{[i]}_{j};y) = \ord(\pi_{i,j}^{*}\F_j,E_j\cup\{i\};y)
\end{equation}
for all $y\in U^{[i]}_{j}$.  Second, note that if $m_k = 0$ and if $i\in I$ is chosen so that $\ell(i) = k$, then the minimality of $N_k$, equation \eqref{eq:transFormula} with $j = k$, and Lemma \ref{lemma:blowupCase0} together imply that
\begin{equation}\label{eq:divDrop}
|d^{[i]}_{k}| < |d_{k}|.
\end{equation}

The following lemma considers a family of sets $\{K^{[i]}\}_{i\in I}$ of the form
\begin{equation}\label{eq:Ki}
K^{[i]}
=
\left\{y\in \pi_{i}^{-1}(K) : \left(|y_i| \leq \delta_i\right)\wedge \left(\bigwedge_{j\in I} |y_j| \leq \epsilon^{\ell(j)-\ell(i)}\right)\right\},
\end{equation}
for some rational numbers $\epsilon > 0$ and $\delta_i > 0$.  Note that for each $i\in I$, $K^{[i]}$ is a compact blowup set. Lemmas \ref{lemma:blowupCover} and \ref{lemma:blowupCoverShrink} imply that
\begin{equation}\label{eq:KiCover}
\{x\in K : \text{$|x_i| \leq \delta_i$ for all $i\in I$}\} \subseteq \bigcup_{i\in I} \pi_i(K^{[i]}).
\end{equation}

\begin{lemma}\label{lemma:completePresBlowup}
Suppose that $\P$ is a complete $\S$-presentation; we use Notation \ref{notation:blowupCompPres}.  Then relative to the approximation and precision oracles for $\S$, we can effectively find rational numbers $\epsilon > 0$ and $\delta_i > 0$ such that if we define $\{K^{[i]}\}_{i\in I}$ as in \eqref{eq:Ki}, then for each $i\in I$ we can effectively find an open blowup set $\tld{U}^{[i]}$ such that $K^{[i]} \subseteq \tld{U}^{[i]} \subseteq U^{[i]}$, $\tld{U}^{[i]}$ is defined by the same sequence of blowings-up as $K^{[i]}$, and the following hold:
\begin{enumerate}{\setlength{\itemsep}{3pt}
\item
For each $j\in\{0,\ldots,\ell(i)-1\}$, $(\F^{[i]}_{j}\Restr{\tld{U}^{[i]}_{j}}, E^{[i]}_{j}; K^{[i]}_{j})$ is $(m_j,N_j,r_j)$-refinable.

\item
If $\ell(i)\neq k$, then
\[
\ord(\F^{[i]}_{\ell(i)},E^{[i]}_{\ell(i)};y) < m_{\ell(i)}
\]
for all $y\in \tld{U}^{[i]}_{\ell(i)}$ such that $|y_j| \leq 2\epsilon$ for all $j\in N_{\ell(i)}\setminus\{i\}$.
}\end{enumerate}
\end{lemma}

\begin{proof}
For each $i\in I$ with $\ell(i)\neq k$, apply Lemma \ref{lemma:refinableBlowup}.2 to $(\F_{\ell(i)},E_{\ell(i)};K_{\ell(i)})$, and use \eqref{eq:orderTransFormula}, to effectively find a rational number $\epsilon_i\in(0,1]$ such that
\[
\ord(\F^{[i]}_{\ell(i)}, E^{[i]}_{\ell(i)}; y) < m_{\ell(i)}
\]
for all $y\in\pi_{i,\ell(i)}^{-1}(K_{\ell(i)})$ such that $|y_j|\leq \epsilon_i$ for all $j\in N_{\ell(i)}\cup\cdots\cup N_k$.  Put $\epsilon = \frac{1}{2}\min\{\epsilon_i : i\in I\}$.  For each $i\in I$, define
\[
K^{[i,\epsilon]} = \left\{y\in\pi_{i}^{-1}(K) : \bigwedge_{j\in I} |y_j| \leq \epsilon^{\ell(j)-\ell(i)}\right\}.
\]
Then for each $i\in I$ with $\ell(i)\neq k$,
\[
\ord(\F^{[i]}_{\ell(i)},E^{[i]}_{\ell(i)};y) < m_{\ell(i)}
\]
for all $y\in K^{[i,\epsilon]}_{\ell(i)}$ such that $|y_j| \leq 2\epsilon$ for all $j\in N_{\ell(i)}$.  Next, for each $i\in I$ and $j\in\{0,\ldots,\ell(i) - 1\}$, apply Lemma \ref{lemma:refinableBlowup}.1 to the co-c.e.\ compact set $K^{[i,\epsilon]}$ to effectively find a rational number $\delta_{i,j} \in (0,2\epsilon]$ such that $(\F^{[i]}_{j},E^{[i]}_{j})$ is $(m_j,N_j,r_j)$-refinable on
\[
K^{[i,\epsilon,\delta_{i,j}]}
=
\left\{y\in \pi_{i}^{-1}(K) : \left(|y_i| \leq \delta_{i,j}\right)\wedge \left(\bigwedge_{j\in I} |y_j| \leq \epsilon^{\ell(j)-\ell(i)}\right)\right\}.
\]
For each $i\in I$, put $\delta_i = \min\{\delta_{i,j} : 0\leq j < \ell(i)\}$, and define $K^{[i]}$ as in \eqref{eq:Ki}.  Thus we have the following for each $i\in I$:
\begin{enumerate}{\setlength{\itemsep}{3pt}
\item
For each $j\in\{0,\ldots,\ell(i)-1\}$, $(\F^{[i]}_{j}, E^{[i]}_{j})$ is $(m_j,N_j,r_j)$-refinable on $K^{[i]}_{j}$.

\item
If $\ell(i)\neq k$, then
\[
\ord(\F^{[i]}_{\ell(i)},E^{[i]}_{\ell(i)};y) < m_{\ell(i)}
\]
for all $y\in K^{[i]}_{\ell(i)}$ such that $|y_j| \leq 2\epsilon$ for all $j\in N_{\ell(i)}\setminus\{i\}$.
}\end{enumerate}
To finish, apply Lemma \ref{lemma:blowupSetVar} to effectively construct the open blowup sets $\tld{U}^{[i]}$ as in the conclusion of the lemma.
\end{proof}

Suppose that $\P$ is complete and that $\epsilon$, $\{\delta_i\}_{i\in I}$, $\{K^{[i]}\}_{i\in I}$, and $\{\tld{U}^{[i]}\}_{i\in I}$ are constructed as in Lemma \ref{lemma:completePresBlowup}.

\begin{definition}\label{def:presAdmissBlowupFamily}
We call
\begin{equation}\label{eq:admFamBlowup}
\{\pi_i:(\tld{U}^{[i]};K^{[i]})\to(U;K)\}_{i\in I}
\end{equation}
an {\bf admissible family of blowings-up for $\P$}, and for each $i\in I$ we define $(\F^{[i]}\Restr{\tld{U}^{[i]}}, E^{[i]};K^{[i]})$ to be the {\bf pullback} of $(\F,E;K)$ by $\pi_i:(\tld{U}^{[i]};K^{[i]})\to(U;K)$.  (Note: This is slightly different from the definition of a pullback of a basic presentation by a coordinate transformation, given in Definition \ref{def:coordTrans}, since $E^{[i]}$ equals $E\cup\{i\}$, not $E$.)
\end{definition}

We are not interested in the members of \eqref{eq:admFamBlowup} directly, but are instead interested in compositions of members of \eqref{eq:admFamBlowup} with certain coordinate transformations for $(\F^{[i]}\Restr{\tld{U}^{[i]}}, E^{[i]};K^{[i]})$, which we know construct.  For each $i\in I$ such that $\ell(i)\neq k$, define
\begin{eqnarray*}
K^{[i]}(\epsilon)
    & = &
    \{y\in K^{[i]} : \text{$|y_j|\leq \epsilon$ for all $j\in N_{\ell(i)}\setminus\{i\}$}\},
\\
\tld{U}^{[i]}(\epsilon)
    & = &
    \{y\in \tld{U}^{[i]} : \text{$|y_j| < 2\epsilon$ for all $j\in N_{\ell(i)}\setminus\{i\}$}\},
\\
K^{[i]}(\epsilon,j,\sigma)
    & = &
    \{y\in K^{[i]} : \sigma y_j\geq \epsilon\},
    \qquad\text{for each $(j,\sigma)\in (N_{\ell(i)}\setminus\{i\})\times\{-1,1\}$,}
\\
\tld{U}^{[i]}(\epsilon,j,\sigma)
    & = &
    \{y\in \tld{U}^{[i]} : \sigma y_j >  \epsilon/2\},
     \qquad\text{for each $(j,\sigma)\in (N_{\ell(i)}\setminus\{i\})\times\{-1,1\}$,}
\end{eqnarray*}
and define
\[
\D_i
=
\left\{\left(\tld{U}^{[i]}(\epsilon);K^{[i]}(\epsilon)\right)\right\}
\cup
\left\{
\left(\tld{U}^{[i]}(\epsilon,j,\sigma); K^{[i]}(\epsilon,j,\sigma)\right) : (j,\sigma)\in (N_{\ell(i)}\setminus\{i\})\times\{-1,1\}
\right\}.
\]
For each $i\in I$ such that $\ell(i) = k$, define
\[
\D_i = \left\{\left(\tld{U}^{[i]}; K^{[i]}\right)\right\}.
\]
Note that for each $i\in I$,
\[
K^{[i]} = \bigcup_{(U';K')\in\D_i} K'
\quad\text{and}\quad
\tld{U}^{[i]} = \bigcup_{(U';K')\in\D_i} U'.
\]
For each $i\in I$ and $(U';K')\in\D_i$, write $\pi_i:(U';K')\to(U;K)$ for the composition of $\pi_i:(\tld{U}^{[i]},K^{[i]})\to(U;K)$ and the inclusion $\iota:(U';K')\to(\tld{U}^{[i]};K^{[i]})$.

\begin{defrmk}\label{defrmk:presAdmBlowupTrans}
Let $i\in I$ and $(U';K')\in\D_i$, and write $\F' = \F^{[i]}\Restr{U'}$ and $E' = E^{[i]}_{U'}$.  Define an admissible sequence of refinements $\{(\F'_j,E'_j;K'_j)\}_{j\in\{0,\ldots,\ell(i)\}}$ for $(\F',E';K')$ by
\begin{eqnarray*}
(\F'_0, E'_0; K'_0)
    & = &
    (\F',E';K'),
\\
(\F'_j, E'_j; K'_j)
    & = &
    \partial^{p_{j-1},m_{j-1}}_{N_{j-1},E'_{j-1}}(\F'_{j-1},E'_{j-1};K'_{j-1}),
\quad
\text{for each $j\in\{1,\ldots,\ell(i)\}$.}
\end{eqnarray*}
Note that Lemma \ref{lemma:completePresBlowup} shows that $(\F'_j,E'_j;K'_j)$ is indeed $(m_j,N_j,r_j)$-refinable for each $j\in\{0,\ldots,\ell(i)-1\}$.  Write
\[
d'_j = \Div(\F'_j,E'_j)
\]
for each $j\in\{0,\ldots,\ell(i)\}$.  We now make a number of observations and definitions, according to various cases:
\begin{enumerate}{\setlength{\itemsep}{5pt}
\item
Suppose that $\ell(i)\neq k$ and that $(U';K') = (U^{[i]}(\epsilon); K^{[i]}(\epsilon))$. Lemma \ref{lemma:completePresBlowup} shows that
\[
\ord(\F'_{\ell(i)},E'_{\ell(i)}) < m_{\ell(i)}.
\]
\begin{enumerate}{\setlength{\itemsep}{3pt}
\item
Suppose that $m_{\ell(i)} = 1$.  Then let $r'_{\ell(i)} = |d'_{\ell(i)}|$. If $|d'_{\ell(i)}| \geq p_{\ell(i)}$, choose $N'_{\ell(i)}\subseteq E'_{\ell(i)}$ minimal such that $|d'_{\ell(i),N'_{\ell(i)}}| \geq p_{\ell(i)}$.  Otherwise let $N'_{\ell(i)} = \emptyset$.

\item
Suppose that $m_{\ell(i)} > 1$.  Then let $N_{\ell(i)} = \emptyset$.
}\end{enumerate}

\item
Suppose that $\ell(i)\neq k$ and that $(U';K') = (U'_i(\epsilon,j,\sigma); K'_i(\epsilon,j,\sigma))$ for some $(j,\sigma)\in (N_{\ell(i)}\setminus\{i\}) \times \{-1,1\}$.  Lemma \ref{lemma:orderBlowup} shows that
\[
\ord(\F'_{\ell(i)},E'_{\ell(i)}) \leq m_{\ell(i)}.
\]
\begin{enumerate}{\setlength{\itemsep}{3pt}
\item
Suppose that $r_{\ell(i)} = 0$.  Then Remark \ref{rmk:refinable}.3 shows that we in fact have
\[
\ord(\F'_{\ell(i)},E'_{\ell(i)}) < m_{\ell(i)}.
\]
\begin{enumerate}{\setlength{\itemsep}{3pt}
\item
Suppose that $m_{\ell(i)} = 1$.  Then let $r'_{\ell(i)} = |d'_{\ell(i)}|$. If $|d'_{\ell(i)}| \geq p_{\ell(i)}$, choose $N'_{\ell(i)}\subseteq E'_{\ell(i)}$ minimal such that $|d'_{\ell(i),N'_{\ell(i)}}| \geq p_{\ell(i)}$.  Otherwise let $N'_{\ell(i)} = \emptyset$.

\item
Suppose that $m_{\ell(i)} > 1$.  Then let $N_{\ell(i)} = \emptyset$.
}\end{enumerate}

\item
Now suppose that $r_{\ell(i)} > 0$.  Apply Special Case (a) in Lemma \ref{lemma:presCoord3} to $(\F',E';K')$.  Relative to the approximation and precision oracles for $\S$, this effectively constructs a finite family $\{G^{(s)}:(V^{(s)};L^{(s)})\to(U';K')\}_{s\in S}$ of admissible coordinate transformations for $\{(\F'_j,E'_j;K'_j)\}_{j\in\{0,\ldots,\ell(i)\}}$ such that if we write
\[
(\G^{(s)},D^{(s)};L^{(s)})
\quad\text{and}\quad \{(\G^{(s)}_{j},D^{(s)}_{j};L^{(s)}_{j})\}_{j\in\{0,\ldots,l_s\}}
\]
for the pullbacks of $(\F',E';K')$ and $\{(\F'_j,E'_j;K'_j)\}_{j\in\{0,\ldots,\ell(i)\}}$ by $G^{(s)}:(V^{(s)};L^{(s)})\to(U';K')$, respectively, then
\[
K' \subseteq \bigcup_{s\in S} G^{(s)}(\Int(L^{(s)})),
\]
and for each $s\in S$, we have $l_s\in\{0,\ldots,\ell(i)-1\}$ and one of the following holds:
\begin{enumerate}{\setlength{\itemsep}{3pt}
\item
We have $\ord(\G^{(s)}_{l_s},D^{(s)}_{l_s}) = 0$ and $|\Div(\G^{(s)}_{l_s},D^{(s)}_{l_s})| \geq p_{l_s}$.  In this case, let $r'_{l_s} = |d'_{l_s}|$ and choose $N'_{l_s}\subseteq D_{l_s}$ minimal such that $|d'_{l_s,N'_{l_s}}| \geq p_{l_s}$.

\item
We have that $\left(\G^{(s)}_{l_s},D^{(s)}_{l_s};L^{(s)}_{l_s}\right)$ is $(m_{l_s}^{(s)},N_{l_s}^{(s)},r_{l_s}^{(s)})$-refinable for some\\ $(m_{l_s}^{(s)},N_{l_s}^{(s)},r_{l_s}^{(s)})$ with $(m_{l_s}^{(s)},r_{l_s}^{(s)})$ lexicographically less than $(m_{l_s},r_{l_s})$.
}\end{enumerate}
}\end{enumerate}

\item
Suppose that $\ell(i) = k$.  (Thus $m_k = 0$ and $i\in E_k$, so $E'_k = E^{[i]}_k = E_k$.)  Then \eqref{eq:divDrop} gives $|d'_k| < |d_k|$.  Define $r'_k = |d'_k|$. If $|d'_k| \geq p_k$, choose $N'_k\subseteq E_k$ minimal such that $|d'_{k,N'_k}| \geq p_k$.  If $|d'_k| < p_k$, let $N'_k = \emptyset$.
}\end{enumerate}
We call $F:(V;L)\to(U;K)$ an {\bf admissible blowup transformation for $\P$} if \begin{itemize}
\item
$F:(V,L)\to(U;K)$ equals $\pi_i:(U';K')\to(U;K)$ for some $i\in I$ and some $(U';K')\in\D_i$ in cases 1, 2(a), or 3,
\end{itemize}
or if
\begin{itemize}
\item
$F:(V,L)\to(U;K)$ equals $\pi_i\circ G^{(s)}:(V^{(s)};L^{(s)})\to(U;K)$ for some $i\in I$, some $(U';K')\in\D_i$ in case 2(b), and some $G^{(s)}:(V^{(s)};L^{(s)})\to(U';K')$ constructed in case 2(b).
\end{itemize}
By considering all the cases above, we categorize admissible blowup transformations for $\P$ into seven types: Type 1(a), Type 1(b), Type 2(a,i), Type 2(a,ii), Type 2(b,i), Type 2(b,ii), and Type 3.  The {\bf pullback} of $(\F,E;K)$ by $F:(V;L)\to(U;K)$ is the basic presentation $(\G,D;L)$ constructed by first taking the pullback $(\F^{[i]}\Restr{\tld{U}^{[i]}}, E^{[i]}; K^{[i]})$ of $(\F,E;K)$ by $\pi_i:(\tld{U}^{[i]};K^{[i]})\to(U;K)$ (for an appropriate $i\in I$), and then taking the pullback of $(\F^{[i]}\Restr{\tld{U}^{[i]}}, E^{[i]}; K^{[i]})$ by any additional coordinate transformations used to construct $F:(V;L)\to(U;K)$ from $\pi_i:(\tld{U}^{[i]};K^{[i]})\to(U;K)$.  The {\bf pullback} of $\P$ by $F:(V;L)\to(U;K)$ is the presentation $\P'$ define as follows, according to the type of $F:(V;L)\to(U;K)$:
\[
\P'
=
\begin{cases}
(\G,D;L : (m_j,N_j,r_j)_{j=0}^{\ell(i)-1}, 0, N'_{\ell(i)}, r'_{\ell(i)}),
    & \text{for Types 1(a) or 2(a,i),}
\vspace*{5pt}\\
(\G,D;L : (m_j,N_j,r_j)_{j=0}^{\ell(i)-1}, m_{\ell(i)} - 1),
    & \text{for Types 1(b) or 2(a,ii),}
\vspace*{5pt}\\
(\G,D;L : (m_j,N_j,r_j)_{j=0}^{l_s-1}, 0, N'_{l_s}, r'_{l_s}),
    & \text{for Type 2(b,i),}
\vspace*{5pt}\\
(\G,D;L : (m_j,N_j,r_j)_{j=0}^{l_s-1}, m^{(s)}_{l_s}, N^{(s)}_{l_s}, r^{(s)}_{l_s}, \infty),
    & \text{for Type 2(b,ii),}
\vspace*{5pt}\\
(\G,D;L : (m_j,N_j,r_j)_{j=0}^{k-1}, 0, N'_k, r'_k),
    & \text{for Type 3.}
\end{cases}
\]
It is important to note that
\begin{equation}\label{eq:rankDropBlowup}
\rank\P' < \rank P.
\end{equation}
\end{defrmk}

\section{Effective Desingularization Theorems}\label{s:desingThms}

This section proves an effective local resolution of singularities theorem, an effective fiber cutting theorem, and an effective theorem of the complement, and then combines these results into an effective parameterization theorem for the $0$-definable sets of $\RR_{\S}$.  Many of the results can be considered to be effective desingularization theorems.

\begin{definition}\label{def:presResolved}
We say that a presentation $\P = (\F,E;K : m_0,\ldots)$ is {\bf resolved} if $m_0=0$, and that $\P$ is {\bf unresolved} if $m_0 > 0$.
\end{definition}

\begin{definition}\label{def:presAdmTrans}
An {\bf admissible transformation} for an unresolved presentation $\P$ is either an admissible coordinate transformation for $\P$ or an admissible blowup transformation for $\P$.  If we write $\P^{(0)} = \P$, and if for each $i\in\{1,\ldots,l\}$ we inductively define $\P^{(i)}$ to be the pullback of the presentation $\P^{(i-1)}$ by an admissible transformation $F^{(i)}:(U^{(i)};K^{(i)}) \to (U^{(i-1};K^{(i-1)})$ for $\P^{(i-1)}$, then we call $F^{(1)}\circ\cdots\circ F^{(l)}:(U^{(l)};K^{(l)}) \to (U;K)$ a {\bf composition of admissible transformations} for $\P$, and we define $\P^{(l)}$ to be the {\bf pullback} of $\P$ by this composition.
\end{definition}

\begin{remark}\label{rmk:rankDrop}
Suppose that $\P'$ is the pullback of a presentation $\P$ by an admissible transformation $F:(U';K')\to(U;K)$ for $\P$.  Remark \ref{rmk:admissCoordTransRankDrop} and \eqref{eq:rankDropBlowup} show that if $\P$ is incomplete, or if $\P$ is complete and $F:(U';K')\to(U;K)$ is either an admissible blowup transformation or is an admissible coordinate transformation constructed as in Special Case (a) of Lemma \ref{lemma:presCoord3}, then $\rank\P' < \rank\P$.
\end{remark}

\begin{proposition}\label{prop:presRankDrop}
Suppose that $\P$ is an unresolved $\S$-presentation.  Then relative to the approximation and precision oracles for $\S$, we can effectively construct a finite family
\begin{equation}\label{eq:presAdmTransFamily}
\{F_j:(U^{(j)},K^{(j)})\to(U;K)\}_{j\in J}
\end{equation}
of admissible transformations for $\P$ such that
\begin{equation}\label{eq:KcoverPresRankDrop}
K \subseteq \bigcup_{j\in J} F_j\left(K^{(j)}\right),
\end{equation}
and such that for each $j\in J$,
\begin{equation}\label{eq:presRankDrop}
\rank\P^{(j)} < \rank\P,
\end{equation}
where $\P^{(j)}$ is the pullback of $\P$ by $F^{(j)}:(U^{(j)};K^{(j)}) \to (U;K)$.
\end{proposition}

\begin{proof}
First suppose that $m_k = \infty$.  By applying Lemma \ref{lemma:SpresZeroDivOrder} to $(\F_k,E_k;K_k)$, we may effectively determine if $\F_k$ contains a nonzero function.  If every function in $\F_k$ is identically zero, then $\P$ is complete; we will show how to handle this case in a moment.  So suppose that $\F_k$ contains a nonzero function.  By applying Lemma \ref{lemma:SpresZeroDivOrder}, we can effectively find some $m'_k\in\NN$ such that $\ord(\F_k,E_k;K_k)\leq m'_k$.  Now apply a Type 0 admissible coordinate transformation for $\P$.

Now suppose that $m_k < \infty$ and that $\P$ is incomplete.  In this case, apply Lemma \ref{lemma:presCoord3}.

Finally, suppose that $\P$ is complete.  Apply the family of all admissible blowup transformations for $\P$, as given in Definition and Remarks \ref{defrmk:presAdmBlowupTrans}.  The images of these maps cover $\{x\in K :\text{$|x_i| \leq \delta_i$ for all $i\in I$}\}$ by \eqref{eq:KiCover}.  To finish, apply Special Case (a) in Lemma \ref{lemma:presCoord3} to the sets $\{x\in K : |x_i|\geq \delta_i\}$, for each $i\in I$.

This completes the construction of \eqref{eq:presAdmTransFamily} such that \eqref{eq:KcoverPresRankDrop} holds; \eqref{eq:presRankDrop} follows from Remark \ref{rmk:rankDrop}.
\end{proof}

\begin{theorem}[Effective Resolution of $\S$-presentations]
\label{thm:presDesing}
Suppose that $\P$ is an $\S$-presentation.  Then relative to the approximation and precision oracles for $\S$, we can effectively construct a finite family $\{F^{(j)}:(U^{(j)},K^{(j)})\to(U;K)\}_{j\in J}$ of compositions of admissible transformations for $\P$ such that
\begin{equation}\label{eq:Kcover}
K \subseteq \bigcup_{j\in J} F^{(j)}\left(K^{(j)}\right),
\end{equation}
and such that $\P^{(j)}$ is resolved for all $j\in J$, where $\P^{(j)}$ is the pullback of $\P$ by $F^{(j)}:(U^{(j)};K^{(j)})\to(U;K)$.
\end{theorem}

\begin{proof}
The proof is by induction on $\rank\P$.  There is nothing to prove if $\P$ is resolved, so assume that $\P$ is unresolved and that the theorem holds for all presentations of rank lower than $\rank\P$.  Let $\{F_j:(U^{(j)},K^{(j)})\to(U;K)\}_{j\in J}$ be the family of admissible transformations for $\P$ given by Proposition \ref{prop:presRankDrop}.  Then $\rank\P^{(j)} < \rank\P$ for each $j\in J$, where $\P^{(j)}$ is the pullback of $\P$ by $F_j:(U^{(j)},K^{(j)})\to(U;K)$.  Now apply the induction hypothesis to each presentation $\P^{(j)}$.
\end{proof}

Theorem \ref{thm:presDesing} is stated in terms of $\S$-presentations, but it applies as well to basic $\S$-presentations by noting that any basic presentation $(\F,E;K)$ always has a trivial presentation, namely $\P = (\F,E;K : \infty)$, which is the presentation for $(\F,E;K)$ of maximal rank.  Thus, relative to the approximation and precision oracles for $\S$, Theorem \ref{thm:presDesing} is an effective local resolution of singularities theorem for $\F$-varieties, where $(\F,E;K)$ is a $\S$-basic presentation.

\begin{definition}\label{def:Sparam}
For each $j$ in a (possibly empty) finite index set $J$, let
\[
\Phi^{(j)} = (\F^{(j)},E^{(j)};K^{(j)} \,|\, F^{(j)},\psi^{(j)},\sigma^{(j)},\tau^{(j)}),
\]
where $(\F^{(j)},E^{(j)};K^{(j)})$ is a basic $\S$-presentation with domain $U^{(j)}\subseteq\RR^{d(j)}$ for some $d(j)\in\NN$; $F^{(j)}:U^{(j)}\to\RR^n$ is a $\C$-analytic map which is computably $C^\infty$ relative to the approximation oracle for $\S$ and which has an $(\S,E^{(j)})$-lifting; $\psi^{(j)}:\RR^{I(j)}\times(\RR\setminus\{0\})^{I(j)^c} \to \RR^n$ is defined by
\begin{equation}\label{eq:psi}
\psi^{(j)}(x) = \left((x_i)_{i\in I(j)}, (x_{i}^{-1})_{i\in I(j)^c}\right)
\end{equation}
for some $I(j)\subseteq\{1,\ldots,n\}$, where $I(j)^c = \{1,\ldots,n\}\setminus I(j)$; and $\sigma^{(j)}:\I(j)\to\{-1,0,1\}$ and $\tau^{(j)}:E^{(j)}\to\{-1,1\}$ are sign maps, where $\I(j)$ is the index set of the family $\F^{(j)} = \{f_i\}_{i\in\I(j)}$.  Define the {\bf locus} of $\Phi^{(j)}$ by
\begin{equation}\label{eq:V}
\Loc(\Phi^{(j)})
= \left\{x\in U^{(j)} : \left(\bigwedge_{i\in\I^{(j)}} \sign((f_i)_{E^{(j)}}(x)) = \sigma^{(j)}(i)\right)
\wedge \left(\bigwedge_{i\in E^{(j)}} \sign(x_i) = \tau^{(j)}(i)\right)\right\}.
\end{equation}
We call $\{\Phi^{(j)}\}_{j\in J}$ an {\bf $\S$-parameterization} if $F^{(j)}(\Loc(\Phi^{(j)})) \subseteq \dom(\psi^{(j)})$ for all $j\in J$.

Assume that $\{\Phi^{(j)}\}_{j\in J}$ is an $\S$-parameterization.  A {\bf representation} for a $\{\Phi^{(j)}\}_{j\in J}$ consists of the following family of discrete data, for each $j\in J$:
\begin{itemize}
\item
a representation for the basic $\S$-presentation $(\F^{(j)},E^{(j)};K^{(j)})$;

\item
a computably $C^\infty$ approximation algorithm for $F^{(j)}$ relative to the approximation oracle for $\S$, and an $(\S,E^{(j)})$-lifting of $F^{(j)}$;

\item
the set $I(j)$ and the number $n$;

\item
the maps $\sigma^{(j)}$ and $\tau^{(j)}$.
\end{itemize}
For a (possibly empty) set $M\subseteq\{1,\ldots,n\}$, we say that $\{\Phi^{(j)}\}_{j\in J}$ is {\bf $M$-bounded} if $M\subseteq I(j)$ for all $j\in J$.  We say that $\{\Phi^{(j)}\}_{j\in J}$ is {\bf resolved} if for all $j\in J$ and all $f\in\F^{(j)}$, $\sign(f_{E^{(j)}}(x))$ is constant on $U^{(j)}$.  We say that $\{\Phi^{(j)}\}_{j\in J}$ is {\bf trivial} if $K^{(j)}$ and $U^{(j)}$ are rational boxes for all $j\in J$.  We say that $\{\Phi^{(j)}\}_{j\in J}$ is {\bf immersive} if it is trivial and resolved, and if for each $j\in J$ there exists a coordinate projection $\Pi^{(j)}:\RR^n\to\RR^{d(j)}$ such that $\Pi^{(j)}\circ F^{(j)}\Restr{\Loc(\Phi^{(j)})} : \Loc(\Phi^{(j)})\to\RR^{d(j)}$ is an immersion.  A {\bf representation} of an immersive $\S$-parameterization $\{\Phi^{(j)}\}_{j\in J}$ consists of a representation for $\{\Phi^{(j)}\}_{j\in J}$, as defined above, along with names for the projections $\Pi^{(j)}$ for each $j\in J$.  For a set $A\subseteq\RR^n$, we call $\{\Phi^{(j)}\}_{j\in J}$ an {\bf $\S$-parameterization of $A$} if
\[
A = \bigcup_{j\in J} \psi^{(j)}\circ F^{(j)}(K^{(j)}\cap \Loc(\Phi^{(j)})) = \bigcup_{j\in J} \psi^{(j)}\circ F^{(j)}(\Loc(\Phi^{(j)})).
\]
\end{definition}

We now define the notion of the dimension of a set and list some of its basic properties, as given in \cite[pg. 4379]{vdDS98} (Note: \cite{vdDS98} works in the analytic category, but the facts remain true in the $C^1$ category as well).  Whenever we call a subset of $\RR^n$ a ``$C^1$-manifold'', we mean an embedded $C^1$-submanifold of $\RR^n$.

\begin{definition}\label{def:dimension}
A set $A\subseteq\RR^n$ {\bf has dimension} if $A$ is a countable union of $C^1$-manifolds.  When $A$ has dimension, define
\[
\dim(A)
=
\begin{cases}
\max\{\dim(M) : \text{$M\subseteq A$ is $C^1$-manifold}\},
    & \text{if $A\neq\emptyset$,} \\
-\infty,
    & \text{if $A = \emptyset$.}
\end{cases}
\]
\end{definition}

This notion of dimension has the following useful properties:
\begin{enumerate}{\setlength{\itemsep}{3pt}
\item
If $A = \bigcup_{i\in\NN} A_i$ and each $A_i$ has dimension, then $A$ also has dimension and $\dim(A) = \max\{\dim(A_i) : i\in\NN\}$.

\item
If $f:M\to\RR^n$ is a $C^1$-map from a $C^1$-manifold $M\subseteq\RR^m$ into $\RR^n$ of constant rank $r$, then $f(M)$ has dimension, and $\dim(f(M)) = r$.

\item
If $A\subseteq\RR^n$ has dimension, $F:U\to\RR^m$ is a $C^1$-map defined on a neighborhood $U$ of $A$, and $F(A)$ has dimension, then $\dim F(A) \leq \dim A$.

\begin{proof}
This is stated in \cite[pg. 4379]{vdDS98} when $F$ is projection map.  To reduce to this case, note that $\dim A = \dim\Graph(F\Restr{A}) \geq \dim F(A)$, where the inequality follows from projecting $\Graph(F\Restr{A})$ onto $F(A)$.
\end{proof}

\noindent Note: The assumption in 3 that $F(A)$ has dimension is actually unnecessary; it follows from the other assumptions.  But we will not use this fact, nor we will prove it.
}\end{enumerate}

\begin{remarks}\label{rmk:Sparam}
Suppose that $\{\Phi^{(j)}\}_{j\in J}$ is an $\S$-parameterization of $A$; we use the notation of Definition \ref{def:Sparam}.
\begin{enumerate}{\setlength{\itemsep}{5pt}
\item
If we are given a representation of $\{\Phi^{(j)}\}_{j\in J}$, then using the $(\S,E^{(j)})$-liftings of $F^{(j)}$ and $\F^{(j)}$ that it gives, we can effectively (actually, quite trivially) write down an existential $\L_{\Delta(\S)}$-formula defining $A$.

\item
{\bf Notational Convention:}
For each $j\in J$, it is frequently notationally more convenient to view $\sigma^{(j)}$ as a function on $\F^{(j)}$ rather than on $\I(j)$, so we will henceforth write $\sigma^{(j)}:\F^{(j)}\to\{-1,0,1\}$ rather than $\sigma^{(j)}:\I(j)\to\{-1,0,1\}$.\vspace*{5pt}

This notational convention is not completely sound because it is possible for there to exist $j\in J$ and distinct $i_1,i_2\in\I(j)$ such that $f_{i_1} = f_{i_2}$ but $\sigma^{(j)}(i_1)\neq\sigma^{(j)}(i_2)$.  But in this case, $\Loc(\Phi^{(j)})$ is empty, so we may simply exclude $j$ from the index set $J$ and still retain the fact that $\Phi$ is an $\S$-parameterization of $A$.  Relative to the approximation and precision oracles for $\S$, this exclusion of $j$ can be done effectively because for each $i,i'\in\I(j)$, the function $f_{i} - f_{i'}$ has an $(\S,E^{(j)})$-lifting, so we can effectively determine if $f_i - f_{i'} = 0$ using Proposition \ref{prop:precOracle} (as in the proof of Lemma \ref{lemma:SpresZero}).  Because of this reduction, we may assume that each $\F^{(j)}$ is an injectively indexed family of functions, which justifies our notational convention.

\item
If $j\in J$ and $f\in\F^{(j)}$, we can effectively determine if $f=0$ using the
approximation and precision oracles for $\S$.  If $0\in\F^{(j)}$ and $\sigma^{(j)}(0) = 0$, then $0$ may be omitted from $\F^{(j)}$ without changing $\Loc(\Phi^{(j)})$.  If $0\in\F^{(j)}$ and $\sigma^{(j)}(0)\in\{-1,1\}$, then $\Loc(\Phi^{(j)}) = \emptyset$, so we may omit the index $j$ from the set $J$.  This justifies the following convention.\vspace*{5pt}

\noindent{\bf Convention:}  We shall assume that $0\not\in\F^{(j)}$ for each $j\in J$.

\item
Suppose that $\{\Phi^{(j)}\}_{j\in J}$ is resolved.  Then for each $j\in J$, either $\Loc(\Phi^{(j)})$ is empty or
\[
\Loc(\Phi^{(j)}) = \left\{x\in U^{(j)} : \bigwedge_{i\in E^{(j)}} \sign(x_i) = \tau^{(j)}(i)\right\}.
\]
Thus $\Loc(\Phi^{(j)})$ is open in $\RR^{d(j)}$.  Also, using the approximation and precision oracles for $\S$, we can effectively determine if $\Loc(\Phi^{(j)})$ is empty, since we can effectively determine the constant sign of $f_{E^{(j)}}$ on $U^{(j)}$ for each $j\in J$ and $f\in\F(j)$.  Because of this, we will use the following convention.\vspace*{5pt}

\noindent{\bf Convention:} When $\{\Phi^{(j)}\}_{j\in J}$ is a resolved $\S$-parameterization of a set $A$, we will henceforth always assume that $\Loc(\Phi^{(j)})$ is nonempty for each $j\in J$, since we may simply omit from the index set $J$ each $j$ for which $\Loc(\Phi^{(j)})$ is empty.

\item
If $\{\Phi^{(j)}\}_{j\in J}$ is resolved and $A$ is nonempty, we can effectively find an $(\S,\emptyset)$-lifting of $\{a\}$ for some $a\in A$.  In particular, we can find an existential $\L_{\Delta(\S)}$-formula defining some point $a\in A$.

\begin{proof}
Choose $j\in J$ and $p\in \QQ^{d(j)}\cap\Loc(\Phi^{(j)})$, and put $a = \psi^{(j)}\circ F^{(j)}(p)$.  Then $a$ has an $(\S,\emptyset)$-lifting.
\end{proof}

\item
Suppose that $\{\Phi^{(j)}\}_{j\in J}$ is immersive.  Then $\Pi^{(j)}\circ F^{(j)}(\Loc(\Phi^{(j)}))$ is open in $\RR^{d(j)}$.  Also, for each $j\in J$, since $K^{(j)}$, $\Loc(\Phi^{(j)})$ and $U^{(j)}$ are just rational boxes, we can effectively find a bounded open rational box $B^{(j)}$ such that $\Loc(\Phi^{(j)})\cap K^{(j)} \subseteq B^{(j)} \subseteq \Loc(\Phi^{(j)})$ and $\cl(B^{(j)}) \subseteq U^{(j)}$.  (To construct $B^{(j)}$, find a bounded open rational box $D$ such that $K^{(j)}\subseteq D$ and $\cl(D)\subseteq U^{(j)}$, and put $B^{(j)} = \Loc(\Phi^{(j)})\cap D$.)  Note that
\begin{equation}\label{eq:simpleSparam}
A = \bigcup_{j\in J} \psi^{(j)}\circ F^{(j)}(B^{(j)}),
\end{equation}
and that for each $j\in J$, $\Pi^{(j)}\circ F^{(j)}\Restr{B^{(j)}}:B^{(j)}\to\RR^{d(j)}$ is an immersion,
$\dim \bd(B^{(j)}) = d(j)-1$, and each set $C$ in the natural stratification of $\cl(B^{(j)})$ is contained in some $M\in\Strat(U^{(j)},E^{(j)})$.  By applying Lemma \ref{lemma:basicLifting}.4 to the inclusion map $C\hookrightarrow M$ and the basic $\S$-lifting of $F^{(j)}$ on $M$ specified by the $(\S,E^{(j)})$-lifting of $F^{(j)}$, we can construct an $(\S,\emptyset)$-lifting of $F^{(j)}\Restr{C}$.
}\end{enumerate}
\end{remarks}

\begin{lemma}\label{lemma:SparamQF}
Let
\[
A = \bigcup_{i\in I} \bigcap_{j=1}^{k_i}\{x\in\RR^n : \sign(f_{i,j}(x)) = \sigma_{i,j}\}
\]
for a finite index set $I$, signs $\sigma_{i,j}\in\{-1,0,1\}$, and functions $f_i = (f_{i,1},\ldots,f_{i,k_i}):\RR^n\to\RR^{k_i}$ given by
\[
f_i(x) = \begin{cases}
g_i(x), & \text{if $x\in D_i$}, \\
0,      & \text{if $x\in\RR^n\setminus D_i$,}
\end{cases}
\]
for some $\S$-polynomial functions $g_i:D_i\to\RR^{k_i}$ with natural domains $D_i\subseteq\RR^n$.  Suppose we are given $M\subseteq\{1,\ldots,n\}$ and $R\in\QQ_{+}^{M}$ such that $\Pi_M(A)\subseteq[-R,R]$.  Then we can effectively find a representation for an $M$-bounded $\S$-parameterization of $A$.
\end{lemma}

\begin{proof}
By extending the tuples of functions $f_i = (f_{i,1},\ldots,f_{i,k_i})$ by the one-variable affine functions used to define the rational box $D_i$, enlarging the index set $I$ so as to include more conjunctive cases, and using the fact that each function $f_{i,j}$ is identically zero on $\RR^n\setminus D_i$, we may rewrite $A$ in the form
\[
A = \bigcup_{i\in I} \bigcap_{j=1}^{k_i}\{x\in D_i : \sign(g_{i,j}(x)) = \sigma_{i,j}\}
\]
for $\S$-polynomial functions $g_i = (g_{i,1},\ldots,g_{i,k_i}):D_i\to\RR^{k_i}$ on their natural domains $D_i$.  It suffices to show that each set in this union has an $M$-bounded $\S$-parameterization, so we may assume that $|I| = 1$ and will write
\begin{equation}\label{eq:AsignCond}
A = \bigcap_{j=1}^{k}\{x\in D : \sign(g_j(x)) = \sigma_j\}.
\end{equation}
There is no harm in replacing $D$ with a rational box contained in the natural domain of $g$, as long as this does not affect the definition of $A$ given in \eqref{eq:AsignCond}.   So by using the given $R\in\QQ_{+}^{M}$ such that $\Pi_M(A)\subseteq[-R,R]$, we may write $D$ in the form $D = [-r,r]\times\RR^{N^c}$ for some set $N$ such that $M\subseteq N\subseteq\{1,\ldots,n\}$ and some $r = (r_i)_{i\in N}\in\QQ_{+}^{N}$.  We now construct an $N$-bounded $\S$-parameterization $\Phi = \{\Phi^{(\alpha,\beta)}\}_{(\alpha,\beta)\in J}$ of the set $A$ by pulling back the functions $g_1,\ldots,g_k$ by a family of maps whose $i$th components are of the form $y_i\mapsto -r_i+y_{i}^{2}$, $0\mapsto -r_i$, $y_i\mapsto y_i$, $0\mapsto r_i$, or $y_i\mapsto r_i - y_{i}^{2}$ if $i\in N$ (in order to cover $[-r_i,r_i]$), and are of the form $y_i\mapsto y_i$ or $y_i\mapsto y_{i}^{-1}$ if $i\in N^c$ (in order to cover $\RR$).  Here are the specifics.

For each $i\in N$, choose $p_i,q_i\in\QQ_+$ such that $\sqrt{\frac{r_i}{2}} < p_i < q_i < \sqrt{r_i}$.  For each $i\in\{1,\ldots,n\}$, $\alpha\in\{-1,0,1\}$, and $\beta\in\{0,1\}$ such that $\beta = 0$ if either $i\in N^c$ or $\alpha = 0$, define
\begin{eqnarray*}
U_{i}^{(\alpha,\beta)}
    & = &
    \begin{cases}
    (-r_i,r_i),
        & \text{if $i\in N$ and $\alpha = 0$,} \\
    (-q_i,q_i),
        & \text{if $i\in N$, $\alpha\in\{-1,1\}$, and $\beta = 0$,} \\
    \RR^0,
        & \text{if $i\in N$, $\alpha\in\{-1,1\}$, and $\beta = 1$} \\
    (-2,2),
        & \text{if $i\in N^c$},
    \end{cases}
\\
K_{i}^{(\alpha,\beta)}
    & = &
    \begin{cases}
    [-\frac{r_i}{2},\frac{r_i}{2}],
        & \text{if $i\in N$ and $\alpha = 0$,} \\
    [-p_i,p_i],
        & \text{if $i\in N$, $\alpha\in\{-1,1\}$, and $\beta = 0$,} \\
    \RR^0,
        & \text{if $i\in N$, $\alpha\in\{-1,1\}$, and $\beta = 1$,} \\
    [-1,1],
        & \text{if $i\in N^c$},
    \end{cases}
\end{eqnarray*}
define $F_{i}^{(\alpha,\beta)}:U_{i}^{(\alpha,\beta)}\to\RR$ by
\[
F_{i}^{(\alpha,\beta)}(t)
=
\begin{cases}
\alpha(r_i - t^2),
    & \text{if $i\in N$, $\alpha\in\{-1,1\}$, and $\beta = 0$,} \\
\alpha r_i,
    & \text{if $i\in N$, $\alpha\in\{-1,1\}$, and $\beta = 1$,} \\
t,
    & \text{otherwise,}
\end{cases}
\]
and define
\[
\psi_{i}^{(\alpha,\beta)}(t)
=
\begin{cases}
t^{-1},
    & \text{if $i\in N^c$ and $\beta\in\{-1,1\}$,} \\
t,
    & \text{otherwise.}
\end{cases}
\]
Thus
\[
\psi_{i}^{(\alpha,\beta)}\circ F_{i}^{(\alpha,\beta)}(t)
=
\begin{cases}
t,
    & \text{if $\alpha=0$,}\\
\alpha(r_i - t^2),
    & \text{if $i\in N$, $\alpha\in\{-1,1\}$, and $\beta = 0$,} \\
\alpha r_i,
    & \text{if $i\in N$, $\alpha\in\{-1,1\}$, and $\beta = 1$,} \\
t^{-1},
    & \text{if $i\in N^c$ and $\alpha\in\{-1,1\}$.} \\
\end{cases}
\]

Now, define the index set
\[
J = \{(\alpha,\beta)\in\{-1,0,1\}^n\times\{0,1\}^n : 
\text{for all $i\in\{1,\ldots,n\}$, if $i\in N^c$ or $\alpha_i = 0$, then $\beta_i = 0$}\}.
\]
For each $(\alpha,\beta)\in J$, define
\begin{itemize}{\setlength{\itemsep}{3pt}
\item
$U^{(\alpha,\beta)} = \prod_{i=1}^{n} U^{(\alpha_i,\beta_i)}_{i}$ and $K^{(\alpha,\beta)} = \prod_{i=1}^{n} K^{(\alpha_i,\beta_i)}_{i}$;

\item
$F^{(\alpha,\beta)}:U^{(\alpha,\beta)}\to\RR^n$ by $F^{(\alpha,\beta)}(y) = (F^{(\alpha_1,\beta_1)}_{1}(y_1),\ldots,F^{(\alpha_n,\beta_n)}_{n}(y_n))$;

\item
$\psi^{(\alpha,\beta)}:\RR^{I(\alpha,\beta)} \times (\RR\setminus\{0\})^{I(\alpha,\beta)^c}$ by $\psi^{(\alpha,\beta)}(x) = (\psi^{(\alpha_1,\beta_1)}_{1}(x_1), \ldots, \psi^{(\alpha_n,\beta_n)}_{n}(x_n))$, \\
where $I(\alpha,\beta) = \{i\in\{1,\ldots,n\} : \text{$i\in N$, or $i\in N^c$ and $\alpha = 0$}\}$;

\item
$\F^{(\alpha,\beta)} = \{g_l\circ F^{(\alpha,\beta)}\}_{l\in\{1,\ldots,k\}}$;

\item
$E^{(\alpha,\beta)}
=
\{i\in N : \alpha_i\in\{-1,1\}, \beta_i = 0\}
\cup
\{i\in N^c : \alpha_i\in\{-1,1\}\}$;

\item
$\sigma^{(\alpha,\beta)}:\F^{(\alpha,\beta)}\to\{-1,0,1\}$ by $\sigma^{(\alpha,\beta)}(g_i\circ F^{(\alpha,\beta)}) = \sigma_i$ for each $i\in\{1,\ldots,k\}$;

\item
$\tau^{(\alpha,\beta)}:E^{(\alpha,\beta)}\to\{-1,1\}$ by $\tau^{(\alpha,\beta)}(i) = \alpha_i$ for each $i\in E^{(\alpha,\beta)}$;

\item
$\Phi^{(\alpha,\beta)}
=
(\F^{(\alpha,\beta)}, E^{(\alpha,\beta)}; K^{(\alpha,\beta)} \,|\, F^{(\alpha,\beta)}, \psi^{(\alpha,\beta)}, \sigma^{(\alpha,\beta)}, \tau^{(\alpha,\beta)})$.
}\end{itemize}
Then $\{\Phi^{(\alpha,\beta)}\}_{(\alpha,\beta)\in J}$ is an $N$-bounded $\S$-parameterization of $A$.

(Note:  Technically speaking, each set $U^{(\alpha,\beta)}$ is an open subset of $(\RR^0)^{D(\alpha,\beta)^c}\times \RR^{D(\alpha,\beta)} = \RR^{D(\alpha,\beta)}$, where $D(\alpha,\beta) = \{i\in \{1,\ldots,n\} : \beta_i = 0\}$, but this may be identified with a subset of $\RR^{|D(\alpha,\beta)|}$ by fixing a bijection from $\{1,\ldots,|D(\alpha,\beta)|\}$ to $D(\alpha,\beta)$.)
\end{proof}

\begin{lemma}\label{eq:SparamExist}
Given an existential $\L_{\Delta(\S)}$-formula defining a set $A\subseteq\RR^m$, and given $M\subseteq\{1,\ldots,m\}$ and $R\in\QQ_{+}^{M}$ such that $\Pi_M(A) \subseteq[-R,R]$, we can effectively find a representation for an $M$-bounded $\S$-parameterization of $A$.
\end{lemma}

\begin{proof}
Consider an existential $\L_{\Delta(\S)}$-formula $\exists y \varphi(x,y)$ defining a set $A\subseteq\RR^m$, where $x = (x_1,\ldots,x_m)$, $y = (y_1,\ldots,y_n)$, and $\varphi(x,y)$ is quantifier-free.  By using a number of rather standard reductions in the quantified variables $y$, we may assume that
\[
\varphi(x,y) = \bigvee_{s\in S} \bigwedge_{j=1}^{k_s} \sign(f_{s,j}(x,y)) = \sigma_{s,j},
\]
where each $f_s = (f_{s,1},\ldots,f_{s,k_s}):\RR^{m+n}\to\RR^{k_i}$ is defined from an $\S$-polynomial by extending it by $0$ off its natural domain, and that $A = \{x\in\RR^m : \exists y\in[-S,S]\,\text{s.t.}\, \varphi(x,y)\}$ for some $S\in\QQ_{+}^{n}$.   By Lemma \ref{lemma:SparamQF} we may effectively find a representation for an $M\cup\{m+1,\ldots,m+n\}$-bounded $\S$-parameterization
\[
\{(\F^{(j)},E^{(j)};K^{(j)} \,|,\ F^{(j)},\psi^{(j)}\times\id,\sigma^{(j)},\tau^{(j)})\}_{j\in J}
\]
of the set
\[
B = \{(x,y)\in\RR^m\times[-S,S] : \varphi(x,y)\},
\]
where $\psi^{(j)}:\RR^{I(j)}\times(\RR\setminus\{0\})^{I(j)^c}\to\RR^m$ for some $M\subseteq I(j)\subseteq\{1,\ldots,m\}$, and $\id:\RR^n\to\RR^n$ is the identity map.  Then
\[
\{(\F^{(j)},E^{(j)};K^{(j)} \,|,\ \Pi_m\circ F^{(j)},\psi^{(j)},\sigma^{(j)},\tau^{(j)})\}_{j\in J}
\]
is an $M$-bounded $\S$-parameterization of $A = \Pi_m(B)$.
\end{proof}

\begin{lemma}\label{lemma:SparamResolve}
Suppose we are given a representation for an $\S$-parameterization
\[
\{\Phi\} = \{(\F,E;K\,|\, F,\psi,\sigma,\tau)\}
\]
of size $1$, where $U\subseteq\RR^d$ is the domain of $\F$.  Then relative to the approximation and precision oracles for $\S$, we can effectively find a representation for a trivial, resolved $\S$-parameterization $\{\Phi^{(j)}\}_{j\in J}$ such that
\[
\Loc(\Phi)\cap K \subseteq \bigcup_{j\in J} F^{(j)}(\Loc(\Phi^{(j)})\cap K^{(j)})
\quad\text{and}\quad
\bigcup_{j\in J} F^{(j)}(\Loc(\Phi^{(j)})) \subseteq \Loc(\Phi),
\]
where for each $j\in J$,
\[
\Phi^{(j)} = \{(\F^{(j)},E^{(j)};K^{(j)}\,|\,
F^{(j)},\id,\sigma^{(j)},\tau^{(j)})\}_{j\in J},
\]
where the domain of $\F^{(j)}$ is $U^{(j)}\subseteq \RR^{d(j)}$, $\id:U\to U$ the identity map, and $F^{(j)}:U^{(j)}\to U$ is such that $\F^{(j)}\Restr{\Loc(\Phi^{(j)})} : \Loc(\Phi^{(j)})\to F^{(j)}(\Loc(\Phi^{(j)}))$ is a $\C$-analytic isomorphism.
\end{lemma}

We need some additional concepts to prove Lemma \ref{lemma:SparamResolve}.

\begin{definition}\label{def:extSparam}
We call $(\Phi,\P)$ a {\bf presented $\S$-parameterization} if $\{\Phi\}$ is an $\S$-parameterization of size $1$ and $\P$ is an $\S$-presentation which are related as follows:
\[
\Phi = (\F,E;K \,|\, F,\psi,\sigma,\tau),
\]
and
\[
\P = (\{g\}, E; K : m_0,\ldots),
\]
where
\[
g = \prod_{f\in\F} f.
\]
(Recall from Remark \ref{rmk:Sparam}.3 that we are assuming that $0\not\in\F$.)
\end{definition}

Note that if $(\Phi,\P)$ is a presented $\S$-parameterization, $\P$ is resolved if and only if $\{\Phi\}$ is resolved.

\begin{definition}\label{def:extSparamTrans}
Suppose that $(\Phi,\P)$ is a presented $\S$-parameterization.  We call $G:(U';K')\to(U;K)$ an {\bf admissible transformation for $(\Phi,\P)$} if it is an admissible transformation for $\P$, or if it is an {\bf admissible inclusion by a center}, which means that $\P$ is complete and that $(U';K) = (U|_I;K|_I)$ and $G(y_{I^c}) = (y_{I^c},0)$, with $0\in\RR^I$ and $I = \bigcup_{j=0}^{k}N_j$, as in the notation specified prior to Definition \ref{def:pres}.  We define the pullback $(\Phi',\P')$ of $(\Phi,\P)$ by $G:(U';K')\to(U;K)$ by defining
\[
\Phi' = (\F',E';K'\,|\, F\circ G,\psi,\sigma',\tau'),
\]
where $E'$, $\sigma'$, $\tau'$ and $\P'$ are defined as follows, according to the type of the transformation:
\begin{enumerate}{\setlength{\itemsep}{5pt}
\item
\emph{Admissible Coordinate Transformation}:

Define $(\F',E';K')$ and $\P'$ to be the pullbacks of $(\F,E;K)$ and $\P$ by $G:(U'K')\to(U;K)$, respectively.  Thus $\F' = G^*\F$ and $E'\subseteq E$.  Define $\sigma':\F'\to\{-1,0,1\}$ by $\sigma'(f\circ G) = \sigma(f)$ for all $f\in\F$, and define $\tau'=\tau\Restr{E'}$.

\item
\emph{Admissible Blowup Transformation}:

Define $(\F',E';K')$ and $\P'$ to be the pullbacks of $(\F,E;K)$ and $\P$ by $G:(U'K')\to(U;K)$, respectively.  Thus $\F' = G^*\F$ and $E'\subseteq (E\setminus\{i\})\cup\{i\}$ for some $i\in I$, where $C = \{x\in U : x_I = 0\}$ is the center of blowing-up and $G$ is constructed from $\pi_i:U_i\to U$, the $i$th standard chart of the blowing-up of $U$ with center $C$.  Define $\sigma':\F'\to\{-1,0,1\}$ by $\sigma'(f\circ G) = \sigma(f)$ for all $f\in\F$, and define $\tau:E'\to\{-1,1\}$ by
\[
\tau'(j) =
\begin{cases}
\tau(j),
    & \text{if $j\in E\setminus\{i\}$,} \\
\xi
    & \text{if $j = i$,}
\end{cases}
\]
for some choice of $\xi\in\{-1,1\}$.  Thus to any admissible blowup transformation for $\P$, we associate two admissible blowup transformations for $(\Phi,\P)$, one with $\xi = 1$ and the other with $\xi = -1$.

\item
\emph{Admissible Inclusion by a Center}:

Define $U' = U|_I$, $\F' = \F|_I$, $E' = E\setminus I$, $K' = K|_I$, and $\P' = (\F',E';K' : \infty)$.  Define $\sigma':\F'\to\{-1,0,1\}$ by $\sigma'(f\Restr{U'}) = \sigma(f)$ for all $f\in\F$, and define $\tau' = \tau\Restr{E'}$.
}\end{enumerate}
\end{definition}

Observe that if $(\Phi',\P')$ is the pullback of $(\Phi,P)$ by an admissible transformation $G:(U';K')\to(U;K)$ for $(\Phi,\P)$, then $G\Restr{\Loc(\Phi')} :\Loc(\Phi')\to\Loc(\Phi)$ is a $\C$-analytic embedding.
Also, if $G:(U';K')\to(U;K)$ is an admissible coordinate transformation for $\P$, then
\[
G(\Loc(\Phi')) = \Loc(\Phi)\cap G(U').
\]
If $G:(U';K')\to(U;K)$ is an admissible blowup transformation for $\P$ with center $C$, and $\{(\Phi'_\xi,\P')\}_{\xi\in\{-1,1\}}$ are the two pullbacks of $(\Phi,\P)$ by $G:(U';K')\to(U;K)$, then
\[
\bigcup_{\xi\in\{-1,1\}} G(\Loc(\Phi'_\xi)) = (\Loc(\Phi)\setminus C) \cap G(U').
\]
Finally, if $G:(U';K')\to(U;K)$ is an admissible inclusion by a center $C$ for $(\Phi,\P)$, then
\[
G(\Loc(\Phi')) = \Loc(\Phi)\cap C.
\]

\begin{proof}[Proof of Lemma \ref{lemma:SparamResolve}]
Write $\Phi = (\F,E;K\,|\, F, \psi, \sigma,\tau)$, where $U\subseteq\RR^n$ is the domain of $\F$.  We can always associate $\Phi$ with the presented $\S$-parameterization $(\Phi,\P)$ with $\P = (\prod_{f\in\F} f, E;K : \infty)$.  Therefore it suffices to assume the more general situation that $(\Phi,\P)$ is a presented $\S$-parameterization for a general $\S$-presentation $\P$, and to prove the lemma by induction of $(n,\rank\P)$, ordered lexicographically.
If $n = 0$, then $\{\Phi\}$ is trivial and resolved.  So assume that $n > 0$.

First suppose that $\P$ is resolved.  Thus $\{\Phi\}$ is resolved.  By using the fact that $K$ is a co-c.e.\ compact subset of the c.e.\ open set $U$, we can effectively find a finite family $\{(U^{(j)}; K^{(j)})\}_{j\in J}$, where each $U^{(j)}$ is an open rational box whose closure is contained in $U$, each $K^{(j)}$ is a nondegenerate compact rational box contained in $U^{(j)}$, and $K \subseteq \bigcup_{j\in J} \Int(K^{(j)})$.  Define an $\S$-parameterization by pulling back $(\Phi,\P)$ by the inclusions $\iota^{(j)} : (U^{(j)}; K^{(j)})\to (U;K)$, which are a type of admissible coordinate transformation.  The resulting $\S$-parameterization is still resolved and is also trivial.

Now suppose that $\P$ is unresolved.  Apply Proposition \ref{prop:presRankDrop} to construct a finite family
\begin{equation}\label{eq:Gfamily}
\{(G^{(j)}:(U^{(j)};K^{(j)})\to(U;K)\}_{j\in J}
\end{equation}
of admissible transformations for $\P$ such that $K\subseteq\bigcup_{j\in J}G^{(j)}(K^{(j)})$ and $\rank\P^{(j)} < \rank\P$ for all $j\in J$, where $\P^{(j)}$ is the pullback of $\P$ by $G^{(j)}:(U^{(j)};K^{(j)})\to(U;K)$.  Let $J'$ be the set of all $j\in J$ for which $G^{(j)}:(U^{(j)};K^{(j)})\to(U;K)$ is an admissible coordinate transformation; thus $J\setminus J'$ is the set of all $j\in J$ for which $G^{(j)}:(U^{(j)};K^{(j)})\to(U;K)$ is an admissible blowup transformation.
For each $j\in J'$, write $(\Phi^{(j)},\P^{(j)})$ for the pullback of $(\Phi,\P)$ by $G^{(j)}:(U^{(j)};K^{(j)})\to(U;K)$, and for each $j\in J\setminus J'$, write $\{(\Phi^{(j,\xi)},\P^{(j,\xi)})\}_{\xi\in\{-1,1\}}$ for the two pullbacks of $(\Phi,\P)$ by $G^{(j)}:(U^{(j)};K^{(j)})\to(U;K)$.  There are now two cases to consider.

First, suppose that $\P$ is incomplete.  Then $J' = J$.  Note that
\[
\Loc(\Phi)\cap K \subseteq \bigcup_{j\in J} G^{(j)}(\Loc(\Phi^{(j)})\cap K^{(j)})
\quad\text{and}\quad
\bigcup_{j\in J} G^{(j)}(\Loc(\Phi^{(j)}) \subseteq \Loc(\Phi),
\]
so we are done by applying the induction hypothesis to each $(\Phi^{(j)},\P^{(j)})$.

Second, suppose that $\P$ is complete, and write $C = \{x\in U : x_I = 0\}$ for the center of blowing-up associated to $\P$.  Note that
\begin{eqnarray*}
(\Loc(\Phi)\setminus C)\cap K
    & \subseteq &
    \left(\bigcup_{j\in J'} G^{(j)}(\Loc(\Phi^{(j)})\cap K^{(j)})\right)
    \\
    &&
    \cup \left(\bigcup_{(j,\xi)\in (J\setminus J')\times\{-1,1\}} G^{(j)}(\Loc(\Phi^{(j,\xi)})\cap K^{(j,\xi)})\right)
\end{eqnarray*}
and
\[
\left(\bigcup_{j\in J'} G^{(j)}(\Loc(\Phi^{(j)})\right)
\cup
\left(\bigcup_{(j,\xi)\in(J\setminus J)'\times\{-1,1\}} G^{(j)}(\Loc(\Phi^{(j,\xi)}))\right)
\subseteq
\Loc(\Phi).
\]
Let $\iota:(U|_I;K|_I)\to(U;K)$ denote the admissible inclusion by the center $C$ , and write $(\Phi|_I,\P|_I)$ for the pullback of $(\Phi,\P)$ by this inclusion.  Then
\[
(\Loc(\Phi)\cap C) \cap K = \iota(\Loc(\Phi|_I)\cap K|_I)
\quad\text{and}\quad
\iota(\Loc(\Phi|_I)) = \Loc(\Phi)\cap C.
\]
We are now done by applying the induction hypothesis to each member of $\{(\Phi|_I,\P|_I)\}\cup\{(\Phi^{(j)},\P^{(j)})\}_{j\in J}$.
\end{proof}

We use the following two observations about the statement of Lemma \ref{lemma:SparamResolve} in the corollary below: (1) each set $F^{(j)}(\Loc(\Phi^{(j)}))$ is a $\C$-analytic manifold; (2) if $\Loc(\Phi)\subseteq K$, then $\Loc(\Phi) = \bigcup_{j\in J} F^{(j)}(\Loc(\Phi^{(j)}))$.

\begin{corollary}\label{cor:dimension}
Let $X\subseteq\RR^n$, and suppose that for each $a\in X$ there exists a neighborhood $U_a$ of $a$ and a countable family of sets $\{Y_{a,i}\}_{i\in I_a}$ such that $X\cap U_a = \bigcup_{i\in I_a} Y_{a,i}$, where for each $i\in I_a$,
\[
Y_{a,i} = \{x\in U_a : f_{a,i}(x) = 0, g_{a,i}^{(1)}(x)>0,\ldots,g_{a,i}^{(k_{a,i})}(x) > 0\}
\]
for some $\C$-analytic functions $f_{a,i},g_{a,i}^{(1)},\ldots,g_{a,i}^{(k_{a,i})}:U_a\to\RR$.  Then $X$ has dimension.
\end{corollary}

\begin{proof}
For each $a\in X$, fix a compact rational box $K_a$ contained in $U_a$ such that $a\in\Int(K_a)$.  Only countably many of the sets $\Int(K_a)$ are needed to cover $X$, so it suffices to fix $a\in\RR^n$ and show that $X\cap\Int(K_a)$ has dimension.  For each $i\in I_a$, define an $\S$-parameterization $\{\Phi_{a,i}\}$ of $Y_{a,i}\cap \Int(K_{a,i})$ by setting $\Phi_{a,i} = (\F_{a,i}, \emptyset; K_a\,|\, \id, \id, \sigma_{a,i}, \emptyset)\}$, where $\F_{a,i}$ consists of $f_{a,i},g_{a,i}^{(1)},\ldots,g_{a,i}^{(k_{a,i})}$ along with all the single-variable affine functions used to define the open rational box $\Int(K_a)$ using positive sign conditions; thus the sign map $\sigma_{a,i}:\F_{a,i} \to\{-1,0,1\}$ is chosen so that $Y_{a,i}\cap\Int(K_{a,i}) = \Loc(\Phi_{a,i}) = \{x\in U_a : \bigwedge_{h\in\F_{a,i}} \sign(h(x)) = \sigma_{a,i}(h)\}$.   For each $i\in I_a$, applying Lemma \ref{lemma:SparamResolve} to $\{\Phi_{a,i}\}$ expresses $Y_{a,i}\cap\Int(K_{a,i})$ as a finite union of $\C$-analytic manifolds.  Therefore each set $Y_{a,i}\cap\Int(K_{a,i})$ has dimension, and hence so does $X\cap\Int(K_a) = \bigcup_{i\in I_a}\left(Y_{a,i}\cap\Int(K_{a,i})\right)$.
\end{proof}

In the following remarks we review some elementary differential geometry in order to prepare for the proof of Proposition \ref{prop:SparamFC}, which uses an effective fiber cutting procedure.

\begin{remarks}\label{rmk:diffGeom}
Let $F:U\to\RR^m$ be a $\C$-analytic map defined on an open set $U\subseteq\RR^d$.  We write $y = (y_1,\ldots,y_d)$ for coordinates on $\RR^d$.  The rank of $F$ at $y\in U$ is, by definition, the rank of the Jacobian matrix $\PD{}{F}{y}(y)$.  Put
\[
r = \max\left\{\rank\PD{}{F}{y}(y) : y\in U\right\}.
\]
Note the following:
\begin{enumerate}{\setlength{\itemsep}{3pt}
\item
The number $r$ is the maximum value of $s\in\{0,\ldots,\min\{d,m\}\}$ for which there exists an $s\times s$ submatrix of $\PD{}{F}{y}(y)$ with nonzero determinant at some $y\in U$.

\item
If $U$ is connected and $V\subseteq U$ has nonempty interior, then the determinant of a square submatrix of $\PD{}{F}{y}(y)$ vanishes identically on $U$ if and only if it vanishes identically on $V$, and hence $r = \max\left\{\rank\PD{}{F}{y}(y) : y\in V\right\}$.

\item
If $F$ is computably $C^\infty$ relative to the approximation oracle for $\S$ and also has an $(\S,E)$-lifting for some $E\subseteq\{1,\ldots,d\}$, then relative to the approximation and precision oracles for $\S$, we can compute $r$ by finding a point $p\in U\cap\QQ^d$ and an $r\times r$ submatrix of $\PD{}{F}{y}(p)$ with nonzero determinant, and if $r < \min\{d,m\}$, by verifying that the determinants of all $(r+1)\times(r+1)$ submatrices of $\PD{}{F}{y}(y)$ are identically equal to $0$.
}\end{enumerate}

For each pair $(\lambda,\mu)$ of increasing maps $\lambda:\{1,\ldots,d-r\}\to\{1,\ldots,d\}$ and $\mu:\{1,\ldots,r\}\to\{1,\ldots,m\}$, write $\lambda':\{1,\ldots,r\}\to\{1,\ldots,d\}$ and $\mu':\{1,\ldots,m-r\}\to\{1,\ldots,m\}$ for their complementary increasing maps, and
put
\begin{equation}\label{eq:U}
U_{\lambda,\mu,\xi} = \left\{y\in U : \sign\left(\det\PD{}{F_\mu}{y_{\lambda'}}(b)\right) = \xi\right\}
\end{equation}
for each $\xi\in\{-1,1\}$.  Thus each set $U_{\lambda,\mu,\xi}$ is open in $U$, and
\[
\left\{x\in U : \rank\PD{}{F}{y}(y) = r\right\} = \bigcup_{\lambda,\mu, \xi} U_{\lambda,\mu,\xi}.
\]

Fix $(\lambda,\mu,\xi)$.  For the rest of the discussion, we shall assume that $U = U_{\lambda,\mu,\xi}$.  Let $b\in U$, and put $a = F(b)$.  By the implicit function theorem for $\C$, there exists a $\C$-analytic function $g$ implicitly defined by
\[
F_\mu\left(y_\lambda, g(x_\mu, y_\lambda)\right) = x_\mu
\quad\text{and}\quad
g(a_\mu,b_\lambda) = b_{\lambda'},
\]
where $g$ is defined in an open box containing $(a_\mu,b_\lambda)$ and maps into $\RR^{\im(\lambda')}$.  Differentiating in $y_\lambda$ gives
\begin{equation}\label{eq:implicitDiff}
\PD{}{F_\mu}{y_\lambda}\left(y_\lambda, g(x_\mu, y_\lambda)\right)
+
\PD{}{F_\mu}{y_{\lambda'}}\left(y_\lambda, g(x_\mu, y_\lambda)\right)
\PD{}{g}{y_\lambda}(x_\mu, y_\lambda)
=
0,
\end{equation}
and hence
\begin{equation}\label{eq:Dg}
\PD{}{g}{y_\lambda}(a_\mu, b_\lambda)
=
- \left[\PD{}{F_\mu}{y_{\lambda'}}(b)\right]^{-1} \PD{}{F_\mu}{y_\lambda}(b)
=
- \frac{\adj\left(\PD{}{F_\mu}{y_{\lambda'}}(b)\right)}
{\det\left(\PD{}{F_\mu}{y_{\lambda'}}(b)\right)}
\, \PD{}{F_\mu}{y_\lambda}(b),
\end{equation}
where $\adj$ is the classical adjoint operator.  Put
\[
h(x_\mu,y_\lambda) = F_{\mu'}\left(y_\lambda, g(x_\mu, y_\lambda)\right).
\]
Note that
\begin{equation}\label{eq:Dh}
\PD{}{h}{y_\lambda}(x_\mu,y_\lambda)
=
\PD{}{F_{\mu'}}{y_\lambda}\left(y_\lambda, g(x_\mu, y_\lambda)\right)
+
\PD{}{F_{\mu'}}{y_{\lambda'}}\left(y_\lambda, g(x_\mu, y_\lambda)\right)
\PD{}{g}{y_\lambda}(x_\mu, y_\lambda).
\end{equation}
Now, for each fixed $y\in U$, we have $r = \rank\PD{}{F_\mu}{y}(y) = \rank\PD{}{F}{y}(y)$, so the rows of $\PD{}{F_{\mu'}}{y}(y)$ are contained in the span of the rows of $\PD{}{F_\mu}{y}(y)$.  Thus \eqref{eq:implicitDiff} and \eqref{eq:Dh} show that $\PD{}{h}{y_\lambda}(x_\mu,y_\lambda) = 0$.  Since $y\in U$ was arbitrary, this means that $h$ only depends on $x_\mu$, so we may write $h(x_\mu)$ instead of $h(x_\mu,y_\lambda)$.

We now consider $a\in F(U)$ to be fixed and study the fiber $F^{-1}(a) = \{y\in U : F(y) = a\}$.  The above discussion shows that $F^{-1}(a)$ is a ($d-r)$-dimensional $\C$-analytic manifold given locally about any $b\in F^{-1}(a)$ as the graph of a $\C$-analytic function $y_\lambda\mapsto g(a_\mu,y_\lambda)$.  (Note:  This is the rank theorem, which was just proved in the course of the discussion.)  Note that the map $\Pi_\mu\circ F:U\to\RR^d$ is an immersion if and only if $r=d$.

Now, consider a $\C$-analytic function $f:U\to\RR$, and suppose that $f\Restr{F^{-1}(a)}$ achieves a local extremum at a point $b\in F^{-1}(a)$.  Then
\[
\PD{}{}{y_\lambda} f(y_\lambda,g(a_\mu,y_\lambda))
=
\PD{}{f}{y_\lambda}(y_\lambda,g(a_\mu,y_\lambda))
+
\PD{}{f}{y_{\lambda'}}(y_\lambda,g(a_\mu,y_\lambda))
\PD{}{g}{y_\lambda}(a_\mu,y_\lambda)
\]
equals $0$ at $y_\lambda = b_\lambda$, for the locally and implicitly defined map $g$ as above.  This and \eqref{eq:Dg} give
\begin{equation}\label{eq:critPt}
\det\left(\PD{}{F_\mu}{y_{\lambda'}}(b)\right) \PD{}{f}{y_{\lambda'}}(b)
-
\PD{}{f}{y_{\lambda'}}(b) \,\adj\left(\PD{}{F_\mu}{y_{\lambda'}}(b)\right) \PD{}{F_\mu}{y_\lambda}(b)
=
0.
\end{equation}
The significance of \eqref{eq:critPt} is that it expresses the fact that $f\Restr{F^{-1}(a)}$ has a critical point at $b$ using only the partial derivatives of the functions $F$ and $f$, without reference to the locally defined map $g$.  Thus for any $b\in F^{-1}(a)$, $b$ is a critical point of $f\Restr{F^{-1}(a)}$ if and only if \eqref{eq:critPt} holds.
\end{remarks}

\begin{proposition}[Effective Desingularization of Existentially Definable Sets]
\label{prop:SparamFC}
Suppose we are given an existential $\L_{\Delta(\S)}$-formula defining a set $A\subseteq\RR^m$, and are given $M\subseteq\{1,\ldots,m\}$ and $R\in\QQ_{+}^{M}$ such that $\Pi_M(A) \subseteq[-R,R]$.  Then relative to the approximation and precision oracles for $\S$, we can effectively find a representation for an immersive, $M$-bounded $\S$-parameterization of $A$.
\end{proposition}

\begin{proof}
By applying Lemmas \ref{eq:SparamExist} and \ref{lemma:SparamResolve}, we obtain a trivial, resolved, $M$-bounded $\S$-parameterization $\{\Phi^{(j)}\}_{j\in J}$ of
$A$.  For each $j\in J$, write
\[
\Phi^{(j)} = (\F^{(j)}, E^{(j)}; K^{(j)} \,|\, F^{(j)}, \psi^{(j)}, \sigma^{(j)},\tau^{(j)}),
\]
where $\F^{(j)}$ has domain $U^{(j)}\subseteq\RR^{d(j)}$.  Thus \begin{equation}\label{eq:AparamKV}
A
= \bigcup_{j\in J}\psi^{(j)}\circ F^{(j)}(\Loc(\Phi^{(j)})\cap K^{(j)})
=\bigcup_{j\in J}\psi^{(j)}\circ F^{(j)}(\Loc(\Phi^{(j)})).
\end{equation}
Suppose that we have a ``distinguished set of indices'' $J'\subseteq J$ such that for each $j\in J'$ we have found a coordinate projection $\Pi^{(j)}:\RR^m\to\RR^{d(j)}$ such that $\Pi^{(j)}\Restr{\Loc(\Phi^{(j)})}:\Loc(\Phi^{(j)})\to\RR^{d(j)}$ is an immersion.    At the beginning of the proof we simply have $J' = \emptyset$, but we are considering a more general situation in order to set up an inductive argument.  Define
\[
d = \begin{cases}
\max\{d(j) : j\in J\setminus J'\},
    & \text{if $J'\neq J$,} \\
0,
    & \text{if $J' = J$.}
\end{cases}
\]
We now proceed by induction on $d$.  Note that if $j\in J$ is such that $d(j) = 0$, then $\Loc(\Phi^{(j)}) = \RR^0$ and $\Pi_0\circ F^{(j)}:\RR^0\to\RR^0$ is a immersion.  Therefore we are done if $d = 0$.  So suppose that $d > 0$, and inductively assume that we can find an immersive, $M$-bounded $\S$-parameterization of $A$ if we can find a trivial, resolved, $M$-bounded $\S$-parameterization $\{\Phi^{(t)}\}_{t\in T}$ of $A$ with a distinguished set of indices $T'\subseteq T$ such that $e < d$ for the number $e$ defined by
\[
e =
\begin{cases}
\max\{d(t) : t\in T\setminus T'\},
    & \text{if $T'\neq T$,} \\
0,
    & \text{if $T' = T$,}
\end{cases}
\]
where for each $t\in T$, $\Phi^{(t)} = (\F^{(t)}, E^{(t)}; K^{(t)} \,|\, F^{(t)}, \psi^{(t)}, \sigma^{(t)}, \tau^{(t)})$ and $\F^{(t)}$ has domain $U^{(t)}\subseteq \RR^{d(t)}$.

For each $j\in J$, compute
\[
r(j) = \max\left\{\rank\PD{}{F^{(j)}}{y}(y) : y\in U^{(j)}\right\}.
\]
(Equivalently, $r(j) = \max\{\rank\PD{}{F^{(j)}}{y}(y) : y\in \Loc(\Phi^{(j)})\}$, since $\Loc(\Phi^{(j)})$ is open in the connected set $U^{(j)}$.)  If $j\in J\setminus J'$ is such that $r(j) = 0$, then there exists $c^{(j)}\in\RR^m$ such that $\psi^{(j)}\circ F^{(j)}(y) = c^{(j)}$ for all $y\in U^{(j)}$, and by choosing any $p\in\QQ^{d(j)}\cap U^{(j)}$ and noting that $\psi^{(j)}\circ F^{(j)}(p) = c^{(j)}$, we see that $c^{(j)}$ has an $(\S,\emptyset)$-lifting.  Thus in this case, we may redefine $\Phi^{(j)}$ to be $\Phi^{(j)} = (\emptyset,\emptyset;\{0\}\,|\, c^{(j)}, \psi^{(j)}, \emptyset, \emptyset)$, and we may remove $j$ from $J\setminus J'$ and place $j$ in $J'$, since the composition of maps $\xymatrix{\RR^0 \ar[r] & \RR^m \ar[r]^-{\Pi_0} & \RR^0} : 0\mapsto c^{(j)} \mapsto 0$ is an immersion.  By this reduction, we may assume that $r(j) > 0$ for all $j\in J\setminus J'$.

Now, fix $j\in J\setminus J'$.  Also fix a bounded open rational box $B^{(j)} = \prod_{i=1}^{d(j)} (a^{(j)}_{i},b^{(j)}_{i})$ such that
\begin{equation}\label{eq:KBV}
K^{(j)}\cap \Loc(\Phi^{(j)}) \subseteq B^{(j)} \subseteq \Loc(\Phi^{(j)})
\end{equation}
and $\cl(B^{(j)}) \subseteq U^{(j)}$ (as constructed in Remark \ref{rmk:Sparam}.6).  Let $\Lambda(j)$ denote the set of all pairs $(\lambda,\mu)$ of increasing maps $\lambda:\{1,\ldots,d(j)-r(j)\}\to\{1,\ldots,d(j)\}$ and $\mu:\{1,\ldots,r(j)\}\to\{1,\ldots,m\}$.  For each $(\lambda,\mu)\in\Lambda(j)$ and $\xi\in\{-1,1\}$, define
\[
W^{(j)}_{\lambda,\mu,\xi}
=
\left\{y\in \Loc(\Phi^{(j)}) : \sign\left(\det \PD{}{F^{(j)}_{\mu}}{y_{\lambda'}}(y)\right) = \xi \right\},
\]
where $\lambda':\{1,\ldots,r(j)\}\to\{1,\ldots,d(j)\}$ is the increasing map complementary to $\lambda$.  Also define
\[
W^{(j)}_{*} = \bigcap_{(\lambda,\mu)\in\Lambda(j)}
\left\{y\in \Loc(\Phi^{(j)}) : \det \PD{}{F^{(j)}_{\mu}}{y_{\lambda'}}(y) = 0 \right\}.
\]
Note that each set $W^{(j)}_{\lambda,\mu,\xi}$ is open in $\Loc(\Phi^{(j)})$, that $\displaystyle \dim W^{(j)}_{*} < d(j)$, and that
\begin{equation}\label{eq:VW}
\Loc(\Phi^{(j)}) = W^{(j)}_{*} \cup \bigcup_{(\lambda,\mu,\xi)\in\Lambda(j)\times\{-1,1\}} W^{(j)}_{\lambda,\mu,\xi}.
\end{equation}

Suppose for the moment that $r(j) < d(j)$, and consider some $(\lambda,\mu,\xi)\in\Lambda(j)\times\{-1,1\}$.  Define $g^{(j)}_{\lambda,\mu,\xi} : U^{(j)}\to\RR$ by
\[
g^{(j)}_{\lambda,\mu,\xi}(y)
=
\xi\left(\det \PD{}{F^{(j)}_{\mu}}{y_{\lambda'}}(y)\right)
\left(\prod_{i=1}^{d(j)}(b^{(j)}_{i} - y_i)(y_i - a^{(j)}_{i})\right).
\]
Note that $g^{(j)}_{\lambda,\mu,\xi}(y) > 0$ on $W^{(j)}_{\lambda,\mu,\xi}\cap B^{(j)}$ and that $g^{(j)}_{\lambda,\mu,\xi}(y) = 0$ on $\bd(W^{(j)}_{\lambda,\mu,\xi}\cap B^{(j)})$.  For each $x\in\RR^m$, write
\[
N_x = \{y\in W^{(j)}_{\lambda,\mu,\xi} \cap B^{(j)} : F^{(j)}(y) = x\}.
\]
Since $F^{(j)}$ has constant rank $r(j)$ on $W^{(j)}_{\lambda,\mu,\xi}$, the rank theorem shows that for each $x\in F^{(j)}(W^{(j)}_{\lambda,\mu,\xi}\cap B^{(j)})$, the set $N_x$ is a nonempty submanifold of $U^{(j)}$ of pure dimension $d(j) - r(j)$; also, the closure of $N_x$ is a compact subset of $U^{(j)}$ because $\cl(N_x)\subseteq\cl(B^{(j)})\subseteq U^{(j)}$, and we have $g^{(j)}_{\lambda,\mu,\xi}(y) > 0$ on $N_x$ and $g^{(j)}_{\lambda,\mu,\xi}(y) = 0$ on $\fr(N_x)$.  Therefore for each $x\in F^{(j)}(W^{(j)}_{\lambda,\mu,\xi}\cap B^{(j)})$, the restriction of $g^{(j)}_{\lambda,\mu,\xi}$ to $N_x$ achieves a maximum value, which necessarily occurs at a critical point $y\in N_x$ satisfying $H^{(j)}_{\lambda,\mu,\xi}(y) = 0$, where
\[
H^{(j)}_{\lambda,\mu,\xi}(y) = \det\left(\PD{}{F^{(j)}_{\mu}}{y_{\lambda'}}(y)\right) \PD{}{g^{(j)}_{\lambda,\mu,\xi}}{y_{\lambda'}}(y)
-
\PD{}{g^{(j)}_{\lambda,\mu,\xi}}{y_{\lambda'}}(y) \,\adj\left(\PD{}{F^{(j)}_{\mu}}{y_{\lambda'}}(y)\right) \PD{}{F^{(j)}_{\mu}}{y_\lambda}(y) .
\]
Put
\[
\tld{W}^{(j)}_{\lambda,\mu,\xi} = \left\{y\in W^{(j)}_{\lambda,\mu,\xi} : H^{(j)}_{\lambda,\mu,\xi}(y) = 0\right\}.
\]
Note that $\dim\tld{W}^{(j)}_{\lambda,\mu,\xi} < d(j)$ since $H^{(j)}_{\lambda,\mu,\xi}$ is a nonzero $\C$-analytic function, and that \begin{equation}\label{eq:tldWimage}
F^{(j)}(\tld{W}^{(j)}_{\lambda,\mu,\xi}\cap B^{(j)}) = F^{(j)}(W^{(j)}_{\lambda,\mu,\xi}\cap B^{(j)}).
\end{equation}

We shall now define a new $\S$-parameterization of $A$, so $j$ is no longer fixed. Define a new index set
\[
S = J'
\cup
\{(j,*) : j\in J\setminus J'\}
\cup
\left\{(j,\lambda,\mu,\xi) : j\in J\setminus J', (\lambda,\mu)\in\Lambda(j), \xi\in\{-1,1\}\right\}.
\]
Define an $\S$-parameterization $\{\Phi^{(s)}\}_{s\in S}$, where for each $s\in S$, \[
\Phi^{(s)} = (\F^{(s)}, E^{(s)}; K^{(s)} \,|\, F^{(s)}, \psi^{(s)}, \sigma^{(s)}, \tau^{(s)}),
\]
where the domain of $\F^{(s)}$ is $U^{(s)}\subseteq \RR^{d(s)}$, and where
\begin{eqnarray*}
d(s)
    & = &
    d(j), \quad\text{if $s \in\{j, (j,*), (j,\lambda,\mu,\xi)\}$,}
\\
U^{(s)}
    & = &
    U^{(j)}, \quad\text{if $s \in\{j, (j,*), (j,\lambda,\mu,\xi)\}$,}
\\
\F^{(s)}
    & = &
    \begin{cases}
    \F^{(j)},
        & \text{if $s = j$,} \\
    \F^{(j)}
    \cup
    \left\{\PD{}{F^{(j)}_{\mu}}{y_{\lambda'}}\right\}_{(\lambda,\mu)\in\Lambda(j)},
        & \text{if $s = (j,*)$,} \vspace*{3pt}\\
    \F^{(j)} \cup \{\det\PD{}{F^{(j)}}{y_{\lambda'}}\},
        & \text{if $s = (j,\lambda,\mu,\xi)$ and $r(j) = d(j)$,} \vspace*{3pt}\\
    \F^{(j)} \cup \{\det\PD{}{F^{(j)}}{y_{\lambda'}}, H^{(j)}_{\lambda,\mu,\xi}\},
        & \text{if $s = (j,\lambda,\mu,\xi)$ and $r(j)  < d(j)$,}
    \end{cases}
\\
E^{(s)}
    & = &
    E^{(j)}, \quad\text{if $s \in\{j, (j,*), (j,\lambda,\mu,\xi)\}$,}
\\
K^{(s)}
    & = &
    \begin{cases}
    K^{(j)},
        & \text{if $s\in\{j,(j,*)\}$, or $s = (j,\lambda,\mu,\xi)$ and $r(j) = d(j)$,} \\
    \cl(B^{(j)}),
        & \text{if $s = (j,\lambda,\mu,\xi)$ and $r(j) < d(j)$,}
    \end{cases}
\\
\psi^{(s)}
    & = &
    \psi^{(j)}, \quad\text{if $s \in\{j, (j,*), (j,\lambda,\mu,\xi)\}$,}
\\
\tau^{(s)}
    & = &
    \tau^{(j)}, \quad\text{if $s \in\{j, (j,*), (j,\lambda,\mu,\xi)\}$,}
\end{eqnarray*}
and $\sigma^{(s)}:\F^{(s)}\to\{-1,0,1\}$ is defined by
\[
\sigma^{(s)}(f)
=
\begin{cases}
\sigma^{(j)}(f),
    & \text{if $s\in\{j, (j,*), (j,\lambda,\mu,\xi)\}$ and $f\in\F^{(j)}$,} \\
0,
    & \text{if $s = (j,*)$ and $f = \PD{}{F^{(j)}_{\mu}}{y_{\lambda'}}$ for some $(\lambda,\mu)$,} \\
\xi,
    & \text{if $s = (j,\lambda,\mu,\xi)$ and $f = \PD{}{F^{(j)}_{\mu}}{y_{\lambda'}}$,} \\
0,
    & \text{if $s = (j,\lambda,\mu,\xi)$, $r(j) < d(j)$, and $f = H^{(j)}_{\lambda,\mu,\xi}$.}
\end{cases}
\]
Note that for each $s\in S$,
\begin{equation}\label{eq:Vs}
\Loc(\Phi^{(s)})
=
\begin{cases}
\Loc(\Phi^{(j)}),
    & \text{if $s = j$,} \\
{\displaystyle W^{(j)}_{*}},
    & \text{if $s = (j,*)$,} \\
W^{(j)}_{\lambda,\mu,\xi},
    & \text{if $j = (\lambda,\mu,\xi)$ and $r(j) = d(j)$,} \\
\tld{W}^{(j)}_{\lambda,\mu,\xi},
    & \text{if $j = (\lambda,\mu,\xi)$ and $r(j) < d(j)$.}
\end{cases}
\end{equation}
It follows from equations \eqref{eq:AparamKV}-\eqref{eq:Vs} that
\[
A
= \bigcup_{s\in S}\psi^{(s)}\circ F^{(s)}(\Loc(\Phi^{(s)})\cap K^{(s)})
=\bigcup_{s\in S}\psi^{(s)}\circ F^{(s)}(\Loc(\Phi^{(s)})).
\]

For each $s\in S\setminus J'$, apply Lemma \ref{lemma:SparamResolve} to $\Phi^{(s)}$.  This constructs a trivial, resolved $\S$-parameterization $\{\Psi^{(t)}\}_{t\in T_s}$ such that
\[
\Loc(\Phi^{(s)})\cap K^{(s)} \subseteq \bigcup_{t\in T_s} G^{(t)}(\Loc(\Psi^{(t)})\cap K^{(t)})
\quad\text{and}\quad
\bigcup_{t\in T_s} G^{(t)}(\Loc(\Psi^{(t)})) \subseteq \Loc(\Phi^{(s)}),
\]
where for each $t\in T_s$,
\[
\Psi^{(t)} = (\F^{(t)}, E^{(t)}; K^{(t)} \,|\, G^{(t)}, \id, \sigma^{(t)}, \tau^{(t)}),
\]
the domain of $\F^{(t)}$ is $U^{(t)}\subseteq\RR^{d(t)}$, and the restriction of $G^{(t)}$ to $\Loc(\Phi^{(t)})$ is an isomorphism onto its image.  We assume that the index sets $T_s$ are all disjoint from $J'$ and from each other.  Define $T = J'\cup\bigcup_{s\in S\setminus J'} T_s$, and define an $\S$-parameterization $\{\Phi^{(t)}\}_{t\in T}$, where for each $s\in S\setminus J'$ and $t\in T_s$,
\[
\Phi^{(t)} = (\F^{(t)}, E^{(t)} ; K^{(t)} \,|\, F^{(t)}, \psi^{(t)}, \sigma^{(t)}, \tau^{(t)}),
\]
with $F^{(t)} = F^{(s)}\circ G^{(t)}$ and $\psi^{(t)} = \psi^{(s)}$.  Note that $\{\Phi^{(t)}\}_{t\in T}$ is a trivial, resolved, $M$-bounded $\S$-parameterization of $A$.

Define a distinguished set of indices $T'\subseteq T$ as follows.  Consider $t\in T$.  If $t = j$ for some $j\in J'$, then $\Pi^{(t)}\circ F^{(t)}\Restr{\Loc(\Phi^{(t)})}:\Loc(\Phi^{(t)})\to\RR^{d(t)}$ is an immersion; in this case, place $t$ in $T'$.  So suppose that $t\in T_s$ for some $s\in S\setminus J'$.  If $s = (j,*)$, or if $s = (j,\lambda,\mu,\xi)$ and $r(j) < d(j)$, then $d(t) \leq \dim \Loc(\Phi^{(s)}) < d(j)$ since $G^{(t)}\Restr{\Loc(\Phi^{(t)})}$ is an isomorphism onto its image $G^{(t)}(\Loc(\Phi^{(t)}) )\subseteq \Loc(\Phi^{(s)})$ and $\dim \Loc(\Phi^{(s)}) < d(j)$; in this case, place $t$ in $T\setminus T'$.  If $s = (j,\lambda,\mu,\xi)$ and $r(j) = d(j)$, then either $d(t) < d(j)$, in which case we place $t$ in $T\setminus T'$, or else $d(t) = d(j)$, in which case $\Pi_{\lambda}\circ F^{(t)}\Restr{\Loc(\Phi^{(t)})} : \Loc(\Phi^{(t)})\to\RR^{d(t)}$ is an immersion so we place $t$ in $T'$.  Define
\[
e =
\begin{cases}
\max\{d(t) : t\in T\setminus T'\},
    & \text{if $T'\neq T$,} \\
0,
    & \text{if $T' = T$.}
\end{cases}
\]
Then $e < d$, so we are done by the induction hypothesis.
\end{proof}

\begin{corollary}\label{cor:existDec}
If we are given an existential $\L_{\Delta(\S)}$-sentence $\varphi$, then relative to the approximation and precision oracles for $\S$, we can decide whether $\varphi$ is true or false in the structure $\RR_{\Delta(\S)}$.
\end{corollary}

\begin{proof}
Let $A$ be the subset of $\RR^0$ defined by $\varphi$; namely, $A = \{0\}$ if $\varphi$ is true, and $A = \emptyset$ if $\varphi$ is false.  Apply Proposition \ref{prop:SparamFC} to construct an immersive $\S$-parameterization $\{\Phi^{(j)}\}_{j\in J}$  of $A$.  Then $\varphi$ is true if and only if $J\neq\emptyset$.
\end{proof}

\begin{corollary}\label{cor:existSetMeet}
If we are given existential $\L_{\Delta(\S)}$-formulas defining sets $A_1,A_2\subseteq\RR^m$, then relative to the approximation and precision oracles for $\S$, we can decide whether $A_1\cap A_2$ is empty or nonempty.
\end{corollary}

\begin{proof}
For each $i\in\{1,2\}$, let $\varphi_i(x,y_i)$ be a quantifier-free $\L_{\Delta(\S)}$-formula such that $A_i = \{x\in\RR^m : \exists y_i\varphi(x,y_i)\}$, where $y_i$ is a tuple of variables.  Then $A_1\cap A_2\neq \emptyset$ if and only if $\exists x\exists y_1\exists y_2\left(\varphi_1(x,y_1)\wedge \varphi_2(x,y_2)\right)$ holds, and this sentence is existential.  Apply Corollary \ref{cor:existDec}.
\end{proof}

\begin{proposition}[Computing Connected Components and Dimension]
\label{prop:CCdim}
Any set $A\subseteq\RR^m$ which is existentially $0$-definable in $\RR_{\Delta(\S)}$ has dimension and has finitely many connected components $C_1,\ldots,C_P$.  In fact, if we are given an existential $\L_{\Delta(\S)}$-formula defining the set $A$, then relative to the approximation and precision oracles for $\S$, we can effectively find representations for immersive $\S$-parameterizations
\[
\{\Phi^{(q)}\}_{q\in Q_p},
\quad\text{for $p=1,\ldots,P$,}
\]
for $C_1,\ldots,C_P$, respectively, where the $Q_p$ are disjoint index sets.
\end{proposition}

The proposition implies the following:
\begin{enumerate}{\setlength{\itemsep}{3pt}
\item
We can compute $P$.

\item
For each $p\in\{1,\ldots,P\}$, we can find an immersive $\S$-parameterization for a point $c_p\in C_p$ (by Remark \ref{rmk:Sparam}.5).

\item
We can compute the dimensions of $C_1,\ldots,C_P$ and $A$ by the following formulas,
\begin{eqnarray*}
\dim C_p
    & = &
    \max\{d(q) : q\in Q_p\}, \quad\text{for each $p\in\{1,\ldots,P\}$,}\\
\dim A
    & = &
    \max\{\dim(C_1),\ldots,\dim(C_P)\}
    =
    \max\left\{d(q) : q\in \bigcup_{p=1}^{P} Q_p\right\},
\end{eqnarray*}
where for each $q$, $\Phi^{(q)} = (\F^{(q)}, E^{(q)}; K^{(q)} \,|\, F^{(q)}, \psi^{(q)}, \sigma^{(q)}, \tau^{(q)})$ and the domain of $\F^{(q)}$ is an open set $U^{(q)}\subseteq\RR^{d(q)}$.
}\end{enumerate}

\begin{proof}
Suppose we are given an existential $\L_{\Delta(\S)}$-formula defining $A$.  Apply Proposition \ref{prop:SparamFC} to construct an immersive $\S$-parameterization
$\{\Phi^{(j)}\}_{j\in J}$ of $A$; we use the notation of Definition \ref{def:Sparam}.  By Remark \ref{rmk:Sparam}.6, we may choose bounded, open, rational boxes $B^{(j)} \subseteq\RR^{d(j)}$ such that $A = \bigcup_{j\in J}\psi^{(j)}\circ F^{(j)}(B^{(j)})$, and such that for each $j\in J$, $\cl(B^{(j)})\subseteq U^{(j)}$ and $\Pi^{(j)}\circ F^{(j)}\Restr{B^{(j)}}:B^{(j)}\to\RR^{d(j)}$ is an immersion.
Each set $\psi^{(j)}\circ F^{(j)}(B^{(j)})$ is connected, so $A$ has finitely many connected components $C_1,\ldots,C_P$, and there exists a partition $\{Q_1,\ldots,Q_P\}$ of $J$ such that $C_p = \bigcup_{q\in Q_p}\psi^{(q)}\circ F^{(q)}(B^{(q)})$ for each $p\in\{1,\ldots,P\}$.  So it suffices to compute the partition $\{Q_1,\ldots,Q_P\}$ of $J$.

Let $j\in J$.  Write $W^{(j)} = \RR^{I(j)}\times(\RR\setminus\{0\})^{I(j)^c}$, which is both the domain and range of $\psi^{(j)}$.  Note that $F^{(j)}(B^{(j)})\subseteq W^{(j)}$, but $F^{(j)}(\cl B^{(j)})$ is not necessarily a subset of $W^{(j)}$.  We claim that
\begin{equation}\label{eq:SparamCL}
\psi^{(j)}(W^{(j)}\cap F^{(j)}(\cl B^{(j)}))
=
\cl\left(\psi^{(j)}\circ F^{(j)}(B^{(j)})\right).
\end{equation}
To see this, note that since $\cl(B^{(j)})\subseteq U^{(j)}$, and since the maps $F^{(j)}:U^{(j)}\to\RR^m$ and $\psi^{(j)}:W^{(j)}\to W^{(j)}$ are continuous, we have $\psi^{(j)}(W^{(j)}\cap F^{(j)}(\cl B^{(j)})) \subseteq \cl\left(\psi^{(j)}\circ F^{(j)}(B^{(j)})\right)$.  To show the reverse inclusion, let $x\in \cl\left(\psi^{(j)}\circ F^{(j)}( B^{(j)})\right)$.  Fix a sequence $\{y_k\}_{k\in\NN}$ in $B^{(j)}$ such that $x = \lim_{k\to\infty} \psi^{(j)}\circ F^{(j)}(y_k)$.  Since $B^{(j)}$ is bounded, by passing to a subsequence we may assume that $\lim_{k\to\infty} y_k = y$ for some $y\in\cl(B^{(j)})$.  Put $b = (b_1,\ldots,b_m)  = F^{(j)}(y)$.  Now, if $b\not\in W^{(j)}$, then $b_i = 0$ for some $i\in I(j)^c$, so the absolute value of the $i$th components of the sequence $\{\psi^{(j)}\circ F^{(j)}(y_k)\}_{k\in\NN}$ diverges to $\infty$, contradicting the fact that this sequence converges to $x$.  Thus $b\in W^{(j)}$, and hence $x = \psi^{(j)}\circ F^{(j)}(y) \in \psi^{(j)}\circ(W^{(j)}\cap F^{(j)}(\cl B^{(j)}))$, which proves the claim.

Write $D^{(j)} = \psi^{(j)}\circ F^{(j)}(B^{(j)})$ for each $j\in J$.  The significance of the claim is that for each $j\in J$, the closure of $D^{(j)}$ is existentially $0$-definable in $\RR_{\Delta(\S)}$, since the left side of \eqref{eq:SparamCL} is clearly so.  Now note the following:
\begin{enumerate}{\setlength{\itemsep}{3pt}
\item
For each $p\in\{1,\ldots,P\}$ and each each $q,r\in Q_p$, there exist $q_0,\ldots,q_k\in Q_p$ such that $q_0 = q$, $q_k = r$, and for each $i\in\{0,\ldots,k-1\}$,
\begin{equation}\label{eq:compChar1}
\cl(D^{(q_i)}) \cap D^{(q_{i+1})} \neq\emptyset
\quad\text{or}\quad
D^{(q_i)} \cap \cl(D^{(q_{i+1})}) \neq\emptyset.
\end{equation}

\item
For each $p,p'\in\{1,\ldots,P\}$ with $p\neq p'$, and each $q\in Q_p$ and $q'\in Q_{p'}$,
\begin{equation}\label{eq:compChar2}
\cl(D^{(q)}) \cap D^{(q')} = \emptyset
\quad\text{and}\quad
D^{(q)} \cap \cl(D^{(q')}) = \emptyset.
\end{equation}
}\end{enumerate}
These two properties uniquely characterize the partition $\{Q_1,\ldots,Q_P\}$ of $J$, and the conditions \eqref{eq:compChar1} and \eqref{eq:compChar2} can be decided using Corollary \ref{cor:existSetMeet}, since the sets $D^{(j)}$ and $\cl(D^{(j)})$ are existentially $0$-definable in $\RR_{\Delta(\S)}$ for each $j\in J$.  We may therefore compute $\{Q_1,\ldots,Q_P\}$.
\end{proof}

\begin{lemma}\label{lemma:compl}
Let $d\in\NN$, and suppose that the number $d$ has the following property:
\begin{quote}
For every set $B\subseteq\RR^d$ which is existentially $0$-definable in $\RR_{\S}$, if we are given an existential $\L_{\S}$-formula defining $B$, then relative to the approximation and precision oracles for $\S$, we can effectively find an existential $\L_{\S}$-formula defining $\RR^d\setminus B$.
\end{quote}
Let $\lambda:\{1,\ldots,d\}\to\{1,\ldots,n\}$ be an increasing map, and suppose we are given $N\in\NN$ and an existential $\L_{\S}$-formula defining a set  $A\subseteq\RR^n$ such that $|A\cap\Pi_{\lambda}^{-1}(x)|\leq N$ for all $x\in\RR^d$.  Then relative to the approximation and precision oracles for $\S$, we can effectively find an existential $\L_{\S}$-formula defining $\RR^n\setminus A$.
\end{lemma}

This lemma of course applies to $\RR_{\Delta(\S)}$ as well, since we could replace $\S$ with $\Delta(\S)$.

\begin{proof}
See the proof of \cite[Lemma 2.5]{vdDS98}.
\end{proof}

\begin{theorem}[Effective Theorem of the Complement]\label{thm:compl}
If we are given an existential $\L_{\Delta(\S)}$-formula defining a set $A\subseteq\RR^m$, then relative to the approximation and precision oracles for $\S$, we can effectively find an existential $\L_{\Delta(\S)}$-formula defining $\RR^m\setminus A$.
\end{theorem}

\begin{proof}
We follow the proof of \cite[Theorem 2.7]{vdDS98}, except we use the notion of an ``$\S$-parameterization'' in place of the ``Gabrielov property'' of \cite{vdDS98}.  The proof is by induction on $m$.  The base case of $m = 0$ is trivial (the effective content of the base case follows from Corollary \ref{cor:existDec}).  So let $m > 0$, and assume that the Theorem holds for all existentially $0$-definable subsets of $\RR^l$ for each $l < m$.

Apply Proposition \ref{prop:SparamFC} to construct an immersive $\S$-parameterization $\{\Phi^{(j)}\}_{j\in J}$ of $A$, with associated projections $\Pi^{(j)}:\RR^m\to\RR^{d(j)}$; we use the notation of Definition \ref{def:Sparam} and Remark \ref{rmk:Sparam}.6.  Thus
\[
A = \bigcup_{j\in J} \psi^{(j)}\circ F^{(j)}(B^{(j)}),
\]
so by DeMorgan's Law, it suffices to fix $j\in J$ and construct an existential $\L_{\Delta(\S)}$-formula defining $\RR^m\setminus \psi^{(j)}\circ F^{(j)}(B^{(j)})$.  Since $j\in J$ is now fixed, we will drop the index $j$, and will simply write $F$, $\psi : \RR^I\times(\RR\setminus\{0\})^{I^c}\to\RR^m$ (for some $I\subseteq\{1,\ldots,m\}$), $\Pi:\RR^m\to\RR^d$, $B\subseteq \RR^d$, and $U\subseteq \RR^d$.  We assume that $A = \psi\circ F(B)$.  Write $W = \RR^I\times(\RR\setminus\{0\})^{I^c}$, which is both the domain and range of $\psi$. Note that $F(B)\subseteq W$ and that $\psi:W\to W$ is a bijection.   Therefore
\[
\RR^m\setminus A = (\RR^m\setminus W) \cup \psi(W\cap (\RR^m\setminus F(B))).
\]
The sets $\RR^m\setminus W$ and $W$, and the function $\psi$, are all quantifier-free definable (and one can easily write down quantifier-free $\L_{\Delta(\S)}$-formulas defining them), so it suffices to construct an existential $\L_{\Delta(\S)}$-formula defining the set $\RR^m\setminus F(B)$.  The proof now breaks into two cases. \vspace*{5pt}

\noindent\emph{Case 1}: $d < m$.

Put $G = \Pi\circ F:U\to\RR^d$.   Note that
\begin{equation}\label{eq:claimConseq}
|\Pi^{-1}(x)\cap F(B)| \leq |G^{-1}(x)\cap B|
\end{equation}
for all $x\in\RR^d$.  We claim that we can effectively find $N\in\NN$ such that
\begin{equation}\label{eq:complClaim}
|G^{-1}(x)\cap B| \leq N
\end{equation}
for all $x\in\RR^d$.  To prove the proposition in Case 1, it suffices to prove the claim, since the proposition follows immediately from the claim, \eqref{eq:claimConseq}, the induction hypothesis, and Lemma \ref{lemma:compl}.

Note that $G\Restr{B}:B\to G(B)$ is a local homeomorphism, with $G(B)$ open in $\RR^d$.  Put $D = G(\bd B)$.  The set $D$ is existentially $0$-definable in $\RR_{\Delta(\S)}$ (and hence has dimension) with $\dim D \leq \dim \bd(B) < d$.  Thus every neighborhood of every point of $D$ contains points in $G(B)\setminus D$.  Thus if $N\in\NN$ is such that $|G^{-1}(x)\cap B| \leq N$ for all $x\in G(B)\setminus D$, then \eqref{eq:complClaim} holds for all $x\in G(B)\cap D$ as well, and hence holds for all $x\in\RR^d$.  So it suffices to find such a constant $N$ for which \eqref{eq:complClaim} holds for all $x\in G(B)\setminus D$.

We now show that the map $G\Restr{B\cap G^{-1}(G(B)\setminus D)} : B\cap G^{-1}(G(B)\setminus D) \to G(B)\setminus D$ is proper.  To see this, fix a compact set $K\subseteq G(B)\setminus D$; we must prove that $B\cap G^{-1}(K)$ is compact.  The set $G^{-1}(K)$ is clearly closed, so $B\cap G^{-1}(K)$ is closed in $B$.  If  $B\cap G^{-1}(K)$ is not closed in $\RR^d$, then $G^{-1}(K) \cap \bd(B) \neq \emptyset$, and hence $K\cap D\neq\emptyset$, contradicting the fact that $K\subseteq G(B)\setminus D$.  Thus $B\cap G^{-1}(K)$ is closed in $\RR^d$.  The set $B$ is also bounded, so $B\cap G^{-1}(K)$ is compact, as desired.

We have shown that $G\Restr{B\cap G^{-1}(G(B)\setminus D)} : B\cap G^{-1}(G(B)\setminus D) \to G(B)\setminus D$ is proper, and this map is also a local homeomorphism.  Therefore $|G^{-1}(x)\cap B|$ takes on a constant finite value on each connected component of $G(B)\setminus D$.  The sets $G(B)$ and $D$ are existentially $0$-definable subsets of $\RR^d$, and $d < m$, so $G(B)\setminus D$ is existentially $0$-definable by the induction hypothesis (and this fact is effective).  By Proposition \ref{prop:CCdim}, we may find existentially $0$-definable points $c_1,\ldots,c_P$ in each of the connected components $C_1,\ldots,C_P$ of $G(B)\setminus D$, respectively.  For each $N\in\NN$ and each $p\in\{1,\ldots,P\}$, the statement ``$|G^{-1}(c_p)\cap B| > N$'' is expressible as an existential $\L_{\Delta(\S)}$-sentence,  so ``$|G^{-1}(c_p)\cap B| \leq N$'' is expressible as the negation of an existential $\L_{\Delta(\S)}$-sentence.  By choosing $N$ large enough and using Corollary \ref{cor:existDec}, we may effectively find $N\in\NN$ such that $|G^{-1}(c_p)\cap B| \leq N$ for all $p\in\{1,\ldots,P\}$, and hence \eqref{eq:complClaim} is true for all $x\in G(B)\setminus D$.  This proves the claim. \vspace*{5pt}

\noindent\emph{Case 2}: $d = m$.

We use the superscript $c$ to denote complementation in $\RR^m$.  The set $F(\bd B)$ is existentially $0$-definable and is of dimension less than $d$, so by case 1, $F(\bd B)^c$ is also existentially definable.  Note that
\[
F(B)^c = F(\cl B)^c \cup (F(\bd B) \setminus F(B)).
\]
The set $F(\bd B)\cap F(B)$ is existentially $0$-definable and of dimension less than $d$, and $F(\bd B)\setminus F(B) = F(\bd B)\setminus(F(\bd B)\cap F(B))$, so by case 1 again, $F(\bd B)\setminus F(B)$ is existentially $0$-definable.  It therefore suffices to show that $F(\cl B)^c$ is existentially $0$-definable.

Note that $F(\cl B) = F(B) \cup F(\bd B)$, so $F(\cl B)^c = F(B)^c \cap F(\bd B)^c \subseteq F(\bd B)^c$.  The set $F(\cl B)^c$ is open in $\RR^m$ and $F(B)^c$ is closed in $\RR^m$ (since $F$ is a local homeomorphism on $B$), so $F(\cl B)^c$ is both open and closed in $F(\bd B)^c$, and hence $F(\cl B)^c$ is the union of a collection of connected components of $F(\bd B)^c$.  By Proposition \ref{prop:CCdim} we may fix existential $\L_{\Delta(\S)}$-formulas for the connected components $C_1,\ldots,C_P$ of $F(\bd B)^c$. Thus for some $J\subseteq\{1,\ldots,P\}$ we have
\begin{eqnarray*}
F(\cl B)^c
    & = &
    \bigcup_{p\in J} C_p,
\\
F(\cl B)\cap F(\bd B)^c
    & = &
    \bigcup_{p\in\{1,\ldots,P\}\setminus J} C_p,
\end{eqnarray*}
since $F(\cl B)^c$ and $F(\cl B)\cap F(\bd B)^c$ are complements of each other relative to $F(\bd B)^c$.  The set $F(\cl B)\cap F(\bd B)^c$ is existentially $0$-definable in $\RR_{\Delta(\S)}$, so by using Corollary \ref{cor:existSetMeet}, we may determine which of the sets $C_p$ meet $F(\cl B)\cap F(\bd B)^c$ and which do not.  This computes $J$, and thereby constructs and existential $\L_{\Delta(\S)}$-formula defining $F(\cl B)^c$.
\end{proof}

\begin{theorem}[Effective Desingularization of Definable Sets]
\label{thm:SparamDefinable}
If we are given an $\L_{\S}$-formula defining a set $A\subseteq\RR^m$, and are given $M\subseteq\{1,\ldots,m\}$ and $R\in\QQ_{+}^{M}$ such that $\Pi_M(A) \subseteq [-R,R]$, then relative to the approximation and precision oracles for $\S$, we can effectively find a representation for an immersive, $M$-bounded $\S$-parameterization of $A$.
\end{theorem}

\begin{proof}
Given an $\L_{\S}$-formula which defines $A\subseteq\RR^n$, we can effectively find an $\L_{\S}$-formula which defines $A$ and is in prenex normal form.  Now repeatedly apply Theorem \ref{thm:compl} to find an existential $\L_{\Delta(\S)}$-formula defining $A$, and then apply Proposition \ref{prop:SparamFC}.
\end{proof}

The following is a generalization of Proposition \ref{prop:denseZeros}.2.

\begin{corollary}\label{cor:denseDef}
If $A\subseteq\RR^n$ is $0$-definable in $\RR_{\C}$, then $A\cap\C_{0}^{n}$ is dense in $A$.
\end{corollary}

\begin{proof}
It suffices to fix an open rational box $U\subseteq\RR^n$ such that $A\cap U \neq\emptyset$ and prove that $U\cap A\cap\C_{0}^{n} \neq\emptyset$.  The set $A\cap U$ is $0$-definable.  Applying Proposition \ref{thm:SparamDefinable} to $A\cap U$, and then Remark \ref{rmk:Sparam}.5, gives a point $a\in A\cap U$ with a $(\C,\emptyset)$-lifting, so $a\in\C_{0}^{n}$ since $\C$ is an IF-system.
\end{proof}

\section{Expanding Implicit Function Systems by Constants}
\label{s:IFconstants}

Recall from Notation \ref{notation:variety} that for a family $\F$ of real-valued functions on a set $A$,
\[
\VV(\F;A) = \{x\in A : \text{$f(x) = 0$ for all $f\in\F$}\}.
\]

\begin{lemma}[Topological Noetherianity]\label{lemma:topNoeth}
Let $\F$ be a family of real-valued $\C$-analytic functions on an open set $U\subseteq\RR^n$, and let $K\subseteq U$ be compact.  Then there exists a finite set $\G\subseteq\F$ such that $\VV(\F;K) = \VV(\G;K)$.
\end{lemma}

\begin{proof}
The proof is by induction on $n$.  The result is trivial when $n = 0$, so let $n > 0$ and assume the lemma holds with $n-1$ in place of $n$.  There is nothing to prove if $\F = \emptyset$ or if $\F$ contains only zero functions, so we may assume that we can fix a nonzero $f\in\F$.  By covering $K$ with finitely many compact rational boxes contained in $U$, we may assume that $K$ is a compact rational box and that $U$ is an open rational box containing $K$, and thus that $\P = (\{f\},\emptyset;K:\infty)$ is a $\C$-presentation.  Apply Theorem \ref{thm:presDesing} to $\P = (\{f\},\emptyset;K:\infty)$.  This gives a finite family $\{F^{(j)}:(U^{(j)};K^{(j)}) \to (U;K)\}_{j\in J}$ of compositions of admissible transformations such that $K \subseteq \bigcup_{j\in J} F^{(j)}(K^{(j)})$ and $\P^{(j)} = (\F^{(j)},E^{(j)};K^{(j)}: 0,\ldots)$ is resolved for each $j\in J$, where $\P^{(j)}$ is the pullback of $\P$ by $F^{(j)}:(U^{(j)};K^{(j)}) \to (U;K)$.  Put $L = \bigcup_{j\in J} F^{(j)}(K^{(j)})$; it suffices to prove the lemma with $L$ in place of $K$.  Note that
\[
\VV(\F;L) = \bigcup_{j\in J} F^{(j)}(\VV(\F^{(j)};K^{(j)})),
\]
so it suffices to fix $j\in J$ and prove the lemma for $\VV(\F^{(j)};K^{(j)})$.  Note that $\VV(\F^{(j)};K^{(j)}) \subseteq \bigcup_{i\in E^{(j)}}\VV(x_i; K^{(j)})$, so we are done by applying the induction hypothesis to the pullback of $\F^{(j)}$ by the maps $x_{\{i\}^c} \mapsto (x_{\{i\}^c},0)$ for each $i\in E^{(j)}$ (where $0\in\RR^{\{i\}}$).
\end{proof}

\begin{definition}\label{def:IFexpand}
For any subgroup $E$ of $(\RR,+)$ such that $\C_0\subseteq E$, define
\[
\C(E) = \bigcup_{m,n\in\NN \atop (r,s)\in\QQ_{+}^{m}\times\QQ_{+}^{n}}
\{f(\cdot,a) : f\in\C_{(r,s)}, a\in E^n\cap[-s,s]\},
\]
where $f(\cdot,a):[-r,r]\to\RR: x\mapsto f(x,a)$.  Write $\C_0(E)$ for the set of constants of $\C(E)$, and write $\C_r(E)$ for the functions in $\C(E)$ defined on $[-r,r]$, for each $r\in\QQ_{+}^{n}$ and integer $n>0$.
\end{definition}

Our running assumption throughout Part IV of the paper has been that $\C$ is a quasianalytic IF-system.  We drop the assumption of quasianalyticity in the following theorem.
\begin{theorem}[Extending IF-systems by constants]\label{thm:IFconstants}
Let $\C$ be a (not necessarily quasianalytic) IF-system, and let $E$ be a subgroup of $(\RR,+)$ such that $\C_0\subseteq E$.
\begin{enumerate}{\setlength{\itemsep}{3pt}
\item
The set $\C(E)$ is the smallest IF-system containing both $\C$ and $E$.

\item
If $\C$ is quasianalytic, then so is $\C(E)$.
}\end{enumerate}
\end{theorem}

\begin{proof}
We first prove 1.  It is clear that $\C(E)$ is contained in every IF-system containing $\C\cup E$, so to prove 1 it suffices to show that $\C(E)$ is an IF-system.  It is easy to see that $\C(E)$ contains the coordinate projection functions and has the extension property.

We now show that $\C_0(E)$ is a field.  For any two numbers $a,b\in\C_0(E)$, there exist $n\in\NN$, $r\in\QQ_{+}^{n}$, $f,g\in\C_r$, and $a',b'\in E^n\cap[-r,r]$ such that $a = f(a')$ and $b = g(b')$.  (The point is that the same $r$ can be chosen for both $f$ and $g$ by adding dummy variables, which is permissible since $\C$ contains the coordinate projection functions and is closed under composition.)  Therefore $\C_0(E)$ is a ring, since each $\C_r$ is a ring.  The fact that $\C_0(E)$ is closed under taking reciprocals follows from Remark \ref{rmk:IFsystem}.2, so $\C_0(E)$ is a field.

We now show that $\C$ is closed under composition.  Let $f\in\C_s(E)$ and $g\in\C_{r}^{m}(E)$, where $s\in\QQ_{+}^{m}$ and $r\in\QQ_{+}^{n}$, and assume that $g([-r,r])\subseteq[-s,s]$.  Put $h = f\circ g$.  We must show that $h$ is in $\C_r(E)$.  Fix functions $F\in\C_{(s,s')}$ and $G\in\C_{(r,r')}^{m}$ and points $a\in E^{m'}\cap[-s',s']$ and $b\in E^{n'}\cap [-r',r']$ such that $f(x) = F(x,a)$ and $g(y) = G(y,b)$, where $s'\in\QQ_{+}^{m'}$ and $r'\in\QQ_{+}^{n'}$, and where we write $(x,x')$ and $(y,y')$ for coordinates on $\RR^m\times\RR^{m'}$ and $\RR^n\times\RR^{n'}$, respectively.  Since $\C$ has the extension property, we may enlarge $s$ to assume that $G([-r,r],b) \subseteq (-s,s)$, and we may enlarge $r'$ to assume that $b\in(-r',r)$.  Fix $b'\in\C_{0}^{n'}$ and $p'\in\QQ_{+}^{n'}$ such that $b\in(b'-p',b'+p')$, $[b'-p',b'+p']\subseteq [-r',r']$, and $G([-r,r]\times[b'-p',b'+p']) \subseteq (-s,s)$.  Define $\tld{G}:[-r,r]\times[-p',p']\to\RR$ by
\[
\tld{G}(y,y') = G(y,y'+b'),
\]
and define $H:[-r,r]\times[-p',p']\times[-s',s'] \to \RR$ by
\[
H(y,y',x') = F(\tld{G}(y,y'), x').
\]
Note that $\tld{G}\in\C_{(r,p')}$ and the image of $\tld{G}$ in contained in $(-s,s)$, so $H\in\C_{(r,p',s')}$.  Finally note that $h(y) = H(y,b-b',a)$ and that $(b-b',a) \in E^{n'+m'}\cap([-p',p']\times[-s's,'])$, since $b\in E^{n'}$, $b'\in\C_{0}^{n'}\subseteq E^{n'}$, and $E$ is an additive group.
So $h\in\C_r(E)$, as desired.

We now show that $\C(E)$ is closed under division by variables.  Let $f\in\C_r(E)$ be such that $f(\tld{x},0) = 0$ on $[-\tld{r},\tld{r}]$, where $r = (r_1,\ldots,r_m)\in\QQ_{+}^{m}$ and $\tld{r} = (r_1,\ldots,r_{m-1})$, and where we write $x = (\tld{x},x_m) = (x_1,\ldots,x_m)$ for coordinates on $\RR^m$.  Let $h:[-r,r]\to\RR$ be the continuous function defined by $f(x) = x_m h(x)$.  We must show that $h\in\C_r(E)$.  Fix $F\in\C_{(r,s)}$ and $a\in E^n\cap[-s,s]$ such that $f(x) = F(x,a)$, where $s\in\QQ_{+}^{n}$.  Write $y$ for coordinates on $\RR^n$.  Define $G:[-r,r]\times[-s,s]\to\RR$ by $G(x,y) = F(x,y) - F(\tld{x},0,y)$.  Note that $f(x) = G(x,a)$ on $[-r,r]$, that $G\in\C_{(r,s)}$, and that $G(\tld{x},0,y) = 0$ on $[-\tld{r},\tld{r}]\times[-s,s]$.  Therefore $G(x,y) = x_m H(x,y)$ for some $H\in\C_{(r,s)}$.  We have $h(x) = H(x,a)$, so $h\in\C_r(E)$, as desired.

We now show that $\C(E)$ is closed under implicit functions.  Let $f\in\C_{(r,s)}(E)$ be such that $\IF(f;r,s)$ holds, where $r\in\QQ_{+}^{m}$ and $s\in\QQ_{+}$; write $(x,y) = (x_1,\ldots,x_m,y)$ for coordinates on $\RR^{m+1}$.  This means that there exists $\sigma\in\{-1,1\}$ such that $\sigma f(x,s) > 0$ and $\sigma f(x,-s) < 0$ on $[-r,r]$, and that $\sigma \det\PD{}{f}{y} > 0$ on $[-r,r]\times[-s,s]$.  Fix $F\in\C_{(r,s,t)}$ and $a\in E^n\cap[-t,t]$ such that $f(x,y) = F(x,y,a)$, where $t\in\QQ_{+}^{n}$.  By performing similar manipulations to what was done in the proof of closure under compositions, we may modify $F$ and $a$ to assume that $\sigma F(x,s,z) > 0$ and $\sigma F(x,-s,z) < 0$ on $[-r,r]\times[-t,t]$, and that $\sigma \det \PD{}{F}{y} > 0$ on $[-r,r]\times[-s,s]\times[-t,t]$; namely, that $\IF(F;(r,t),s)$ holds.  So there exists a unique function $G:[-r,r]\times[-t,t]\to(-s,s)$ such that $F(x,G(x,z),z) = 0$ on $[-r,r]\times[-t,t]$, and we have $G\in\C_{(r,s)}$.  Note that $g(x) = G(x,a)$, so $g\in\C_r(E)$, as desired.  This completes the proof of 1.

We now prove 2.  Assume that $\C$ is quasianalytic.  Let $f\in\C_r(E)$ and $a\in[-r,r]$ be such that $\PDn{\alpha}{f}{x}(a) = 0$ for all $\alpha\in\NN^m$, where $r\in\QQ_{+}^{m}$.  We must show that $f = 0$.  Fix $F\in\C_{(r,s)}$ and $b\in E^n\cap[-s,s]$ such that $f(x) = F(x,a)$, where $s\in\QQ_{+}^{n}$.  Let $\F = \{\PDn{\alpha}{F}{x} : \alpha\in\NN^m\}$, and write $\VV(\F) := \VV(\F;[-r,r]\times[-s,s])$.  By Lemma \ref{lemma:topNoeth}, there exists a finite $\G\subseteq\F$ such that
$\VV(\F) = \VV(\G)$, and by Proposition \ref{prop:denseZeros}, $\VV(\G)\cap\C_{0}^{m+n}$ is dense in $\VV(\G)$.  Since $(a,b)\in\VV(\F)$, we may fix a sequence $\{(a_i,b_i)\}_{i\in\NN}\subseteq \VV(\F)\cap \C_{0}^{m+n}$ which converges to $(a,b)$.  Each function $x\mapsto F(x,b_i)$ is in $\C_r$ and has a zero Taylor series at $a_i$, so by the quasianalyticity of $\C$, $F(x,b_i) = 0$ for all $i\in\NN$ and all $x\in[-r,r]$.  Thus for all $x\in[-r,r]$, $f(x) = F(x,b) = \lim_{i\to\infty} F(x,b_i) = 0$.  This proves 2.
\end{proof}

\section{Model-Theoretic Consequences of Desingularization}\label{s:MT}

This section proves the five theorems from the Introduction.  The following is a restatement of Theorem \ref{introThm:main}.

\begin{theorem}[Characterization of Decidability]\label{thm:charDec}
The theory of $\RR_{\S}$ is decidable if and only if the approximation and precision oracles for $\S$ are decidable.
\end{theorem}

\begin{proof}
This follows from Theorem \ref{thm:SparamDefinable} in the same way that Corollary \ref{cor:existDec} follows from Proposition \ref{prop:SparamFC}.
\end{proof}

Various application of Theorem \ref{thm:charDec} to decidability are given in \cite{DJMconstrDec}. For the remaining of the paper, we set all effectivity issues aside and use the tools developed to prove Theorem \ref{thm:charDec} to deduce theorems about the model theory of the structures $\RR_{\S}$ and $\RR_{\C}$.

Due to the length of its statement, here we will only prove Theorem \ref{introThm:param} but do not restate it.

\begin{proof}[Proof of Theorem \ref{introThm:param}]
This follows immediately from Theorem \ref{thm:SparamDefinable} and Remark \ref{rmk:Sparam}.6.
\end{proof}

The following is Theorem \ref{introThm:MC}.

\begin{theorem}\label{thm:MC}
The structure $\RR_{\Delta(\S)}$ is model complete.
\end{theorem}

\begin{proof}
This follows immediately from Theorem \ref{thm:compl}.  (It also follows immediately from the last sentence in the statement of Theorem \ref{introThm:param}.)
\end{proof}

\begin{lemma}[Curve Selection]\label{lemma:curveSelection}
Let $A\subseteq\RR^m$ be definable in $\RR_{\S}$, and let $a\in\cl(A)$.  Then there exists a definable $\C$-analytic function $g:[0,1]\to\RR^m$ such that $g(0) = a$ and $g((0,1])\subseteq A$.
\end{lemma}

\begin{proof}
Fix a definable set $A$ and $a\in\cl(A)$.  By using Theorem \ref{thm:IFconstants} to replace $\C$ with $\C(\RR)$, we may assume that $\RR\subseteq\C$.  By intersecting $A$ with a box which is a neighborhood of $a$, we may assume that $A$ is bounded.  Applying Theorem \ref{introThm:param} to $\RR_{\S\cup\RR}$ gives a finite family of immersions $\{F^{(j)}:U^{(j)}\to\RR^m\}_{j\in J}$ and bounded open rational boxes $\{B^{(j)}\}_{j\in J}$ such that
\begin{equation}\label{eq:SparamSimple}
A = \bigcup_{j\in J} F^{(j)}(B^{(j)}),
\end{equation}
where for each $j\in J$, the set $U^{(j)}$ is open in $\RR^{d(j)}$ for some $d(j)\leq m$, the closure of $B^{(j)}$ is contained in $U^{(j)}$, the map $F^{(j)}$ is $0$-definable in $\RR_{\S}$, and $F^{(j)}\Restr{B^{(j)}}$ is an immersion.  Fix $j\in J$ such that $a\in\cl(F^{(j)}(B^{(j)}))$.  Since $B^{(j)}$ is bounded, $\cl(F^{(j)}(B^{(j)})) = F^{(j)}(\cl B^{(j)})$, so we may fix $b\in \cl(B^{(j)})$ such that $a = F^{(j)}(b)$.  Choose $c\in B^{(j)}$, and define $h:[0,1]\to\RR^m$ by $h^{(j)}(t) = (1-t)b + tc$.  Then $h(0) = b$ and $h((0,1])\subseteq B^{(j)}$.  The desired function $g:[0,1]\to A$ is given by $g = F^{(j)}\circ h$.
\end{proof}

\begin{definition}
If $\R$ is an expansion of the real field, the {\bf field of definable exponents} of $\R$ is the set of all $r\in\RR$ for which the power function $(0,\infty):t\mapsto t^r$ is definable in $\R$.  (It is easy to show that this set is indeed a field.)
\end{definition}

The following is a restatement of Theorem \ref{introThm:OminPolyUnif}.

\begin{theorem}\label{thm:OminPolyUnif}
The structure $\RR_{\S}$ is a polynomially bounded o-minimal expansion of the real field with $C^\infty$-uniformization, and the field of definable exponents of $\RR_{\S}$ is $\QQ$.
\end{theorem}

\begin{proof}
By using Theorem \ref{thm:IFconstants} to replace $\C$ with $\C(\RR)$, we may assume that $\RR\subseteq\C$.  Applying Theorem \ref{introThm:param} to the structure $\RR_{\S\cup\RR}$ shows that every definable set of $\RR_{\S}$ has finitely connected components, so $\RR_{\S}$ is o-minimal.\footnote{Alternatively, one could apply Theorem \ref{introThm:param} to $\RR_{\S}$ to see that the $0$-definable sets of $\RR_{\S}$ have finitely many connected components, and then invoke \cite[Proposition 1]{CMSpeissOpenCore} to see that $\RR_{\S}$ is o-minimal.}

We now show that $\RR_{\S}$ is polynomially bounded and that $\QQ$ is the field of definable exponents of $\RR_{\S}$.  The later implies the former, and to show the later it suffices to fix a map $f:(0,\epsilon)\to\RR$ definable in $\RR_{\S}$, for some $\epsilon > 0$, and to prove that $\lim_{t\to 0} f(t)/t^r = c$ for some $r\in\QQ$ and nonzero $c\in\RR$.  By o-minimality, $L = \lim_{t\to 0}f(t)$ exists in $\RR\cup\{-\infty,+\infty\}$.  If $L\in\RR\setminus\{0\}$, we may take $r = 0$, and if $L = \pm\infty$, we may replace $f$ with $1/f$.  So we may assume that $L = 0$.
By Lemma \ref{lemma:curveSelection} we may fix a $\C$-analytic map $g = (g_1,g_2):[0,1]\to\RR^2$ such that $g_1(0) = g_2(0) = 0$ and $g((0,1])\subseteq \Graph(f)$.  Thus $g_1(t) = (t u(t))^d$ and $g_2(t) = (t v(t))^e$ for some positive integers $d$ and $e$ and some $\C$-analytic functions $u,v:[0,1]\to\RR$ such that $u(0)\neq 0$ and $v(0)\neq 0$.  We have
\[
f((t u(t))^d) = (t v(t))^e
\]
for all $t\in(0,1]$.  Put $s = t\, u(t)$, and note that by the implicit function theorem, $t = s\, w(s)$ for all $s$ in some interval $[0,\delta)$, where $\delta > 0$ and $w$ is $\C$-analytic with $w(0)\neq 0$.  Thus
\[
f(s^d) = (s\, w(s) v(s\, w(s)))^e.
\]
So by letting $x = s^d$, we get
\[
f(x) = x^{e/d} \left(w(x^{1/d})  v(x^{1/d} w(x^{1/d}))\right)^e,
\]
so $\lim_{x\to 0} f(x)/x^{e/d} = \left(w(0) v(0)\right)^e \in\RR\setminus\{0\}$.

Finally, we show that $\RR_{\S}$ has $C^\infty$-uniformization.  Let $A\subseteq\RR^m$ be compact and $0$-definable in $\RR^{\S}$.  Consider the representation of $A$ given in \eqref{eq:SparamSimple} in the proof of Lemma \ref{lemma:curveSelection}.  Since $A$ is closed, we have
\[
A = \bigcup_{j\in J}F^{(j)}(\cl(B^{(j)})).
\]
For each $j\in J$, let $M^{(j)}$ be a torus which projects naturally onto $\cl(B^{(j)})$; for instance, if we write $\cl(B^{(j)}) = \prod_{i=1}^{d(j)}[a_i-r_i,a_i+r_i]$ for some $a\in\QQ^{d(j)}$ and $r\in\QQ_{+}^{d(j)}$, then we could let
\[
M^{(j)} = \left\{(x,y)\in\RR^{d(j)}\times\RR^{d(j)} : \bigwedge_{i=1}^{d(j)} (x_i-a_i)^2 + y_{i}^{2} = r_{i}^{2}\right\},
\]
which projects onto $B^{(j)}$ via $(x,y)\mapsto x$.  Let $M = \coprod_{j\in J} M^{(j)}$ (where $\coprod$ means disjoint union, with the understanding that we consider $M$ to be a subset of some Euclidean space).  Define $f:M\to A$ to be the composition of the projections $M^{(j)}\to \cl(B^{(j)})$ and the maps $F^{(j)}:\cl(B^{(j)})\to A$.  Note that $M$ is a compact $C^\infty$-manifold, that $\dim M = \max\{\dim(B^{(j)}) : j\in J\} = \dim A$, and that $f:M\to A$ is $0$-definable, surjective, and $C^\infty$.  Thus $\RR_{\S}$ has $C^\infty$-uniformization.
\end{proof}

The previous theorem with $\S = \C$ has the following converse, which is a more specific form of Theorem \ref{introThm:OminPolyUnifChar}.

\begin{theorem}\label{thm:strucChar}
If $\R$ is a polynomially bounded o-minimal expansion of the real field with $C^\infty$-uniformization, and if $\C$ is the IF-system given in Example \ref{ex:IFsystems}.6, then $\C$ is quasianalytic and $\R$ is definably equivalent to $\RR_{\C}$.
\end{theorem}

\begin{proof}
As already explained in Example \ref{ex:IFsystems}.6, $\C$ is a quasianalytic IF-system since $\R$ is a polynomially bounded o-minimal expansion of the real field.  Every set which is $0$-definable in $\RR_{\C}$ is $0$-definable in $\R$, so we need to prove the converse.  To show this, since $\R$ expands the real field, it suffices to show that every bounded set which is $0$-definable in $\R$ is $0$-definable in $\RR_{\C}$.  If $A$ is bounded and $0$-definable in $\R$, then $\cl(A)$ and $\fr(A)$ are compact and $0$-definable in $\R$, and $A = \cl(A)\setminus \fr(A)$.  So in fact, it suffices to prove that every compact set which is $0$-definable in $\R$ is $0$-definable in $\RR_{\S}$.

So let $A$ be compact and $0$-definable in $\R$.  Fix a surjective $C^\infty$ map $f:M\to A$ which is $0$-definable in $\R$, where $M$ is a compact $C^\infty$-manifold.  By pulling back $f$ by the inverses of finitely many charts for $M$, we may fix a finite family $\{F^{(j)} : (-s^{(j)},s^{(j)})\to A\}_{j\in J}$ of $0$-definable $C^\infty$ maps and family of compact boxes $\{[-r^{(j)},r^{(j)}]\}_{j\in J}$ such that $A = \bigcup_{j\in J}F_j((-r^{(j)},r^{(j)}))$, where each $r^{(j)},s^{(j)}\in\QQ_{+}^{d(j)}$ with $r^{(j)} < s^{(j)}$.  Each map $F^{(j)}\Restr{[-r^{(j)},r^{(j)}]}$ is in $\C$, so $A$ is $0$-definable in $\RR_{\C}$.
\end{proof}

\begin{definition}\label{def:defConst}
If $\R$ is an expansion of a field with universe $R$, we call
\[
\{a\in R : \text{$a$ is $0$-definable in $\R$}\}
\]
the {\bf field of definable constants} of $\R$.  (It is easy to see that this set is indeed a field.)
\end{definition}

\newpage
\begin{theorem}[Definable Constants]\label{thm:defConst}
\hfill
\begin{enumerate}{\setlength{\itemsep}{3pt}
\item
The field of definable constants of $\RR_{\S}$ is
\begin{equation}\label{eq:SdefConst}
\{a\in\RR : \text{$a$ has an $(\S,\emptyset)$-lifting}\}.
\end{equation}

\item
The field of definable constants of $\R_{\C}$ is $\C_0$.
}\end{enumerate}
\end{theorem}

Note that saying that $a\in\RR$ has an $(\S,\emptyset)$-lifting simply means that there exist an integer $n > 0$, an $\S$-polynomial function $f:D\to\RR^n$ with $D\subseteq\RR^n$, $b = (b_1,\ldots,b_n)\in D$, and $i\in\{1,\ldots,n\}$ such that $f(b) = 0$, $\det\PD{}{f}{x}(b)\neq 0$, and $a = b_i$.

\begin{proof}
To prove 1, note that every point of $\RR$ with an $(\S,\emptyset)$-lifting is $0$-definable in $\RR_{\S}$, and Theorem \ref{thm:SparamDefinable} implies the converse.  Statement 2 follows from 1 by taking $\S = \C$ and noting that if $a\in\RR$ has an $(\C,\emptyset)$-lifting, then $a\in\C_0$ since $\C$ is an IF-system.
\end{proof}

For a first-order structure $\R$, we say that a model $\P$ of $\Th(\R)$ is {\bf prime} if $\P$ embeds elementarily into every model of $\Th(\R)$.  If $\R$ is an o-minimal expansion of an ordered field, then $\R$ has a prime model $\P$ which is unique up to isomorphism (see Pillay and Steinhorn \cite[Theorem 5.1]{PillaySteinhorn-I}). Since o-minimal theories have definable Skolem functions, it follows that the submodel of $\R$ whose universe is the field of definable constants of $\R$ is a prime model of $\Th(\R)$.  Therefore Theorem \ref{thm:defConst} describes the universes of the prime models of $\RR_{\S}$ and $\RR_{\C}$.

The fact that $\C_0$ is the universe of the prime model of $\RR$ was first proven by the author in his thesis \cite{DJMthesis} using a much simpler variant of the RSW-construction, for this reason: he first showed that if $\C$ is a Weierstrass system (meaning that $\C$ is an IF-system of analytic functions closed under Weierstrass preparation) and if $\RR\subseteq\C$, then the Lion-Rolin preparation theorem \cite{LR97} holds for $\RR_{\C}$ (see also \cite{DJMprep}).  By using a version of Theorem \ref{thm:IFconstants} for Weierstrass systems and the fact $\RR_{\C}$ is elementarily equivalent to its submodel with universe $\C_0$, he showed that it is actually unnecessary to assume that $\RR\subseteq\C$.  Thus Theorem \ref{thm:defConst} is a good place to conclude the paper, for with it we have come full circle and arrived at the origins the paper.

\bibliographystyle{amsplain}
\bibliography{bibliotex}
\end{document}